\documentclass{article}
\usepackage{geometry}
\usepackage[utf8]{inputenc}
\usepackage[english]{babel}
\usepackage[hidelinks]{hyperref}
\usepackage{graphicx}

\usepackage{amssymb, amsmath, amsthm, mathrsfs, yhmath}
\usepackage{enumerate}

\newtheorem{theorem}{Theorem}[section]
\newtheorem{corollary}[theorem]{Corollary}
\newtheorem{lemma}[theorem]{Lemma}
\newtheorem{proposition}[theorem]{Proposition}

\theoremstyle{definition}
\newtheorem{definition}[theorem]{Definition}
\newtheorem{remark}[theorem]{Remark}

\usepackage{chngcntr}
\counterwithin{table}{section}
\usepackage{multirow}

\renewcommand{\epsilon}{\varepsilon}
\renewcommand{\theta}{\vartheta}

\usepackage{mathtools}
\usepackage{etoolbox}
\DeclareFontFamily{U}{matha}{\hyphenchar\font45}
\DeclareFontShape{U}{matha}{m}{n}{
      <5> <6> <7> <8> <9> <10> gen * matha
      <10.95> matha10 <12> <14.4> <17.28> <20.74> <24.88> matha12
      }{}
\DeclareSymbolFont{matha}{U}{matha}{m}{n}
\DeclareFontSubstitution{U}{matha}{m}{n}

\DeclareFontFamily{U}{mathx}{\hyphenchar\font45}
\DeclareFontShape{U}{mathx}{m}{n}{
      <5> <6> <7> <8> <9> <10>
      <10.95> <12> <14.4> <17.28> <20.74> <24.88>
      mathx10
      }{}
\DeclareSymbolFont{mathx}{U}{mathx}{m}{n}
\DeclareFontSubstitution{U}{mathx}{m}{n}

\DeclareMathDelimiter{\vvvert}{0}{matha}{"7E}{mathx}{"17}
\DeclarePairedDelimiterX{\normi}[1]
  {\vvvert}
  {\vvvert}
  {\ifblank{#1}{\:\cdot\:}{#1}}

\usepackage[
    backend=biber,
    citestyle=numeric,
    bibstyle=authoryear,
    dashed=false
]{biblatex}
\usepackage{csquotes}

\makeatletter
\input{numeric.bbx}
\makeatother

\DeclareFieldFormat[article,inbook,incollection,inproceedings,patent,thesis,unpublished, periodical]{title}{{#1\isdot}}
\DeclareFieldFormat[article,inbook,incollection,inproceedings,patent,thesis,unpublished, periodical]{volume}{\mkbibbold{#1}}

\renewbibmacro{in:}{}
\DeclareFieldFormat{pages}{#1}

\title{{Multidimensional Stability of Planar Travelling Waves for Stochastically Perturbed Reaction-Diffusion Systems}}
\author{M.\,van den Bosch$^{\,\rm a,}$\footnote{Corresponding author.\\Email addresses: \url{vandenboschm@math.leidenuniv.nl} and \url{hhupkes@math.leidenuniv.nl}.}\,\,, H.\,J.\,Hupkes$^{\,\rm a}$}
\date{\today}

\numberwithin{equation}{section}
\begin{document}


\maketitle

\begin{center}\small
    \textsc{
    $^{\mathrm a}$Mathematical Institute,  Leiden University,\\ P.O. Box 9512, 2300 RA Leiden, The Netherlands}
\end{center}





\

\begin{abstract}
\noindent We consider reaction-diffusion systems with   multiplicative noise on a spatial domain of dimension two or higher. 
The noise process is  white in time, coloured in space, and invariant under translations. 
Inspired by previous works on the real line, 
we establish the multidimensional stability of planar waves on a cylindrical domain on time scales that are   exponentially long with respect to the noise strength. This is achieved by means of a stochastic phase tracking mechanism that can be maintained over such long time scales.  
The corresponding mild formulation of our problem features stochastic integrals with respect to anticipating integrands, which hence cannot be understood within the well-established setting of Itô-integrals. To circumvent this problem, we exploit and extend recently developed theory concerning forward integrals.

\end{abstract}

\,



\section{Introduction}\label{sec:intro}
In this paper we set out to study the multidimensional  stability  of planar travelling wave solutions to stochastic reaction-diffusion systems of the form
\begin{equation}
    \mathrm du=[D\Delta u + f(u)]\mathrm  dt+\sigma g(u)\mathrm d
W_t^Q,\label{eq:u_intro}
\end{equation}
posed on the cylindrical spatial domain
$\mathcal D=\mathbb R\times \mathbb T^{d-1}$ for some $d\geq 2$, where $\mathbb T$ denotes a
one-dimensional torus of arbitrary fixed size. The waves of interest will be propagating in the direction of the first (unbounded) coordinate. In addition,
for $\textbf{x}\in\mathcal D$ and $t>0$ we have  $u(\mathbf{x},t)\in \mathbb R^n$ for some $n\geq 1$,
and the matrix $D$ is diagonal with strictly positive entries. 
The deterministic dynamics is driven by a ``generalised Gaussian'' noise process  which is considered  white in time and coloured in space.   The noise satisfies the formal relations
\begin{equation}
    \begin{aligned}
    \mathbb E[\mathrm dW_t^Q(\textbf{x})]&=0,\\
        \mathbb E[\mathrm dW_t^Q(\textbf{x})\mathrm dW_{t'}^Q(\textbf{x}')]&=\delta(t-t')q(\textbf{x}-\textbf{x}'),\label{eq:math_trans}
    \end{aligned}
\end{equation}
for some smooth covariance function $q$ that characterises the correlation in space. In particular, 
the noise and hence system \eqref{eq:u_intro} are translationally  invariant. The $Q$ refers to the convolution operator induced by $q.$ As currently written, the noise in \eqref{eq:u_intro} should be interpreted in the It\^o sense, but for many physical applications|especially involving external noise|it is more natural to interpret  stochastic terms in the  Stratonovich sense \cite{kloeden1992stochastic,van1981ito}. We denote such stochastic reaction-diffusion systems as 
\begin{equation}
    \mathrm du=[D\Delta u + f(u)]\mathrm  dt+\sigma g(u)\circ \mathrm d
W_t^Q,\label{eq:u_intro_Strat}
\end{equation}
and note that these can also be incorporated into our framework. This paper is based on preceding works where $d=1$; see \cite{hamster2019stability,hamster2020diag,hamster2020expstability,hamster2020}. We emphasise 
that in \cite{hamster2020} the mathematical expressions \eqref{eq:u_intro}, \eqref{eq:math_trans}, and \eqref{eq:u_intro_Strat} are linked back to the relevant equations and notation  commonly used in the  physics literature.

Our aim is to explore the influence of the multiplicative noise term for small $\sigma$ on the propagation of planar wave solutions. As is discussed in \cite{garcia2012noise} and several references therein, multiplicative noise is associated with external fluctuations. Typical scenarios are whenever a control parameter of the otherwise deterministic system fluctuates around its intended value. For example, in chemical systems|especially within  experimental setups|the reaction rates serve as control parameters that are highly sensitive to local variations in temperature, illumination conditions in photochemical systems and other environmental variables \cite{kadar1998noise,mikhailov1983stochastic,steinbock1993wave}. 
As an illustration, let us consider the effect of fluctuations in the intensity of the light source driving  photosensitive Belousov-Zhabotinskii chemical reactions. The authors in \cite{sendina1998wave} demonstrate,  both numerically and experimentally, that for an effectively one-dimensional front the velocity decreases when random spatial fluctuations of the light intensity are present. In the two-dimensional setting, however, the front becomes
distorted and curvature seems to favour an increase in the wave velocity.

One may  mathematically model the latter by stochastically perturbing, in an appropriate manner, the two-component system
\begin{equation}
    \begin{aligned}
        \partial_tu&= \Delta u+\frac1\epsilon\left(u-u^2-(\alpha v+\beta)\frac{u-\gamma}{u+\gamma}\right),\\
        \partial_tv&= \delta \Delta v+u-v, \label{eq:oregenator}
    \end{aligned}
\end{equation}
where $u$ and $v$ are dimensionless versions of the concentrations of bromous acid and the catalyst, respectively, while the parameters satisfy
$\alpha,\beta,\gamma,\delta,\epsilon > 0$.  These equations are derived in \cite{krug1990analysis} from a modified version of the  Oregenator model by performing a quasi-steady state approximation. These (minor) mod\-i\-fi\-ca\-tions are motivated by the fact that inhibiting effects were  observed after irradiating  photosensitive chemicals
with both ultraviolet and visible light \cite{vavilin1968effect}.

The control parameter $\beta$ in  system \eqref{eq:oregenator} is proportional to the applied light intensity.  
Note that existence and temporal stability of travelling waves, and other spatial patterns such as spirals, have been extensively studied  
in the light-insensitive case (i.e., $\beta = 0$); see \cite{gomez1994vulnerability,kessler1990stability,leach1993initiation,merkin2009travelling} and references therein.
%
%
%
Allowing
$\beta$ to fluctuate randomly 
by performing the formal substitution $\beta \mapsto \beta-\epsilon\sigma \frac{\partial}{\partial t}  W_t^Q $,
we obtain the stochastically perturbed system
\begin{equation}
    \begin{aligned}
        \mathrm du&=\left[ \Delta u+\frac1\epsilon\left(u-u^2-(\alpha v+\beta)\frac{u-\gamma}{u+\gamma}\right)\chi(u)\right]\mathrm dt+\sigma \frac{u-\gamma}{u+\gamma}\pi(u)\circ \mathrm dW_t^Q,\\
        \mathrm dv&=\left[\delta\Delta v+u-v\right]\mathrm dt,\label{eq:stoch-oregenator}
    \end{aligned}
\end{equation}
after 
introducing cut-off functions $\chi$ and $\pi$ that satisfy
$\chi(u)=\pi(u)=1$ for all values of $u$ that are chemically relevant.
We  exploit such cut-off functions to enforce convenient \textit{pointwise} Lipschitz properties on certain nonlinearities and to ensure
that the noise does not affect the homogeneous background states
of the front solutions. 
We refer to \cite{hamster2019stability, hamster2020} for many other examples of model systems that fit within our framework.




\paragraph{Deterministic setting} 
From now on, we shall consider the setting where \eqref{eq:u_intro} with $\sigma =0$ admits a spectrally stable travelling front or pulse in one spatial dimension. More specifically, the operator associated to the linearisation about the wave in the one-dimensional setting is assumed  to have only the translational eigenvalue at zero, with the remainder of the spectrum bounded away from the imaginary axis in the left half plane. 
Such spectrally stable waves are known to exist under quite general hypotheses \cite{gardner1982existence,hamster2019stability,kapitula2013spectral,sandstede2002stability}. A common analytical approach for establishing this spectral gap condition
is to invoke methods from geometric singular perturbation theory \cite{hale2000exact,van2008pulse}, which rely on a strict separation between the diffusive length scales. In terms of the example system \eqref{eq:oregenator} this would mean $\delta\ll 1$, which for instance is a natural assumption when modelling the propagation of waves in a silica gel where the catalyst is being immobilised \cite{krug1990analysis,sendina1998wave}.

The spectral gap allowed Kapitula \cite{kapitula1994stability,kapitula1997} 
 to use semigroup methods to prove under very mild conditions that the associated planar waves on $\mathbb R^d$ are orbitally stable for every $d\geq 2$. Our main purpose here is to use the spirit of his approach to establish similar conclusions for stochastically forced systems. 
  Earlier deterministic approaches, such as  \cite{levermore1992multidimensional}, depend  heavily  on the maximum principle and energy methods which are generally not applicable to most systems, which we therefore also choose to avoid. It is also worth pointing out  that the spectrum of the operator associated to the linearisation of a planar wave $(d\geq 2)$ is no longer   bounded away from the imaginary axis. In particular, the resulting algebraic decay of perturbations required the use of a refined decomposition in which the phase of the wave plays a crucial role.

\paragraph{Cylindrical spatial domain} 

   The first main reason to restrict to a cylindrical spatial domain is our desire to consider noise that is translationally invariant.
   In more detail, suppose  $(W_t^Q)_{t\geq 0}$ is a cylindrical $Q$-Wiener process in\footnote{In fact, the noise process is rigorously constructed in an extended space $\mathcal W_{\rm ext}\supset \mathcal W$, see \S\ref{sec:forward} and \cite{hamster2020}.} some space $\mathcal W$,  where $Q:\mathcal W\to\mathcal W$ is a linear, symmetric,  positive semi-definite operator $Q:\mathcal W\to\mathcal W$. All these assumptions on $Q$ are basically essential to make sense of it as a covariance operator. For us to be able to interpret either system \eqref{eq:u_intro} or  \eqref{eq:u_intro_Strat},
   it is important that stochastic integrals of the form
\begin{equation}
    \int_0^tg(\Phi_0)\mathrm dW_s^Q\label{eq:welldefined}
\end{equation}
are well-defined in a sense suitable for our  analysis. 
More specifically, we need to be able to interpret  $g(\Phi_0)$ as a Hilbert-Schmidt operator from $\mathcal W_Q=Q^{1/2}(\mathcal W)$ to the  Sobolev space  $H^k(\mathcal D;\mathbb R^n)$, for some $k\geq 0$.
Consider now $\mathcal W=L^2(\mathcal{D};\mathbb R^m)$ to be the space of square-integrable functions, where $m>0$  denotes the number of noise components; we have $n=2$ and $m=1$ in system \eqref{eq:stoch-oregenator}. This subsequently allows us to interpret the $n\times m$ matrix $g(\Phi_0)$  as a Nemytskii operator, i.e., we proceed via the pointwise multiplication $g(\Phi_0)[\xi](\textbf{x})=g(\Phi_0(\textbf{x}))\xi(\textbf{x})$,   for any $\xi\in \mathcal W_Q$.

If $Q$ is of trace class, i.e., Tr$(Q)=\|Q^{1/2}\|_{HS(L^2;H^k(\mathcal D;\mathbb R^n))}^2<\infty,$ then it suffices to demand
\begin{equation}  \sup_{\textbf{x}\in\mathcal{D}}|g(\Phi_0(\textbf{x}))|<\infty,
\end{equation}
which is  likely to hold in many situations, since it is supposed that  the non-linearity $g$ vanishes at the endpoints of the wave. This is also true  if we replace $\mathcal D$ by   $\mathbb R^d$. However, for our setting \eqref{eq:math_trans} in combination with the fact that our domain is unbounded, we see that
 the convolution operator $Q$, which acts as $Qv= q * v$,
cannot be of trace class  (even for $d=1$; see \cite{hamster2020}). The  computations in {\S}\ref{subsec:nl:prlm} 
show that the
well-posedness condition for \eqref{eq:welldefined} becomes
\begin{equation}
    \|g(\Phi_0)\|_{H^k(\mathcal{D};\mathbb R^{n\times m})}<\infty.\label{eq:condition}
\end{equation}
Note that for $\mathcal{D} = \mathbb{R}^d$ this condition  holds if and only if $d=1$,
since $\Phi_0$ depends only on the first spatial coordinate. 
Hence, we consider domains of the form $\mathcal D=\mathbb R\times \mathbb T^{d-1}$, where \eqref{eq:condition} does hold for  dimensions $d\geq 2.$

The second main reason is that on  $\mathbb{R}^d$, with $d \geq 2,$ one may only expect algebraic decay of perturbations \cite{kapitula1994stability}.  In addition, if the disturbance is not sufficiently ``localised'' then one need not have any decay at all \cite{BHM}. One therefore needs to carefully balance the technical requirements for the noise term with the machinery necessary to handle the slowly decaying terms. Scalar noise for instance, also known as spatially homogeneous noise (take $\mathcal W=\mathbb R^m$ and $Q$ a positive semi-definite matrix as in the previous works \cite{hamster2019stability,hamster2020diag}), is of trace class but simply reduces the problem back to the one-dimensional case and hence generates no decay in the transverse direction.

We do intend to investigate the impact of coloured noise on travelling waves evolving on $\mathbb R^d$ in the future, and the present paper can be seen as a preparatory study. For example, 
our work here can be used to extract detailed information
concerning the stochastic behaviour of the phase and the dependence on the size $|\mathbb{T}|$ of the torus.
We envision that the translational invariance 
with respect to the transverse direction
 will need to be loosened, utilising 
localised or weighted noise as  interesting alternatives.

\paragraph{Main result} We prove that a spectrally stable planar wave 
on $\mathcal D=\mathbb R\times \mathbb T^{d-1}$
survives in a suitable sense under the influence of the small multiplicative noise terms in \eqref{eq:u_intro} or \eqref{eq:u_intro_Strat}.
This is achieved by analysing perturbations of the form
\begin{equation}
    v(x,y,t)=u(x+\gamma(t),y,t)-\Phi_\sigma(x),\label{eq:perturbation}
\end{equation}
and following the spirit of the procedure developed in \cite{hamster2019stability,hamster2020}. We extend the pair $(\Phi_0,c_0)$ to a branch of so-called instantaneous stochastic waves $(\Phi_\sigma,c_\sigma)$ 
that satisfy\footnote{Throughout this paper, we  shall often use the abbreviations $L^2=L^2(\mathcal D;\mathbb R^n)$ and $H^k=H^k(\mathcal D;\mathbb R^n)$. At times we  also need to consider other domains and codomains, in which case we will always be explicit to prevent any confusion.}
\begin{equation}
    \|\Phi_\sigma-\Phi_0\|_{H^k}+|c_\sigma-c_0|=\mathcal O(\sigma^2),\label{eq:facts}
\end{equation}
which only feel stochastic forcing at onset.
\pagebreak The phase shift $\gamma(t)$ is intended
to
stochastically ``freeze'' the solution $u$ by 
enforcing the orthogonality condition
\begin{equation}
    \langle v(t),\psi_{\rm tw}\rangle_{L^2}=0,
\end{equation}
for some function $\psi_{\rm tw}$ related the adjoint of operator associated to the linearisation about the wave. 
In particular, the phase shift $\gamma(t)$ satisfies the stochastic (ordinary) differential equation
\begin{equation}
    \mathrm d\gamma = \big[c_\sigma+\mathcal O(\|u(t)-\Phi_\sigma(\cdot - \gamma(t))\|_{H^k})\big]\mathrm dt+\mathcal O(\sigma)\mathrm dW_t^Q,\label{eq:phaseshift-int}
\end{equation}
and can, loosely speaking, be seen as the
``location'' of the wave after spatially averaging over the
transverse coordinates. For a rigorous construction and more information, we refer to {\S}\ref{sec:main}. 


We are interested in 
the probabilistic behaviour of the exit-time
%
\begin{equation}
    t_{\rm st}(\eta;k)=\inf\{t\geq 0:\|v(t)\|_{H^k}^2+\int_0^te^{-\epsilon(t-s)}\|v(s)\|_{H^{k+1}}^2\mathrm ds>\eta\},
\end{equation}
which measures at what time the solution $u(t)$ deviates too much 
from the expected 
``location'' of the profile $\Phi_\sigma$.
By taking $k > d/2$ and using a standard Sobolev embedding,
the first term ensures that the pointwise size of $v(t)$ stays under control until this stopping time. The second term (with a small regularisation parameter $\epsilon > 0$) provides integrated control over higher derivatives.
This is an important bonus feature of our method and can be seen as a type of optimal regularity result. For example,  this enables us to provide $H^1$-results in dimensions $2\le d \le 4$, where no pointwise control is available. In addition, it could be exploited to consider nonlinearities $f$ that include a dependence on the first derivatives $\nabla u$.

Our  metastability result below not only improves upon the main findings presented in \cite{hamster2020expstability}, but also extends these  to higher spatial dimensions.  In particular, we may  take $\mathcal D=\mathbb R$ when $d=1$ (see Table \ref{table:1}), 
and the probability bound in \eqref{eq:bound:tst} is now in line with those found in the context of large deviations  theory \cite{da2014stochastic,sowers1992large,swiȩch2009pde,varadhan}.
In words, the theorem states that with probability exponentially close to one,
our exit-time is
exponentially long with respect to the parameter $1/\sigma$. 




%

\noindent \begin{theorem}[{Interpretation of Theorem \ref{thm:main}}] \label{thm:main_rough} Under certain technical assumptions, planar fronts and pulses in stochastically perturbed reaction-diffusion equations evolving over $\mathcal D=\mathbb R\times \mathbb T^{d-1}$ persist on exponentially long times scales. In particular, 
there exists a constant $0<\mu <1$ so that for all sufficiently small  $\sigma>0$  
there is a stochastic scalar process $\gamma(t)$ such that
\begin{equation}
    \mathbb P(t_{\rm st}(\eta;k)< T) \le 2T\exp\left(-\frac{\mu\eta}{\sigma^2}\right)\label{eq:bound:tst}
\end{equation}
holds for any integer $T \ge 2$,
any sufficiently small exit value $\eta > 0$, and any
    initial value $u(0)$  that satisfies $\|u(0)-\Phi_\sigma\|_{H^{k}} ^2<\mu\eta$.
\end{theorem}

In one spatial dimension, several distinct approaches have been developed to study the behaviour of patterns under stochastic perturbations, using various alternative definitions
for the stochastic phase 
\cite{adams2022isochronal,inglis2016general,kruger2014front,lang2016multiscale,lang2016l2,lord2012computing,stannat2013stability,stannat2014stability}.
An alternative technique that also leads to stability results on exponentially long timescales uses renormalisation to reset the phase tracking mechanism at suitable time points
\cite{eichinger2022multiscale,gnann2024solitary,kruger2014front,maclaurin2023phase}. In two dimensions, the latter approach has been used to study rotating spirals \cite{kuehn2022stochastic}  on a compact spatial domain. As we will explain in the sequel, our spatially unbounded setting introduces complications that are not covered by these results.

\paragraph{Local existence and uniqueness}  While rigorous results on the impact of stochastic forcing on deterministic patterns are still relatively scarce \cite{kuehn2020travelling},  general existence and uniqueness results for SPDEs have been well-developed over the past decades; see  \cite{chow2007stochastic,da1996ergodicity,da2014stochastic,evans2012stoch,gawarecki2010stochastic,hairer2009introduction,hairer2014theory,hytonen2023analysis,jentzen2011taylor,prevot2007concise,walsh} and the references therein. Nevertheless, 
in order to be able to talk about $H^k$-valued (local) variational solutions of \eqref{eq:u_intro}, we need to exploit developments that are quite recent.

In spatial dimension $d=1$ it has been shown|after  substitution of $u=z+\Phi_{\rm ref}$ into system \eqref{eq:u_intro} with $\Phi_{\rm ref}=\Phi_0$|that
\begin{equation}
   \mathrm dz=[\partial_x^2z+\Phi_{\rm ref}''+f(z+\Phi_{\rm ref})]\mathrm dt+\sigma g(z+\Phi_{\rm ref})\mathrm d W_t^Q\label{eq:z}
    \end{equation}  admits a global unique solution \cite{hamster2019stability,hamster2020}, under suitable conditions on the nonlinearities $f$ and $g$. In particular, it  crucially involves the 
    sign condition
    \begin{equation}
      \label{eq:int:mono:sign}
        \langle f(u) - f(v), u - v \rangle \le K |u - v|^2
    \end{equation}
    to control the super-linear terms in $f$; consider, e.g., the Allen-Cahn nonlinearity $f(u) = u-u^3$. This can be used to derive local monotonicity properties
    such as
\begin{equation}\label{eq:int:mono:prop}\begin{aligned}
    \langle f(\Phi+v_A)-f(\Phi+v_B), v_A - v_B \rangle_{L^2}\leq K\|v_A-v_B\|_{L^2}^2,\end{aligned}
\end{equation}
     which should be contrasted with the estimate\footnote{This follows directly from the proof of \cite[Lem. 2.3.1]{hamster2019stability}, with the slight modification that instead of invoking the Sobolev embedding $\|u\|_\infty\leq C\|u\|_{H^1} $, we use   the Gagliardo-Nirenberg  inequality $\|u\|_\infty\leq \|u\|_{L^2}^{1/2}\|u\|_{H^1}^{1/2}$; see \cite{liu2023best,nagy1941integralungleichungen}.} 
  \begin{equation}\begin{aligned}
    \|f(\Phi+v_A)-f(\Phi+v_B)\|_{L^2}\leq K\big(1+\|v_A\|_{L^2}\|v_A\|_{H^1}+\|v_B\|_{L^2}\|v_B\|_{H^1}\big)\|v_A-v_B\|_{L^2},\label{eq:cross}\end{aligned}
\end{equation}
    that is available for general cubic nonlinearities.
    To be more precise, this global solution fits into the classical variational framework as described in Liu and Röckner \cite{liu2010spde}, which uses the Gelfand triple $(H^1,L^2,H^{-1})$ to embed the SPDE and its solution
    and relies heavily on these monotonicity properties.
    For additive noise, these
    conditions have been weakend by the same authors to
    allow cross terms as in \eqref{eq:cross} \cite[Sec. 5.2]{liu2015stochastic}. Further refinements are also possible
    in the setting where the Gelfand triple embeddings
    are compact \cite{rockner2022well}, yet both these extensions do not apply to our spatially extended system
    with multiplicative noise.

    The key issue is that our stability arguments require pointwise control over the perturbation $v(t)$, requiring us to work in $H^k$ with $k>d/2$. Due to the fact that derivative terms such as $f(u)' = f'(u) u'$ automatically involve products, the
    monotonicity property \eqref{eq:int:mono:prop}  fails to hold in $H^k$, even if the sign condition \eqref{eq:int:mono:sign} is satisfied. Recently, a new variational framework by Agresti and Veraar \cite{agresti2022critical}, the so-called critical variational framework, has been developed which basically replaces local monotonicity assumptions by local Lipschitz conditions. Our approach here is hybrid: we first
    construct a global solution with respect to
    $(H^1,L^2,H^{-1})$ 
    and then apply
    the results in \cite{agresti2022critical}
    to conclude that these global solutions exist locally in $(H^{k+1},H^k,H^{k-1})$. For convenience, we impose global 
     Lipschitz-smoothness on the nonlinearities in a \textit{pointwise} fashion,  which can readily be obtained by  multiplying any sufficiently smooth nonlinearity with a smooth cut-off function. However, we caution the reader that this does \textit{not} carry over to the nonlinearities viewed as operators on the relevant function spaces. 
     In {\S}\ref{subsec:small:dim}   we also explain how cubic nonlinearities can be accomodated in dimensions $1\leq d\leq 4.$

\paragraph{Mild formulation}  
After performing some computations, we conclude that the perturbation $v(t)$ is a variational solution 
to 
an initial value problem of the form
\begin{equation}
    \begin{cases}
\mathrm dv(t)&=[A\big(t, v(t) \big)v(t)+F(t,v(t))]\mathrm dt+B(t,v(t))\mathrm dW^Q_t,\\
        v(0)&=v_0, \label{eq:initial}
    \end{cases}
\end{equation}
in which the unbounded operator-valued function $A$ now also carries a path-dependence that cannot be transformed away, in contrast to \cite{hamster2019stability,hamster2020diag,hamster2020expstability,hamster2020} where $d = 1$. Nevertheless, we want to 
pass to a mild formulation of the form
\begin{equation}
   \label{forward} v(t)=E(t,0,\omega)v_0+\int_0^tE(t,s,\omega)F(s,v(s))\mathrm ds+\underbrace{\int_0^tE(t,s,\omega)B(s,v(s))\mathrm dW^{Q}_s}_{\textnormal{Q: How to interpret this integral?}},
\end{equation}
to exploit the stability of the linearised flow encoded in the evolution family $E(t,s)$. Through the parameter $\omega \in \Omega$ this family depends explicitly on the
probability space $(\Omega,\mathcal F,\mathbb P)$, which should be seen as a direct consequence of the fact that $A$ depends on $v(t)$. Consequently, the integrand $E(t,s)B(s,v(s))$ is anticipating, and hence we cannot make sense of the stochastic integral in \eqref{forward} with the usual It\^o calculus. Mild formulations where $E(t,s)$ is independent of $\omega\in\Omega$ have been studied thoroughly in the survey \cite{Neerven2020MaximalEF}. In the previous papers
\cite{hamster2019stability,hamster2020diag,hamster2020expstability,hamster2020} it was possible to write
$E(t,s)=S_{\rm tw}(t-s)$, using the semigroup $S_{\rm tw}(t)$ associated to the linearisation of the travelling wave, for which general theory has been available for some time \cite{da2014stochastic,gawarecki2010stochastic, hairer2009introduction}.



The Skorokhod integral is one of the most well-studied extensions of the It\^o integral to a certain set of anticipating integrands 
\cite{skorokhod1976generalization}. However, this is known to be inconvenient 
when studying mild solutions to  parabolic equations with a random family of operators, primarily due to the introduction of a drift and the reliance
on so-called Mallivian derivatives which results in technical assumptions that cannot be easily verified; see 
\cite{leon1998stochastic} and \cite[Rem. 6.7]{van2021maximal}. 

\paragraph{Forward integrals} 
 In our setting it turns out to be appropriate to use so-called  forward integrals, 
 introduced by Russo and Vallois \cite{russo1991integrales,russo1993forward,russo1995generalized}. 
 In the operator-valued setting, 
 L\'eon and Nualart \cite{leon1998stochastic} were the first (and  one of the few) to study forward integrals
 and apply them to stochastic evolution equations.
 Their existence and uniqueness results have been obtained by linking the forward integral back to the Skorokhod integral,
 which transfers the associated inconveniences.
 Recent progress has been made through the observation
 \cite[eq. (5.4)]{pronk2015forward} that links
 these integrals to the theory of pathwise mild solutions as developed by Pronk and Veraar \cite{PRONK20143634};
 see for example \cite[Sec. 6]{van2021maximal}
 and the work of Kuehn and Neamţu \cite{KUEHN20202185} on stochastically perturbed quasilinear problems. 
 




In {\S}\ref{sec:mild_proof}   we will show that the variational solution $v(t)$
to \eqref{eq:initial}
satisfies the mild formulation
\begin{equation}  v(t)=E(t,0,\omega)v_0+\int_0^tE(t,s,\omega)F(s,v(s))\mathrm ds+\underbrace{\int_0^tE(t,s, \omega)B(s,v(s))\mathrm dW^{-}_s}_{\textnormal{A: As a forward integral!}}
\end{equation}
involving a forward integral. To appreciate
this concept, let us  consider
%
a standard  Brownian motion $\beta=(\beta(t))_{t\geq 0}$. 
Formally, we  have 
\begin{equation}\begin{aligned}
\int_0^tg(s)\mathrm d\beta(s)&=\int_0^tg(s)\frac{\mathrm d\beta(s)}{\mathrm ds}\mathrm ds
=\int_0^tg(s)\lim_{h\searrow0}\frac{\beta(s+h)-\beta(s)}{h}\mathrm ds\\&=\lim_{h\searrow0}\frac1h\int_0^tg(s)[\beta(s+h)-\beta(s)]\mathrm ds,\label{eq:intermediate:steps}
\end{aligned}
\end{equation}
which motivates the definition
\begin{equation}
    \int_0^tg(s)\mathrm d\beta^-(s):=\lim_{n\to \infty}n\int_0^tg(s)[\beta(s+1/n)-\beta(s)]\mathrm ds\label{eq:forward:1d}
\end{equation}
for the forward integral, interpreting the limit
in the ucp-topology. The intermediate steps in \eqref{eq:intermediate:steps} do not  make sense, since a Brownian motion is nowhere differentiable, yet we do have the identity $\int_0^tg(s)\mathrm d\beta(s)=\int_0^tg(s)\mathrm d\beta^-(s)$ for adapted processes $g$ \cite[Prop 1.1]{russo1995generalized}.
Also, anticipating processes $g$ exist for which the forward integral is well-defined, implying that  it is indeed a proper extension of the It\^o integral.
In the same spirit,
 Russo and Vallois  also studied so-called backward and symmetric integrals, which are extensions of the backward It\^o and the Stratonovich integral, respectively.

Maximal inequalities for stochastic convolutions (a type of Burkholder-Davis-Gundy inequalities \cite{karatzas2012brownian,book:protter}) play an important role in the series \cite{hamster2019stability,hamster2020diag,hamster2020expstability,hamster2020}. These inequalities are also essential in the current paper, now underlying both the well-posedness of our mild formulation and our stability analysis. For a  general random $C_0$-evolution family $E(t,s)=E(t,s,\omega)$ of contractions, one can use \cite[Thm. 6.4]{van2021maximal} to obtain bounds of the form
\begin{equation}\mathbb{E}\sup_{0\leq t\leq T}\left\|\int_0^t E(t, s) B(s) \mathrm d W^-_s\right\|_{H^k}^{2p}  \leq K^{2p} p^p \mathbb E \left[\int_0^t\|B(s)\|_{HS(\mathcal W_Q; H^{k})}^2\mathrm ds\right]^{p},\label{eq:optimal}\end{equation} 
which hold for any $p\geq 1$ and some $K>0$ independent of time $T$. This
recent result sharpens 
the earlier bounds obtained by 
Léon and Nualart employing a factorisation method 
\cite[Thm. 4.4]{leon1998stochastic}.

One of the contributions of this paper is that 
we derive estimates for forward integrals 
that work for certain  $C_0$-evolution families that are only eventually contractive, i.e., satisfy the bound $\|E(t,s)\|_{H^k\to H^k}\leq M$ with $M>1$. 
To achieve this, we exploit the fact that the evolution families we encounter 
can  be decomposed into a family of contractions and the semigroup associated to the 
linearisation about the wave in one spatial dimension,
which admits $H^\infty$-calculus after projecting out the translational eigenfunction.
For certain one-component systems, it is known that the linearised flow is immediately contractive in the direction orthogonal to the translational eigenfunction. This  is exploited by many authors in order to prove (nonlinear) stabilty results in stochastically perturbed systems \cite{inglis2016general,lang2016multiscale,lang2016l2,stannat2013stability,stannat2014stability}, but generally speaking it is not true or unclear whether the semigroup  associated to the linearisation about the wave   is contractive \cite{hamster2019stability,veraar2011note}. In addition, we provide novel maximal regularity estimates for these forward integrals, allowing us to control their $H^{k+1}$-norm in an integrated sense. This is a subtle task due to the delicate limiting process underlying their definition.
In any case, several of our results generalise naturally to $\mathbb R^d$ and are also applicable to other settings that involve random evolution families.

\paragraph{Expansions in $\sigma$} 
One of the main features of our freezing method
is that the fluctuations around the 
instantaneous stochastic wave $(\Phi_\sigma,c_\sigma)$
can be readily investigated by expanding our
equations for $v(t)$ and $\gamma(t)$ in powers of $\sigma$.
For example, we may follow \cite{hamster2020} to find the expansions
\begin{equation}
    \begin{aligned}
        \gamma(t)&=c_0 t+\sigma \gamma_{1}(t) + \mathcal O(\sigma^2),\qquad
        c_\sigma=c_0+\sigma^2 c_{0;2}+\mathcal O(\sigma^4),
    \end{aligned}
\end{equation}
with the explicit expressions
\begin{equation}
\gamma_1(t) = -\frac{1}{|\mathbb T|^{d-1}}\int_0^t\langle g(\Phi_0)\mathrm dW_s^Q,\psi_{\rm tw}\rangle_{L^2(\mathcal D;\mathbb R^n)}
\end{equation}
and
\begin{equation}
\begin{aligned}
c_{0;2}  &=  
-\,\frac1{2|\mathbb T|^{d-1}}\langle\Phi_0'',\psi_{\rm tw}\rangle_{L^2(\mathbb R;\mathbb R^n)}\langle g(\Phi_0)Qg(\Phi_0)^\top\psi_{\rm tw},\psi_{\rm tw}\rangle_{L^2(\mathbb R;\mathbb R^n)}\\[-.05cm]
& \qquad\qquad-\frac{1}{|\mathbb T|^{d-1}}\langle g(\Phi_0)Qg(\Phi_0)^\top\psi_{\rm tw},\psi_{\rm tw}'\rangle_{L^2(\mathbb R;\mathbb R^n)}-\langle h(\Phi_0),\psi_{\rm tw}\rangle_{L^2(\mathbb R;\mathbb R^n)}.
\end{aligned}
\end{equation}
As before, the function $h$ denotes the It\^o-Stratonovich correction term. Also note that
\begin{equation}
\begin{aligned}
    \textrm{Var}\big[\gamma_1(t)\big]&=\frac{1}{(|\mathbb T|^{d-1})^2}\mathbb E\left[\int_0^t\langle g(\Phi_0)\mathrm dW^Q_s,\psi_{\rm tw}\rangle_{L^2(\mathcal D;\mathbb R^n)}\right]^2
    \\&=\frac{1}{|\mathbb T|^{d-1}}\langle g(\Phi_0)Qg(\Phi_0)^\top \psi_{\rm tw},\psi_{\rm tw}\rangle_{L^2(\mathbb R;\mathbb R^n)}t.
\end{aligned}
\end{equation}

To illustrate the effects of the dimensions transverse to the propagation of the wave, we simply take $m=n=1$ and 
consider a convolution $Qv=q*v$  whose kernel $q$
can be factorised as follows: $q(x,y)=q_{\rm wv}(x)q_{\perp}(y)$ for $(x,y)\in\mathcal D$.
Writing $Q_{\rm wv}v = q_{\rm wv} *v$ for the one-dimensional convolution operator along the wave direction yields
\begin{equation}
 Qg(\Phi_0)^\top\psi_{\rm tw}
 = \mathfrak{q}_{\rm avg}|\mathbb T|^{d-1}   Q_{\rm wv} g(\Phi_0)\psi_{\rm tw},\qquad \mathfrak{q}_{\rm avg} = 
    \frac{1}{|\mathbb T|^{d-1}} \int_{\mathbb T^{d-1} } q_{\perp}(y) \, \mathrm d y.
\end{equation}
In addition, we define
$h(u)=\frac\mu2q(0)g'(u)g(u)$ for $\mu\in\{0,1\}$ and recall that $\mu=0$ corresponds with the It\^o interpretation \eqref{eq:u_intro} while $\mu=1$ 
encodes the Stratonovich interpretation \eqref{eq:u_intro_Strat}.
These choices  allow us to find
\begin{equation}
\begin{aligned}
    \textrm{Var}\big[\gamma_1(t)\big]&=
    \mathfrak{q}_{\rm avg}
    \langle g(\Phi_0)Q_{\rm wv} g(\Phi_0) \psi_{\rm tw},\psi_{\rm tw}\rangle_{L^2(\mathbb R;\mathbb R)}t \label{eq:gamma_1}
\end{aligned}
\end{equation}
together with
\begin{equation}
\begin{aligned}
c_{0;2} & = 
-\,
\frac{1}2\mathfrak{q}_{\rm avg}
\langle\Phi_0'',\psi_{\rm tw}\rangle_{L^2(\mathbb R;\mathbb R^n)}\langle g(\Phi_0)Q_{\rm wv}g(\Phi_0)\psi_{\rm tw},\psi_{\rm tw}\rangle_{L^2(\mathbb R;\mathbb R)}
\\
& \qquad\qquad-
\mathfrak{q}_{\rm avg}
\langle g(\Phi_0)Q_{\rm wv}g(\Phi_0)\psi_{\rm tw},\psi_{\rm tw}'\rangle_{L^2(\mathbb R;\mathbb R)}
-\frac{\mu}{2} \langle g'(\Phi_0)g(\Phi_0),\psi_{\rm tw}\rangle_{L^2(\mathbb R;\mathbb R)}.\label{eq:c0;2}
\end{aligned}
\end{equation}

To recover the one-dimensional results in \cite[Sec. 2.3]{hamster2020}, one  simply sets $q_\perp \equiv 1$ which implies $\mathfrak{q}_{\rm avg} = 1$, independent of the size of the torus $\mathbb T.$ 
Indeed, this choice  models spatially homogeneous noise in the transverse direction.
The situation changes however if the noise correlation decays in the $y$-direction; for example, take a kernel $q_\perp$ that is supported on $[-1,1]$ and  does not depend on the value $|\mathbb T|$. In this case one has $\mathfrak{q}_{\rm avg} \to 0$ as $|\mathbb T| \to \infty$. Consequently,
we see that
\begin{equation}
    \mathrm{Var} \big[ \gamma_1(t) \big] \to 0,
    \qquad 
    c_{0;2} \to 
    - \frac{\mu}{2}\langle g'(\Phi_0) g(\Phi_0),\psi_{\rm tw}\rangle_{L^2(\mathbb R;\mathbb R^n)}.
\end{equation}
The vanishing variance of $\gamma_1(t)$ can intuitively be appreciated by noting that the energy inserted by the noise can be dissipated in the transverse direction rather than causing fluctuations in the phase of the wave. This seems to suggests
that on the whole space $\mathcal D = \mathbb R^d$ one can proceed in the spirit of the deterministic approach \cite{kapitula1997} and replace the
global phase function by local phase functions.

Under appropriate parameter regimes, the examples in \cite{hamster2020} displayed a change
of sign for $c_{0;2}$ when switching between 
$\mu =0 $ and $\mu=1$ (i.e., $\mathfrak{q}_{\rm avg}= 1$). Now, in the Stratonovich setting $\mu = 1$,
it is conceivable that $c_{0;2}$ changes sign
as the parameter $|\mathbb T|$ is adjusted between 0 and infinity. This leads to a potential explanation for the dimension-dependent speed changes observed
in \cite{sendina1998wave}; recall the discussion above
in the context of \eqref{eq:oregenator}.
We intend to examine this further in a forthcoming paper, which will  include higher order expansions that are able to capture curvature-driven effects. This will guide the development of an appropriate theory for the full case $\mathcal{D} = \mathbb{R}^d$.

\paragraph{Organisation} The structure of this paper is outlined as follows. In  {\S}\ref{sec:main} we formulate our assumptions, state our main results and provide an overview of the main steps in the proof. In {\S}\ref{sec:mild} we study random evolution families, introduce the concept of the forward integral, and obtain key estimates and regularity results for the associated stochastic convolutions. In {\S}\ref{sec:nl:ests} we provide bounds for our nonlinearities, which we use
in {\S}\ref{sec:variational} to construct global and local solutions to our problem. In {\S}\ref{sec:pert} we work towards a mild representation for the perturbation $v(t)$,
for which we then provide stability estimates in {\S}\ref{sec:stability}.

\paragraph{Acknowledgements} The corresponding author wishes to thank  Mark Veraar for the very fruitful discussion in 2023.

\section{Main results}\label{sec:main}
In this section we state our findings for the existence and metastability  of  planar 
wave solutions 
to stochastic reaction-diffusion systems of the form
\begin{equation}
    \mathrm du=[D\Delta u + f(u)+\sigma^2h(u)]\mathrm dt+\sigma g(u)\mathrm d
W_t^Q,\label{eq:u}
\end{equation}
where $u(x,y,t)\in \mathbb R^n$ evolves in time $t\geq 0$ on a cylindrical domain $\mathcal D=\mathbb R\times \mathbb T^{d-1}\ni (x,y) $ with dimension $d\geq 2$, and is driven by a translationally  invariant noise process $(W^Q_t)_{t\geq 0}$. The Laplacian acts in a standard fashion on the spatial coordinates $(x,y)$ as
\begin{equation}
\Delta u = \partial_x^2u+\Delta_y u, 
\end{equation}
and the diffusion matrix $D$ is diagonal with strictly positive diagonal elements.  In particular, we will denote a derivative with respect to 
$x$ interchangeably by 
$\partial_x$ and a prime $'$. Note that throughout this work, we also allow $d=1$ with $\mathcal D=\mathbb R$.

\begin{remark}
    To improve the readability of our arguments, 
    we will assume that $D = I_n$ from this point forward,
    where $I_n$ is the $n \times n$ identity matrix. 
    Indeed, the approach used in \cite{hamster2020diag,hamster2020} to handle varying constants on the diagonal also works in the present multidimensional setting. 
\end{remark}

To set the stage, we impose several conditions on the nonlinearity $f$ and also the deterministic planar wave that travels in the $x$-direction. 
In {\S}\ref{sec:stoch}, we formulate conditions on the noise term and the It\^o-Stratonovich correction term $h$, guarantee\-ing the existence and uniqueness of solutions in various variational settings. In {\S}\ref{sec:overview}, we couple an extra SDE to the SPDE above that will serve as a phase-tracking mechanism.
This enables us to formulate and discuss our main stability results. 


\subsection{Deterministic setup}\label{sec:determ}

We start by formulating our conditions for the nonlinearity $f$, which involve an arbitrary integer $k \ge 0$ that   varies depending on the context. This parameter is associated to the degree of smoothness (in the $H^k$-sense) that we can expect our solutions to have. Our most general result that is valid in arbitrary dimensions $d \ge 1$ will require $k > d/2$ and hence provide pointwise control. In this case, we need to impose global Lipschitz conditions on $f$ and its derivatives.  We do point out that the pointwise control in our final stability result means that we can safely modify our nonlinearities outside the region of interest to enforce these conditions.

\begin{itemize}
    \item[(Hf-Lip)] We have $f\in C^{k+2}(\mathbb R^n;\mathbb R^n)$ with $f(u_-)=f(u_+)=0$  for some pair $u_\pm\in\mathbb R^n.$ In addition, there is a  constant $K_f>0$ such that 
    \begin{equation}
        |f(u_A)-f(u_B)|+\ldots+|D^{k+1}f(u_A)-D^{k+1}f(u_B)|\leq K_f|u_A-u_B|\label{eq:lipschitz}
    \end{equation} 
    holds for all  $u_A,u_B\in\mathbb R^{n}$. 
\end{itemize}

In lower dimensions $1 \le d \le 4$ it is also possible to consider solutions (and noise) with a lower degree of smoothness. In fact, we will be able to take $k=1$ and allow our nonlinearity $f$ to have cubic growth. Observe that for $2 \le d \le 4$ this can no longer be artifically imposed by modifying $f$ since we no longer have pointwise control over our $H^1$-valued solutions.
\begin{itemize}
        \item[(Hf-Cub)]We have $1 \le d \le 4$ and $f\in C^3(\mathbb R^n,\mathbb R^n)$ with $f(u_-)=f(u_+)=0$ for some pair $u_\pm \in \mathbb R^n$. In addition, there exists a constant $K_f>0$ so that the bound
        \begin{equation}
            |D^3f(u)|\leq K_f
        \end{equation} holds for all $u\in \mathbb R^n.$
    \end{itemize}

We assume the existence of a planar wave solution $u(x,y,t) = \Phi_0(x - c_0 t)$ that approaches its limits $u_\pm \in \mathbb R^n$
%
at an exponential rate; a common assumption connected to asymptotic hyperbolicity \cite{kapitula1997,sandstede2002stability} that holds in many applications. 
Together with (Hf-Lip) or (Hf-Cub), this condition implies that
$\Phi_0\in C^{k+4}(\mathbb R;\mathbb R^n)$ holds  together with
 $|\Phi_0^{(\ell)}(x)|\to 0$ exponentially fast as $|x|\to\infty$, for any $1\leq \ell\leq k+4$,
showing that $\Phi_0' \in H^{k+3}(\mathbb R;\mathbb R^n)$. 

\begin{itemize}
    \item[(HTw)] There exists a waveprofile $\Phi_0 \in C^{2}(\mathbb R;\mathbb R^n)$  and a wavespeed $c_0\in\mathbb R$ 
    that satisfy the travelling wave ODE
    \begin{equation}
    \label{eq:mr:trv:wave:ode}
    \Phi_0''+c_0\Phi_0'+f(\Phi_0)=0.
    \end{equation}
    In addition, there is a constant $K>0$ together with exponents $\nu_\pm>0$ so that the bound
    \begin{equation}
        |\Phi_0(x)-u_-|+|\Phi_0'(\xi)|\leq Ke^{-\nu_-|x|}
    \end{equation}
    holds for all $x\leq 0,$ whereas the bound
    \begin{equation}
        |\Phi_0(x)-u_+|+|\Phi_0'(\xi)|\leq Ke^{-\nu_+|x|}
    \end{equation}
    holds for all $x\geq 0.$  
\end{itemize}

Linearising  \eqref{eq:mr:trv:wave:ode}  around the travelling wave $(\Phi_0,c_0)$ leads to the linear operator
\begin{equation}
    \mathcal L_{\rm tw}:H^2(\mathbb R;\mathbb R^n)\to L^2(\mathbb R;\mathbb R^n)
\end{equation}
that acts as
\begin{equation}
    [\mathcal L_{\rm tw}u](x)=u''(x)+c_0u'(x)+Df(\Phi_0(x))u(x).
\end{equation}
In particular, observe $\mathcal L_{\rm tw}\Phi_0'=0$. Its associated adjoint operator, which we will denote by 
\begin{equation}
    \mathcal L_{\rm tw}^{\rm adj}:H^2(\mathbb R;
    \mathbb R^n)\to L^2(\mathbb R;\mathbb R^n),
\end{equation}
acts as
\begin{equation}
    [\mathcal L_{\rm tw}^{\rm adj}w](x)=w''(x)-c_0w'(x)+Df(\Phi_0(x))w(x),
\end{equation}
and it is easily verified that $\langle \mathcal L_{\rm tw}v,w\rangle_{L^2(\mathbb R;\mathbb R^n)}=\langle v,\mathcal L_{\rm tw}^{\rm adj}w\rangle_{L^2(\mathbb R;\mathbb R^n)}$ holds for $v,w\in H^2(\mathbb R;\mathbb R^n).$ 

We continue  with a standard spectral stability condition 
for the operator $\mathcal L_{\rm tw}$. In particular,
we demand that the translational eigenvalue at zero is an isolated simple eigenvalue and that the remainder of the spectrum can be strictly bounded to the left of the imaginary axis, i.e., there is a spectral gap.

\begin{itemize}
    \item[(HS)] The operator $\mathcal L_{\rm tw}:H^{2}(\mathbb R;\mathbb R^n)\to L^2(\mathbb R;\mathbb R^n)$ has a simple eigenvalue at $\lambda=0$ and there exists a constant $\beta>0$ so that the operator $\mathcal L_{\rm tw}-\lambda:H^{2}(\mathbb R;\mathbb R^n)\to L^{2}(\mathbb R;\mathbb R^n)$ is invertible for all $\lambda \in \mathbb C$ satisfying Re $\lambda\geq -2\beta$. 
\end{itemize}

Assuming (Hf-Lip) or (Hf-Cub)  together with (HTw) and (HS), 
we note that the operators $\mathcal L_{\rm tw}$ and $\mathcal L_{\rm tw}^{\rm adj}$ map  $H^{\ell+2}(\mathbb R;\mathbb R^n)$ into $H^\ell(\mathbb R;\mathbb R^n)$, for any  $0 \le \ell \le k+1 $.  
This follows either directly from inspection or general (Sobolev tower) interpolation  theory \cite{engel2000one,lunardi2018interpolation,book:Triebel}. 
Furthermore, the resolvent set of $\mathcal L_{\rm tw}$ restricted to the domain $H^{k+3}(\mathbb R;\mathbb R^n)$  contains the resolvent set of $\mathcal L_{\rm tw}$ seen as the original operator with domain $H^2(\mathbb R;\mathbb R^n).$  Consequently,   the operator $\mathcal L_{\rm tw}:H^{k+3}(\mathbb R;\mathbb R^n)\to H^{k+1}(\mathbb R;\mathbb R^n)$ is Fredholm with index zero,  we have  
    \begin{equation}
        \ker (\mathcal L_{\rm tw})=\text{span}\{\Phi_0'\}\subset H^{k+3}(\mathbb R;\mathbb R^n),\quad \ker (\mathcal L_{\rm tw}^{\rm adj})=\text{span}\{\psi_{\rm tw}\}\subset H^{k+3}(\mathbb R;\mathbb R^n),
    \end{equation}
    for some $\psi_{\rm tw}$  that satisfies the normalised identity 
    \begin{equation}
        \langle \Phi_0',
        \psi_{\rm tw}\rangle _{L^2(\mathbb R;\mathbb R^n)}=1,
    \end{equation}
 and $|\psi_{\rm tw}^{(\ell)}(x)|\to 0$ exponentially fast  as $|x|\to \infty$ for any $0\leq \ell\leq k+3$.

Since $\mathcal L_{\rm tw}$ is a lower order perturbation of the diffusion operator $\partial_x^2$, we see that 
$\mathcal L_{\rm tw}$ is sectorial\footnote{Following a definition as in \cite[Sec. 10]{hytonen2018analysis}, we have that $-\mathcal L_{\rm tw}$ is sectorial.} in $L^2(\mathbb R;\mathbb R^n)$, $H^k(\mathbb R;\mathbb R^n)$ and $H^{k+1}(\mathbb R;\mathbb R^n)$ (the latter two after restriction), and hence generates analytic semigroups
on these spaces, which we all denote by $(S_{\rm tw}(t))_{t\geq 0}$ since they agree where they overlap \cite{engel2000one,lunardi2018interpolation,book:Triebel}.
The  same conclusions hold for the  adjoint operator $\mathcal L_{\rm tw}^{\rm adj}$.


Lastly, the orthogonal projection $P_{\rm tw}:H^k(\mathbb R;\mathbb R^n)\to H^k(\mathbb R;
\mathbb R^n)$, defined by
\begin{equation}
    P_{\rm tw}v=\langle v,\psi_{\rm tw}\rangle_{L^2(\mathbb R;\mathbb R^n)}\Phi_0', \label{eq:Ptw-exp-intro}
\end{equation}
plays a crucial role for our stability analysis. It is used to project out the translational eigenfunction in order to circumvent the neutral mode of the semigroup $S_{\mathrm{tw}}(t)$. 




\subsection{Stochastic setup}\label{sec:stoch}
In this part we formulate  conditions to ensure that the stochastic terms in \eqref{eq:u} are well-defined.
We start by considering the covariance function $q$
that governs the noise process,
writing $\hat q$ for the Fourier transform of $q$ (see also Appendix \ref{appendix:Fourier}).
Recall, the integer $m$ corresponds
to the dimension of the space from which the noise will be sampled.




\begin{itemize}
    \item[(Hq)] We have  $q\in H^\ell(\mathcal D;\mathbb R^{m\times m})\cap L^1(\mathcal D;\mathbb R^{m\times m})$ for some integer $\ell >2k+d/2,$ with $q(-\mathbf{x})=q(\mathbf{x})$ and $q^\top(\mathbf{x})=q(\mathbf{x})$ for all $\mathbf{x}=(x,y)\in \mathcal D=\mathbb R\times \mathbb T^{d-1}$. Further, for  $\boldsymbol \xi =(\omega,\xi) \in \widehat D=\mathbb R\times \mathbb Z^{d-1}$ the $m\times m$ matrix $\hat q(\boldsymbol{\xi})$ is   non-negative definite.   
\end{itemize}

Thanks to Young's inequality, the integrability of the kernel $q$ allows us  to introduce the bounded convolution  operator $Q$ on $L^2(\mathcal D;\mathbb R^m)$ that acts as
\begin{equation}
    [Qv](\mathbf{x})=[q*v](\mathbf{x})=\int_{\mathcal D}q(\mathbf{x}-\mathbf{x}')v(\mathbf{x}')\mathrm d\mathbf{x}'=\int_{\mathbb R}\int_{\mathbb T^{d-1}}q(x-x',y-y')v(x',y')\mathrm dx'\mathrm dy'.
\end{equation}
The properties in (Hq) imply that  $Q$ is symmetric and that $\langle Qv,v\rangle_{L^2(\mathcal D;\mathbb R^m)}\geq 0$ holds for every $v\in L^2(\mathcal D;\mathbb R^m)$, showing that $Q$ is indeed a covariance operator. Consequently, we can follow \cite{da2014stochastic,gawarecki2010stochastic,hairer2009introduction,hamster2020,karczewska2005stochastic,prevot2007concise} and  construct a cylindrical $Q$-Wiener process $W^Q=(W^Q_t)_{t\geq 0}$ that is defined on a filtered probability space $(\Omega,\mathcal F,\mathbb F,\mathbb P)$ 
and takes values in (an extended space containing) the Hilbert space $L^2(\mathcal D;\mathbb R^m)$. For more information, we refer to {\S}\ref{sec:forward}.

In order to perform $H^k$-valued stochastic integration with respect to $W^Q$, we need to interpret the noise term $g(u)$
in \eqref{eq:u} as a Hilbert-Schmidt operator from $L^2_Q$ into $H^k(\mathcal D;\mathbb R^n)$,
in which we have introduced the Hilbert space
\begin{equation}
L_Q^2=Q^{1/2}\big(L^2(\mathcal D;\mathbb R^m)\big).
\end{equation} 
We proceed by viewing $g$ as a Nemytskii operator that
acts in a pointwise fashion as
\begin{equation}
    (g(u)[\xi])(\mathbf x)=g(u(\mathbf x))\xi(\mathbf x),\label{eq:pointwise}
\end{equation}
for any $\xi \in  L^2_Q$.  We impose the following pointwise conditions on $g$, which most notably require $g$ to vanish at the limits $u_\pm$ for the waveprofile.

\begin{itemize}
    \item[(HSt)] We have $g\in C^{k+1}(\mathbb R^{n};\mathbb R^{n\times m})$ with $g(u_-)=g(u_+)=0.$ In addition, there is a constant $K_g>0$ such that 
    \begin{equation}
        |g(u_A)-g(u_B)|+\ldots+|D^{k+1}g(u_A)-D^{k+1}g(u_B)|\leq K_g|u_A-u_B|
    \end{equation}
    holds for all  $u_A,u_B\in\mathbb R^{n}.$
\end{itemize}

The results in  {\S}\ref{subsec:nl:prlm}
use the $H^\ell$-smoothness of $q$ to show that $g(u)$
has the desired properties for all $u$ in the affine spaces
\begin{equation}
    \mathcal U_{H^k(S;\mathbb R^n)}=\Phi_{\rm ref}+H^k( S;\mathbb R^n),\quad S\in\{\mathbb R,\mathcal D\},
\end{equation}
where $\Phi_{\rm ref}$  is a sufficiently smooth reference function that has $\Phi_{\rm ref}(\pm\infty)=u_\pm$; for example $\Phi_{\rm ref}=\Phi_0$.
Indeed, in general we cannot work directly with $H^k$ since the wave $\Phi_0$ is not included in this space.
On the other hand, in the special case $u_+=u_- =0$ one can simply take $\Phi_{\rm ref}=0$.

The purpose of  the function $h$  in  \eqref{eq:u} is to allow It\^o-Stratonovich correction
terms to be incorporated into our framework. For example, in the scalar case $n = m = 1$ the choice
$h(u) = \frac12
q(0)g'(u)g(u)$ allows us to interpret \eqref{eq:u} as the Stratonovich SPDE
\begin{equation}
    du = [\Delta u + f(u)]\mathrm dt + \sigma g(u)\circ  \mathrm dW_t^Q.
\end{equation}
We refer to 
\cite{evans2012stoch,tessitore2006wong,twardowska2004relation}
for further information. In view of the example above, we require one order of smoothness less on $h$ as compared to $g$. 


\begin{itemize}
    \item[(HCor)] We have $h\in C^{k}(\mathbb R^n;\mathbb R^n)$ with    $h(u_-)=h(u_+)=0$. In addition, there is a   constant $K_h>0$  such that
    \begin{equation}
        |h(u_A)-h(u_B)|+\ldots+|D^{k}h(u_A)-D^{k}h(u_B)|\leq K_h|u_A-u_B|
    \end{equation} 
    holds for all  $u_A,u_B\in\mathbb R^{n}$. 
\end{itemize}

\subsection{Overview}\label{sec:overview}
Our goal here is to provide a
comprehensive overview of the key steps and 
intermediate results that lead up to Theorem \ref{thm:main}.
In particular, it will become clear in what sense our solutions are defined, how our phase-tracking mechanism is implemented, and which expectation estimates tend towards the main (meta)stability estimate.
Our results cover different combinations of the parameters $(k,d)$, as summarised in Table \ref{table:1}.


\begin{table}[!ht]
\centering
\begin{tabular}{|c|rcl|rcl|c|c|} 
 \hline
 non-linearity $f$&\multicolumn{3}{|c}{smoothness $k$} & \multicolumn{3}{|c|}{dimension $d$} &existence/uniqueness&(meta)stability \\  
 \hline\hline
 &&&&&&&&\\[-.3cm]
 \multirow{2}{*}{(Hf-Cub)}&$k$&{$=$}&{$0$} & {$d$}&{$=$}&{$1$} &{\,\,global$^*$}& {\,\,\, yes$^{**}$}\\ 
&$k$&{$=$}&{$1$} & {$d$}&{$\leq$}&{$4$} &{global}& {yes}\\[.1cm] 
 \hline\hline
 &&&&&&&&\\[-.3cm]
 \multirow{5}{*}{(Hf-Lip)}&\,\,\,\,$k$&=&$0$ & \,\,$d$&$=$&$1$ &global& {\,\,\, yes$^{**}$} \\ 
&$k$&=&$0$ & $d$&$\geq$&$2$ &global& no \\ 
  &$k$&=&$1$ & $d$&$\le$&$4$ &global& yes \\ 
 &$k$&$>$&$1$ & $d$&$=$&$1$ &local& yes \\
 &$k$&$>$&$\tfrac d2$ & $d$&$\geq$&$ 2$ &local&yes\\[.075cm] 
 \hline
\end{tabular}
\caption{List of compatible combinations of the smoothness parameter $k$ and  dimension $d$, with (meta)stability on exponentially long time scales. $(*)$: The case (Hf-Cub) with $k=0$ and $d=1$ is treated in \cite{hamster2020expstability,hamster2020} and requires  the one-sided inequality \eqref{eq:int:mono:sign} as an additional assumption. $(**)$: (meta)stability for the case $k=0$ and $d=1$ is discussed in Remark \ref{rem:st:sharpening:d1}. In particular, the techniques in this paper can be used to strengthen the bound obtained in \cite{hamster2020expstability}.}
\label{table:1}
\end{table}
%
\paragraph{Step 1: Phase-tracking mechanism} 
Let us continue the discussion from the introduction.
We couple the SDE
\begin{equation}
   \mathrm d\gamma = [c+a_\sigma(u,\gamma;c)]\mathrm dt+\sigma b(u,\gamma) \mathrm d  W_t^Q\label{eq:gamma}
\end{equation}
to the  SPDE \eqref{eq:u}.  
The two  functions \begin{equation}a_\sigma:\mathcal U_{H^k(\mathcal D;\mathbb R^n)}\times \mathbb R\to \mathbb R\quad\text{and}\quad b:\mathcal U_{H^k(\mathcal D;\mathbb R^n)}\times \mathbb R\to HS(L^2_Q;\mathbb R)\end{equation}
are defined  in Appendix \ref{list},
based on technical considerations to ensure
that our perturbation does not feel the neutral
translational mode of the linear flow; 
see Step 3 for further information. The velocity
parameter $c$ is discussed further in Step 4.




\paragraph{Step 2: Existence, uniqueness, and regularity} Our first result will clarify what we mean by a solution to system \eqref{eq:u}. In particular, under the assumption (Hf-Lip) we will
obtain continuous $L^2$-valued solutions 
that are unique up to indistinguishability \cite{book:protter}.
As a matter of fact,  there exists a version of $z(t)$, let us say $\tilde z(t),$ which satisfies (i)-(iii) and for which the map $\tilde z:(0,T]\times \Omega \mapsto H^1(\mathcal D;\mathbb R^n)$ is progressively measurable \cite{agresti2022critical,liu2015stochastic,prevot2007concise}. However $\tilde z(t)$ is  not continuous in $H^1(\mathcal D;\mathbb R^n)$. We refer to {\S}\ref{sec:variational} for the precise interpretation of the diffusion operator in the integral equation \eqref{eq:equality}, which should be seen in the context of Gelfand triples \cite[Sec. 5.9]{evans2022partial}.





\begin{proposition}[{see {\S}\ref{sec:variational}}]\label{prop:existLip} Fix $k=0$ and assume that 
\textnormal{(Hf-Lip)}, \textnormal{(HSt)}, \textnormal{(HCor)},
\textnormal{(HTw)} and \textnormal{(Hq)} are satisfied. 
Pick $T>0$ together with $0\leq \sigma\leq 1$. Then for any initial condition
%
\begin{equation}
    (z_0,\gamma_0)\in L^2(\mathcal D;\mathbb R^n)\times \mathbb R,
\end{equation}
there are progressively measurable maps 
\begin{equation}
    z:[0,T]\times \Omega\to L^2(\mathcal D;\mathbb R^n),\quad \gamma:[0,T]\times \Omega\to\mathbb R,
\end{equation}
that satisfy the following properties:
\begin{enumerate}[\rm (i)]
    \item For almost every $\omega\in\Omega$, the map
    $ 
        t\mapsto (z(t,\omega),\gamma(t,\omega))
    $
    is of class $C([0,T];L^2(\mathcal D;\mathbb R^n)\times \mathbb R)$;
   \item We have the integrability condition 
   $(z,\gamma)\in L^2([0,T]\times\Omega,\mathrm dt\times\mathbb P;H^1(\mathcal D;\mathbb R^n)\times\mathbb R)$;
   \item The $L^2(\mathcal D;\mathbb R^n)$-valued identity\footnote{At first, the equality in \eqref{eq:equality} should be understood as an equality in the dual space of $H^1$, thus  in the analytically weak sense, but by (ii) we can conclude that we have equality in $L^2$. See also \cite{hamster2019stability}, \cite[Rem. 4.2.2]{liu2015stochastic} and {\S}\ref{sec:variational} for more information.
   }
   \begin{equation}
       \begin{aligned}
           z(t)=z_0+\int_0^t [\Delta z(s)+\Phi_{\rm ref}'']\mathrm ds&+\int_0^tf(z(s)+\Phi_{\rm ref})\mathrm ds\\&+\sigma^2\int_0^th(z(s)+\Phi_{\rm ref})\mathrm ds+\sigma \int_0^tg(z(s)+\Phi_{\rm ref})\mathrm dW_s^Q,
       \end{aligned}\label{eq:equality}
   \end{equation}
   together with the scalar identity
   \begin{equation}
\gamma(t)=\gamma_0+\int_0^t[c+a_\sigma(z(s)+\Phi_{\rm ref},\gamma(s);c)]\mathrm ds+\sigma\int_0^t b(z(s)+\Phi_{\rm ref},\gamma(s))\mathrm dW_s^Q,
   \end{equation}
   hold $\mathbb P$-a.s. for all $0\leq t\leq T$;
   \item Suppose there are another progressively measurable maps $\tilde z$ and $\tilde \gamma$ that satisfy \textnormal{(i)--(iii)}. Then for almost all $\omega\in\Omega,$ we have
   \begin{equation}
       \tilde z(t,\omega)=z(t,\omega)\quad\text{and}\quad \tilde \gamma(t,\omega)=\gamma(t,\omega),\quad \text{for all } 0\leq t\leq T.
   \end{equation}
\end{enumerate}
Upon fixing $k=1$ and assuming  \textnormal{(Hf-Cub)}  instead of \textnormal{(Hf-Lip)}, the same results hold after replacing the pair $(L^2(\mathcal D;\mathbb R^n),H^1(\mathcal D;\mathbb R^n) )$ by $(H^1(\mathcal D;\mathbb R^n),H^2(\mathcal D;\mathbb R^n) )$.
\end{proposition} 

In {\S}\ref{sec:regularity} we investigate the regularity of the solution described in Proposition \ref{prop:existLip}. We observe that the solution $z(t)$ with a smooth initial condition $z_0\in H^k(\mathcal D;\mathbb R^n)$ remains $H^k$-smooth on (at least) a short realisation-dependent time interval $[0,\tau(\omega))$, but could lose smoothness over time. 
An interesting corollary of our stability result in Step 5 is that, with very high probability, solutions remain smooth for a very long time.


\paragraph{Step 3: Evolution of the perturbation} 
For the purposes of this step, we may choose any pair $(\Phi,c)$ that is sufficiently close to $(\Phi_0,c_0)$
as the basis for our perturbation argument. In particular, we impose the following condition
and refer to Appendix \ref{list} for more information.

\begin{itemize}
    \item[(HPar)]  The conditions (HTw) and (HS) hold and the pair $(\Phi,c)\in \mathcal U_{H^{k+2}(\mathbb R;\mathbb R^n)}\times \mathbb R$ satisfies 
    \begin{equation}
        \|\Phi-\Phi_0\|_{H^{k+2}(\mathbb R;\mathbb R^n)}\leq 
        \min \{ 1,  |\mathbb T|^{\frac{d-1}{2}} [4 \|\psi\|_{H^1(\mathbb R;\mathbb R^n)}]^{-1} \}
        ,\quad |c-c_0|\leq 1.
    \end{equation}
    \vspace{-.3cm}
\end{itemize}

We now decompose $u(t)$ as
\begin{equation}u(t)=T_{\gamma(t)}\Phi+T_{\gamma(t)} v(t),\end{equation}
or equivalently,  define the perturbation $v(t)$ by writing
\begin{equation}
    v(t)=T_{-\gamma(t)}u(t)-\Phi,
\end{equation}
where $T_\delta,$ for any $\delta\in\mathbb R$, is the right-shift operator that act  on the first coordinate only, i.e.,
\begin{equation}
    T_\delta u(x,y,t)=u(x-\delta,y,t),\quad \delta\in\mathbb R.
\end{equation}
The computations in {\S}\ref{sec:pert_SPDE} exploit the translational invariance of our noise to show that $v(t)$ satisfies the  system
\begin{equation}\label{eq:v_orig}
    \mathrm dv=\mathcal R_\sigma(v;c,\Phi)\mathrm dt+\sigma \mathcal S(v;\Phi)\mathrm dW_t^Q,
\end{equation}
with the deterministic part being of the form
\begin{equation}
\label{eq:mr:r:sigma:form}
    \mathcal R_\sigma(v;c,\Phi)=
    \Delta_y v +
    \kappa_\sigma(\Phi +v , 0) 
    \big[ \partial_x^2 v + \Phi'' + \mathcal{J}_\sigma( \Phi + v, 0; c)
    \big] 
    + a_\sigma(\Phi+v,0;c)\partial_x(\Phi+v),
\end{equation}
while the stochastic part reads
\begin{equation}\mathcal S(v;\Phi)[\xi]=g(\Phi+v)[\xi]+\partial_x(\Phi+v)  b(\Phi+v,0)[\xi],\quad \xi\in L^2_Q.\end{equation}
The exact definitions of the scalar function $\kappa_\sigma$ and the nonlinearity $\mathcal{J}_\sigma$
are provided in Appendix \ref{list}, yet the main take away is that these terms are (more or less) the same as in the $d=1$ case \cite{hamster2020} and do not involve second order derivatives. In addition, we have $\mathcal{R}_0(0, c_0, \Phi_0)= 0$.

Our goal is to ensure that the orthogonality condition
\begin{equation}\langle v(t),\psi_{\rm tw}\rangle_{L^2(\mathcal D;\mathbb R^n)}=0\label{eq:orthog}\end{equation}
is satisfied as long as possible, because
this causes the perturbation to not feel  any of the dangerous neutral modes of our higher spatially dimensional system (see {\S}\ref{sec:forward}).
In order to achieve this, we set out to enforce the  conditions
 \begin{equation}\langle {\mathcal R_\sigma(v;c,\Phi)},\psi_{\rm tw}\rangle_{L^2(\mathcal D;\mathbb R^n)}=0,\label{eq:cond_R}\end{equation}
 together with
 \begin{equation}\label{eq:cond_S}\langle {\mathcal S(v;\Phi)[\xi]},\psi_{\rm tw}\rangle_{L^2(\mathcal D;\mathbb R^n)}=0,\quad \xi\in L_Q^2,\end{equation}
 which is possible whenever $\|v\|_{L^2(\mathcal D;\mathbb R^n)}$ is sufficiently small 
 and  by choosing the functions $a_\sigma$ and $b$
 appropriately. 
 Indeed, 
 condition \eqref{eq:cond_R} is satisfied whenever
\begin{equation}
\begin{aligned}
a_\sigma(\Phi+v,0;c)=-\kappa_\sigma(\Phi +v , 0) 
 \frac{\langle \partial_{xx} v  + \Phi'' + \mathcal J_\sigma(\Phi+v,0;c),\psi_{\rm tw}\rangle_{L^2(\mathcal D;\mathbb R^n)}}{\langle \partial_x(\Phi+v),\psi_{\rm tw}\rangle_{L^2(\mathcal D;\mathbb R^n)}},
 \end{aligned}
\end{equation}
     which features only  derivatives with respect to the coordinate in the direction of the wave. This is due to the fact that \begin{equation}\langle \Delta_y v,\psi_{\rm tw}\rangle_{L^2(\mathcal D;\mathbb R^n)}=\int_{\mathbb R}\psi_{\rm tw} \int_{\mathbb T^{d-1}}\nabla_y\cdot(\nabla_yv)\, \mathrm dy\,\mathrm dx=0,\label{eq:Gauss}\end{equation}
  as a consequence of Gauss' divergence theorem, or more simply due to the fact that  $\nabla_y v$ is periodic and because the integral of the derivative of a periodic function vanishes. 
%
Likewise, 
condition \eqref{eq:cond_S} is satisfied whenever
\begin{equation}
b(\Phi+v,0)[\xi]=-\frac{\langle g(\Phi+v)\xi,\psi_{\rm tw}\rangle_{L^2(\mathcal D;\mathbb R^n)}}{\langle \partial_x(\Phi+v),\psi_{\rm tw}\rangle_{L^2(\mathcal D;\mathbb R^n)}},\quad \xi\in L^2_Q.
\end{equation}
This briefly motivates the expressions for $a_\sigma$ and $b$  found in Appendix \ref{list}, which require less regularity on $v$
and also contain cut-off functions to account for the cases where $v$ is not sufficiently small.

\paragraph{Step 4: Instantaneous stochastic waves} 
In this step we use our remaining degrees of freedom
to pick the initial phase $\gamma_0$ from Step 2
and the profile $(\Phi,c)$ appearing in Step 3.
To this end,
we construct a branch of profiles and speeds
$(\Phi_\sigma,c_\sigma)$ that turns out to be highly convenient for our computations and Taylor expansions. 
In particular, we will demand
\begin{equation}\label{eq:K=0}
\Phi_{\sigma}'' + \mathcal{J}_\sigma(\Phi_\sigma,0;c_\sigma) =0,
\end{equation}
which by construction implies that
\begin{equation}
     \mathcal R_\sigma(0;c_\sigma,\Phi_\sigma)=0\quad\text{and}\quad a_\sigma(\Phi_\sigma,0;c_\sigma)=0.\label{eq:consequence}
 \end{equation}
In particular, when $u(t)=T_{\gamma(t)}\Phi_\sigma$ holds
for some $t$, the phase $\gamma(t)-c_\sigma t$ 
and perturbation $v(t)$ will instantaneously only
feel stochastic forcing, motivating their name. 
Referring to Appendix \ref{list}, we note that
$\mathcal{J}_0(\Phi,0;c) =  c \Phi' + f(\Phi)$,
which in view of \eqref{eq:mr:trv:wave:ode} implies that 
the branch $(\Phi_\sigma,c_\sigma)$ reduces naturally
to the deterministic pair $(\Phi_0, c_0)$ at $\sigma = 0$.



\begin{proposition}\label{prop:wave}
Suppose that either \textnormal{(Hf-Lip)} is satisfied with $k \ge 0$ or 
that $k=1$ and \textnormal{(Hf-Cub)} is satisfied. Assume furthermore that
\textnormal{(Hq)}, \textnormal{(HSt)}, \textnormal{(HCor)}, \textnormal{(HTw)} and \textnormal{(HS)}  hold and pick a suffi\-ciently large $K > 0$.
%
Then there exists a $\delta_\sigma>0$ so that for every $0\leq \sigma\leq \delta_\sigma$ there is a unique pair  \begin{equation}
    (\Phi_\sigma,c_\sigma)\in \mathcal U_{H^{k+2}(\mathbb R;\mathbb R^n)}\times \mathbb R \end{equation}
that satisfies  
\begin{equation}\label{eq:K=0:prop}
\Phi_\sigma'' + \mathcal J_\sigma(\Phi_\sigma,0;c_\sigma)=0
\end{equation}
and  admits the bound
\begin{equation}
    \|\Phi_\sigma-\Phi_0\|_{H^{k+2}(\mathbb R;\mathbb R^n)}+|c_\sigma-c_0|\leq K \sigma^2.
\end{equation}
\end{proposition}
\begin{proof}
Note that this is a problem in one spatial dimension. In particular,
the result readily follows from a similar fixed point argument as  in the proof of \cite[Prop. 2.2.2]{hamster2019stability}; see also \cite[Sec. 4.5]{hamster2020}. 
\end{proof}



We are interested in solutions with an initial condition $u(0)=u_0$ close to $\Phi_\sigma.$ The result below shows that the initial phase $\gamma(0)=\gamma_0$ 
can be chosen in such a way that the orthogonality condition \eqref{eq:orthog} holds at $t=0$ (and hence for all later times until we lose control over the size of $v$).

\begin{proposition}
Suppose that either \textnormal{(Hf-Lip)} is satisfied with $k \ge 0$ or 
that $k=1$ and \textnormal{(Hf-Cub)} is satisfied. Assume furthermore that 
\textnormal{(Hq)}, \textnormal{(HSt)}, \textnormal{(HCor)}, \textnormal{(HTw)} and \textnormal{(HS)} all hold.
 Then there exist constants  $\delta_0>0, \delta_\sigma>0$ and $K>0$ so that for every $0\leq \sigma\leq \delta_\sigma$ and any initial condition $u_0\in\mathcal U_{H^k(\mathcal D;\mathbb R^n)}$ that satisfies
\begin{equation}
    \|u_0-\Phi_\sigma\|_{H^k(\mathcal D;\mathbb R^n)}<\delta_0,
\end{equation}
there exists a $\gamma_0\in\mathbb R$ for which the function
\begin{equation}
    v_{\gamma_0}=T_{-\gamma_0}u_0-\Phi_\sigma
\end{equation}
satisfies the orthogonality condition 
$
    \langle v_{\gamma_0},\psi_{\rm tw}\rangle_{L^2(\mathcal D;\mathbb R^n)}=0 $
and admits the bound
\begin{equation}
    |\gamma_0|+\|v_{\gamma_0}\|_{H^k(\mathcal D;\mathbb R^n)}\leq K\|u_0-\Phi_\sigma\|_{H^k(\mathcal D;\mathbb R^n)}.
\end{equation}
    
\end{proposition}
\begin{proof}
The proof is  analogous to   the proof of \cite[Prop. 2.2.3]{hamster2019stability}.
\end{proof}

\paragraph{Step 5: Mild formulation}

In order to utilise the stability properties of the semigroup $S_{\rm tw}(t)$ associated to the one-dimensional travelling wave, we need to establish some form of mild 
representation for the perturbation $v(t)$. Due to the
quasilinear structure of our problem, this  turns out to be rather subtle. In {\S}\ref{sec:time_transform}-\ref{sec:mild_proof} we perform
a (stochastic) time transformation  allowing
us to arrive at the representation\footnote{For notational clarity, we continue to use $t$ for the transformed time in this section.}
\begin{equation}
        v(t)=E(t,0) v(0)+\int_0^tE(t,s)\mathcal N_\sigma(v(s))\mathrm ds+\sigma\int_0^tE(t,s)\mathcal M_\sigma( v(s))\mathrm dW_s^-.\label{eq:mr:forward_g}
    \end{equation}
Observe that $E(t,s)$ denotes the random evolution family generated
by the time-dependent family of linear operators
\begin{equation}
t \mapsto \mathcal{L}_\sigma(v(t)) := \mathcal{L}_{\rm tw} + \kappa_{\sigma}(\Phi_\sigma + v(t),0)^{-1} \Delta_y.
\end{equation} 
In addition, the nonlinearities $\mathcal{N}_\sigma$ and $\mathcal{M}_\sigma$ are defined in \eqref{eq:app:def:n:m:sigma} and satisfy
\begin{equation}
\label{eq:mr:def:n:m:sigma}
    \begin{array}{lcl}
        \mathcal{N}_\sigma(v)
         & = & 
        \kappa_{\sigma}( \Phi_\sigma + v, 0)^{-1} \mathcal{R}_{\sigma}(v;c_\sigma,\Phi_\sigma) - \mathcal{L}_\sigma(v)v,
        \\[0.2cm]
        \mathcal{M}_{\sigma}(v)
         & = & \kappa_{\sigma}( \Phi_\sigma + v, 0)^{-1/2}\mathcal{S}(v;\Phi_\sigma).
    \end{array}
\end{equation}
We note that the $\kappa_{\sigma}$ factors are generated by the time transformation.
Due to the dependence of the behaviour of the solution on the interval $[s,t]$, the function $s \mapsto E(t,s)$ is no longer progressively measurable. In particular, the stochastic integral in \eqref{eq:mr:forward_g} is a so-called forward integral, which we discuss at length in {\S}\ref{sec:forward}. 

As explained in Appendix \ref{list}, the alternative representation \eqref{eq:app:repr:n:sigma:final} can be used to conclude that also $\langle \mathcal{N}_\sigma(v), \psi_{\rm tw} \rangle_{L^2} = 0$
holds whenever $\|v\|_{L^2}$ is small. Clearly, we also have  $\langle \mathcal{M}_\sigma(v), \psi_{\rm tw} \rangle_{L^2} = 0$
whenever $\|v\|_{L^2}$ is small. These orthogonality conditions are crucial to circumvent the neutral modes of the random evolution family $E(t,s),$ resulting in long-time stability.

\paragraph{Step 6: Long-time stability}

As explained in the introduction, our strategy to obtain a stability result for the pair $(\Phi_\sigma, c_\sigma)$
is to control the size of the perturbation $v(t)$ in terms of the expression
\begin{equation}
\label{eq:mr:def:n:eps:k}
N_{\epsilon;k}(t)=\|v(t)\|_{H^k}^2+\int_0^te^{-\epsilon(t-s)}\|v(s)\|_{H^{k+1}}^2\mathrm ds,
\end{equation}
in which $\epsilon > 0$ is a small fixed parameter.
%
For any $\eta > 0$, the associated stopping time is given by
\begin{equation}
\label{eq:mr:def:n:eps:k:st:time}
    t_{\rm st}(\eta;k)=\inf\{t\geq 0: 
    N_{\epsilon;k}(t) > \eta \}.
\end{equation}
\pagebreak
The choice of $k$ becomes important here, as we will need to achieve control over the nonlinear terms in
\eqref{eq:mr:def:n:m:sigma}.
In more detail,
in {\S}\ref{sec:nl:ests} we take $k > d/2$ when (Hf-Lip) is assumed
and show that
\begin{equation}
    \|\mathcal N_\sigma(v)\|_{H^k}\leq 
    K \|v\|_{H^k}^2 (1 + \|v\|_{H^{k+1}})
        + \sigma^2 K \|v\|_{H^{k+1}} 
\end{equation}
holds for some $K>0$, whenever $\|v\|_{H^k}\leq 1$.
Alternatively, assuming (Hf-Cub), we obtain
\begin{equation}
    \|\mathcal N_\sigma(v)\|_{H^1}\leq 
    K \|v\|_{H^2}^2 
        + \sigma^2 K (1 + \|v\|_{H^2} ) \|v\|_{H^2},
\end{equation}
 whenever $\|v\|_{H^1}\leq 1$ holds. In {\S}\ref{sec:stability} we shall use the fact that both cases satisfy
\begin{equation}
    \|\mathcal N_\sigma(v)\|_{H^{k+1}}
\leq K\|v\|_{H^{k+1}}^2+\sigma^2K\|v\|_{H^{k+1}},\quad \|v\|_{H^k}\leq 1,\end{equation}
%
 for some appropriate value of $k$.
 
Our main result here shows how the expected 
supremum of $N_{\epsilon;k}(t)$ behaves as we increase $T$. 



%

\begin{proposition}[{see {\S}\ref{sec:stability}}]
\label{prop:stab-overview}
 Suppose that either \textnormal{(Hf-Lip)} is satisfied with $k > d/2$ or 
that $k=1$ and \textnormal{(Hf-Cub)} is satisfied. Assume furthermore that
\textnormal{(Hq)}, \textnormal{(HSt)}, \textnormal{(HCor)}, \textnormal{(HTw)} and \textnormal{(HS)} hold 
%
and pick  $\epsilon>0$ sufficiently small.
Then there exist constants $\delta_\eta>0$, $\delta_\sigma>0$, and $K>0$ so that, for any integer $T\geq 2,$ any $0<\eta<\delta_\eta,$ any $0\leq \sigma\leq \delta_\sigma,$ and any integer $p\geq 1$, we have the moment bound
 \begin{equation}
 \label{eq:mr:res:moment:bnd:n}
     \mathbb E\left[\sup_{0\leq t\leq t_{\rm st}(\eta;k)\wedge T}|N_{\epsilon;k}(t)|^p\right]\leq K^{p}\Bigg[\|v(0)\|_{H^{k}}^{2p}+\sigma^{2p}(p^p+\log(T)^p)\Bigg].
 \end{equation}
\end{proposition}
We  conclude that we can choose the time $T$ to be exponentially large in $1/\sigma.$ 
More specifically, for any $0\leq T\leq \exp[\delta_\sigma^2/\sigma^2]$, the estimate above shows 
\begin{equation}
     \mathbb E\left[\sup_{0\leq t\leq t_{\rm st}(\eta;k)\wedge T}|N_{\epsilon,k}(t)|^p\right]\lesssim \|v(0)\|_{H^k}^{2p}+\mathcal O(\delta_\sigma^{2p}).
 \end{equation}
The control
on all powers of $N_{\epsilon;k}(t)$ enables us to obtain the probability bound 
\eqref{eq:second}, which is an improvement upon the result for $k=0$ and $d=1$ in \cite{hamster2020expstability} and in line with the 
estimates found in the literature regarding large deviations  theory \cite{da2014stochastic,sowers1992large,swiȩch2009pde,varadhan}. The proof is relatively direct, using
an exponential Markov-type inequality.

\begin{theorem}\label{thm:main} 
Suppose that either \textnormal{(Hf-Lip)} is satisfied with $k > d/2$ or 
that $k=1$ and \textnormal{(Hf-Cub) }is satisfied. Assume furthermore that
\textnormal{(Hq)}, \textnormal{(HSt)}, \textnormal{(HCor)}, \textnormal{(HTw)} and \textnormal{(HS)} hold 
%
and pick  $\epsilon>0$ sufficiently small.
Then   there exist  constants $0<\mu<1$, $\delta_\eta>0$, and $\delta_\sigma>0$ such that, for any integer $T\geq 2,$ any $0<\eta\leq \delta_\eta$, any $0\leq\sigma\leq \delta_\sigma$, and any initial value $u(0)\in\mathcal U_{H^{k}}$  that satisfies $\|u(0)-\Phi_\sigma\|_{H^{k}                             }^2<\mu\eta$,   we have
    \begin{equation}
        \mathbb P(t_{\rm st}(\eta;k)<T)\leq 2T\exp\left(-\frac{\mu\eta}{\sigma^2}\right)\label{eq:second}.
    \end{equation}
\end{theorem}
\begin{proof}
Introducing the random variable
\begin{equation}
\label{eq:mr:proof:exit:bnd:def:z:t}
    Z_T=\sup_{0\leq t\leq t_{\rm st(\eta;k)}\wedge T}|N_{\epsilon;k}(t)|,
\end{equation}
we observe that $
    \mathbb P(t_{\rm st}(\eta;k)< T)=\mathbb P(Z_T\geq \eta) $ holds.
    Using the bounds \eqref{eq:mr:res:moment:bnd:n} for all integers $p$, we may apply the exponential Markov-type inequality in Lemma \ref{lem:moment:to:tail}
    with
\begin{equation}\Theta_1=K\sigma^2,\quad \Theta_2=K\|v(0)\|_{H^{k}}^2+K\sigma^2\log(T),\end{equation}
    to obtain
    \begin{equation}
        \mathbb P(Z_T\geq \eta)\leq 3T^{1/2e}\exp\left(-\frac{\eta-2eK\|v(0)\|_{H^{k}}^2}{2eK\sigma^2}\right).
    \end{equation}
   The bound \eqref{eq:second} now readily follows upon choosing $\mu \leq (4eK)^{-1} $ and by noting that $3T^{1/2e}\leq 2T$ holds for $T\ge 2$.
\end{proof}

    

Finally, we remark that the control of the second moment $(p=1)$ 
in \eqref{eq:mr:res:moment:bnd:n} alone already allows us to show that the exit probability increases logarithmically in time. Indeed,
applying the standard Markov inequality 
to the random variable \eqref{eq:mr:proof:exit:bnd:def:z:t}
yields
    \begin{equation}
    \mathbb P(t_{\rm st}(\eta;k)< T)=\mathbb P(Z_T\geq \eta)
    \leq \eta^{-1}\mathbb EZ_T
\leq
\eta^{-1}K\left[\|u(0)-\Phi_\sigma\|_{H^{k}           }^2+\sigma^2\log(T)\right].\label{eq:first}
    \end{equation}

\section{Random evolution families and the forward integral}\label{sec:mild}
In this preparatory section, we consider a type of random evolution family associated to $\mathcal L_{\rm tw}$ that we will encounter throughout this paper. In particular,
we consider the spatial domain $\mathcal D=\mathbb R\times \mathbb T^{d-1}$ for some fixed dimension $d\geq 1$
and introduce the family of random linear operators $
    \mathcal L_{\nu}(t):\Omega\to\mathscr L(H^{2}, L^2)$
that act as
\begin{equation}
    [\mathcal L_\nu(t)(\omega)u](x,y)=\mathcal [\mathcal L_{\rm tw}u(\cdot,y)](x)+\nu(t,\omega)[\Delta_yu(x,\cdot)](y),\label{eq:lin_gen}
\end{equation}
for $x\in\mathbb R$, $y\in\mathbb T^{d-1},$ $0\leq t\leq T,$ and $\omega\in\Omega$. We impose the following conditions on the coefficient function $\nu$ and  the general setting that we consider in this section. 
%

\begin{itemize}
    \item[(H$\nu$)] The function $\nu:[0,T]\times\Omega\mapsto \mathbb R$ is progressively measurable and continuous with respect to the time variable $\mathbb P$-almost surely. In addition, there exist two positive constants $k_\nu, K_\nu>0$ such that $k_\nu \leq \nu(t) \leq K_\nu$ holds  for all $t\in[0,T]$. 
    \item[(HE)] Either (Hf-Lip) is satisfied with $k \ge 0$ or we have $k=1$ with (Hf-Cub). Furthermore, 
\textnormal{(HTw)}, \textnormal{(HS)} and \textnormal{(H$\nu$)} all hold and we have $T \ge 1$.
\end{itemize}

It is common practice to suppress the dependency of $\omega\in\Omega$, and we will do that too unless we want to be explicit. Since both $\mathcal L_{\rm tw}$ and $\Delta_y$ can be seen as unbounded operators on $L^2(\mathcal D;\mathbb R^n)$ that act pointwise on the `other' coordinate, we will typically write \eqref{eq:lin_gen} in the shorter form
\begin{equation}\mathcal L_\nu(t)=\mathcal L_{\rm tw}+\nu(t)\Delta_y.
\end{equation}
At times we will restrict the operators $\mathcal L_\nu(t)$ to spaces of smoother functions 
while using the same notation, writing $
    \mathcal L_{\nu}(t):\Omega\to\mathscr L(H^{k+2}, H^k)$, for example. Observe  that $\mathbb P$-a.s. the operators
    $\mathcal L_\nu(t)$ are sectorial for all  $0\leq t\leq T$, as they are lower order perturbations to the diffusive operators $\partial_x^2+\nu (t)\Delta_y$ \cite[Prop. 3.2.2(iii)]{lunardi2004linear}.

In  {\S}\ref{sec:random_evol} we   show that any family of random linear operators defined in \eqref{eq:lin_gen} satisfying (H$\nu$) generates an adapted $C_0$-evolution family with convenient analytic properties. For these $\omega$-dependent evolution families, we introduce in  {\S}\ref{sec:forward}  a non-standard type of stochastic integration|referred to as  forward integrals|and provide an essential maximal inequality. We proceed in {\S}\ref{subsec:fw:decay} by exploiting the decaying part of the evolution to formulate a maximal inequality in $H^k$ with respect to a weight that decays exponentially in time. This  allows us in {\S}\ref{sec:maxreg} to formulate a maximal regularity result that achieves control over the integrated $H^{k+1}$-norm. Ultimately, in {\S}\ref{sec:supremum} we study the time-dependence of our bounds and provide conditions that gaurantee a logarithmic growth rate for stochastic convolutions with our evolution family.


\subsection{Basic properties}\label{sec:random_evol}
In order to understand the random flow $E(t,s)$ generated by the
family of random linear operators in \eqref{eq:lin_gen},
we consider the linear initial value problem
\begin{equation}
        \partial_tv=\mathcal L_{\nu}(t)v,\quad v(s)=v_s.\label{eq:linear_eq}\end{equation}
Writing $\hat v$ to denote the Fourier transform with respect to the transverse direction, i.e.,
    \begin{equation}
        \hat v(x,\xi)=\frac{1}{|\mathbb T|^{d-1}}\int_{\mathbb T^{d-1}}e^{-\frac{2\pi i}{|\mathbb T|}\langle y,\xi\rangle }v(x,y)\,\mathrm dy,\quad \xi\in\mathbb Z^{d-1},\label{eq:transversal_Fourier}
    \end{equation}
the initial value problem \eqref{eq:linear_eq} transforms into 
    \begin{equation}
        \partial_t\hat v=(\mathcal L_{\rm tw}-\lambda_1\nu(t)|\xi|^2)\hat v,\quad \hat v(s)=\hat v_s,
    \end{equation}
where $\lambda_1=4\pi^2/|\mathbb T|^2$ denotes the first non-zero eigenvalue of the Laplacian $\Delta_y$.
    An explicit solution of the equation above in terms of the analytic $C_0$-semigroup $S_{\rm tw}(t)$
    generated by $\mathcal L_{\rm tw}$
    is  given by
     \begin{equation}
         \hat v(t)=S_{\rm tw}(t)e^{-\lambda_1|\xi|^2\int_s^t\nu(r)\mathrm dr}\hat v_s.
     \end{equation}
     Applying the inversion formula
\begin{equation}
    v(x,y)=\sum_{\xi\in \mathbb Z^{d-1}}e^{\frac{2\pi i}{|\mathbb T|}\langle y, \xi\rangle}\hat v(x,\xi),
\end{equation} 
we hence obtain the semi-explicit expression
\begin{equation}
        [E(t,s)v](x,y)=\sum_{\xi\in \mathbb Z^{d-1}}e^{\frac{2\pi i}{|\mathbb T|}\langle y, \xi\rangle }S_{\rm tw}(t-s)e^{-\lambda_1|\xi|^2\int_{s}^t\nu(r)\mathrm dr}\hat v(x,\xi) \label{eq:evolution}
    \end{equation}
    for the evolution family associated to \eqref{eq:linear_eq}.    
    Note that \eqref{eq:evolution} reduces to the expression found in \cite{kapitula1997} when considering the autonomous case $\nu \equiv 1$. 
    
Our first result establishes several useful properties
of this family $E(t,s)$, which is defined on the set
${\triangle} = \{(s,t) \in [0,T]^2:s\leq t\}$.
In summary, $E(t,s)$
is an adapted $C_0$-evolution family in $H^k=H^k(\mathcal D;\mathbb R^n)$ that inherits several important features from the analytic semigroup $S_{\rm tw}(t)$.


\begin{proposition}\label{prop:evolfam}
Suppose that \textnormal{(HE)} holds. 
%
Then  the random family $(E_{\rm}(t,s,\omega))_{(s,t)
    \in\triangle,
    \omega\in\Omega}$ satisfies the following properties:
    \begin{enumerate}[\upshape(i)]
        \item There exists a constant $M > 0$ that does not depend on $T$ so that $\|E(t,s)\|_{\mathscr{L}(H^k)} \le M$ 
       holds  for all $(s,t) \in \triangle$;
        \item $E(s,s) = I$ for all $s \in [0,T]$;
        \item $E(t,s) = E(t,r)E(r,s)$ for all $0 \leq  s \leq r \leq  t \leq T$;
        \item The mapping $\triangle \ni (t,s) \mapsto E(t,s)$ is strongly continuous, i.e., 
         $(t,s)\mapsto E(t,s)v$ is continuous for all $v\in H^k(\mathcal D;\mathbb R^n)$;
         \item $E(t,s)v$ is  $\mathcal F_t$-measurable for all $(s,t)\in\triangle $ and $v\in H^k(\mathcal D;\mathbb R^n)$;
         \item For every $s<t,$ one has $\frac {\mathrm d}{\mathrm dt}E(t,s)=\mathcal L_\nu(t)E(t,s)$ and $\frac {\mathrm d}{\mathrm ds}E(t,s)=-E(t,s)\mathcal L_\nu(s)$, and there exists a constant $C>0$, independent of $\omega\in\Omega$,  such that
         \begin{equation}
             \|\mathcal L_\nu(t)E(t,s)\|_{\mathscr L(H^k)}\leq C(t-s)^{-1}.
         \end{equation}
    \end{enumerate}
\end{proposition}
\begin{proof}
Observe first that for $k>0$, we have
\begin{equation}
        \partial_x^\alpha \partial_y^\beta E(t,s)v(x,y)=\sum_{\xi\in \mathbb Z^{d-1}}e^{\frac{2\pi i}{|\mathbb T|}\langle y, \xi\rangle }\partial_x^\alpha S_{\rm tw}(t-s)e^{-\lambda_1|\xi|^2\int_s^t\nu(r)\mathrm dr}\widehat {\partial_y^\beta v}(x,\xi),\label{eq:derivatives}
    \end{equation}
    where $\alpha\in\mathbb Z_{\geq 0}$ and $\beta \in\mathbb Z_{\geq 0}^{d-1}$ with $\alpha+|\beta|\leq k.$ Applying Plancherel's theorem  twice yields
\begin{equation}
\begin{aligned}
\|\partial^\alpha_x\partial_y^\beta E(t,s)v\|_{L^2(\mathcal D;\mathbb R^n)}&=\|\partial_x^\alpha E(t,s)\partial_y^\beta v\|_{L^2(\mathbb T^{d-1};L^2(\mathbb R;\mathbb R^n))}\\
&\leq M\|\partial_y^\beta v\|_{L^2(\mathbb T^{d-1};H^\alpha(\mathbb R;\mathbb R^n))}
\\&\leq M\|v\|_{H^k(\mathcal D;\mathbb R^n)},
    \end{aligned}
\end{equation}
where we choose a sufficiently large $M\geq 1$ for which $\|S_{\rm tw}(t)\|_{\mathscr L(H^\alpha (\mathbb R;\mathbb R^n))} \leq M$ holds for all $\alpha\leq k$. This yields (i).
Properties (ii), (iii) and (v) are immediate from the definition \eqref{eq:evolution}.

Turning to (iv), the uniform bound in (i) implies that it suffices to establish the continuity for $v$ in a dense set $W\subset H^k(\mathcal D;\mathbb R^n)$; see, e.g., the proof of \cite[Prop I.5.3]{engel2000one}. Without loss, we restrict ourselves to $n=1$ and choose $W$ to be the set
of product functions $w(x,y)=w_1(x)w_2(y)$. Here 
$w_1(x)$ is any element in $C^\infty_c(\mathbb R;\mathbb R )$ and $w_2(y)$ is either the real or imaginary part of \begin{equation} e^{\frac{2\pi i}{|\mathbb T|}\langle y, \zeta\rangle},\quad \zeta\in \mathbb Z^{d-1}.\end{equation} 
Exploiting Plancherel's theorem once more, together with the identity \eqref{eq:derivatives}, we find
 \begin{equation}
 \begin{aligned}
      \|E(\tau,\sigma)w-E(t,s)w\|_{H^k(\mathcal D;\mathbb R)}&= \|S_{\rm tw}(\tau-\sigma)w_1h(\tau,\sigma;\zeta)-S_{\rm tw}(t-s)w_1h(t,s;\zeta)\|_{H^k(\mathcal D;\mathbb R)}\\
      &\leq \|S_{\rm tw}(\tau-\sigma)w_1-S_{\rm tw}(t-s)w_1\|_{H^k( \mathbb R;\mathbb R)}\\
      &\quad\quad\quad + M\|w_1\|_{H^k(\mathbb R;\mathbb R)}|h(\tau,\sigma;\zeta)-h(t,s;\zeta)|,
      \end{aligned}
 \end{equation}
 where $h(t,s;\zeta)=e^{-\lambda_1|\zeta|^2\int_s^t\nu(r)\mathrm dr}$, which satisfies $|h(t,s;\zeta)|\leq 1.$ Since $S_{\rm tw}(t)$ is strongly continuous in $H^k(\mathbb R;\mathbb R^n),$ it remains to show that $\Delta\to\mathbb R, (t,s)\mapsto h(t,s;\zeta)$ is continuous for every $\zeta\in\mathbb Z^{d-1}$ fixed. Since $|1-e^{ax}|\leq 2e^a|x|$ for $x\leq 1$, we obtain
\begin{equation}
\begin{aligned}
    |h(\tau,\sigma;\zeta)-h(t,s;\zeta)|&\leq \big|1-e^{\lambda_1|\zeta|^2\int_\sigma^\tau\nu(r)\mathrm dr-\lambda_1|\zeta|^2\int_s^t\nu(r)\mathrm dr}\big| 
    \\&\leq 2\lambda_1|\zeta|^2\left|{\int_\sigma^\tau\nu(r)\mathrm dr-\int_s^t\nu(r)\mathrm dr}\right|   
    \\&\leq 2\lambda_1|\zeta|^2K_\nu\big[|\tau -t|+|\sigma-s|\big]<\varepsilon,
\end{aligned}\end{equation}
assuming $|t-\tau|,|s-\sigma|<\varepsilon/4\lambda_1|\zeta|^2K_\nu$, where  $\varepsilon\ll1 $ needs to be sufficiently small.
 

Finally, property (vi) follows from the fact that we can swap differentiation and summation in this situation, that $S_{\rm tw}(t)$ satisfies the properties 
\begin{equation}
    \|\tfrac {\mathrm d}{\mathrm dt}S_{\rm tw}(t-s)v\|_{H^k(\mathbb R;\mathbb R^n)}=\|\mathcal L_{\rm tw}S_{\rm tw}(t-s)v\|_{H^k(\mathbb R;\mathbb R^n)}\leq (t-s)^{-1}\|v\|_{H^k(\mathbb R;\mathbb R^n)},
\end{equation}
and 
that the
elementary estimate $ue^{-2u}\leq 1$ for $u\geq 0$ allows us to deduce  the inequality \begin{equation}
\lambda_1|\xi|^2e^{-2\lambda_1|\xi|^2k_\nu t}\leq t/k_\nu ,\quad t\geq 0,
\end{equation} 
on account of (H$\nu$).
\end{proof}

\paragraph{Decompositions of $E(t,s)$.} In view of the semi-explicit expression \eqref{eq:evolution}
for the random evolution family $E(t,s)$, it is convenient to  introduce the bounded linear operators $F(t,s,\omega):H^k(\mathbb T^{d-1};\mathbb R^n)\to H^k(\mathbb T^{d-1};\mathbb R^n)$ that act as 
\begin{equation}
\label{eq:fw:df:f:t:s}
        F(t,s)v(y)=\sum_{\xi\in \mathbb Z^{d-1}}e^{\frac{2\pi i}{|\mathbb T|}y\cdot \xi}e^{-\lambda_1|\xi|^2\int_{s}^t\nu(r)\mathrm dr}\hat v(\xi),
    \end{equation}
for any $(s,t)\in\triangle$ and $\omega\in\Omega$. Recall, $\hat v$ now denotes the usual Fourier transform on $L^2(\mathbb T^{d-1};\mathbb R^n).$
This allows us to obtain the commuting decomposition
\begin{equation}
\label{eq:frwd:decomp:E}
    E(t,s)=F(t,s)S_{\rm tw}(t-s)=S_{\rm tw}(t-s)F(t,s)
\end{equation}
for any $(s,t)\in\triangle$.
Here $S_{\rm tw}(t)$ and $F(t,s)$ can be interpreted
as a semigroup and evolution family, respectively,
on the full space $H^k(\mathcal D;\mathbb R^n)$
that act pointwise on the ``other'' coordinate, i.e.,
\begin{equation}
        [S_{\rm tw}(t)v](x,y)=[S_{\rm tw}(t)v(\cdot,y)](x),\quad [F(t,s)v](x,y)=[F(t,s)v(x,\cdot)](y).
\end{equation}
Indeed, we have
\begin{equation}
\label{eq:fw:norm:s:on:full:space}
    \begin{aligned}
        \|S_{\rm tw}(t)v\|_{H^k(\mathcal D;\mathbb R^n)}&\leq C \sum_{\ell=0}^k \|S_{\rm tw}(t)v\|_{H^\ell(\mathbb T^{d-1};H^{k-\ell}(\mathbb R;\mathbb R^n))}\\
        &\leq CM\sum_{\ell=0}^k\|v\|_{H^\ell(\mathbb T^{d-1};H^{k-\ell}(\mathbb R;\mathbb R^n))}\\
        &\leq CMc^{-1}\|v\|_{H^k(\mathcal D;\mathbb R^n)},
    \end{aligned}
    \end{equation}
    for some constants $C,c>0$.
In fact, $F(t,s)$ represents a $C_0$-family of contractions, since
    \begin{equation}
        \|F(t,s)\|_{\mathscr L(H^k(\mathcal D;\mathbb R^n))}=\|F(t,s)\|_{\mathscr L(H^k(\mathbb T^{d-1};\mathbb R^n))}\leq 1.
    \end{equation}
    However in general $E(t,s)$ is not a family of contractions, since $S_{\rm tw}(t)$ is not. Note that for $\nu \equiv 1$ we simply have $F(t,s)=e^{(t-s)\Delta_y}$. 

Importantly, note that we can actually decompose $E$ even further by writing
    \begin{equation}
    E(t,s,\omega)=H(t,s,\omega)G(t-s)S_{\rm tw}(t-s),\label{eq:newcomp}
    \end{equation}
    using the contractive random evolution family
\begin{equation}
        H(t,s,\omega)v(y)=\sum_{\xi\in \mathbb Z^{d-1}}e^{\frac{2\pi i}{|\mathbb T|}y\cdot \xi}e^{-\lambda_1|\xi|^2\int_{s}^t[\nu(r)-\frac12k_\nu]\mathrm dr}\hat v(\xi)
    \end{equation}
    and the contractive analytic semigroup
    \begin{equation}
        G(t)=e^{\frac{1}{2} k_\nu t \Delta_y}.
    \end{equation}
    The key benefit of this decomposition is that
    \begin{equation}
        E_{\rm tw}(t)=G(t)S_{\rm tw}(t)
    \end{equation}
    defines a (deterministic) analytic semigroup on $H^k(\mathcal D;\mathbb R^n)$ that hence also has smoothening properties in the $y$-direction; see, e.g., \eqref{eq:HB} and \eqref{eq:fw:log:bnd:i2}.

\paragraph{Spectral projections.} Upon introducing the (one-dimensional) spectral projection
\begin{equation}
    P_{\rm tw}u=\langle u,\psi_{\rm tw}\rangle_{L^2(\mathbb R;\mathbb R^n)}\Phi_0' 
    \label{eq:Ptw-exp}
\end{equation}
together with its complement $P^\perp_{\rm tw} = I - P_{\rm tw}$,
we note that the spectral gap assumption (HS)
implies the exponential decay
\begin{equation}
    \|S_{\rm tw}(t)P^\perp_{\rm tw}u\|_{H^k(\mathbb R;\mathbb R^n)} \le M e^{-\beta t}\|P^\perp_{\rm tw}u\|_{H^k(\mathbb R;\mathbb R^n)}; \label{eq:decay-tw}
\end{equation}
see for example \cite{lunardi2004linear}. On the other hand,
after introducing the averaging operator
\begin{equation}
    P_{\rm avg}u= \frac{1}{|\mathbb T|^{d-1}}\int_{\mathbb T^{d-1}}u(y)\,\mathrm dy\label{eq:Pavg-exp}
\end{equation}
with respect to the $y$-direction
and writing $P_{\rm avg}^\perp=I-P_{\rm avg}$,
it is clear 
from \eqref{eq:fw:df:f:t:s} that
\begin{equation}
\label{eq:fw:bnd:f:s:avg}
    \| F(t,s) P_{\rm avg}^\perp v\|_{H^k(\mathbb T^{d-1};\mathbb R^n )} \le e^{-\lambda_1 \kappa_\nu (t-s)}
    \|P_{\rm avg}^\perp v\|_{H^k(\mathbb T^{d-1};\mathbb R^n )} ,
\end{equation}
since the $\xi = 0$ component of the Fourier transform is averaged out and hence vanishes.

We extend these projections (which are bounded on $H^k(\mathbb R; \mathbb R^n)$ and $H^k(\mathbb T^{d-1};\mathbb R^n)$, respectively) to bounded operators on the full function space $H^k(\mathcal{D}; \mathbb R^n)$ by writing 
\begin{equation}
    [P_{\rm tw}u](x,y)=\langle u(\cdot,y),\psi_{\rm tw}\rangle_{L^2(\mathbb R;\mathbb R^n)}\Phi_0'(x)\quad\text{and}\quad  [P_{\rm avg}u](x,y)= \frac{1}{|\mathbb T|^{d-1}}\int_{\mathbb T^{d-1}}u(x,y')\,\mathrm dy'.
\end{equation}
Observe that the bounds in \eqref{eq:decay-tw}
and \eqref{eq:fw:bnd:f:s:avg} carry over to the full spatial norm $\|\cdot\|_{H^k(\mathcal D;\mathbb R^n)}$ by estimates
such as \eqref{eq:fw:norm:s:on:full:space}
and direct inspection of the Fourier
representation \eqref{eq:fw:df:f:t:s}.

An important role in this paper is reserved for the
 bounded linear operator $P$ on $H^k(\mathcal D;\mathbb R^n)$
 that acts as
\begin{equation}
    P=P_{\rm tw}P_{\rm avg}=P_{\rm avg}P_{\rm tw},
\end{equation}
implying that it is  also a projection. Observe that  $Pv=0$ holds if and only if $\langle v,\psi_{\rm tw}\rangle_{L^2(\mathcal D;\mathbb R^n)}=0.$ Further, introducing the complement
\begin{equation}
\label{eq:fw:id:for:p:perp}
    P^\perp = I - P = P^\perp_{\rm tw}+P_{\rm tw}P^\perp_{\rm avg},
\end{equation}
we see that $E(t,s) P^\perp$ admits exponential decay.

\begin{lemma}\label{lem:decay_E}
Suppose that \textnormal{(HE)} holds. 
 Then there exists a constant $M\geq 1$ for which we have 
    \begin{equation}
    \label{eq:fw:exp:decay:e:t:s}
        \|E(t,s)P^\perp\|_{\mathscr L(H^k,H^k)}\leq M e^{-\mu (t-s)},
    \end{equation}
    together with 
    \begin{equation}
        \|E(t,s)P^\perp\|_{\mathscr L(H^{k},H^{k+1})}\leq M\max((t-s)^{-1/2},1)e^{-\mu (t-s)}, \label{eq:k+1->k}
    \end{equation}
where $\mu= \min\{\beta,\lambda_1 k_\nu\}>0.$ 
\end{lemma}
\begin{proof}
The bound \eqref{eq:fw:exp:decay:e:t:s} follows directly from the representation
\eqref{eq:fw:id:for:p:perp} together with the estimates
\begin{equation}
\begin{array}{lcl}
    \| E(t,s) P^\perp_{\rm tw} v \|_{H^k}
     & \le & \| F(t,s) \|_{\mathcal{L}(H^k)} \| S(t-s) P^\perp_{\rm tw} v\|_{H^k}
     \\[0.2cm]
     & \le & M e^{-\beta (t-s)} \| P^\perp_{\rm tw} v\|_{H^k}
\end{array}
\end{equation}
and
\begin{equation}
\begin{array}{lcl}
    \| E(t,s) P_{\rm tw} P^\perp_{\rm avg}v \|_{H^k}
     & \le & \| S(t-s) P_{\rm tw}\|_{\mathcal{L}(H^k(\mathcal D; \mathbb R)} \| F(t,s) P^\perp_{\rm avg} v\|_{H^k}
     \\[0.2cm]
     & \le & M e^{-\lambda_1 k_\nu (t-s)} \| P^\perp_{\rm avg} v\|_{H^k},
\end{array}
\end{equation}
increasing $M$ if necessary. On the other hand,
\eqref{eq:k+1->k} follows from the standard semigroup bound
\begin{equation}
    \|S_{\rm tw}(t) \|_{\mathscr L(H^k(\mathbb R;\mathbb R^n),H^{k+1}(\mathbb R;\mathbb R^n))}
    \leq M\max\{t^{-1/2},1\}
\end{equation}
and inspection of the identity \eqref{eq:derivatives}.
\end{proof}


\begin{remark}
Note that we do not assume any H\"older continuity on our coefficient $\nu$, as opposed to the setting of 
\cite[Ch. VI.9]{engel2000one} and \cite{acquistapace1988evolution,Acquista_parabolic,PRONK20143634, veraar2010non}. Assuming more regularity would allow us to write the evolution family as 
\begin{equation}
    E(t,s)=e^{(t-s)\mathcal L_\nu(s)}+\int_s^tZ(t,r)\,\mathrm dr,
\end{equation}
where $Z(t,r)$, see  \cite[eq. (2.7)]{acquistapace1988evolution} for instance, is expressed in terms of inductively defined operators that are related to the resolvent of $\mathcal L_\nu(s)$.
\end{remark}


\subsection{Forward integrability and maximal inequalities}\label{sec:forward}
We will now set out to construct a suitable notion for stochastic convolutions against the random evolution family $E(t,s)$. The key issue that needs to be addressed is that $E(t,s)$ is only measurable with respect to $\mathcal{F}_t$ and not $\mathcal{F}_s$, precluding the use of the regular It\^o integral. In addition, $E(t,s)$ is not a family of contractions, which prevents us from appealing directly to existing results.

As customary, we let $\mathcal W$ be a real separable Hilbert space   with orthonormal basis $
(e_k)_{k\geq 0}$ and assume $Q\in\mathscr L(\mathcal W)$ to be a non-negative symmetric operator. We can then consider the Hilbert space $\mathcal W_Q=Q^{1/2}(\mathcal W)$ endowed with its  natural inner product
 \begin{equation}
     \langle v,w\rangle_{\mathcal W_Q}=\langle Q^{-1/2}v,Q^{-1/2}w\rangle_{\mathcal W},
 \end{equation}
 which has $(\sqrt Qe_k)_{k\geq 0}$ as an orthonormal basis.\footnote{We tacitly neglect any possible zero element of the set $(\sqrt Qe_k)_{k\geq 0}$ caused by the fact that $Q$ is only a non-negative and not a positive operator.} 

 Following \cite{da2014stochastic,gawarecki2010stochastic,hairer2009introduction,hamster2020,karczewska2005stochastic,prevot2007concise} we consider a filtered probability space $(\Omega,\mathcal F,\mathbb F,\mathbb P)$ \cite{book:protter}
 and set to construct a cylindrical $Q$-Wiener process $W^Q=(W^Q_t)_{t\geq 0}$ that is adapted to the filtration $\mathbb F$. We consider a set $(\beta_k)_{k\geq 0}$ of independent standard Brownian motions adapted to $\mathbb F$ and write
 \begin{equation}
W_t^Q=\sum_{k=0}^\infty\sqrt{Q}e_k\beta_k(t),
\end{equation}
which converges in $L^2(\Omega;\mathcal{W}_{\rm ext})$ for some larger (abstract) space $\mathcal{W}\subset \mathcal{W}_{\rm ext}$ that is guaranteed to exist by the discussion in \cite{hairer2009introduction}; 
see also \cite[Sec. 5.1]{hamster2020} for  additional background information.

For any Hilbert space $\mathcal H$ and $p \ge 2$,
we introduce the class of processes
\begin{equation}
\begin{aligned}
    \mathcal N^p([0,T];\mathbb F;HS(\mathcal W_Q;\mathcal H))&=\{B\in L^p(\Omega ; L^2([0, T] ; HS(\mathcal W_Q; \mathcal H))) :\\&\quad\quad\quad\quad\quad B\text{ has a progressively measurable version}\},
    \end{aligned}
\end{equation}\pagebreak 

\noindent for which It\^o stochastic integrals with respect to $W^Q$ can be defined. In fact, 
we have the identity 
\begin{equation}
    \int_0^tB(s)\mathrm dW_s^Q=\lim_{n\to\infty}\sum_{k=1}^n \int_0^t B(s)[\sqrt Qe_k]\mathrm d\beta_k(s),
\end{equation}
where the convergence is in $L^p(\Omega;\mathcal H)$ and hence also in probability \cite{karczewska2005stochastic}. We remark that for any measurable and adapted process, we can find a progressive measurable version with an indistinguishable stochastic integral \cite[p. 68]{book:meyer}. 

Setting out to drop the requirement concerning progressive measurability, we follow \cite{leon1998stochastic}
and introduce the notion of a forward integral
with respect to the cylindrical $Q$-Wiener process $W^Q$.
We restrict ourselves to the Hilbert space setting, referring the reader to \cite{van2021maximal} for the analogous definitions for Banach spaces.

 
\begin{definition}\label{def:forward} Suppose $G:[0,T]\times \Omega \to HS(\mathcal W_Q;\mathcal H)$ is an $\mathcal F$-measurable process such that for each $\xi\in\mathcal W$ we $\mathbb P$-a.s. have $G[\xi]\in L^1([0,T];\mathcal H)$. For any integer $n \ge 1$ we define \begin{equation} I^{-}(G, n)=n \sum_{k=1}^n \int_0^T G(s) [\sqrt Q {e_k}]( \beta_k(s+1/n)-\beta_k(s)) \mathrm d s. \end{equation} 
If the sequence $\left(I^{-}(G, n)\right)_{n=1}^{\infty}$ converges in probability, then  $G$ is said to be forward integrable (with respect to $\mathcal{H}$), and we denote the limiting process $\lim_{n\to\infty}I^-(G,n)$ either by $I^{-}(G)$ or
\begin{equation}
\int_0^TG(s)\mathrm dW_s^-.
\end{equation}
\end{definition}

In the special case  $G\in \mathcal N^p([0,T];\mathbb F;HS(\mathcal W_Q;\mathcal H))$,
the infinite sequence $(I^{-}(G,n))_{n=1}^\infty$ converges in probability and the limit coincides with the  It\^o-integral \cite[Prop. 3.2]{pronk2015forward}.
In other words, the forward integral is indeed a proper extension of the It\^o-integral.
One of the advantages over other extensions of the It\^o-integral, such as the Skorohod integral, is that one may simply pull any random operator $A: \Omega \rightarrow \mathscr{L}(\mathcal H,\mathcal K)$ out of the integral. That is,  when $G$ is forward integrable with respect to $\mathcal H$,  then $A G$ is  forward integrable with respect to $\mathcal K$, and 
\begin{equation} 
\label{eq:fw:pull:out:A}
\int_0^t A G(s) \mathrm d W^{-}_s=A \int_0^t G(s) \mathrm d W^{-}_s,
\end{equation}
 for all $0\leq t\leq T.$


Turning towards  stochastic convolutions with $E(t,s)$ in $\mathcal H=H^k=H^k(\mathcal D;\mathbb R^n)$,
we now follow \cite{van2021maximal} by first considering adapted finite-rank step processes $B$ 
that map into $HS(\mathcal{W}_Q; H^{k+2})$, providing extra
smoothness as compared to the target space  $H^k$ in which we want to understand the convolution.
These processes can be written as
\begin{equation}
    B(t)[\sqrt Qe_\ell]= \sum_{i=1}^I\sum_{j=1}^J \mathbf{1}_{\left(t_{i-1}, t_i\right]}(t) \mathbf{1}_{A_{i j}} {h_{i j \ell}}, 
\end{equation}
where $A_{ij}\in\mathcal F_{t_{i-1}}$ and $h_{ij\ell}\in H^{k+2}$, for any $1\leq i\leq I$, $1\leq j\leq J$ and $1\leq \ell\leq L$, for some $I,J,L<\infty$. For the remaining $\ell>L$ we set $B(s)[\sqrt{Q}e_\ell]=0.$

\begin{proposition} \label{prop:well-defined}
Suppose that \textnormal{(HE)} holds.
Consider
any adapted finite-rank  step process $B$ that takes values in $HS(\mathcal W_Q;H^{k+2})$. Then the stochastic process  $(E(t, s) B(s))_{s \in[0, t]}$ is forward integrable with re\-spect to $H^k$ on the interval $[0, t]$, for any $0\leq t\leq T,$ and $\mathbb P$-a.s. we have
\begin{equation}
    \int_0^t E(t, s) B(s)\mathrm dW_s^{-}=E(t, 0) \int_0^t B(s) \mathrm dW_s^Q+\int_0^t \partial_s E(t, s) \int_s^t B(r) \mathrm d W_r^Q \,\mathrm d s. \label{eq:pathwise}
\end{equation}
    Moreover, the process $\left(\int_0^t E(t, s) B(s)\mathrm dW_s^{-}\right)_{t \in[0, T]}$ has a continuous version in $H^k$.
\end{proposition}
\begin{proof}
    Since the process $B$ takes values in the the domain of  $\mathcal L_\nu$, item (vi) of Proposition \ref{prop:evolfam} leads to the bound
    \begin{equation}
        \|E(t,s,\omega)v\|_{W^{1,1}([0,t];H^k)}\leq  C\|v\|_{H^{k+2}}
    \end{equation}
    for some constant $C>0$, independent of $\omega\in\Omega$ and $(s,t)\in\triangle$.
    In particular, the right-hand side of \eqref{eq:pathwise} is well-defined 
    and can indeed be identified as the forward integral of $E(t, \cdot) B(\cdot)$ by applying \cite[Cor. 5.3]{pronk2015forward}. The pathwise continuity of the It\^o integral $\int_0^tB(s)\mathrm dW_s^Q$ in $H^{k+2}$  
     directly implies that the forward integral 
    \eqref{eq:pathwise}
    admits a continuous version in $H^k$.
\end{proof}
Combining \eqref{eq:fw:pull:out:A} and Proposition \ref{prop:well-defined}, we see that for any
$B$ that satisfies the conditions of the latter we may split stochastic convolutions, in the sense that
for any $0\leq \sigma\leq \tau\leq t$ we may write
 \begin{equation}
\begin{aligned}\label{eq:forward-property}
    \int_\sigma^\tau E(t,s)B(s)\mathrm dW_s^-&=\int_0^\tau E(t,s)B(s)\mathrm dW_s^--\int_0^\sigma E(t,s)B(s)\mathrm dW_s^-\\
    &=E(t,\tau)\int_0^\tau E(\tau,s)B(s)\mathrm dW_s^--E(t,\sigma)\int_0^\sigma E(\sigma,s)B(s)\mathrm dW_s^-.
\end{aligned}
\end{equation}
More precisely, the integrals on the right hand side are well-defined (when assuming the conditions in Proposition \ref{prop:well-defined}) and can hence be seen as a definition for the left hand side of \eqref{eq:forward-property}.

A major step towards interpreting stochastic convolutions of $E(t,s)$ with general processes $B$ is the derivation
of some  maximal inequality, which 
provides an estimate for \eqref{eq:pathwise} that
does not rely on the additional $H^{k+2}$-smoothness.
We refer to \cite{Neerven2020MaximalEF,veraar2011note}
for discussions on maximal inequalities for ordinary It\^o convolutions, which typically rely on factorisation or dilation arguments that require the semigroups to be contractive or to admit an $H^\infty$-calculus.
In our case, we
 exploit the decomposition \eqref{eq:frwd:decomp:E} and utilise the fact that $S_{\rm tw}(t)$ admits an $H^\infty$-calculus after projecting out the neutral mode. Note that the factorisation method, although  applicable in a much wider setting, leads to a  bound that is less sharp (see also the discussion after Theorem \ref{thm:maximal}).


In particular, we proceed by splitting the process $B$ as 
\begin{equation}
\label{eq:fw:decomp:B}
 B(s)=P_{\rm tw}B(s)+P_{\rm tw}^\perp B(s),   
\end{equation}
noting that the spectral projection acts as
\begin{equation}
P_{\rm tw}B(s)[w]=\langle B(s)[w],\psi_{\rm tw}\rangle_{L^2(\mathbb R;\mathbb R^n)}\Phi_0'
\end{equation}
for all $w\in\mathcal W_Q$.
Since $S_{\rm tw}(t)\Phi_0'=\Phi_0'$ holds, we have
\begin{equation}
\label{eq:fw:e:on:p:tw}
E(t,s)P_{\rm tw}B(s)=F(t,s)P_{\rm tw}B(s),
\end{equation}
allowing us to exploit the fact that $F$ is a contractive (random) evolution family.

Turning to the complementary process $P_{\rm tw}^\perp B$,
we shall make us of the subspaces
\begin{equation}
    H^k_\perp(\mathbb R; \mathbb R^n)=\{v\in H^k(\mathbb R;\mathbb R^n):P_{\rm tw}v=0\},
\end{equation}
for any $k\ge 0$, which are again Hilbert spaces when endowed with the norm $\|\cdot\|_{H^k}$. 
Considering the restricted operator 
$\mathcal L_{\rm tw}: H^{k+2}_\perp(\mathbb R; \mathbb R^n) \to H^k_\perp(\mathbb R; \mathbb R^n)$,
we note that the perturbation arguments in 
\cite[Sec. 8]{weis2006h} can be combined with the fact that the restriction of $-\partial_x^2$
to $H^{k+2}(\mathbb R;\mathbb R^n)$ admits a bounded $H^\infty$-calculus \cite[Prop 10.2.23]{hytonen2018analysis},
to conclude that $-\mathcal L_{\rm tw}$ is sectorial in $H^k_\perp(\mathbb R; \mathbb R^n)$ and admits a bounded $H^\infty$-calculus of angle strictly smaller then $\pi/2$. The details can be found in \cite[Lem. 2.9.7]{hamster2019stability} where the $k=0$ case was considered, relying crucially on the fact that the simple translational eigenvalue at zero has been removed from the spectrum of  $-\mathcal L_{\rm tw}$ .

In view of \cite[Prop. 3.1]{veraar2011note}, there exists\footnote{
With regards to the notation used in \cite{ kalton2015hinftyfunctional, van2007stochastic,veraar2011note},
we point out that $\gamma(\mathbb R_+;H^k(\mathbb R;\mathbb R^n))=L^2(\mathbb R_+;H^k(\mathbb R;\mathbb R^n))$, since  $H^k(\mathbb R;\mathbb R^n)$ is a Hilbert space, which is of type 2 as well as cotype 2 \cite[Rem. 4.7]{kalton2015hinftyfunctional}.
} 
an equivalent norm $\vvvert \cdot \vvvert_{H^k_\perp(\mathbb R; \mathbb R^n)}$
on $H^k_\perp(\mathbb R; \mathbb R^n)$, which is given by  
\begin{equation}
    \vvvert v\vvvert_{H^k_\perp(\mathbb R; \mathbb R^n)}=
    \Bigg[ \int_0^\infty \| \mathcal L_{\rm tw}^{1/2}S_{\rm tw}(r)v \|^2_{H^k(\mathbb R;\mathbb R^n)} \, \mathrm dr \Bigg]^{1/2},
\end{equation}
ensuring that $S_{\rm tw}(t)$ restricted to $H^k_\perp(\mathbb R;\mathbb R^n)$
is contractive with respect to this norm.
Turning to the full spatial domain, we now introduce the notation
\begin{equation}
H^k_\perp(\mathcal{D}; \mathbb R^n)=\{v\in H^k(\mathcal D ;\mathbb R^n):P_{\rm tw}v=0\}    
\end{equation}
and provide it with the norm
\begin{equation}
\vvvert v \vvvert_{H^k_\perp(\mathcal{D}; \mathbb R^n)} 
= \left[ \sum_{ \beta : |\beta| \le k  } 
\int_{\mathbb T^{d-1}} \vvvert \partial_y^\beta v( \cdot, y) \vvvert_{H^{k - |\beta|}_{\perp}(\mathbb R; \mathbb R^n)}^2 \, d y \right]^{1/2},
\end{equation}
where the sum is with respect to multi-indices $\beta \in \mathbb Z^{d-1}_{\ge 0}$. This norm is again equivalent to the usual $\| \cdot \|_{H^k}$ norm, which means that there exist constants $C,c>0$ for which
\begin{equation}
\label{eq:fw:norm:eqv}
    c\|v\|_{H^k(\mathcal D;\mathbb R^n)}\leq \vvvert v\vvvert_{H^k_\perp(\mathcal D ;\mathbb R^n)}\leq C\|v\|_{H^k(\mathcal D;\mathbb R^n)},
\end{equation}
for all $v \in H^k_\perp(\mathcal D, \mathbb R^n)$. The full evolution family $E(t,s)$ can be restricted to this subspace, where it is contractive.

\begin{lemma}
\label{lem:fw:E:perp:contr}
Suppose that \textnormal{(HE)} holds.  
Then for any $(s,t) \in \triangle $ and any $v \in H^k_\perp(\mathcal{D}; \mathbb R^n)$, we have $E(t,s)v \in H^k_\perp(\mathcal{D}; \mathbb R^n) $ together with the bound
\begin{equation}
    \vvvert E(t,s)v \vvvert_{H^k_\perp(\mathcal{D}; \mathbb R^n) } \le \vvvert v \vvvert_{H^k_\perp(\mathcal{D}; \mathbb R^n) }. 
\end{equation}
\end{lemma}
\begin{proof}
Let us  write $H^k_\perp=H^k_\perp(\mathcal D;\mathbb R^n).$ The invariance of $H^k_\perp $ under the evolution $E(t,s)$
follows from the fact that inner products with respect to $x$ commute with the Fourier transform with respect to $y$.
As in the proof of Proposition \ref{prop:evolfam},
the operator $\partial_y^\beta$ commutes
with both $F(t,s)$ and $S_{\rm tw}(t)$.
Hence, 
for any $v\in H^k_\perp$
we  compute
\begin{equation}
\begin{aligned}
    \vvvert E(t,s)v\vvvert_{H^k_\perp}^2
    & = 
    \sum_{ \beta : |\beta| \le k  } 
\int_{\mathbb T^{d-1}} \vvvert [\partial_y^\beta
F(t,s)S_{\rm tw}(t-s) v] ( \cdot, y) \vvvert_{H^{k - |\beta|}_{\perp}(\mathbb R; \mathbb R^n)}^2 \, \mathrm d y 
\\& =     
    \sum_{ \beta : |\beta| \le k  } 
    \int_{\mathbb T^{d-1}}
    \int_0^\infty
\| \mathcal L_{\rm tw}^{1/2}S_{\rm tw}(r) 
[\partial_y^\beta
F(t,s)S_{\rm tw}(t-s) v](\cdot, y) 
\|^2_{H^{k - |\beta|}(\mathbb R; \mathbb R^n)} \, \mathrm dr \, \mathrm dy
\\& =     
    \sum_{ \beta : |\beta| \le k  } 
    \int_0^\infty
    \int_{\mathbb T^{d-1}}
\| 
[ F(t,s) \mathcal L_{\rm tw}^{1/2}S_{\rm tw}(r) 
S_{\rm tw}(t-s) \partial_y^\beta v](\cdot, y) 
\|^2_{H^{k - |\beta|}(\mathbb R; \mathbb R^n)}  \, \mathrm dy \, \mathrm dr
\\& \le     
    \sum_{ \beta : |\beta| \le k  } 
    \int_0^\infty
    \int_{\mathbb T^{d-1}}
\| 
[\mathcal L_{\rm tw}^{1/2} 
S_{\rm tw}(r + t-s) \partial_y^\beta v](\cdot, y) 
\|^2_{H^{k - |\beta|}(\mathbb R; \mathbb R^n)}  \, \mathrm dy \, \mathrm dr,
\end{aligned}
\end{equation}
where we used the fact that $F(t,s)$ is a contraction
which, in addition, acts pointwise with respect to $x$. 
By performing the substitution $r + t - s \mapsto r$, we find 
\begin{equation}
\begin{aligned}
\vvvert E(t,s)v\vvvert_{H^k_\perp}^2
& \le     
    \sum_{ \beta : |\beta| \le k  } 
    \int_{t-s}^\infty
    \int_{\mathbb T^{d-1}}
\| 
[\mathcal L_{\rm tw}^{1/2}S_{\rm tw}(r) 
 \partial_y^\beta v](\cdot, y) 
\|^2_{H^{k-|\beta|}(\mathbb R; \mathbb R^n)}  \, \mathrm dy \, \mathrm dr
\\
& \le     
    \sum_{ \beta : |\beta| \le k  } 
    \int_{0}^\infty
    \int_{\mathbb T^{d-1}}
\| 
[\mathcal L_{\rm tw}^{1/2}S_{\rm tw}(r) 
 \partial_y^\beta v](\cdot, y) 
\|^2_{H^{k-|\beta|}(\mathbb R; \mathbb R^n)}  \, \mathrm dy \, \mathrm dr
\\& =     
    \sum_{ \beta : |\beta| \le k  } 
   \int_{\mathbb T^{d-1}}
   \int_0^\infty
\| 
[\mathcal L_{\rm tw}^{1/2}S_{\rm tw}(r) 
 \partial_y^\beta v](\cdot, y) 
\|^2_{H^{k-|\beta|}(\mathbb R; \mathbb R^n) }  \, \mathrm dr \, \mathrm dy
\\& =     
    \sum_{ \beta : |\beta| \le k  } 
   \int_{\mathbb T^{d-1}}
   \vvvert 
 [\partial_y^\beta v](\cdot, y) 
\vvvert^2_{H^{k-|\beta|}_\perp(\mathbb R; \mathbb R^n) }   \, \mathrm dy
\\& =    \vvvert v\vvvert_{H^k_\perp}^2,
    \end{aligned}
\end{equation}
as desired.
\end{proof}

We are now ready to state our main result here
and provide a maximal inequality for forward integrals. It can be seen as an extension of \cite[Thm. 6.4]{van2021maximal}, in the sense that $E(t,s)$ itself is not a family of contractions on the full space $H^k=H^k(\mathcal{D};\mathbb R^n)$. It is worth noting that for $p=1$ 
one can use an alternative direct approach utilising the Fourier transform.

\begin{theorem}[{maximal inequality}]\label{thm:maximal}
Suppose that  \textnormal{(HE)}   holds. 
%
Then there exists a constant $K_{\rm cnv} > 0$ that does not depend on $T$
so that for any adapted finite-rank step process $B$
that takes values in $HS(\mathcal W_Q;H^{k+2})$,
and for every $p \ge 1$, we have the bound \begin{equation}\mathbb{E} \sup _{0\leq t\leq T}\left\|\int_0^t E(t, s) B(s) \mathrm d W^-_s\right\|_{H^k}^{2p}  \leq  p^p K_{\rm cnv}^{2p}\mathbb E \left[\int_0^T\|B(s)\|_{HS(\mathcal W_Q; H^{k})}^2\mathrm ds\right]^{p}.\label{eq:maximal}\end{equation} 
\end{theorem}

\begin{proof}
In light of the decomposition \eqref{eq:fw:decomp:B},
we write
\begin{equation}
\begin{array}{lcl}
     \mathcal I &=& \mathbb{E} \sup_{0\leq t\leq T}\left\|\int_0^t E(t, s) P_{\rm tw}B(s) \mathrm d W^-_s\right\|_{H^k}^{2p},
     \\[.2cm]
    \mathcal I_\perp &=&\mathbb{E} \sup_{0\leq t\leq T}\left\|\int_0^t E(t, s) P_{\rm tw}^\perp B(s) \mathrm d W^-_s\right\|_{H^k}^{2p},
\end{array}
\end{equation}
and note that
\begin{equation}\mathbb{E} \sup_{0\leq t\leq T}\left\|\int_0^t E(t, s) B(s) \mathrm d W^-_s\right\|_{H^k}^{2p}  \leq  2^{2p-1}\left(\mathcal I+\mathcal I_\perp\right).
\end{equation} 
In view of \eqref{eq:fw:e:on:p:tw} and the contractivity of $F(t,s)$, we may apply
\cite[Thm. 6.4]{van2021maximal} to find
\begin{equation}\begin{aligned}
    \mathcal I&\leq  \mathbb E\sup_{t\in[0,T]}\left\|\int_0^t F(t,s)P_{\rm tw}B(s)\mathrm dW_s^-\right\|_{H^k}^{2p}\\
    &\leq p^p K_{F}^{2p}  \mathbb E \left[\int_0^T\|P_{\rm tw}B(s)\|_{HS(\mathcal W_Q; H^{k})}^2\mathrm ds \right]^{p}\\
    &\leq p^pK_F^{2p}\|P_{\rm tw}\|_{\mathscr L(H^k(\mathbb R;\mathbb R^n))}^{2p} \mathbb E \left[\int_0^T\|B(s)\|_{HS(\mathcal W_Q; H^{k})}^2\mathrm ds\right]^{p}
\end{aligned}
\end{equation}
for some constant $K_F$. In a similar fashion,
we may exploit Lemma \ref{lem:fw:E:perp:contr}
and the equivalence \eqref{eq:fw:norm:eqv}
to compute
\begin{equation}
\begin{aligned}
\mathcal I_\perp&\leq C^{2p}\mathbb E\sup_{0\leq t\leq T}\left\vvvert\int_0^TE(t,s){P_{\rm tw}^\perp B(s)}\mathrm dW_s^-\right\vvvert_{H^k_\perp(\mathcal D;\mathbb R^n)}^{2p}\\&\leq p^pC^{2p}K_E^{2p}\mathbb E\left[\int_0^T \vvvert P_{\rm tw}^\perp B(s)\vvvert^2_{HS(\mathcal W_Q;H^k_\perp(\mathcal D;\mathbb R^n))}\mathrm ds\right]^{p}\\
&\leq p^p C^{2p}K_E^{2p}c^{-2p}\mathbb E\left[\int_0^T \| P_{\rm tw}^\perp B(s)\|^2_{HS(\mathcal W_Q;H^k)}\mathrm ds\right]^{p}\\
&\leq p^p C^{2p}K_E^{2p}c^{-2p}\|P_{\rm tw}^\perp \|_{\mathscr L(H^k(\mathbb R;\mathbb R^n))}^{2p}\mathbb E\left[\int_0^T \|B(s)\|^2_{HS(\mathcal W_Q;H^k)}\mathrm ds\right]^{p}
\end{aligned}
\end{equation}
for some constant $K_E$. By setting $
    K_{\rm cnv}=2\max\{K_F\|P_{\rm tw}\|_{\mathscr L(H^k(\mathbb R;\mathbb R^n))},CK_Ec^{-1}\|P_{\rm tw}^\perp\|_{\mathscr L(H^k(\mathbb R;\mathbb R^n))}\}, $
the desired bound follows.
\end{proof}
In retrospect, the  representation \eqref{eq:pathwise} allows us (for $T\geq 1$) to derive the (crude) pathwise bound
    \begin{equation}
        \sup_{0\leq t\leq T}\left\|\int_0^tE(t,s)B(s)\mathrm dW_s^-\right\|_{H^k}\leq KT\sup_{0\leq t\leq T}\left\|\int_0^tB(s)\mathrm dW_s^Q\right\|_{H^{k+2}},
    \end{equation}
for some $K>0.$ Taking  $L^{2p}(\Omega)$-norms on both sides, and
   appealing to \cite[Prop. 2.1 and Rem. 2.2]{van2021maximal}, results into a similar inequality as in  \eqref{eq:maximal},
   but then with $HS(\mathcal W_Q;H^{k+2})$ instead of $HS(\mathcal W_Q;H^{k})$ and an additional factor of $T^{2p}$ on the right hand side, which would both be detrimental to our future estimates. 
   Note that the factorisation method does not lead to a loss of regularity, but the maximal inequality that one obtains is less sharp. Indeed, the constant $K_{\rm cnv}$ would  depend on time, the power $p$ on the right hand side of \eqref{eq:maximal} would  be inside of the integral, and the result would only hold for $p>1.$



Since the embeddings $H^{k+2} \hookrightarrow H^k$ are  dense, for any $k\geq 0,$ we can use the maximal inequality above to extend the definition of the forward  integral for progressively measurable processes that attain values in $H^k$. This is achieved in a standard fashion by approximating these processes by adapted step processes of finite-rank with values in $H^{k+2}$.

\begin{corollary}[{unique extension}]\label{cor:important} 
Suppose that  \textnormal{(HE)}   holds.
Then for all $p \ge 2$,
the mapping \begin{equation} B \mapsto \int_0^t E(t, s) B(s) \mathrm{d} W_s^{-} \end{equation} has a unique extension to a continuous linear operator 
\begin{equation}\label{eq:fwd:def:J:p}
J_p: \mathcal N^p([0,t];\mathbb F;HS(\mathcal W_Q;H^k)) \rightarrow L^p(\Omega ; C([0, t] ; H^k)), \end{equation}
for any $0\leq t\leq T.$
\end{corollary}
\begin{remark}We continue to denote the generalised forward integral $J_p B$ as 
$
    \int_0^t E(t,s)B(s)\mathrm dW_s^-,
$
and emphasise the fact that the maximal inequality in \eqref{eq:maximal} 
remains valid. However, for general $B \in \mathcal N^p$, 
it is not yet known whether the stochastic process $J_p B$ is given by a forward stochastic convolution again, i.e., whether $I^-\big(E(t,\cdot)B(\cdot)\big)$ exists in the sense of Definition \ref{def:forward}. On the other hand, if the latter does exist, then it is necessarily equal to $J_p B$.
\end{remark}

\subsection{Weighted decay}
\label{subsec:fw:decay}

Although of fundamental importance, the maximal inequality 
\eqref{eq:maximal} in its current form does not allow us to exploit the decay of the evolution family $E(t,s) P^\perp$. Let us introduce the shorthand notation 
\begin{equation}
\mathcal Z^\perp[B](t) =     \int_0^t E(t,s) P^\perp B(s) \,\mathrm  d W_s^-.\label{eq:fw:itg}
\end{equation}A first step is taken here by considering the process $B$ in a space that is exponentially weighted with respect to time. In this case one can also include the weight in the supremum bound.

\begin{proposition}[{weighted decay estimate}] \label{prop:decay_embedding:sup}
Suppose that   \textnormal{(HE)}   holds
and pick $\epsilon \in(0,2 \mu)$.
Then there exists a constant $K_{\rm dc} > 0$ that does not depend on $T$
so that for any process
\begin{equation}
\label{eq:prp:wt:dec:space:for:b}
B \in \mathcal N^{2p}\big( [0,T]; \mathbb F ; HS(\mathcal{W}_Q; H^k \big),
\end{equation}
and every $p \ge 1$, we have the bound 
\begin{equation}
\label{eq:fw:decay:est}
    \mathbb E \sup_{0 \le t \le T}  \left\|  e^{ -\epsilon/2 (T-t)} 
 \mathcal Z^\perp[B](t)\right\|_{H^k}^{2p}
 \leq p^p K_{\rm dc}^{2p}\mathbb E\left[\int_0^T e^{-\epsilon(T-s)}\|B(s)\|_{HS(\mathcal W_Q;H^k)}^2\,\mathrm ds\right]^p.
\end{equation}
In particular, for any $0\leq t\leq T$, we have
\begin{equation}
    \mathbb E\left\|\mathcal Z^\perp[B](t)\right\|_{H^k}^{2p}\leq p^pK^{2p}_{\rm dc}\mathbb E\left[\int_0^t e^{-\epsilon(t-s)}\|B(s)\|_{HS(\mathcal W_Q;H^k)}^2\,\mathrm ds\right]^p.\label{eq:fw:decay:est:nosup}
\end{equation}
\end{proposition}

The key ingredient is that   representation \eqref{eq:pathwise} allows us to split up convolutions in such a way that the decay becomes visible.
In particular, we start again by considering an adapted finite-rank step process $B$ that takes values in $HS(\mathcal W_Q;H^{k+2})$. This allows us 
to consider a time $N \le t \le N+1$, for some integer $N$, and split up the convolution
as
\begin{equation}
\label{eq:fw:splitting:itg}
\mathcal Z^\perp[B](t) = I_{N;I}(t) + I_{N;II}(t),
\end{equation}
where 
\begin{equation}
\begin{array}{lcl}
    I_{N;I}(t) & = & \sum_{j=1}^N E(t, j) P^\perp \int_0^j E(j, s)  B(s) \mathbf{1}_{j-1 \le s \le j} \, \mathrm d W_s^- ,
    \\[0.2cm]
    I_{N;II}(t) & = & 
     \int_N^t E(t,s) P^\perp B(s) \, \mathrm d W_s^-.
\end{array}    
\end{equation}
Note that, in the above, we have exploited 
the commutation relation $P^\perp E(t,s) = E(t,s) P^\perp$.

\begin{proof}[Proof of Proposition \ref{prop:decay_embedding:sup}]
Without loss we assume $T \in \mathbb N$. We first
consider an adapted finite-rank step process $B$
that takes values in $HS(\mathcal W_Q;H^{k+2})$.
In particular, the splitting \eqref{eq:fw:splitting:itg} holds.
Using
the bound \eqref{eq:fw:exp:decay:e:t:s},
we hence obtain the initial estimate
\begin{equation}
\begin{aligned}
    \left\| I_{N;I}(t) \right\|_{H^k}&\leq \sum_{j=1}^N \|E(t,j)P^\perp\|_{\mathscr L(H^k)}\left\|\int_{j-1}^j E(j,s)B(s)\mathrm dW_s^-\right\|_{H^k}
    \\
    &\leq Me^{ -\mu (t - N)} \sum_{j=1}^N e^{-\mu(N-j)}\left\|\int_{j-1}^j E(j,s)B(s)\mathrm dW_s^-\right\|_{H^k} .
    \\
\end{aligned}
\end{equation}
An application of H\"older's inequality yields
\begin{align}
    \left\|I_{N;I}(t) \right\|_{H^k}^{2p}
    &\nonumber\leq M^{2p}e^{-2\mu p(t-N)}\left[\sum_{j=1}^N e^{-\frac{(2\mu-\epsilon)}{2}(N-j)}\cdot e^{-\frac\epsilon2(N-j)}\left\|\int_{j-1}^j E(j,s)B(s)\mathrm dW_s^-\right\|_{H^k}\right]^{2p}\\
    &\nonumber\leq M^{2p}
    e^{-2\mu p(t-N)}
    \left[\sum_{j=1}^N e^{-\frac{p(2\mu-\epsilon)}{2p-1}(N-j)}\right]^{2p-1}\sum_{j=1}^N e^{-p\epsilon(N-j)}\left\|\int_{j-1}^j E(j,s)B(s)\mathrm dW_s^-\right\|_{H^k}^{2p}\\
    &\leq \frac{M^{2p}e^{-2\mu p(t-N)}}{\big(e^\frac{p(2\mu-\epsilon)}{2p-1}-1\big)^{2p-1}}\sum_{j=1}^N e^{-p\epsilon(N-j)}\left\|\int_{j-1}^j E(j,s)B(s)\mathrm dW_s^-\right\|_{H^k}^{2p}.
\end{align}
Writing $C=Me^{1/(2(2\mu-\epsilon)^2)}$ and observing
\begin{equation}
    \frac{M^{2p}}{(e^{\frac{p(2\mu-\epsilon)}{2p-1}}-1)^{2p-1}}\leq M^{2p} e^{ p/(2\mu - \epsilon)^2} = C^{2p},
\end{equation}
we hence see that
\begin{equation}
\begin{aligned}
     \| e^{ \frac{1}{2} \epsilon t } I_{N;I}(t) \|_{H^k}^{2p}&
    \leq C^{2p} e^{ \epsilon p N}
    \sum_{j=1}^N e^{-p\epsilon(N-j)}\,\left\|\int_{j-1}^j E(j,s)B(s)\mathrm dW_s^-\right\|_{H^k}^{2p}\\
    &\leq C^{2p}\sum_{j=1}^N e^{p\epsilon j}\sup_{j'\in[j-1,j]}\left\|\int_{j-1}^{j'} E(j',s)B(s)\mathrm dW_s^-\right\|_{H^k}^{2p}.
\end{aligned}
\end{equation}
In addition, we observe that 
\begin{equation}
\begin{aligned}
\| e^{ \frac{1}{2} \epsilon t} I_{N;II}(t)\|_{H^k}^{2p}
&\le 
      e^{\epsilon p t} \left\|\int_{N}^t E(t,s)B(s)\mathrm dW_s^- \right\|_{H^k}^{2p}
\\[0.2cm]
& \le e^{\epsilon p (N+1)}
\sup_{j'\in[N,N+1]}\left\|\int_{N}^{j'} E(j',s)B(s)\mathrm dW_s^-\right\|_{H^k}^{2p}.
\end{aligned}
\end{equation}
Note that these estimates no longer explicitly depend on $t$.
As a consequence, this yields
\begin{equation}
\begin{aligned}
   & \sup_{0 \le t \le T} \big[ e^{\frac{1}{2} \epsilon t} 
    \|\mathcal{Z}^\perp[B](t)\|_{H^k} \big]^{2p}
     \le2^{2p} 
     C^{2p}\sum_{j=1}^T e^{p \epsilon j} 
    \sup_{j'\in[j-1,j]}\left\|\int_{j-1}^{j'} E(j',s)B(s)\mathrm dW_s^-\right\|_{H^k}^{2p} .
\end{aligned}
\end{equation}
Applying the maximal inequality \eqref{eq:maximal}, we may hence compute
\begin{align}
    \mathbb E \sup_{0 \le t \le T} \big[ e^{\frac{1}{2} \epsilon t} 
    \|\mathcal{Z}^\perp[B](t)\|_{H^k} \big]^{2p}
    &\nonumber\leq 2^{2p} C^{2p}K_{\rm cnv}^{2p}p^p\sum_{j=1}^T e^{p\epsilon j}\ \mathbb E\left[\int_{j-1}^{j} \|B(s)\|_{HS(\mathcal W_Q;H^k)}^2\,\mathrm ds\right]^p\\
    &\leq p^p 2^{2p}C^{2p}K_{\rm cnv}^{2p}e^{p\epsilon}\mathbb E\sum_{j=1}^T\left[\int_{j-1}^{j} e^{\epsilon s }\|B(s)\|_{HS(\mathcal W_Q;H^k)}^2\,\mathrm ds\right]^p\\
    &\nonumber\leq {p^p} [2 C K_{\rm cnv}e^{\epsilon/2}]^{2p} \mathbb E\left[\int_{0}^{T} e^{\mu s}\|B(s)\|_{HS(\mathcal W_Q;H^k)}^2\,\mathrm ds\right]^p,
\end{align}
using the standard sequence space inequality $\|\cdot\|_{\ell^p}^p \le \|\cdot\|_{\ell^1}^p$ in the last step.
In a standard fashion, we now use Corollary \ref{cor:important} to extend the result to general processes $B$ that satisfy \eqref{eq:prp:wt:dec:space:for:b}.
\end{proof}

\subsection{Maximal regularity}\label{sec:maxreg}
Our goal here is to examine the integrated $H^{k+1}$-norm
of stochastic convolutions with our evolution family $E(t,s)$. In particular, we consider the integral
\begin{equation}
	\mathcal I^\perp_B=\int_0^Te^{-\epsilon(T-s)}\|\mathcal Z^\perp[B](s)\|_{H^{k+1}}^2\mathrm ds,
\end{equation}
with $\mathcal Z^\perp[B]$ defined as in \eqref{eq:fw:itg}. 
Our  result here states that in a certain sense the
$H^k$-supremum estimates from Theorem \ref{thm:maximal}
and Proposition \ref{prop:decay_embedding:sup}
can be combined to infer integrated control
over the $H^{k+1}$-norm of the convolution.

\begin{proposition}[maximal regularity estimate]
\label{prop:max:reg:main:bnd}
Suppose that  \textnormal{(HE)}  holds and pick $\epsilon\in(0,2\mu).$
Then there exists a constant $K_{\rm mr} > 0$ that does not depend on $T$
so that for any 
\begin{equation}
\label{eq:prp:max:reg:space:for:b}
    B \in \mathcal N^{2p}\big( [0,T]; \mathbb F ; HS(\mathcal{W}_Q; H^k \big),
\end{equation}
and every integer $p \ge 1$, we have the bound 
\begin{equation}
\label{eq:fw:max:reg:bnd:i:perp:b}
\begin{aligned}
\mathbb E [\mathcal{I}^\perp_B ]^{p}
 &\le  
    K_{\rm mr}^{p} \, \mathbb E   \sup_{0 \le t \le T} 
      \|\mathcal{Z}^\perp[B](t)\|_{H^k}^{2p} 
    \\&\displaystyle\qquad + p^{p/2} K_{\rm mr}^{p} \mathbb E
    \left[\int_0^T e^{-\epsilon(T-r)} \| B(r) \|_{HS(\mathcal{W}_Q;H^k )}^2 \, \mathrm dr \right]^{p} .
\end{aligned}
\end{equation}
\end{proposition}

Instead of appealing directly to a result for forward integrals, 
we will need to dive
deeper into the limiting process used in 
\cite{van2021maximal} to obtain maximal inequalities. This procedure crucially relies upon the introduction of a (small) delay $\delta > 0$ into the arguments of the random evolution family. In particular, we introduce the notation $x_+=\max\{0,x\}$
and write
\begin{equation}
 E^\delta(t,s)=E\big((t-\delta)_+,(s-\delta)_+ \big),
 \qquad 
	\nu^{\delta}(t ) = \nu( (t- \delta)_+ ), 
\end{equation}
for $0 \le s \le t \le T$.

We proceed under the assumption that $B$ is an
adapted finite-rank step process 
that takes values in $HS(\mathcal W_Q;H^{k+3})$, instead of $HS(\mathcal W_Q;H^{k+2})$. Upon defining
the process
\begin{equation}
Z_\delta(t)=\int_0^t E^{\delta}(t,s) P^\perp B(s)\mathrm dW_s^-,    
\end{equation}
our main task is to establish bounds for the truncated integral
\begin{equation}
    I_{\delta} = \int_\delta^T e^{-\epsilon(T-s)} \|Z_\delta(s)\|_{H^{k+1}}^2\mathrm ds
\end{equation}
which only involve $H^k$-norms. To see that this integral is well-defined, we follow step 3 in the proof of \cite[Thm. 6.4]{van2021maximal} and remark that the arguments in the proof of Proposition \ref{prop:well-defined} can be used to 
obtain 
the
alternative representation
\begin{equation}
\label{eq:fw:max:reg:alt:repr:z:delta}
    Z_{\delta}(t) = E( (t-\delta)^+, 0) P^\perp\int_0^t B(s) \, \mathrm dW_s^Q
    + \int_0^{(t-\delta)^+} \partial_s E ( (t-\delta)^+, s) P^\perp \int_{s+\delta}^t B(r) \, \mathrm dW_r^Q \, \mathrm ds.
\end{equation}
This implies that $Z_{\delta}$ has continuous paths in $H^{k+1}$ and that we have  the pathwise bounds
\begin{equation}
\label{eq:fw:max:reg:alt:repr:z:delta:bnd}
    \sup_{0 \le t \le T} \|Z_{\delta}(t)\|_{H^{\ell}}
    \le  K T \sup_{0 \le t \le T} \left\| \int_0^t B(r) \, \mathrm dW_r^Q\right\|_{H^{\ell+2}},
\end{equation}
for any $0\leq \ell\leq k+1$, with
a constant $K>0$ that is independent of $\delta$. 

The key towards eliminating the dependence on the 
$H^{k+1}$-norms
lies within the introduction of the alternative equivalent inner products\footnote{In this definition and the subsequent identities \eqref{eq:Hk+1nubounds} and \eqref{eq:fw:deriv:e:delta:delta} it is also allowed to take $\delta =0$. }
\begin{equation}
	\langle v , w \rangle_{s;\delta;k+1}
	= \langle v, w \rangle_{H^k}
	+ \langle v_x, w_x \rangle_{H^k}
	+ \nu^\delta(s) \langle \nabla_y v, \nabla_y w \rangle_{H^k},\label{eq:Hk+1nu}
\end{equation}
for any $v,w\in H^{k+1}$ and any $0 \le s \le T$.
For now, we readily observe
\begin{equation}
\max\{1,k_\nu\} 	\langle v, v \rangle_{H^{k+1}}
	\le \langle v, v \rangle_{s;\delta;k+1}\leq d\max\{1,K_\nu\}\langle v, v \rangle_{H^{k+1}},\label{eq:Hk+1nubounds}
\end{equation}
which implies that
\begin{equation}
\label{eq:fw:max:reg:bnd:i:delta:new:ip}
	\begin{aligned}
		I_{\delta}&\leq C_\nu \int_\delta ^Te^{-\epsilon(T-s)}\langle Z_\delta(s),Z_\delta(s)\rangle_{s;\delta;k+1}\mathrm ds,
\end{aligned}
\end{equation}
where $C_\nu=\min\{1,k_\nu^{-1}\}$. To appreciate the benefit of these inner products, we note that for $t \ge \delta$ 
we may compute
\begin{equation}
\label{eq:fw:deriv:e:delta:delta}
	\begin{aligned}
	\frac{\mathrm d}{\mathrm dt}\langle E^\delta(t,s)v,E^\delta(t,s)w\rangle_{H^{k}}&= \langle [\mathcal L_{\rm tw}+\nu^\delta(t)\Delta_y]E^\delta(t,s)v,E^\delta(t,s)w\rangle_{H^{k}}\\&\quad\quad\quad+	\langle E(t,s)v,[\mathcal L_{\rm tw}+\nu^\delta(t)\Delta_y]E^\delta(t,s)w\rangle_{H^{k}}\\
	&=\langle E^\delta(t,s)v, [\mathcal L_{\rm tw}+\mathcal L_{\rm tw}^{\rm adj}-2\partial_x^2]E^\delta(t,s)w\rangle_{H^{k}}\\&\quad\quad\quad+	2\langle E^\delta(t,s)v,[\partial_x^2+\nu^\delta(t)\Delta_y]E^\delta(t,s)w\rangle_{H^{k}}\\
	&=\langle E^\delta(t,s)v, [\mathcal L_{\rm tw}+\mathcal L_{\rm tw}^{\rm adj}-2\partial_x^2]E^\delta(t,s)w\rangle_{H^{k}}\\&\quad\quad\quad+	2\langle E^\delta(t,s)v,E^\delta(t,s)w\rangle_{H^{k}}\\
	&\quad\quad\quad\quad\quad\quad-	2\langle E^\delta(t,s)v,E^\delta(t,s)w\rangle_{t;\delta;k+1}.
	\end{aligned}
\end{equation}
In order to extract this derivative, we introduce
the bilinear form
\begin{equation}
\label{eq:fw:max:reg:def:j:delta:s:v:w}
	\mathcal{J}^{\delta}(s_2, s_1)[v, w]
	= \int_{s_1}^{s_2} e^{-\epsilon(s_2-s)} \langle  E^{\delta}(s,s_1) v,
	E^{\delta}(s, s_1)  w \rangle_{s;\delta;k+1} \mathrm ds,
\end{equation}
for any $\delta \leq s_1\leq s_2\leq T$ and any pair $v,w\in H^{k+1}$. Performing an integration by parts,
we obtain the decomposition
\begin{equation}
    \mathcal{J}^{\delta}(s_2, s_1)[v,w]
    =  \mathcal{J}^{\delta}_{\mathrm{expl}}(s_2,s_1)[v,w]
    + \mathcal{J}^\delta_{\mathrm{res}}(s_2, s_1)[v,w]
\end{equation}
in terms of the explicit and residual bilinear forms
\begin{equation}
\begin{array}{lcl}
\mathcal{J}^{\delta}_{\mathrm{expl}}(s_2,s_1)[v,w] &  = & 
-\frac{1}{2} \langle E^\delta(s_2,s_1) v, E^\delta(s_2, s_1) w \rangle_{H^k} + \frac{1}{2} e^{-\epsilon(s_2-s_1)}\langle  v,  w \rangle_{H^k},
\\[0.2cm]
    \mathcal{J}^\delta_{\mathrm{res}}(s_2, s_1)[v,w]
    & = & \frac{1}{2} \int_{s_1}^{s_2}  e^{-\epsilon(s_2-s)}
          \langle E^{\delta}(s,s_1) v, A_\epsilon E^{\delta}(s,s_1) w \rangle_{H^k}
     \, \mathrm ds,
\end{array}
\end{equation}
where $A_\epsilon= (2 + \epsilon) +  
      \mathcal L_{\rm tw}+\mathcal L_{\rm tw}^{\rm adj}-2\partial_x^2\in \mathscr L(H^k).$
The main point is that these expressions can be bounded using $H^k$-norms only.

\begin{lemma}
\label{lem:max:reg:bnds:j:delta}
Consider the setting of Proposition \ref{prop:max:reg:main:bnd}.
Then there exists a constant $K >0$  that does not depend on $T$ so that for all sufficiently small $\delta \ge 0$ we have the bounds
\begin{equation}
    \begin{array}{lcl}
    | \mathcal J^\delta_{\mathrm{res}}(s_2,s_1)[v, w] |
     & \le & K \|v\|_{H^k}\|w\|_{H^k}\int_{s_1}^{s_2}e^{-\epsilon(s_2-s)}\mathrm ds,
     \\[0.2cm]
    | \mathcal J^\delta(s_2,s_1)[v, w] |
     & \le & K \|v\|_{H^k}\|w\|_{H^k},
    \end{array}
\end{equation}
 for any $\delta \le s_1 \le s_2 \le T$ and any  $v,w \in H^{k+1}$.
\end{lemma}
\begin{proof}
    These bounds follow directly from inspection.
\end{proof}

Setting out to obtain a bound for the integral $I_{\delta}$ for $\delta > 0$ fixed,
we choose an arbitrary partition
$\pi=\{r_0,r_1,\ldots,r_N\}$, where $0=r_0<r_1<\ldots<r_N=T$ and $|r_j-r_{j-1}|< \delta$ for $j=1,\ldots,N$. We ensure that $\delta \in \pi$ holds, i.e., $r_{j_{\delta}-1} = \delta$ for some integer $j_{\delta}$. 
Furthermore, we introduce the shorthand notation $B^{j}(r)=B(r)\textbf{1}_{r_{j-1}\leq r\leq r_j}$ together with the processes
\begin{equation}
\label{eq:fw:max:reg:def:Y:delta}
    Y_{\delta}^{j}(t)=\int_{r_{j-1}}^t E^\delta(t,r)B^j(r)\mathrm dW_r^Q,
\end{equation}
defined on $r_{j-1} \le t \le r_{j}$. On account of the delay and the mesh spacing constraint, we see that $E^{\delta}(t,r)v$ is $\mathcal F_{r_{j-1}}$-measurable for all $v \in H^k$ and 
$r_{j-1} \le r \le t \le r_j$. In particular, the stochastic integrals in \eqref{eq:fw:max:reg:def:Y:delta} are indeed well-defined in the It\^o sense. In parallel to step 2 of the proof of \cite[Thm. 6.4]{van2021maximal}, we note that we have the recursive relations
\begin{equation}
\label{eq:fw:max:reg:recurrent}
	Z_\delta(t)=E^\delta(t,r_{j-1})Z_\delta(r_{j-1})+Y_\delta^j(t),\qquad r_{j-1}\leq t\leq r_j.
\end{equation}
Exploiting the bound \eqref{eq:fw:max:reg:bnd:i:delta:new:ip} and breaking up the integral using the partition $\pi$ leads 
to the bound
	\begin{align}\nonumber
		\mathcal I_\delta
   		&\le C_\nu
  \sum_{j=j_\delta}^N\int_{r_{j-1}}^{r_j}e^{-\epsilon(T-s)}\langle E^{\delta}(s,r_{j-1})Z_\delta(r_{j-1})
  +Y_\delta^j(t),
  E^{\delta}(s,r_{j-1})Z_\delta(r_{j-1})+Y_\delta^j(s)\rangle_{s;\delta;k+1}\mathrm ds\\
		&=C_\nu\big[\mathcal H_I+2\mathcal H_{II}+\mathcal H_{III}\big],
	\end{align}
 in which we have introduced the expressions
\begin{align}
\mathcal H_I&= \sum_{j=j_\delta}^N\int_{r_{j-1}}^{r_j}e^{-\epsilon(T-s)} \langle  E^{\delta}(s,r_{j-1}) Z_\delta(r_{j-1}),
	E^{\delta}(s, r_{j-1})  Z_\delta(r_{j-1}) \rangle_{s;\delta;k+1} \mathrm  ds,\label{eq:HI}
 \\
	\mathcal H_{II}&=\sum_{j=j_\delta}^N\int_{r_{j-1}}^{r_j}e^{-\epsilon(T-s)}\langle E^\delta(s,r_{j-1})Z_\delta(r_{j-1}),Y_\delta^j(s)\rangle_{s;\delta;k+1}\mathrm ds,\label{eq:HII}
 \\
	\mathcal H_{III}&=\sum_{j=j_\delta}^N\int_{r_{j-1}}^{r_j}e^{-\epsilon(T-s)}\langle Y_\delta^j(s),Y_\delta^j(s)\rangle_{s;\delta;k+1}\mathrm ds.\label{eq:HIII}
\end{align}
Recalling the definition \eqref{eq:fw:max:reg:def:j:delta:s:v:w}
and introducing the expressions
\begin{equation}
\begin{array}{lcl}
\mathcal{H}_{I;A} &=& 
    \sum_{j=j_\delta}^N e^{-\epsilon(T-r_j)} \mathcal J^\delta_{\mathrm{res}}(r_j,r_{j-1})[Z_\delta(r_{j-1}),Z_\delta(r_{j-1})],
\\[0.2cm]
    \mathcal{H}_{I;B} & = & 
    \sum_{j=j_\delta}^N e^{-\epsilon(T-r_j)} \mathcal J^\delta_{\mathrm{expl}}(r_j,r_{j-1})[Z_\delta(r_{j-1}),Z_\delta(r_{j-1})],
\end{array}   
\end{equation}
we see that $\mathcal{H}_{I} = \mathcal{H}_{I;A} + \mathcal{H}_{I;B}$.

We will treat the first term $\mathcal{H}_{I;A}$ in a different|more direct|fashion than the others. This will be especially
convenient in {\S}\ref{sec:supremum}, since the maximum of
this term over a set of different values for $T$ reduces to the value for the largest $T$. Note that this term would disappear if one considers the simple grid $\pi = \{0, \delta, T\}$ and sends $\delta \downarrow 0$, which would be possible when considering deterministic evolution families.
Intuitively,
this term collects the long-range cross talk between different elements of the set $\{B^j\}_{j=1}^N$. In the regular setting these contributions are uncorrelated, but this is no longer the case here due to the probabilistic forward-looking nature of $\nu(t)$. More concretely, quadratic terms involving $Y_\delta$ can be analysed with a mild It\^o formula; this is not possible for quadratic terms involving $Z_{\delta}$.

\begin{lemma}
\label{lem:max:reg:bnd:h:i:a}
Consider the setting of Proposition \ref{prop:max:reg:main:bnd}
together with the partition $\pi$. 
Then there exists a constant $K > 0$ that does not depend on $T$ so that we have the pathwise estimate
\begin{equation}
   \textstyle \mathcal H_{I;A}
    \le  K  \sup_{0 \le t \le T }  \|Z_{\delta}(t)\|_{H^k}^2 . 
\end{equation}
\end{lemma}
\begin{proof}
Recalling Lemma \ref{lem:max:reg:bnds:j:delta}, a direct computation yields
\begin{equation}
    \begin{array}{lcl}
    \mathcal{H}_{I;A}
    & \le & K \sum_{j=j_\delta}^N \|Z_{\delta}(r_{j-1})\|_{H^k}^2  e^{-\epsilon(T-r_j)} \int_{r_{j-1}}^{r_j} e^{-\epsilon(r_j - s)} \, \mathrm ds 
    \\[0.2cm]
    & \le & K \sup_{0 \le t \le T }  \|Z_{\delta}(t)\|_{H^k}^2\sum_{j=j_\delta}^N \int_{r_{j-1}}^{r_j} e^{-\epsilon(T-s)} \, \mathrm ds
    \\[0.2cm]
    & = & K \sup_{0 \le t \le T }  \|Z_{\delta}(t)\|_{H^k}^2 \epsilon^{-1}(1 - e^{-\epsilon (T-\delta)}),
    \end{array}
\end{equation}
which provides the bound.
\end{proof}

After appropriate transformations, the remaining terms will all have a similar structure featuring various bilinear forms $\Gamma^j$ on each mesh interval. All these satisfy the following structural constraint, allowing for a streamlined
estimation procedure.

\begin{itemize}
    \item[(H$\Gamma$)] {
       There exists $K_\Gamma > 0$ so that 
       for every $1 \le j \le N$, any $r \in [r_{j-1}, r_{j}]$,
       and any $\omega \in \Omega$, the map $\Gamma^j(r, \omega):H^k \times H^k \to \mathbb R$ is a bilinear form  that satisfies the bound
       \begin{equation}
          \Gamma^j(r) [ v, w] \le K_\Gamma
          \|v\|_{H^k} \|w\|_{H^k}.  
       \end{equation}
       In addition, the map $\omega \mapsto \Gamma^j(r, \omega)[v,w]$ is $\mathcal{F}_{r_{j-1}}$-measurable for 
       any integer $1 \le j \le N$, any $r \in [r_{j-1}, r_j]$, \pagebreak  and any $v,w\in H^k$.
    }
\end{itemize}

\begin{lemma}
\label{lem:fw:max:reg:bnd:with:gamma:j}
Consider the setting of Proposition \ref{prop:max:reg:main:bnd}
together with the partition $\pi$. 
Then there exists a constant $K > 0$ that does not depend on $T$ so that for any integer
$p \ge 1$, any $\Gamma$  satisfying \textnormal{(H$\Gamma$)},
and any progressively measurable $G \in L^{2p}( \Omega; C([0,T];H^k) )$,
the integral expressions
\begin{equation}
\begin{array}{lcl}
    \mathcal{I}_s[G, \Gamma]
    & = & \sum_{j=j_\delta}^N e^{-\epsilon(T-r_j)}
\int_{r_{j-1}}^{r_j} \Gamma^j(r)\big[ G(r),  B^j(r)[\,\cdot\,] \big] \,  \mathrm dW_r^Q ,
\\[0.2cm]
\mathcal{I}_d[G, \Gamma]
     &= & \sum_{k =0}^\infty \sum_{j=j_\delta}^N e^{-\epsilon(T-r_j)}
\int_{r_{j-1}}^{r_j} \Gamma^j(r)[ B^j(r) \sqrt{Q} e_k,  B^j(r) \sqrt{Q} e_k ] \,  \mathrm d r
\end{array}
\end{equation}
satisfy the estimates
\begin{equation}
\begin{array}{lcl}
\mathbb E \, \mathcal{I}_s[G, \Gamma]^{p}
& \le & \displaystyle p^{p/2} (K K_\Gamma)^{p}\mathbb E  \sup_{0 \le t \le T } 
\Big[ e^{-\epsilon (T-t) } \|G(t)\|_{H^k}^2 \Big]^{p} 
\\[0.2cm]
& & \qquad\,\,\displaystyle
+\,p^{p/2} (K K_\Gamma)^{p} \mathbb E \left[\int_0^T e^{-\epsilon(T-r)} \| B(r) \|_{HS( \mathcal{W}_Q;H^k )}^2 \, \mathrm dr \right]^{p} ,\label{eq:EIs}
\\[0.2cm]
\mathcal{I}_d[G, \Gamma]
& \le &   \displaystyle
K K_\Gamma  \int_0^T e^{-\epsilon(T-r)} \| B(r) \|_{HS(\mathcal{W}_Q;H^k )}^2 \, \mathrm dr.
\end{array}
\end{equation}
\end{lemma}
\begin{proof}
Upon introducing the functions
\begin{equation}
    \chi^j(r)= e^{-\epsilon(T-r_j)}\mathbf{1}_{r_{j-1} \le r \le r_j} \Gamma^j(r)[ G(r), B^{j}(r) ]
\end{equation}
and writing $\chi(r)=\sum_{j=1}^N\chi^j(r)$,
we see that
\begin{equation}
	\mathcal I_s[G,\Gamma]=\sum_{j=1}^N\int_0^t
 \chi^j(r) \mathrm dW_r=\int_0^t\chi(r)\mathrm dW_r^Q.
\end{equation}
Applying the Burkholder-Davis-Gundy inequality \cite[Prop. 2.1]{veraar2011note} yields\footnote{The stated $p$-dependence holds for all $p \ge 2$ in view of \cite[Rem. 2.1]{veraar2011note}. We absorb the (single) extra case $p=1$ into the constant $K$, but remark that one cannot
extend the bound uniformly to $p \in (1,2)$. This is why we restrict our analysis to integer values of $p$.} 
\begin{equation}
\begin{array} {lcl}
	\mathbb E \, \mathcal I_s[G,\Gamma]^{p}
 & \le &
 p^{p/2} K^{p}
\mathbb E\left[\int_0^t\|\chi(r)\|_{HS(\mathcal{W}_Q;\mathbb R)}^2\mathrm dr\right]^{p/2}
\\[.2cm]
 & = & p^{p/2} K^{p} \mathbb E\left[\sum_{j=1}^N\int_{r_{j-1}}^{r_j}\|\chi^j(r)\|_{HS(\mathcal{W}_Q;\mathbb R)}^2\mathrm  dr\right]^{p/2},
\end{array}
\end{equation}
where we implicitly have used the identity $\mathbf 1_{[r_{j-1},r_j)}\mathbf 1_{[r_{k-1},r_k)}=\delta_{jk}$. In view of the estimate
\begin{equation}
\|\chi^j(r)\|_{HS(\mathcal{W}_Q;\mathbb R)}^2\leq 
K_\Gamma^2 e^{-2\epsilon(T-r_{j})}\|G(r)\|_{H^k}^2\|B^j(r)\|_{HS(\mathcal{W}_Q;H^k)}^2,
\end{equation}
we may write $M_{\epsilon}[G] = \sup_{0\leq t \leq T}e^{-\epsilon(T-t)}\|G(t)\|_{H^k}^2$
and compute 
\begin{equation}
\begin{aligned}
	\sum_{j=1}^N\int_{r_{j-1}}^{r_j}\|\chi^j(r)\|_{HS(\mathcal{W}_Q;\mathbb R)}^2\mathrm dr&\leq 
 K_\Gamma^2 e^{\epsilon \delta}
 M_{\epsilon}[G] \sum_{j=1}^Ne^{-\epsilon(T-r_{j})}\int_{r_{j-1}}^{r_j}\|B^j(r)\|_{HS(\mathcal{W}_Q;H^k)}^2\mathrm dr\\
	&\leq
 K_\Gamma^2 e^{2 \epsilon \delta} M_{\epsilon}[G]
 \sum_{j=1}^N\int_{r_{j-1}}^{r_j}e^{-\epsilon(T-r)}\|B^j(r)\|_{HS(\mathcal{W}_Q;H^k)}^2\mathrm dr\\
	&= K_\Gamma^2 e^{2 \epsilon \delta}
 M_{\epsilon}[G]
  \int_{0}^{T}e^{-\epsilon(T-r)}\|B(r)\|_{HS(\mathcal{W}_Q;H^k)}^2\mathrm dr,\\
\end{aligned}
\end{equation}
exploiting the disjoint supports of $B^j$. The first
estimate in \eqref{eq:EIs} now follows from the elementary bound
$ab \le \frac{1}{2}(a^2 + b^2)$.
The deterministic bound follows in a  similar but more straightforward fashion.
\end{proof}

As a final preparation, we shall expand the inner products involving pairs of $Y^j_{\delta}$. In particular,
treating the inner product parameter $s$ as fixed
and by noting that $\nu^{\delta}(s)$ is $\mathcal{F}_{r_{j-1}}$-measurable for $r_{j-1} \le s \le r_j$, 
we may apply a mild It\^o formula \cite{da2019mild} and find
\begin{equation}
\label{eq:fw:max:req:mild:ito}
	\begin{aligned}
	\langle Y_{\delta}^{j}(s), Y_{\delta}^{j}(s) \rangle_{s; \delta; k+1 }
	&=    
	\sum_{k=0}^\infty \int_{r_{j-1}}^s \langle E^{\delta}(s, r)B^j(r) \sqrt{Q} e_k,
	E^{\delta}(t, r) B^j(r) \sqrt{Q} e_k \rangle_{s;\delta;k+1} \, \mathrm dr
	\\&\qquad\qquad\qquad+ 2 \int_{r_{j-1}}^t \langle E^{\delta}(s, r) Y_{\delta}^{j}(r), E^{\delta}(s, r) B^j(r)[\,\cdot\,]   \rangle_{s; \delta;k+1} \mathrm dW_r^Q.
\end{aligned}
\end{equation}
We are now ready to estimate the remaining expressions. We note that
$\mathcal{H}_{I;B}$ involves a telescoping argument
that only generates cross-terms between neighbouring mesh intervals. Due to the delay $\delta > 0$, these can still intuitively be interpreted as uncorrelated. 

\begin{lemma}
\label{lem:max:reg:bnd:h:i:b}
Consider the setting of Proposition \ref{prop:max:reg:main:bnd}
together with the partition $\pi$. 
Then there exists a constant $K > 0$ that does not depend on $T$ so that for any integer $p \ge 1$ we have the estimate
\begin{equation}
\begin{aligned}
    \mathbb E [\mathcal H_{I;B}]^{p}
    &\le  p^{p/2} K ^{p}\mathbb E  \sup_{0 \le t \le T } 
\left[ e^{-\epsilon (T-t) } \|Z_{\delta}(t)\|_{H^k}^2 \right]^{p}
\\
&  \qquad
+ p^{p/2} K^{p} \mathbb E \left[\int_0^T e^{-\epsilon(T-r)} \| B(r) \|_{HS(\mathcal{W}_Q;H^k )}^2 \, \mathrm dr \right]^{p} .\label{eq:HIB}
\end{aligned}
\end{equation}
\end{lemma}
\begin{proof}
For convenience, we define $\widetilde{\mathcal H}_{I;B}=-2\mathcal H_{I;B}$ and use the recurrence relation \eqref{eq:fw:max:reg:recurrent}
to compute
\begin{equation} \begin{aligned}
    \widetilde{\mathcal{H}}_{I;B} & = 
\sum_{j=j_\delta}^N\Big[e^{-\epsilon(T-r_{j})}\langle Z_\delta(r_j) -Y_\delta^j(r_j),Z_\delta(r_j)-Y_\delta^j(r_j)\rangle_{H^k}\\&\qquad\qquad\qquad\qquad\qquad\qquad\quad
    -e^{-\epsilon(T-r_{j-1})}\langle Z_\delta(r_{j-1}),Z_\delta(r_{j-1})\rangle_{H^k}\Big].
    \end{aligned}
\end{equation}    
Upon introducing the shorthand notations
    \begin{equation}
\begin{array}{lcl}
    \mathcal{H}_{I; Ba} &  = & \sum_{j=j_\delta}^N e^{-\epsilon (T-r_j)}
     \langle E^{\delta}(r_j,r_{j-1}) Z_\delta(r_{j-1}), Y_{\delta}^j(r_j) \rangle_{H^k},
\\[0.2cm]
\mathcal{H}_{I; Bb} &  = &
\sum_{j=j_\delta }^N e^{-\epsilon(T-r_{j})}\langle Y_\delta^j(r_j),Y_\delta^j(r_j)\rangle_{H^k} ,
\end{array}
\end{equation}
we notice that
%
\begin{align}
    \widetilde{\mathcal{H}}_{I;B} & = \displaystyle
	\sum_{j=j_\delta}^N\left[e^{-\epsilon(T-r_{j})}\langle Z_\delta(r_j),Z_\delta(r_j)\rangle_{H^k}-e^{-\epsilon(T-r_{j-1})}\langle Z_\delta(r_{j-1}),Z_\delta(r_{j-1})\rangle_{H^k}\right] -2 \mathcal{H}_{I;Ba} - \mathcal{H}_{I; Bb}\nonumber
 \\
 & = 
	\langle Z_{\delta}(T), Z_{\delta}(T) \rangle_{H^k}
 - e^{-\epsilon(T-\delta)}\langle Z_{\delta}(\delta), Z_{\delta}(\delta) \rangle_{H^k}
 -2 \mathcal{H}_{I;Ba} - \mathcal{H}_{I; Bb}.
    \end{align}
It hence remains for us to establish bounds for $\mathbb E[\mathcal H_{I;Ba}]^p$ and $\mathbb E[\mathcal H_{I;Bb}]^p$ that can be absorbed into \eqref{eq:HIB}. Firstly, observe that we have
\begin{equation} 
    \mathcal{H}_{I;Ba} 
 = 
\sum_{j=j_\delta}^N e^{-\epsilon (T-r_j)}
\int_{r_{j-1}}^{r_j} \langle E^{\delta}(r_j,r) E^{\delta}(r,r_{j-1}) Z_\delta(r_{j-1}), E^{\delta}(r_j,r) B^j(r) \, \mathrm dW_r^Q \rangle_{H^k},
\end{equation}
which can be estimated accordingly by applying Lemma \ref{lem:fw:max:reg:bnd:with:gamma:j} with $G(r)=G_{I;Ba}(r),$ where
\begin{equation}
\label{eq:fw:max:req:bnd:G:I:Ba}
    G_{I;Ba}(r) = E^{\delta}(r, r_{j-1}) Z_\delta(r_{j-1}), \quad r_{j-1} \le r \le r_{j}.
\end{equation}
Indeed, we see that
\begin{equation}
\begin{aligned} 
    e^{-\epsilon(T-r)} \|G_{I;Ba}(r)\|_{H^k}^2
    & \le M e^{\epsilon \delta} e^{-\epsilon (T - r_{j-1})} \|Z_{\delta}(r_{j-1})\|_{H^k}^2 
    \\
    & \le M e^{\epsilon \delta} \sup_{0 \le t \le T} e^{-\epsilon (T - t)} \|Z_{\delta}(t)\|_{H^k}^2.
\end{aligned}
\end{equation}
 Secondly, using the mild It\^o representation \eqref{eq:fw:max:req:mild:ito} 
we obtain
\begin{equation}
\begin{aligned} 
  \mathcal{H}_{I;Bb} &  = \sum_{k=0}^\infty \sum_{j=j_\delta}^N e^{-\epsilon (T-r_j)}
     \int_{r_{j-1}}^{r_j} \langle E^{\delta}(r_j, r)B^j(r) \sqrt{Q} e_k, E^{\delta}(r_j, r)B^j(r) \sqrt{Q} e_k \rangle_{H^k} \, \mathrm dr
\\
& \qquad \quad
+ 2 \sum_{j=j_\delta}^N e^{-\epsilon (T-r_j)}
\int_{r_{j-1}}^{r_j} \langle E^{\delta}(r_j,r) Y_{\delta}^j(r), E^{\delta}(r_j,r) B^j(r) \, \mathrm dW_r^Q \rangle_{H^k}.
\end{aligned}
\end{equation}
Both terms can  be treated using Lemma \ref{lem:fw:max:reg:bnd:with:gamma:j} again, yet now with $G(r)=G_{I;Bb}(r)$, where
\begin{equation}
\label{eq:fw:max:req:bnd:G:I:Bb}
    G_{I;Bb}(r) = Y_\delta^j(r) = Z_{\delta}(r) -E^{\delta}(r, r_{j-1}) Z_\delta(r_{j-1}), \quad r_{j-1} \le r \le r_{j}.
\end{equation}
Indeed, the latter satisfies the estimate
\begin{equation}
\label{eq:fw:max:req:bnd:for:G:Z:delta}
\begin{aligned} 
    e^{-\epsilon(T-r)} \|G_{I;Bb}(r)\|_{H^k}^2 
    &\le  
    2e^{-\epsilon(T-r)} \|Z_\delta(r)\|_{H^k}^2
    +
    2 M e^{\epsilon \delta} e^{-\epsilon (T - r_{j-1})} \|Z_{\delta}(r_{j-1})\|_{H^k}^2
\\
& \le    2(M e^{\epsilon \delta} + 1) \sup_{0 \le t \le T} e^{-\epsilon (T - t)} \|Z_{\delta}(t)\|_{H^k}^2.  
\end{aligned}
\end{equation}
This completes the proof. 
\end{proof}

\begin{lemma}
\label{lem:max:reg:bnd:h:ii}
Consider the setting of Proposition \ref{prop:max:reg:main:bnd}
together with the partition $\pi$. 
Then there exists a constant $K > 0$ that does not depend on $T$ so that for any integer $p \ge 1$ we have the estimate
\begin{equation}
\begin{aligned} 
    \mathbb E [\mathcal H_{II}]^{p}
    &\le  p^{p/2} K ^{p}\mathbb E  \sup_{0 \le t \le T } 
\Big[ e^{-\epsilon (T-t) } \|Z_{\delta}(t)\|_{H^k}^2 \Big]^{p}
\\
& \qquad
+ p^{p/2} K^{p} \mathbb E \Big[\int_0^T e^{-\epsilon(T-r)} \| B(r) \|_{HS(\mathcal{W}_Q;H^k )}^2 \, \mathrm dr \Big]^{p}.
\end{aligned}
\end{equation}
\end{lemma}
\begin{proof}
We first note that
\begin{equation}
\mathcal H_{II}=\sum_{j=j_\delta}^N	
\int_{r_{j-1}}^{r_j} e^{-\epsilon(T-s)} \int_0^s\langle E^{\delta}(s, r_{j-1}) Z_\delta(r_{j-1}) , E^{\delta}(s,r) B^{j}(r) \mathrm d W_r^Q \rangle_{s;\delta;k+1} \mathrm ds.
\end{equation}
Reversing the order of integration yields
\begin{equation}
\begin{aligned} 
\mathcal H_{II}
& = \sum_{j=j_\delta}^N
	\int_{r_{j-1}}^{r_j}  \int_r^{r_j}\ e^{-\epsilon(T-s)}  \langle E^{\delta}(s, r_{j-1}) Z_\delta(r_{j-1}),  E^\delta(s,r)B^{j}(r) [\,\cdot\,] \rangle_{s;\delta;k+1} \mathrm ds \,\mathrm dW_r^Q 
 \\
& =  \sum_{j=j_\delta}^N	\int_{r_{j-1}}^{r_j}  \mathcal{J}^{\delta}({r_{j}, r})\big[ E^{\delta}(r, r_{j-1}) Z_\delta(r_{j-1}), B^{j}(r)[\,\cdot\,] \big]  \mathrm dW_r^Q.
\end{aligned}
\end{equation}
The bound follows by applying Lemma \ref{lem:fw:max:reg:bnd:with:gamma:j} 
with $G(r)$ as in \eqref{eq:fw:max:req:bnd:G:I:Ba}.
\end{proof}

\begin{lemma}
\label{lem:max:reg:bnd:h:iii}
Consider the setting of Proposition \ref{prop:max:reg:main:bnd}
together with the partition $\pi$. 
Then there exists a constant $K > 0$ that does not depend on $T$ so that for any integer $p \ge 1$ we have the estimate
\begin{equation}
\begin{aligned} 
    \mathbb E [\mathcal H_{III}]^{p}
    &\le  p^{p/2} K ^{p}\mathbb E  \sup_{0 \le t \le T } 
\Big[ e^{-\epsilon (T-t) } \|Z_{\delta}(t)\|_{H^k}^2 \Big]^{p}
\\
& \qquad
+ p^{p/2} K^{p} \mathbb E \Big[\int_0^T e^{-\epsilon(T-r)} \| B(r) \|_{HS(\mathcal{W}_Q;H^k )}^2 \, \mathrm dr \Big]^{p}.
\end{aligned}
\end{equation}
\end{lemma}
\begin{proof}
Applying the mild It\^o representation \eqref{eq:fw:max:req:mild:ito}, we see that
\begin{equation}
\begin{aligned}
    \mathcal{H}_{III}
    & =  \sum_{k = 0} ^\infty \sum_{j=j_\delta}^N \int_{r_{j-1}}^{r_j}
    e^{-\epsilon(t-s)} \int_{r_{j-1}}^s \langle E^{\delta}(s, r)B^j(r) \sqrt{Q} e_k,
	E^{\delta}(s, r) B^j(r) \sqrt{Q} e_k \rangle_{s;\delta;k+1} \mathrm dr \, \mathrm ds
 \\
 &  \qquad
	+ 2 \sum_{j=j_\delta}^N \int_{r_{j-1}}^{r_j} e^{-\epsilon(t-s)}
 \int_{r_{j-1}}^s \langle E^{\delta}(s, r) Y_\delta^{j}(r), E^{\delta}(s, r) B^j(r)  \mathrm dW_r^Q \rangle_{s;\delta;k+1} \, \mathrm ds.
 \end{aligned}
\end{equation}
Reversing the order of integration  gives us
\begin{equation}
\begin{aligned}
    \mathcal{H}_{III}
    & =  \sum_{k=0}^\infty \sum_{j=j_\delta}^N \int_{r_{j-1}}^{r_j}
    \int_{r}^{r_j}
    e^{-\epsilon(t-s)}  \langle E^{\delta}(s, r)B^j(r) \sqrt{ Q} e_k,
	E^{\delta}(s, r) B^j(r) \sqrt{Q} e_k \rangle_{s;\delta;k+1} \, \mathrm ds \, \mathrm dr 
 \\
 &  \qquad
	+ 2 \sum_{j=j_\delta}^N \int_{r_{j-1}}^{r_j} 
 \int_{r}^{r_j} e^{-\epsilon(t-s)} \langle E^{\delta}(s, r) Y_\delta^{j}(r), E^{\delta}(s, r) B^j(r)[\,\cdot\,]  \rangle_{s;\delta;k+1} \, \mathrm ds
 \, \mathrm dW_r^Q ,
 \end{aligned}
\end{equation}
    which can hence be written as
\begin{equation}
\begin{aligned} 
    \mathcal{H}_{III}
    & =  \sum_{ k =0}^\infty \sum_{j=j_\delta}^N \int_{r_{j-1}}^{r_j}
         \mathcal{J}^\delta(r_j,r)[B^j(r)\sqrt{Q} e_k,B^j(r) \sqrt{Q} e_k] \, \mathrm dr 
 \\
 &  \qquad \qquad \qquad \qquad \qquad
	+ 2 \sum_{j=j_\delta}^N \int_{r_{j-1}}^{r_j} 
 \mathcal{J}^\delta(r_j,r)\big[Z^j(r),B^j(r)[\,\cdot\,]\big]
 \, \mathrm dW_r^Q  .
 \end{aligned}
\end{equation}
The bound now follows by appealing to Lemma \ref{lem:fw:max:reg:bnd:with:gamma:j}, 
with $G(r)$ as in \eqref{eq:fw:max:req:bnd:G:I:Bb}.
\end{proof}

\begin{proof}[Proof of Proposition \ref{prop:max:reg:main:bnd}]
We first consider adapted finite-rank step processes $B$
that take values in $HS(\mathcal W_Q;H^{k+3})$.
Following step 3 in the proof of \cite[Thm.  6.4]{van2021maximal},
the representation 
\eqref{eq:fw:max:reg:alt:repr:z:delta}
can be used to show that $Z_\delta( \cdot,\omega)\to \mathcal{Z}^\perp[B]( \cdot,\omega)$ holds in $C([0,T];H^{k+1})$, as $\delta \downarrow 0$, for almost any $\omega\in\Omega$.
Moreover,  for these $\omega\in\Omega$ we have
the convergence
\begin{equation}
  I_{\delta}(\omega)
    \to \mathcal{I}^\perp_B(\omega),
\end{equation}
 as $\delta\downarrow 0,$
since we can find a constant $K_{\omega}>0$ so that $\sup_{0 \le t \le T} \|Z_{\delta}(t,\omega)\|_{H^{k+1}} \le K_{\omega}$  holds for all 
small $\delta>0$. 

Applying Fatou's lemma, we see that
\begin{equation}
   \mathbb E [\mathcal{I}^\perp_B]^{p} \le \liminf_{\delta \downarrow 0} \mathbb E [I_{\delta}]^{p}.
\end{equation}
Clearly, the convergence of  $Z_\delta $ to $\mathcal Z^\perp[B]$ also holds  $\mathbb P$-a.s.\ in $C([0,T];H^k)$. Recalling the bound  \eqref{eq:fw:max:reg:alt:repr:z:delta:bnd},
where the right hand side is contained in $L^{2p}(\Omega)$, we may apply the dominated convergence theorem to 
show that
\begin{equation}
    \mathbb E   \sup_{0 \le t \le T} \|Z_{\delta}(t)\|^{2p}_{H^k}
    \to \mathbb  E \sup_{0 \le t \le T} \|\mathcal{Z}^\perp[B](t)\|_{H^k}^{2p},
\end{equation}
as $\delta \downarrow 0$, together with
\begin{equation}
\mathbb E  \sup_{0 \le t \le T} \big[ e^{-\epsilon(T-t)} \|Z_{\delta}(t)\|^2_{H^k}\big]^{p}
    \to \mathbb E  \sup_{0 \le t \le T} \big[e^{-\epsilon (T-t)} \|\mathcal{Z}^\perp[B](t)\|_{H^k}^2\big]^{p}.
\end{equation} 
The desired bound  follows by 
appealing to 
\eqref{eq:fw:decay:est} and combining
Lemma's \ref{lem:max:reg:bnd:h:i:a},
\ref{lem:max:reg:bnd:h:i:b},
\ref{lem:max:reg:bnd:h:ii}
and
\ref{lem:max:reg:bnd:h:iii}.
In a standard fashion, the result can be extended
to general processes $B$ satisfying \eqref{eq:prp:max:reg:space:for:b},  by means of
Corollary \ref{cor:important}.
\end{proof}

\subsection{Bounding suprema with forward integrals logarithmically in time}\label{sec:supremum}


In this section, we  examine how fast the unweighted supremum over $[0,T]$ of the convolution $\mathcal{Z}^\perp[B]$  grows when increasing $T$. For instance, if we assume that $B$ is constant and take $\epsilon = 0$
in \eqref{eq:fw:decay:est:nosup}, we obtain a prediction of order $T^p$. 
Our main result here states that this can be improved to yield growth estimates for the expressions
\begin{equation}
\mathcal Z^\perp[B](t)=\int_0^t E(t,s) P^\perp B(s)\mathrm dW_s^-
\quad\text{and}\quad  
	\mathcal I^\perp[B](t)=\int_0^te^{-\epsilon(t-s)}\|\mathcal Z^\perp[B](s)\|_{H^{k+1}}^2\mathrm ds
\end{equation}
that are logarithmic in $T$. This will be achieved by 
imposing the following a priori pathwise bounds on $B$,
which will arise in a natural fashion during our stability analysis in {\S}\ref{sec:stability}.

\begin{itemize}
    \item[(HB)] For the process $B \in \mathcal{N}^2\big([0, T] ; \mathbb F  ; H S(\mathcal W_Q;H^k  ) \big)$ 
there exists a constant $\Theta_*>0$ such that 
the pathwise bounds
\begin{equation}
\int_0^t e^{-\varepsilon(t-s)}\|B(s)\|_{H S(\mathcal W_Q;H^k  )}^2 \mathrm d s \leq \Theta_*^2\quad\text{and}\quad \|E_{\rm tw}(1)B(t)\|_{H S(\mathcal W_Q;H^k  )}^2 \leq \Theta_*^2\label{eq:HB}
\end{equation}
hold for all $0 \leq t \leq T$.
\end{itemize}
Observe that \eqref{eq:HB} in (HB) implies that $B\in \mathcal{N}^p\big([0, T] ; \mathbb F  ; H S(\mathcal W_Q;H^k  ) \big)$ for all $p\geq 2.$ We recall that the constant $\mu=\min\{\beta,\lambda_1k_\nu\}>0$ is defined for the first time in Lemma \ref{lem:decay_E}.

\begin{proposition}\label{prop:E_B}
Suppose that  \textnormal{(HE)} holds. Assume that $T \ge 2$ is an integer and pick $\epsilon \in (0, 2 \mu)$.
Then there exists a constant $K_{\rm gr} > 0$ that does not depend on $T$
so that for any process $B$ that satisfies \textnormal{(HB)}, and every integer $p \ge 1$,
we have the growth bound
\begin{equation}
\mathbb E \sup _{0 \leq t \leq T}\left\|\mathcal{Z}^\perp[B](t)\right\|_{H^k(\mathcal W_Q;H^k  )}^{2p}
+ \mathbb E\sup_{0\leq t\leq T}\mathcal I^\perp[B](t)^{p}
\leq K^{2p}_{\rm gr} \Theta_*^{2p} (p^p+\log (T)^p).
\label{eq:E_B}
\end{equation}
\end{proposition}

We will primarily follow the approach of 
\cite[Sec. 3.1]{hamster2020expstability},
but the forward integrals require subtle modifications to the procedure at several points, which we highlight below.
In addition, our bound for $\mathcal I^\perp[B]$ is sharper than the related bound in \cite{hamster2020expstability}
due to  Proposition \ref{prop:max:reg:main:bnd}. The approach
relies on integral splittings, which
are only available for finite-rank processes. We will therefore first impose the following condition and use a limiting argument in the final step to lift the restriction.

\begin{itemize}
    \item[(HB*)] The process $B$ satisfies condition (HB) and, in addition, is a finite-rank process that takes values in $HS(\mathcal W_Q;H^{k+2})$.
\end{itemize}

Under this assumption, we
can use \eqref{eq:forward-property}
to make the splitting
\begin{equation}
\mathcal{Z}^\perp[B](t)=\mathcal{E}_B^{\mathrm{lt}}(t)+\mathcal{E}_B^{\mathrm{sh}}(t),
\end{equation}
as in \cite{hamster2020expstability}.
Here the long time (lt) and short time (sh) contributions are given by
\begin{equation}
\mathcal{E}_B^{\mathrm{lt}}(t)=\int_0^{t-1} E(t,s) P^\perp B(s) \mathrm d W_s^-\quad\text{and} \quad \mathcal{E}_B^{\mathrm{sh}}(t)=\int_{t-1}^t E(t,s)P^\perp B(s) \mathrm d W_s^-,
\end{equation}
respectively, where we interpret the boundary $t-1$ as $\max \{t-1,0\}$ if necessary. Both these terms need to be handled using separate techniques.

\begin{lemma}[short-term bound]
     \label{lem:sh}  
     Consider the setting of Proposition \ref{prop:E_B}.
     Then there exists a constant $K_{\rm sh}>0$
     that does not depend on $T$
     so that for any process $B$ that satisfies
     \textnormal{(HB*)}, and every integer $p\geq 1$, we have the supremum  bound
    \begin{equation}
        \mathbb E\sup_{0\leq t\leq T}\|\mathcal E_B^{\rm sh}(t)\|_{H^k }^{2p}\leq (p^p+\log(T)^p)K_{\rm sh}^{2p}\Theta_*^{2p}.
    \end{equation}
\end{lemma}
\begin{proof}
    The result is obtained by generalising \cite[Lem. 3.4]{hamster2020expstability} in a completely analogous way, followed by
     using the estimate on maximum expectations in either Corollary \ref{cor:moment:tail} or \cite[Cor. 2.4]{hamster2020expstability}.  
\end{proof}

For any pair $0 \leq t_1 \leq t_2 \leq T$,  we split 
the long term increment into two parts:
\begin{equation}
\mathcal{E}_B^{\mathrm{lt}}(t_1)-\mathcal{E}_B^{\mathrm{lt}}(t_2)=\mathcal{I}_1(t_1, t_2)+\mathcal{I}_2(t_1, t_2),
\end{equation}
where
\begin{align}
& \mathcal{I}_1(t_1, t_2)=\int_0^{t_1-1}[E(t_2,s)-E(t_1,s)] P^\perp B(s) \mathrm d W_s^-, \\
& \mathcal{I}_2(t_1, t_2)=\int_{t_1-1}^{t_2-1} E(t_2,s) P^\perp B(s) \mathrm d W_s^-.
\end{align}
Note that this is in line with \cite{hamster2020expstability}.
The first integral can be analysed by exploiting the regularity of the random evolution family $E(t,s)$ for $t-s \geq 1$. The second integral requires a supremum bound on the smoothened process $E_{\rm tw}(1) B$, which motivates the condition in \eqref{eq:HB}.

\begin{lemma}\label{lem:710} 
Consider the setting of Proposition \ref{prop:E_B}
and assume that \textnormal{(HB*)} holds. Then for any $1 \leq t_1 \leq t_2 \leq T$ and any integer $p \geq 1$ we have the bound
\begin{equation}
\mathbb E\left\|\mathcal{I}_1(t_1, t_2)\right\|_{H^k}^{2 p} \leq p^p e^{\epsilon p}K_{\mathrm{dc}}^{2 p}  M^{4 p} \Theta_*^{2 p}\left|t_2-t_1\right|^{2 p} .
\end{equation}
\end{lemma}
\begin{proof}
Observe first that by Cauchy-Schwarz, we have
\begin{align}
\mathbb E\left\|\mathcal{I}_1\left(t_1, t_2\right)\right\|_{H^k}^{2 p} & \nonumber \leq\mathbb E\left[\left\|\left[E\left(t_2,t_1\right)-I\right] E(t_1,t_1-1)\right\|_{\mathscr{L}\left(H^k\right)}^{2 p}  \left\|\int_0^{t_1-1} E(t_1-1,s) P^\perp B(s) \mathrm d W_s^-\right\|_{H^k}^{2 p} \right]\\
& \leq M^{2 p}\left|t_2-t_1\right|^{2 p} \mathbb E\left\|\int_0^{t_1-1} E\left(t_1-1,s\right) P^\perp B(s) \mathrm d W_s^-\right\|_{H^k}^{2 p}\\
&\nonumber \leq  M^{4 p}\left|t_2-t_1\right|^{2 p}p^p K_{\mathrm{dc}}^{2 p}  \mathbb E\left[\int_0^{t_1-1} e^{-\epsilon\left(t_1-1-s\right)}\|B(s)\|_{HS(\mathcal W_Q; H^k)}^2 \mathrm d s\right]^p.
\end{align}
The second inequality exploits the fact that $\|(E(t,r)-I)E(r,r-1)\|_{\mathscr L(H^k)}\leq M|t-r|$ holds, which is a generalised version of the final estimate in \cite[Lem 2.1]{hamster2020expstability}. The ultimate inequality follows from  \eqref{eq:fw:decay:est:nosup} in Proposition \ref{prop:decay_embedding:sup}. The assertion  follows in view of \eqref{eq:HB}.
\end{proof} 

 We remind the reader that $E(t,s)=E(t,s,\omega)$ depends on $\omega\in\Omega.$ Consequently, the proof of the lemma below is slightly different as compared to \cite[Lem 3.5]{hamster2020expstability}. In particular, we need the  expanded decomposition \eqref{eq:newcomp}.
 
\begin{lemma}\label{lem:711}
Consider the setting of Proposition \ref{prop:E_B}
and assume that \textnormal{(HB*)} holds.
Then for any $1 \leq t_1 \leq t_2 \leq T$ and any integer $p \geq 1$ we have the bound
\begin{equation}
\mathbb E\left\|\mathcal{I}_2(t_1, t_2\right)\|_{H^k}^{2 p} \leq p^p K_{\mathrm{dc}}^{2 p} \|P^\perp\|_{\mathscr L(H^k)}^{2p}M^{2 p} \Theta_*^{2 p}\left|t_2-t_1\right|^p .
\end{equation}
\end{lemma}
\begin{proof}
An application of Cauchy-Schwarz yields
\begin{equation}
\label{eq:fw:log:bnd:i2}
\begin{aligned}
\mathbb E\|&\mathcal{I}_2(t_1, t_2)\|_{H^k}^{2 p}  \\&\leq \|P^\perp\|_{\mathscr L(H^k)}^{2p}\mathbb E\left[\left\|E(t_2,t_2-1)\int_{t_1-1}^{t_2-1}E(t_2-1,s)B(s)\mathrm dW_s^-\right\|_{H^k}\right]^{2p}\\&
\leq \|P^\perp\|_{\mathscr L(H^k)}^{2p}\mathbb E\|H(t_2,t_2-1) \|_{\mathscr L(H^k)}^{2p}\mathbb E\left[\left\|\int_{t_1-1}^{t_2-1} E(t_2-1,s) E_{\rm tw}(1)B(s) \mathrm d W_s^-\right\|_{H^k}\right]^{2 p} \\
& \leq \|P^\perp\|_{\mathscr L(H^k)}^{2p}M^{2 p}p^p K_{\mathrm{dc}}^{2 p} \left|t_2-t_1\right|^p \mathbb E\left[\sup _{t_1-1 \leq s \leq t_2-1}\|E_{\rm tw}(1)B(s)\|_{H S(\mathcal W_Q; H^k)}^2\right]^p.
\end{aligned}
\end{equation}
By applying  \eqref{eq:HB}, we hence arrive at the stated bound.\end{proof}



The previous two results give a handle on small increments $\left|t_2-t_1\right| \leq 1$. For larger increments, one simply exploits the decay of the random evolution family.

\begin{lemma}\label{lem:712}
Consider the setting of Proposition \ref{prop:E_B}
and assume that \textnormal{(HB*)} holds.
    Then for any $0 \leq t \leq T$ and any integer $p \geq 1$ we have the bound
\begin{equation}
\mathbb E\left\|\mathcal{E}_B^{\mathrm{lt}}(t)\right\|_{H^k}^{2 p} \leq p^p e^{\epsilon p}K_{\mathrm{dc}}^{2 p} M^{2 p} \Theta_*^{2 p} .
\end{equation}
\end{lemma} 

\begin{proof} 
Using  the decay estimate \eqref{eq:fw:decay:est:nosup} in  Proposition \ref{prop:decay_embedding:sup}, it follows that
\begin{equation}
\begin{aligned}
\mathbb E\left\|\mathcal{E}_B^{\mathrm{lt}}(t)\right\|_{H^k}^{2 p} &\leq \|E(t,t-1)\|_{\mathscr L(H^k)}^{2p}\mathbb E\left\|\int_0^{t-1}E(t-1,s)P^\perp B(s)\mathrm dW_s^-\right\|_{H^k}\\& 
\leq M^{2 p}  p^p K_{\mathrm{dc}}^{2 p} \mathbb E\left[\int_0^{t-1} e^{-\epsilon(t-1-s)}\|B(s)\|_{H S(\mathcal W_Q; H^k)}^2 \mathrm ds\right]^p.
\end{aligned}
\end{equation}
The assertion again follows from \eqref{eq:HB}.
\end{proof}

\begin{corollary}\label{cor:E_lt_diff}
Consider the setting of Proposition \ref{prop:E_B}
and assume that \textnormal{(HB*)} holds.
    Then for any $0 \leq t_1 \leq t_2 \leq T$ and any integer $p \geq 1$ we have the bound
\begin{equation}
\mathbb E\left\|\mathcal{E}_B^{\mathrm{lt}}\left(t_1\right)-\mathcal{E}_B^{\mathrm{lt}}\left(t_2\right)\right\|_{H^k}^{2 p}\leq  2^{2p}p^p e^{\epsilon p} K_{\mathrm{dc}}^{2 p}\|P^\perp\|_{\mathscr L(H^k)}^{2p} M^{4 p} \Theta_*^{2 p} \min \left\{\left|t_2-t_1\right|^{1 / 2}, 1\right\}^{2 p} .
\end{equation}
\end{corollary}
\begin{proof} This follows from the standard inequality $(a+b)^{2 p} \leq 2^{2 p-1}\left(a^{2 p}+b^{2 p}\right)$ and a combination of the estimates from Lemma \ref{lem:710}, Lemma \ref{lem:711}, and Lemma \ref{lem:712}.
\end{proof}

As in \cite{hamster2020expstability}, we will proceed by using the results from Dirksen \cite{dirksen2015tail}, which is again   based  on  the chaining principle developed by Talagrand \cite{talagrand2005generic}.

\begin{lemma}[long-term bound]\label{lem:Talagrand}
Consider the setting of Proposition \ref{prop:E_B}.
     Then there exists a constant $K_{\mathrm{lt}} > 0$ 
     that does not depend on $T$ so that
     for any process $B$ that satisfies \textnormal{(HB*)}, and any integer $p \geq 1$,  we have the supremum bound
\begin{equation}
\mathbb E \sup _{0 \leq t \leq T}\left\|\mathcal{E}_B^{\mathrm{lt}}(t)\right\|_{H^k}^{2 p} \leq\left(p^p+\log (T)^p\right) K_{\mathrm{lt}}^{2 p} \Theta_*^{2 p} .
\end{equation}

\end{lemma}
\begin{proof} Upon writing 
$d_{\max }=2 \sqrt{e} K_{\mathrm{dc}} M^2 \|P^\perp\|_{\mathscr L(H^k)}\Theta_*$ together with
\begin{equation}
d\left(t_1, t_2\right)=d_{\max } \min \left\{\sqrt{\left|t_2-t_1\right|}, 1\right\} ,
\end{equation}
an application of \cite[Lemma 2.2]{hamster2020expstability} to Corollary \ref{cor:E_lt_diff} provides the bound
\begin{equation}
P\left(\left\|\mathcal{E}_B^{\mathrm{lt}}\left(t_1\right)-\mathcal{E}_B^{\mathrm{lt}}\left(t_2\right)\right\|_{H^k}>\vartheta\right) \leq 2 \exp \left[-\frac{\vartheta^2}{2 d\left(t_1, t_2\right)^2}\right] .
\end{equation}
Furthermore, recall that the map $
    [0,T] \ni t\mapsto \mathcal E_B^{\rm lt}(t) \in H^k
$
has continuous paths $\mathbb P$-almost surely, as a consequence of Proposition \ref{prop:well-defined}.  This 
suffices to establish the equality 
\begin{equation}
\mathbb E \sup_{0\leq t\leq T}\|\mathcal E_B^{\rm lt}(t)\|_{H^k}^{2p} = 
 \textstyle    \sup_{K\subset[0,T],|K|<\infty}\mathbb E\sup_{t\in K}\|\mathcal E_B^{\rm lt}(t)\|_{H^k}^{2p},
\end{equation}
consequently allowing us to use the results in \cite{dirksen2015tail}, since these are only stated for suprema over sets of finite cardinality.
In particular, for any integer $p\geq 1$, there is a constant $C_{\rm ch}>0$ such that 
\begin{equation}
    \mathbb E \sup_{0\leq t\leq T}\|\mathcal E_B^{\rm lt}(t)\|_{H^k}^{2p}\leq C_{\rm ch}^{2p} \left[\int_0^{\infty} \sqrt{\log (N(T, d, u))} \,\mathrm d u\right]^{2p} +C_{\rm ch}^{2p} d_{\rm max}^{2p}p^p,
\end{equation}
where $N(T,d,u)$ denotes the smallest number of intervals of length at most $u>0$, in the metric $d$, required to cover the interval $[0,T].$ 
This follows  by  choosing $\alpha =  2$  in  \cite[eq.  (3.2)] {dirksen2015tail}, together with \cite[eq.  (2.3)] {dirksen2015tail}, and  by applying the  final inequality in the proof of \cite[Thm. 3.2] {dirksen2015tail}.

Conform to the proof in \cite[Lem. 3.9]{hamster2020expstability},  the 
Dudley entropy integral can be bounded by
\begin{align}
\int_0^{\infty} \sqrt{\log (N(T, d, u))}\,\mathrm  d u \leq d_{\rm max}(\sqrt{2\log(T)}+\sqrt{\pi})\leq 4d_{\rm max}\sqrt{\log(T)},
\end{align}
as $T\geq 2.$
The desired estimate now follows directly.
\end{proof}

\begin{proof}[Proof of Proposition \ref{prop:E_B}]
Combining the short time result in Lemma \ref{lem:sh} with the long time result in Lemma \ref{lem:Talagrand} yields the supremum bound for $\mathcal{Z}^\perp[B]$ for processes that satisfy (HB*). A standard limiting argument
using Corollary \ref{cor:important} then generalises the bound to processes that satisfy (HB).


Turning towards the supremum bound for $\mathcal I^\perp[B]$, note that
for any $0\leq t\leq T,$  Proposition \ref{prop:max:reg:main:bnd} implies that
   for every integer $p \ge 1$ we have
\begin{equation}
\begin{aligned} 
  \mathbb E \, \mathcal{I}^\perp[B](t) ^{p}
 &\le 
    K^{p} \, \mathbb E  \big[  \sup_{0 \le r \le t} 
     \|\mathcal{Z}^\perp[B](r)\|_{H^k}^{2p} \big]
     \\
     &   \qquad
    + p^{p/2} K^{p} \mathbb E
    \Big[\int_0^t e^{-\epsilon(t-r)} \| B(r) \|_{HS(\mathcal{W}_Q;H^k )}^2 \, dr \Big]^{p},
\end{aligned}
\end{equation}
for some $K>0$. Using the estimate
\eqref{eq:E_B} for $\mathcal{Z}^\perp[B]$ which we have just established, we obtain
\begin{equation}
    \mathbb E \, \mathcal I^\perp[B](t)^{p} \le K^{p}   \Theta_*^{2p}(p^{p}+\log(T)^{p})+p^{p/2}K^{p}\Theta_*^{2p}\leq p^p\Theta_1^{p}+\Theta_2^{p},\label{eq:I_B^p}
\end{equation}
for some updated $K>0$, with $\Theta_1=2 K\Theta_*^2$ and $\Theta_2=\frac{1}{2}\Theta_1\log(T).$ 
Corollary \ref{cor:moment:tail} now implies
\begin{equation}
    \mathbb E\max_{i\in\{1,\ldots,T\}}\mathcal I^\perp[B](i)^{p}\leq  K^{p}\Theta_*^{2p}(p^{p}+\log(T)^{p}),
\end{equation}
after updating $K>0$ again.
To conclude the proof, it  suffices to observe
that 
\begin{align}
     \sup_{0\leq t\leq T} \mathcal I^\perp[B](t)&= \max_{i\in\{1,\ldots,T\}}\sup_{i-1\leq t\leq i}\mathcal I^\perp[B](t)\leq e^\varepsilon \max_{i\in\{1,\ldots,T\}}\mathcal I^\perp[B](i)
\end{align} 
holds for any integer $T\geq 2$.
\end{proof}

\section{Bounds on nonlinearities}
\label{sec:nl:ests}

In this section, we obtain estimates on the nonlinearities that we encounter throughout this paper. 
We start in {\S}\ref{subsec:nl:prlm} with some useful preparatory bounds, proceeding
in {\S}\ref{subsec:nl:l2} with estimates that hold in $L^2(\mathcal{D};\mathbb R^n)$. As a consequence
of our pointwise global Lipschitz assumptions, we will be able to extract more information than was possible in \cite{hamster2019stability,hamster2020}.
We proceed in {\S}\ref{subsec:nl:st} and {\S}\ref{subsec:nl:det} with $H^k$-based estimates for the terms appearing in our stochastic and deterministic expressions, respectively. Finally, in {\S}\ref{subsec:nl:low} we consider the low dimensional setting $1 \le d \le 4$ under the cubic growth condition (Hf-Cub). We remind the reader that the full list of functions can be found in Appendix \ref{list}. Recall that we abbreviate the notation for function spaces only if the domain is $\mathcal D=\mathbb R\times \mathbb T^{d-1}$ and the codomain is $\mathbb R^n.$

\subsection{Preliminaries}
\label{subsec:nl:prlm}

 Let us start by considering $H^k$-based estimates for 
the Nemytskii operators that we use. For this, we consider a sufficiently smooth function $\Theta: \mathbb R^n \to \mathbb R^N$ 
and assume that $\Phi$ is bounded and sufficiently smooth on $\mathcal{D} = \mathbb R \times \mathbb T^{d-1}$.
As an illustration, we compute
\begin{equation}
\label{eq:nl:id:for:partial:theta}
\begin{aligned}
 \partial_x[\Theta(\Phi+v_A)-\Theta(\Phi+v_B)]
     & = D\Theta(\Phi+v_A)[\partial_x v_A-\partial_x v_B] 
     \\
     &  \qquad + \big(D\Theta(\Phi+v_A)-D\Theta(\Phi+v_B))[\partial_x \Phi+ \partial_x v_B].
\end{aligned}
\end{equation}
Since $\Phi$ is bounded, 
this provides the pointwise estimate
\begin{equation}
\label{eq:nl:pw:est:for:partial:x}
    |\partial_x[\Theta(\Phi+v_A)-\Theta(\Phi+v_B)]|
    \le  K\Big(  |\partial_x v_A-\partial_x v_B|
   +  |v_A-v_B| ( |\partial_x \Phi| + |\partial_x v_B| ) \Big)
\end{equation}
under the assumption that $\Theta$ and $D\Theta$
are globally Lipschitz,
which automatically implies that $D\Theta$ is bounded.
We can hence not expect global Lipschitz bounds to hold in $H^k$ and the cross-terms will rapidly become more involved as $k$ is increased. To control these expressions, we recall that for any $k > d/2$
we can find a constant $K > 0$ so that
for any bounded $\ell$-linear map  $\Lambda: (\mathbb R^n)^{\ell} \to \mathbb R^N$, the bound
\begin{equation}
\label{eq:nl:bnd:multilinear:general}
    \| \Lambda[ \partial^{\alpha_1} v_1, \ldots, \partial^{\alpha_\ell} v_{\ell}] \|_{L^2(\mathcal D;\mathbb R^N)}
    \le  K |\Lambda| \|v_1\|_{H^k} \cdots \|v_\ell\|_{H^k}
\end{equation}
holds for any tuple $(v_1, \ldots, v_\ell) \in (H^k)^{\ell}$, provided that $|\alpha_1| + \ldots + |\alpha_\ell| \le k$. This is related to the fact that $H^k$ is an algebra under multiplication, i.e., $\|vw\|_{H^k} \le K \|v\|_{H^k}\|w\|_{H^k}$, for $k > d/2$.

\begin{lemma}\label{lem:Hk:func}
Pick $k > d/2$, assume that $\Phi$ is bounded with $\Phi' \in H^k$, and consider a $C^k$-smooth function $\Theta: \mathbb R^n \to \mathbb R^N$ for which $D^\ell \Theta$
is globally Lipschitz for all $0 \le \ell \le k$. Then 
there exists a constant $K > 0$ so that for each pair $v_A, v_B \in H^k$ 
we have the bound
\begin{equation}
\label{eq:nl:bnd:delta:theta}
    \|\Theta(\Phi+v_A)-\Theta(\Phi+v_B)\|_{H^k(\mathcal D;\mathbb R^N)}\leq K(1+\|v_A\|_{H^k}^k+\|v_B\|_{H^k}^k)\|v_A-v_B\|_{H^k}.
\end{equation}
\end{lemma}
\begin{proof}
Consider a multi-index $\alpha \in \mathbb{Z}^d_{\ge 0}$
with $|\alpha| \le k$.
We now claim that it is possible to write the spatial
derivative $\partial^\alpha [\Theta(\Phi+ v_A) - \Theta(\Phi + v_B)]$ 
as a finite sum of expressions of two types.
The first type is given by
\begin{equation}
\label{eq:nl:def:i:1:theta}
\mathcal{I}_I =     D^{\ell}\Theta(\Phi + v_A)[\partial^{\beta_1}(v_A - v_B), \partial^{\beta_2} (\Phi + v_{\#_1}), 
    \ldots,
    \partial^{\beta_{\ell}} (\Phi + v_{\#_{\ell}})
    ],
\end{equation}
with $\#_i \in \{A , B\}$ and multi-indices $\{\beta_i\}_{i=1}^\ell \in \mathbb{Z}^d_{ \ge 0} $ that satisfy $|\beta_i| \ge 1$, for each $1 \le i \le \ell\leq |\alpha|$, together with $|\beta_1| + \ldots + |\beta_\ell| = |\alpha|$.
The second type is given by
\begin{equation}
\label{eq:nl:def:i:2:theta}
\mathcal{I}_{II} =
\Big(D^{\ell}\Theta (\Phi + v_A) - D^{\ell} \Theta (\Phi + v_B)\Big) \Big[ \partial^{\beta_1} ( \Phi + v_B), \dots, \partial^{\beta_\ell} (\Phi + v_B)    \Big],
\end{equation}
with the same conditions on $\{\beta_i\}_{i=1}^\ell$.
This can be readily verified with induction.

Using \eqref{eq:nl:bnd:multilinear:general}
together with the global Lipschitz properties of $\Theta$, which automatically imply that $D^{\ell}\Theta$ is bounded for $1 \le \ell \le k$,
we obtain the bounds
\begin{equation}
\begin{array}{lcl}
   \| \mathcal{I}_{I} \|_{L^2(\mathcal D;\mathbb R^N)} 
   & \le & K 
   \|v_A - v_B\|_{H^k} \big[ 1 + \|v_A\|^{k-1}_{H^k} + \|v_B\|^{k-1}_{H^k} \big] ,
   \\[0.2cm]
   \| \mathcal{I}_{II}\|_{L^2(\mathcal D;\mathbb R^N) } 
   & \le & K \|v_A - v_B\|_{\infty} \big[ 1 + \|v_B\|^k_{H^k} \big].
\end{array}
\end{equation}
Both terms can be absorbed in \eqref{eq:nl:bnd:delta:theta}
in view of the Sobolev embedding $H^k \hookrightarrow L^\infty$.
\end{proof}

\begin{lemma}\label{lem:Hk:func:hkp1}
Pick $k > d/2$, assume that $\Phi$ is bounded with $\Phi' \in H^{k+1}$, and consider a $C^{k+1}$-smooth function $\Theta: \mathbb R^n \to \mathbb R^N$ for which $D^\ell \Theta$
is globally Lipschitz for all $0 \le \ell \le k + 1$. Then 
there exists a constant $K > 0$ so that for each pair $v_A, v_B \in H^{k+1}$ 
we have the bound
\begin{equation}
\label{eq:nl:bnd:delta:theta:hkp1}
\begin{array}{lcl}
    \|\Theta(\Phi+v_A)-\Theta(\Phi+v_B)\|_{H^{k+1}(\mathcal D;\mathbb R^N)}
    & \leq & K(1+\|v_A\|_{H^k}^k +\|v_B\|_{H^k}^k
    )
\\[0.2cm]
& & \qquad \qquad
 \times \big(1 + \|v_A\|_{H^{k+1}}  
    +  \|v_B\|_{H^{k+1}}
    \big) 
\\[0.2cm]
& & \qquad \qquad \times
    \|v_A-v_B\|_{H^k}
\\[0.2cm]
& & \qquad
    +K(1+\|v_A\|_{H^k}^{k} +\|v_B\|_{H^k}^{k})
\\[0.2cm]
& & \qquad \qquad \qquad \times
    \|v_A-v_B\|_{H^{k+1}}.
\end{array}
\end{equation}
\end{lemma}
\begin{proof}
Inspecting the terms $\mathcal{I}_I$
and $\mathcal{I}_{II}$
in \eqref{eq:nl:def:i:1:theta}
and \eqref{eq:nl:def:i:2:theta},
but where now $|\alpha| = k+1$ is also allowed,
we see that  each term can be covered by appealing
to the bound in  \eqref{eq:nl:bnd:multilinear:general}.
It is only necessary to replace one of the $v_{i}$
in this bound
by the differentiated version $\partial^\gamma v_{i}$
for some multi-index $\gamma$ that has a single component, i.e., $|\gamma| = 1$. This leads directly to the stated estimates.
\end{proof}

When we additionally have $\Theta(\Phi) \in H^k(\mathcal{D};\mathbb R^N)$, we may use \eqref{eq:nl:bnd:delta:theta} to obtain the bound
\begin{equation}
\label{eq:bnd:theta:hk:additional}
    \| \Theta(\Phi + v) \|_{H^k(\mathcal{D};\mathbb R^N)} \le K \big[ 1 + \|v\|_{H^k}^{k+1} \big],
\end{equation}
for all $v \in H^k$, possibly after increasing $K$. In the same fashion, whenever $\Theta(\Phi) \in H^{k+1}(\mathcal{D};\mathbb R^N)$, we may use \eqref{eq:nl:bnd:delta:theta:hkp1}
to obtain
\begin{equation}
\label{eq:bnd:theta:hk:additional:hkp1}
    \| \Theta(\Phi + v) \|_{H^{k+1}(\mathcal{D};\mathbb R^N)} \le K \big[ 1 + \|v\|_{H^k}^{k+1} \big] \big[ 1 + \|v\|_{H^{k+1}} \big],
\end{equation}
for all $v \in H^{k+1}$.
This will be the case when we consider the functions $f$, $g$ and $h$.

\begin{lemma}\label{lem:Hk:func:quadr}
Pick $k > d/2$, assume that $\Phi$ is bounded with $\Phi' \in H^k$, and consider a $C^{k+2}$-smooth function $\Theta: \mathbb R^n \to \mathbb R^N$ for which $D^\ell \Theta$
is globally Lipschitz for all $0 \le \ell \le k+2$. Then 
there exists a constant $K > 0$ so that for every $v \in H^k$ 
we have the bound
\begin{equation}
\label{eq:nl:bnd:delta:theta:drv}
    \|\Theta(\Phi+v)-\Theta(\Phi) - D\Theta(\Phi)[v] \|_{H^k(\mathcal D;\mathbb R^N)}\leq K(1+\|v\|_{H^k}^{k})\|v\|^2_{H^k}.
\end{equation}
\end{lemma}
\begin{proof}
Note first that we have the pointwise representation
\begin{equation}
\begin{aligned} 
    \Theta(\Phi + v) - \Theta(\Phi) - D\Theta(\Phi )[ v ]
    & =  \int_0^1 \Big( D \Theta(\Phi + t v) [ v ] - D\Theta(\Phi)[v] \Big) \,\mathrm dt
    \\
    & =  \int_0^1 \left(t \int_0^1 D^2 \Theta (\Phi + s t v) [v,v] \, \mathrm ds \right)  \mathrm dt.\label{eq:pointwise:D2Theta}
\end{aligned}
\end{equation}
Consider a multi-index $\alpha \in \mathbb{Z}^d_{\ge 0}$
with $|\alpha| \le k$.
We now claim that the spatial
derivative 
$\partial^\alpha D^2 \Theta (\Phi + s t v) [v,v]$
can be written as a finite sum of expressions of the form
\begin{equation}
 \mathcal{I} =   D^{\ell + 2} \Theta( \Phi + st v)\Big[
    \partial^{\beta_1} ( \Phi + st v), \partial^{\beta_2} ( \Phi + st v),
    \ldots,  \partial^{\beta_{\ell +1}} v,
    \partial^{\beta_{\ell + 2}} v
    \Big],
\end{equation}
with multi-indices $\{\beta_i\}_{i=1}^{\ell+2} \in \mathbb Z^d_{ \ge 0}$
that satisfy $|\beta_i| \ge 1$, for each $1 \le i \le \ell+2\leq |\alpha|$,
together with  $|\beta_{1}| + \ldots + |\beta_{\ell +2 }| = |\alpha| $. This can again readily be verified using induction. Using the global Lipschitz properties,
we obtain the bound
\begin{equation}
    \| \mathcal{I} \|_{L^2} \le 
    K 
    [ 1 + \|v\|_{H^k}^k ] \|v\|_{H^k}^2,
\end{equation}
as desired.
\end{proof}

We now turn to preliminary observations regarding the Hilbert-Schmidt norms of operators that map $L^2_Q$ into $H^k$, extending
the previous results obtained in \cite{hamster2020} for $d=1$. Assuming that (Hq) is satisfied, we can (formally) introduce the function $p: \mathcal{D} \to \mathbb{R}^{m \times m}$
by taking the Fourier inverse of $\sqrt{\hat{q}}$. We first show that this map is well-defined, confirming that indeed
$\hat p=\sqrt{\hat{q}}$ and hence that $p$ can be interpreted as the convolution kernel for $\sqrt{Q}:L^2_Q\to L^2$, i.e., $\sqrt{Q}v=p*v$ for $v\in L^2_Q$.

\begin{lemma}
\label{lem:p:in:hk}
Pick $k \ge 0$ and assume that \textnormal{(Hq)} is satisfied.
Then $p$ is well-defined and, in addition, we have $p\in H^k(\mathcal D; \mathbb R^{m \times m})$.
\end{lemma}
\begin{proof} 

%
Pick a multi-index $\alpha=(\alpha_1,\ldots,\alpha_d) \in \mathbb Z^{d}_{\ge 0}$ with $|\alpha| \le k$
and
write $\textbf{z}^\alpha=z_1^{\alpha_1}z_2^{\alpha_2}\cdots z_d^{\alpha_d}$.
Applying Cauchy-Schwartz yields
\begin{equation}\int_{\widehat{\mathcal D}}  \boldsymbol\xi^{2 \alpha}|\hat q(\boldsymbol\xi)| \mathrm d\boldsymbol\xi=\int_{\widehat{\mathcal D}}  \frac{\boldsymbol\xi^{2 \alpha}}{(1+|\boldsymbol\xi|^2)^{\ell/2}}(1+|\boldsymbol\xi|^2)^{\ell/2}|\hat q(\boldsymbol\xi)|\mathrm d\boldsymbol\xi\leq K(\alpha,\ell)\|q\|_{H^\ell(\mathcal D;\mathbb R^{m\times m})},
\end{equation}
in which we have introduced the expression
\begin{equation}
    K(\alpha,\ell) =
    \int_{\widehat{\mathcal D}}  \frac{\boldsymbol\xi^{4 \alpha}}{(1+|\boldsymbol\xi|^2)^{\ell}}\mathrm  d\boldsymbol\xi=\sum_{\xi\in\mathbb Z^{d-1}}\int_{\mathbb R}\frac{(\omega,\xi)^{4\alpha}}{(1+\omega^2+|\xi|^2)^\ell}\mathrm d\omega.
\end{equation}
We now claim that $K(\alpha,\ell)<\infty$, since  $\ell>2|\alpha|+d/2$ by assumption (Hq). To see this, note that
\begin{equation}
    K(\alpha,\ell)\leq \int_{\mathbb R^d}\frac{\textbf{z}^{4\alpha}}{(1+|\mathbf z|^2)^\ell}\mathrm d\mathbf z\leq \int_{\mathbb R^d}\frac{|\textbf{z}|^{4|\alpha|}}{(1+|\mathbf z|^2)^\ell}\mathrm d\mathbf z.
\end{equation}
Introducing the spherical coordinate $r=|z|$ with the associated integration factor $r^{d-1}$, we obtain
\begin{equation}
\begin{aligned}
    \int_{\mathbb R^d}\frac{|\textbf{z}|^{4|\alpha|}}{(1+|\mathbf z|^2)^\ell}\mathrm d\mathbf z&= V_d \int_{0}^\infty\frac{r^{4|\alpha|}}{(1+r^2)^\ell}r^{d-1}\mathrm dr\\&=V_d\int_0^\infty\frac{u^{2|\alpha|+d/2-1}}{(1+u)^\ell}\mathrm du\\&=V_d\frac{\Gamma(\ell-2|\alpha|-\tfrac d2)\Gamma(2|\alpha|+\tfrac d2)}{\Gamma(\ell)},
    \end{aligned}
\end{equation}
where $\Gamma$ denotes the Gamma-function and where $V_{d}={2 \pi^{\frac{d}{2}}/\Gamma(\frac{d}{2})}$ indicates the volume  of the unit $d$-sphere.

Taking $\alpha = 0$, we see that 
$\boldsymbol\xi\mapsto \sqrt{\hat q(\boldsymbol{\xi})}$ is in $L^2(\widehat{\mathcal D};\mathbb R^{m\times m})$.
This allows us to take the Fourier inverse, ensuring
that $p$ is a well-defined map.
Lastly, 
Plancherel's identity (see Appendix \ref{appendix:Fourier}) implies 
\begin{equation} 
\begin{aligned}\|p\|_{H^k(\mathcal D;\mathbb R^{m\times m})}^2&=\sum_{|\alpha| \leq k}\left\|\partial^\alpha p\right\|_{L^2(\mathcal D;\mathbb R^{m\times m})}^2\\&=\frac1{|\mathbb T|^{d-1}}\sum_{|\alpha|\leq k}\left\|\boldsymbol\xi\mapsto \boldsymbol\xi^{\alpha}\hat p(\boldsymbol\xi)\right\|_{L^2(\widehat{\mathcal D};\mathbb R^{m\times m})}^2\\&=\frac1{|\mathbb T|^{d-1}}\sum_{|\alpha|\leq k}\int_{\widehat{\mathcal D}} \boldsymbol \xi^{2 \alpha}|\hat q(\boldsymbol\xi)|\mathrm d\boldsymbol \xi, 
\end{aligned}
\end{equation}
confirming that indeed $p \in H^k(\mathcal D;\mathbb R^{m \times m})$.
\end{proof}

\begin{lemma}\label{lem:HS:z}
Pick $k \ge 0$ and suppose that \textnormal{(Hq)} is satisfied. Then 
there exists a constant $K > 0$ so that
any $z\in H^k(\mathcal D;\mathbb R^{n\times m})$ can be interpreted as a Hilbert-Schmidt operator from $L^2_Q$ into $H^k$ that acts via the pointwise multiplication $z[w](x,y)=z(x,y)w(x,y)$ and admits the bound
\begin{equation}
    \| z\|_{HS(L^2_Q ; H^k)}
    \le K \| z\|_{H^k(\mathcal D; \mathbb R^{n \times m}) }.
\end{equation}
\end{lemma}
\begin{proof} Without loss of generality we take $n=m=1$ to ease the notational complexity. 
Consider an orthonormal basis $
(e_\ell)_{\ell \geq 0}$ for $L^2(\mathcal{D}; \mathbb R)$
and introduce the functions
\begin{equation}
    p^{(\ell)}(x,y) = [\sqrt{Q} e_\ell](x,y)
     = \langle p( x - \cdot, y -\cdot) , e_{\ell}(\cdot, \cdot) \rangle_{L^2(\mathcal{D};\mathbb R)},
\end{equation}
together with the notation
\begin{equation}
    z_{\ell}(x,y) = z[\sqrt{Q}e_\ell](x,y)
    = z(x,y) [\sqrt{Q} e_\ell](x,y)
    = z(x,y) p^{(\ell)}(x,y).
\end{equation}
By definition, we have
\begin{align*}
\|z\|_{HS(L^2_Q;H^k(\mathcal D;\mathbb R))}^2&=\sum_{\ell =0}^\infty\|z_\ell\|_{H^k(\mathcal D; \mathbb R)}^2=\sum_{\ell =0}^\infty\sum_{|\alpha|\leq k}\int_{\mathcal D}|\partial^\alpha z_\ell(x,y)|^2 \, \mathrm dx \, \mathrm dy\\
    &=\sum_{\ell =0}^\infty\sum_{|\alpha|\leq k}\int_{\mathcal D}\left|\sum_{\beta\leq \alpha}\frac{\alpha!}{\beta!(\alpha-\beta)!}\partial ^\beta z(x,y) \partial ^{\alpha-\beta}p^{(\ell)}(x,y)\right|^2.
\end{align*}
Because $
(e_\ell)_{\ell \geq 0}$ is an orthonormal basis, and due to the  translation invariance of the integral, we obtain
$
     \sum_{\ell = 0}^\infty 
     | \partial^{\alpha-\beta} p^{(\ell)}(x,y) |^2 
     = \sum_{\ell = 0}^\infty 
     \langle \partial^{\alpha-\beta} p (x - \cdot, y - \cdot), e_{\ell}(\cdot, \cdot) \rangle_{L^2(\mathcal{D}; \mathbb R)}^2
     = \|\partial^{\alpha-\beta} p\|_{L^2(\mathcal D; \mathbb R)}^2,
$
which is independent of the coordinate $(x,y) \in \mathcal{D}$. This allows us to compute
\begin{equation}
\begin{aligned}
    \|z\|_{HS(L^2_Q;H^k(\mathcal D;\mathbb R))}^2
    &\leq (k!)^3\sum_{\ell =0}^\infty\sum_{|\alpha|\leq k}\sum_{\beta\leq \alpha}\int_{\mathcal D}|\partial^\alpha z(x,y)|^2|\partial^{\alpha-\beta}p^{(\ell)}(x,y)|^2 \, \mathrm dx \, \mathrm dy\\
    &= (k!)^3\sum_{|\alpha|\leq k}\sum_{\beta\leq \alpha}\int_{\mathcal D}|\partial^\alpha z(x,y)|^2\|\partial^{\alpha-\beta}p\|^2_{L^2(\mathcal D;\mathbb R)} \, \mathrm dx \, \mathrm dy\\
    &\leq (k!)^4\|p\|_{H^k(\mathcal D;\mathbb R)}^2\sum_{|\alpha|\leq k}\int_{\mathcal D}|\partial^\alpha z(x,y)|^2 \, \mathrm dx \, \mathrm dy\\
    &= (k!)^4\|p\|_{H^k(\mathcal D; \mathbb R)}^2\|z\|_{H^k(\mathcal D;\mathbb R)}^2,
\end{aligned}
\end{equation}
which in view of Lemma \ref{lem:p:in:hk} completes the proof.
\end{proof}

\subsection{Bounds in \texorpdfstring{$L^2$}{L2}}
\label{subsec:nl:l2}

The pointwise global Lipschitz bounds on our nonlinearities $f$, $g$ and $h$ will enable us to obtain estimates for the 
cut-offs $\chi_h$ and $\chi_l$ and the scalar functions $a_\sigma$, $b$ and $\kappa_\sigma$ that are entirely based on $L^2$-norms. In fact, several of the estimates obtained in previous work for $d=1$ do not rely on Sobolev embeddings and will carry over conveniently to the present context. We start by stating a basic consequence of the global Lipschitz bounds on $f$, $g$ and $h$.

\begin{lemma}
\label{lem:nw:l2:ests:f:g:h}
Pick $k =0$ and suppose that \textnormal{(Hf-Lip)}, \textnormal{(HSt)}, \textnormal{(HCor)} and \textnormal{(HPar)} are satisfied. Then there exists a constant $K > 0$, which does not
depend on the pair $(\Phi, c)$, so that the following holds true.
For any $v \in L^2$ 
we have the bounds
\begin{equation}
 \label{eq:nl:bnds:chi:f:g:h}
    \begin{array}{lcl}
    \|f(\Phi + 
         v)\|_{L^2(\mathcal{D}; \mathbb{R}^{n })}
    & \le & K [1 + \|v\|_{L^2}],
\\[0.2cm]
\|g(\Phi + 
         v)\|_{L^2(\mathcal{D}; \mathbb{R}^{n \times m})}
    & \le & K [1 + \|v\|_{L^2}],
\\[0.2cm]
\|h(\Phi + 
         v)\|_{L^2(\mathcal{D}; \mathbb{R}^{n})}
    & \le & K [1 + \|v\|_{L^2}],
\\[0.2cm]
    \end{array}
\end{equation}
while for any pair $v_A, v_B  \in L^2 $
we have the estimates
\begin{equation}
  \label{eq:est:chig:lip:bnds}
    \begin{array}{lcl}
    \|f(\Phi + 
         v_A) - f(\Phi + v_B)\|_{L^2(\mathcal{D}; \mathbb{R}^{n})}
    & \le & K \|v_A - v_B\|_{L^2},
\\[0.2cm]
\|g(\Phi + 
         v_A) - g(\Phi + v_B)\|_{L^2(\mathcal{D}; \mathbb{R}^{n \times m})}
    & \le & K \|v_A - v_B\|_{L^2},
\\[0.2cm]
    \|h(\Phi + 
         v_A) - h(\Phi + v_B)\|_{L^2(\mathcal{D}; \mathbb{R}^{n})}
    & \le & K \|v_A - v_B\|_{L^2}.
    \end{array}
\end{equation}
\end{lemma}
\begin{proof}
    This follows immediately from the fact that $f$, $g$ and $h$ are all globally Lipschitz, that $\Phi(x)$ approaches its spatial limits at an exponential rate as $x \to \pm \infty$, and that $\mathbb{T}^{d-1}$ is bounded.
\end{proof}
\begin{corollary} \label{cor:g:HS:L2}
    Pick $k =0$ and suppose that \textnormal{(Hf-Lip)}, \textnormal{(HSt)}, \textnormal{(HCor)} and \textnormal{(HPar)} are satisfied. Then there exists a constant $K > 0$, which does not
depend on the pair $(\Phi, c)$, so that the following holds true.
For any $v \in L^2$ 
we have the bounds
\begin{equation}
    \begin{array}{lcl}
\|g(\Phi + 
         v)\|_{HS(L^2_Q;L^2)}
    & \le & K [1 + \|v\|_{L^2}],
    \end{array}
\end{equation}
while for any pair $v_A, v_B  \in L^2 $
we have the estimates
\begin{equation}
    \begin{array}{lcl}
\|g(\Phi + 
         v_A) - g(\Phi + v_B)\|_{HS(L^2_Q;L^2)}
    & \le & K \|v_A - v_B\|_{L^2}.
    \end{array}
\end{equation}
\end{corollary}
\begin{proof}
    The assertion follows from Lemmas \ref{lem:HS:z} and \ref{lem:nw:l2:ests:f:g:h}.
\end{proof}

We proceed by considering the scalar cut-off functions
$\chi_h$ and $\chi_l$ defined
in \eqref{eq:list:def:chi:h:l}. In addition, we provide
$L^2_Q$-estimates for the 
auxiliary function
\begin{equation}
    \widetilde{\mathcal{K}}_C(u, \gamma) =
       \chi_l(u,\gamma)
         \chi_h(u,\gamma) Q g(u)^\top T_{\gamma} \psi_{\rm tw} ,
\end{equation}
which turn out to be highly convenient for our analysis of $b$ and $\mathcal{K}_C$ below.

\begin{lemma}
\label{lem:nl:chi:wtkc:l2:ests}
Pick $k=0$ and suppose that \textnormal{(HSt)}, \textnormal{(Hq)} and \textnormal{(HPar)} are satisfied. 
Then there exists a constant $K > 0$, which does not
depend on the pair $(\Phi, c)$, so that the following holds true.
For any $v \in L^2$ and $\gamma \in \mathbb{R}$
we have the bound
\begin{equation}
 \label{eq:nl:bnds:chi:hl:k:tw:l2}
    \begin{array}{lcl}
      |\chi_h(\Phi + v , \gamma)|
      + 
      |\chi_l(\Phi + v , \gamma)|
      + 
      \|\widetilde{\mathcal{K}}_C(\Phi + v,\gamma)\|_{L^2_Q} 
      & \le &  K.
    \end{array}
\end{equation}
In addition, for any pair $v_A, v_B  \in L^2$
and any pair $\gamma_A,\gamma_B \in \mathbb{R}$,
we have the estimates
\begin{equation}
  \label{eq:nl:est::lip:bnds:chi:l2}
    \begin{array}{lcl}
|\chi_h(\Phi + v_A , \gamma_A)
- \chi_h(\Phi + v_B, \gamma_B)|
      & \le &  K \big[ \|v_A - v_B\|_{L^2}
      + |\gamma_A - \gamma_B| \big],
\\[0.2cm]
    |\chi_l(\Phi + v_A , \gamma_A)
- \chi_l(\Phi + v_B, \gamma_B)|
      & \le &  K 
      \big[ \|v_A - v_B\|_{L^2}
       + (1 + \|v_A\|_{L^2}) | \gamma_A - \gamma_B | \big] ,
    \end{array}
\end{equation}
while the expression
\begin{equation}
\begin{array}{lcl}
    \Delta_{AB}\widetilde{\mathcal{K}}_C  
     & = & \widetilde{\mathcal{K}}_C(\Phi + v_A,\gamma_A)
        - \widetilde{\mathcal{K}}_C(\Phi + v_B,\gamma_B)
\end{array}
\end{equation}
satisfies the bound 
\begin{equation}
  \label{eq:nl:lip:bnd:delta:k:tilde:c:l2}
    \begin{array}{lcl}
\|\Delta_{AB} \widetilde{\mathcal{K}}_C\|_{L^2_Q} 
        & \le &
        K \big[  \|v_A - v_B\|_{L^2}
        +   (1 + \|v_A\|_{L^2})|\gamma_A - \gamma_B| \big] .
    \end{array}
\end{equation}
\end{lemma}
\begin{proof}
The results for $\chi_h$ and $\chi_l$
follow directly from their definitions in \eqref{eq:list:def:chi:h:l}, as in \cite[App. A]{hamster2020}. Turning to $\widetilde{\mathcal{K}}_C$, we note
that for any $z \in L^2(\mathcal{D};\mathbb{R}^{m \times n})$
and any $\psi \in L^\infty(\mathbb R;\mathbb R^n)$
we have
\begin{equation}
   \| Q z \psi \|_{L^2_Q}^2 = \langle Q z  \psi, z \psi \rangle_{L^2(\mathcal{D};\mathbb{R}^m)}
    \le \|q\|_{L^1(\mathcal{D};\mathbb{R}^{m \times m})} \|z\|^2_{L^2( \mathcal{D};\mathbb{R}^{m \times n} )}
    \|\psi\|_\infty^2,
\end{equation}
which upon taking $z = g(\Phi +v)^\top$ and $\psi = T_{\gamma} \psi_{\rm tw}$ shows that 
\begin{equation}
    \|\widetilde{\mathcal{K}}_C(\Phi + v,\gamma)\|_{L^2_Q} 
    \le K \| g(\Phi + v) \|_{L^2(\mathcal D; \mathbb R^{n \times m})}
\end{equation}
for any $\gamma \in \mathbb R$. 
We now observe
that whenever $\chi_h(\Phi + v, \gamma) \neq 0$ holds, the function
\begin{equation}
    \tilde{v} = \Phi + v - T_{\gamma} \Phi_{\rm ref}
\end{equation}
necessarily satisfies $\|\tilde{v} \|_{L^2 }\le 3 + \|\Phi_0 - \Phi_{\rm ref}\|_{L^2}$,  
which  leads to the uniform a priori bound
\begin{equation}
    \| g(\Phi + v) \|_{L^2(\mathcal D; \mathbb R^{n \times m})} =
    \| g(T_{\gamma} \Phi_{\rm ref} + \tilde{v}) \|_{L^2(\mathcal D; \mathbb R^{n \times m})} 
    \le C' 
\end{equation}
for some $C' > 0$
on account of the global Lipschitz smoothness of $g$ and the fact that the quantity $\|g(T_{\gamma} \Phi_{\rm ref})\|_{L^2(\mathcal D; \mathbb R^{n \times m}) }$ is finite (and independent of $\gamma$).

Turning to the Lipschitz bound \eqref{eq:nl:lip:bnd:delta:k:tilde:c:l2} in $L^2_Q$,
we first compute
\begin{equation}
\begin{aligned}
    \Delta_{AB} \widetilde{\mathcal{K}}_C & = 
    \big[\chi_l( \Phi + v_A, \gamma_A)
      - \chi_l( \Phi + v_B, \gamma_B)\big]
         Q \chi_h( \Phi + v_A, \gamma_A)  g(\Phi + v_A)^\top 
         T_{\gamma_A}\psi_{\rm tw}
    \\
    &  \qquad
     +\chi_l( \Phi + v_B, \gamma_B)
        Q ( \chi_h(\Phi+ v_A, \gamma_A)
        - \chi_h (\Phi + v_B, \gamma_B)  )g(\Phi + v_A)^\top
         T_{\gamma_A} \psi_{\rm tw} 
\\ 
    &\qquad
     +\chi_l( \Phi + v_B, \gamma_B)
        Q  
         \chi_h (\Phi + v_B, \gamma_B) 
         \big( g(\Phi +v_A)^\top - g(\Phi + v_B)^\top  \big)
         T_{\gamma_B} \psi_{\rm tw}  
\\
    &  \qquad
     +\chi_l( \Phi + v_B, \gamma_B)
        Q  
         \chi_h (\Phi + v_B, \gamma_B) 
          g(\Phi +v_B)^\top   
         \big(T_{\gamma_B} \psi_{\rm tw} -T_{\gamma_A} \psi_{\rm tw}   \big)
         .\label{eq:Delta:AB:KtildeC}
\end{aligned}
\end{equation}
Assuming without loss that $\chi_h(\Phi + v_A, \gamma_A) \neq 0$ (reverse $A$ and $B$ otherwise),
the result now follows by applying \eqref{eq:nl:est::lip:bnds:chi:l2} and the observations above, noting that $\psi_{\rm tw}'$ is uniformly bounded. 
\end{proof}

We are now ready to consider the function
$\mathcal{K}_C$ defined in \eqref{eq:list:def:wt:k:c},
which can be written as
\begin{equation}
\label{eq:nl:id:for:kc:wtkc}
    \mathcal{K}_C(u,\gamma) = - \chi_h(u,\gamma) g(u) \widetilde{\mathcal{K}}_C(u, \gamma),
\end{equation}
together with the function $b$ defined in \eqref{eq:b}.
Note that the computations in
\cite[App. A]{hamster2020} provide the 
convenient relationship
\begin{equation}
\label{eq:nl:id:for:b}
    \| {b}(\Phi + v, \gamma) \|^2_{HS(L^2_Q; \mathbb{R})}
     = \chi_h(\Phi + v , \gamma)^2 \| \widetilde{\mathcal{K}}_C(\Phi + v, \gamma)\|_{L^2_Q}^2.
\end{equation}

\begin{corollary}
\label{cor:nl:bnds:for:b:and:kc}
Pick $k=0$ and suppose that \textnormal{(HSt)}, \textnormal{(Hq)} and \textnormal{(HPar)} are satisfied. 
Then there exists a constant $K > 0$, which does not
depend on the pair $(\Phi, c)$, so that the following holds true.
For any $v \in L^2$ and $\gamma \in \mathbb{R}$
we have the bound
\begin{equation}
 \label{eq:nl:est:kc:b}
    \begin{array}{lcl}
\|\mathcal{K}_C(\Phi + v,\gamma)\|_{L^2}    
+ \|b(\Phi+v,\gamma)\|_{HS(L^2_Q;\mathbb R)}
      & \le & K.
    \end{array}
\end{equation}
In addition, for any pair $v_A, v_B  \in L^2$
and any pair $\gamma_A,\gamma_B \in \mathbb{R}$,
the expressions
\begin{equation}
\begin{array}{lcl}
    \Delta_{AB}\mathcal{K}_C  
     & = & \mathcal{K}_C(\Phi + v_A,\gamma_A)
        - \mathcal{K}_C(\Phi + v_B,\gamma_B),
\\[0.2cm]
\Delta_{AB} b  
     & = & b(\Phi + v_A,\gamma_A)
        - b(\Phi + v_B,\gamma_B)
\end{array}
\end{equation}
satisfy the bounds 
\begin{equation}
  \label{eq:nl:est:b:kc:lip:bnds}
    \begin{array}{lcl}
\|\Delta_{AB} \mathcal{K}_C\|_{L^2} 
        & \le &
        K \big[  \|v_A - v_B\|_{L^2}
        +   (1 + \|v_A\|_{L^2})|\gamma_A - \gamma_B| \big] ,
        \\[0.2cm]
        \|\Delta_{AB} b\|_{HS(L^2_Q;\mathbb R)}
        & \leq &
        K \big[   \|v_A-v_B\|_{L^2}
        + \big( 1+\|v_A\|_{L^2} \big) |\gamma_A-\gamma_B| \big].
    \end{array}
\end{equation}
\end{corollary}
\begin{proof}
The estimate \eqref{eq:nl:est:kc:b} follows directly from Lemma \ref{lem:nl:chi:wtkc:l2:ests},
the identities \eqref{eq:nl:id:for:kc:wtkc}--\eqref{eq:nl:id:for:b}, and from the general inequality
\begin{equation}
    \| \mathcal{K}_C(u,\gamma) \|_{L^2}
    \le  \| \chi_h(u,\gamma)  g(u) \|_{HS(L^2_Q; L^2)}
      \|\widetilde{\mathcal{K}}_C(u, \gamma) \|_{L^2_Q}.
\end{equation}
The computations in \cite[App. A]{hamster2020}
yield
\begin{equation}
\label{eq:nl:id:for:delta:b}
     \begin{aligned}
         \hspace{-.05cm}\|\Delta_{AB} b\|_{HS(L^2_Q;\mathbb R)}^2
         =
   \|\chi_h(\Phi+v_A,\gamma_A)\widetilde{\mathcal K}_C(\Phi+v_A,\gamma_A)-\chi_h(\Phi+v_B,\gamma_B)\widetilde{\mathcal K}_C(\Phi+v_B,\gamma_B)\|_{L^2_Q}^2,
     \end{aligned}
\end{equation}
which leads to  the Lipschitz bound for $\Delta_{AB} b$ in $L^2$ after appealing to Lemma \ref{lem:nl:chi:wtkc:l2:ests}. The estimate for $\Delta_{AB} \mathcal{K}_C$ follows analogously.
\end{proof}

We now consider the function
\begin{equation}
\kappa_\sigma(u,\gamma)=1+\frac{\sigma^2}{2}\|b(u,\gamma)\|^2_{HS(\mathcal W_Q ; \mathbb R)},
\end{equation}
together with the associated quantities $\nu_\sigma^{(\vartheta)}$
defined in \eqref{eq:list:def:nu}.
These  are the final ingredient required towards analysing the function $a_\sigma$ that describes the deterministic part of the evolution of the phase.

\begin{lemma}\label{lem:nl:kappa_sigma}
Pick $k=0$ together with $\theta \in \{-1,  -1/2, 1\}$ and suppose that \textnormal{(HSt)}, \textnormal{(Hq)} and \textnormal{(HPar)} are satisfied. 
Then there exists a constant $K > 0$, which does not
depend on the pair $(\Phi, c)$, so that the following holds true.
For any $\sigma \ge 0$, any $v \in L^2$ and any $\gamma \in \mathbb{R}$
we have the bound
    \begin{equation}
    \label{eq:nl:bnd:on:n:sigma}
        |\nu_\sigma^{(\theta)}(\Phi+v,\gamma)|\leq \sigma^2 K ,
    \end{equation}
In addition, for any $\sigma \ge 0$, any pair $v_A, v_B  \in L^2 $,
and any pair $\gamma_A,\gamma_B \in \mathbb{R}$,
we have the estimate
    \begin{equation}
    \label{eq:nl:bnd:lip:on:n:sigma}
        |\nu_\sigma^{(\theta)}(\Phi+v_A,\gamma_A)-\nu_\sigma^{(\theta)}(\Phi+v_B,\gamma_B)|\leq \sigma^2 K \big[ \|v_A-v_B\|_{L^2}
        +\big( 1+\|v_A\|_{L^2} \big) |\gamma_A-\gamma_B| \big] .
    \end{equation}
\end{lemma}
\begin{proof}
Using Corollary \ref{cor:nl:bnds:for:b:and:kc},
the bounds can be established as in the proof of 
\cite[Lem. 3.9]{hamster2019stability}.
\end{proof}

\begin{lemma}\label{lem:a_sigma}    Pick  $k=0$ and suppose that \textnormal{(Hf-Lip)}, \textnormal{(HSt)}, \textnormal{(HCor)}, \textnormal{(Hq)} and \textnormal{(HPar)}  all hold.  Then there exists a constant $K_a>0$, which does not depend on the pair $(\Phi,c)$, such that for any $v\in L^2$, any $\gamma\in\mathbb R$, and any $0 \le \sigma \le 1$, we have
    \begin{equation}
        |a_\sigma(\Phi+v,\gamma;c)|\leq K_a(1+\|v\|_{L^2}).\label{eq:bound_a}
    \end{equation}
        In addition, for any pair $v_A,v_B\in L^2$, any $\gamma_A,\gamma_B\in\mathbb R$, 
        and any $0 \le \sigma \le 1$, we have
    \begin{equation}\label{eq:nl:a:lip:l2}\begin{aligned}
|a_\sigma(\Phi+v_A,\gamma_A&;c)-a_\sigma(\Phi+v_B,\gamma_B;c)|\leq \\
&K_a\big[1+\|v_A\|_{L^2}^2\big]
\big[\|v_A-v_B\|_{L^2}+|\gamma_A-\gamma_B|\big].
\end{aligned}
    \end{equation}

\end{lemma}
\begin{proof}
Recalling the definition \eqref{eq:list:def:a}
and introducing the functions
\begin{align}\mathcal E_I(u,\gamma)&=\langle f(u),T_\gamma \psi_{\rm tw}\rangle _{L^2}+\sigma^2\langle h(u),T_\gamma \psi_{\rm tw}\rangle _{L^2},\\
\mathcal E_{II}(u,\gamma;c)&=c\langle u,T_\gamma \psi_{\rm tw}'\rangle _{L^2}+\sigma^2\langle \mathcal K_C(u,\gamma;c),T_\gamma \psi_{\rm tw}'\rangle _{L^2},\\\label{eq:focus}
\mathcal E_{III}(u,\gamma)&=\kappa_\sigma(u,\gamma)\langle u,T_\gamma \psi_{\rm tw}''\rangle _{L^2},\end{align}
together with the shorthand notation
\begin{equation}
    \mathcal{E}(u,\gamma;c) = \mathcal E_I(u,\gamma)-\mathcal E_{II}(u,\gamma;c)+\mathcal E_{III}(u,\gamma), 
\end{equation}
we see that
\begin{equation}
    a_\sigma(u,\gamma;c)=-\chi_l(u,\gamma) \mathcal{E}(u,\gamma;c). 
\end{equation}
In particular, we have
\begin{equation}
\begin{aligned}
        |a_\sigma(u_A,\gamma_A;c)-a_\sigma(u_B,\gamma_B;c)|\leq |\chi_l(u_A&,\gamma_A)-\chi_l(u_B,\gamma_B)\|\mathcal E(u_A,\gamma_A;c)|\\&+|\chi_l(u_B,\gamma_B)\|\mathcal E(u_A,\gamma_A;c)-\mathcal E(u_B,\gamma_B;c)|.\label{eq:a_sigma}
\end{aligned}
\end{equation}
Note that (HPar) implies
\begin{equation}
|c|+\|\Phi\|_{H^1(\mathbb R;\mathbb R^n)}\leq C_1,
\end{equation}
which using Lemmas \ref{lem:nw:l2:ests:f:g:h} and \ref{lem:nl:kappa_sigma}, Corollary \ref{cor:nl:bnds:for:b:and:kc}, and the fact that $\psi_{\rm tw}\in H^2(\mathbb R;\mathbb R^n)$, 
yields
\begin{equation}
\begin{aligned}
    |\mathcal E(\Phi+v,\gamma;c)|&\leq C_2(1+\|v\|_{L^2}).
    \end{aligned}
\end{equation}
Together with the uniform bound
\eqref{eq:nl:bnds:chi:hl:k:tw:l2} for $\chi_l$,
 we hence obtain \eqref{eq:bound_a} as well as an estimate for the first term of \eqref{eq:a_sigma} that can be absorbed in \eqref{eq:nl:a:lip:l2}.

We continue by only focusing on the term $\mathcal E_{III}$, remarking that $\mathcal E_{I}$ and $\mathcal E_{II}$ can be handled in a similar fashion.
Note that
\begin{equation}
\begin{aligned}
    &|\mathcal E_{III}(\Phi+v_A,\gamma_A)-\mathcal E_{III}(\Phi+v_B,\gamma_B)|\leq \\& \qquad\quad|\kappa_\sigma(u_A,\gamma_A)-\kappa_\sigma(u_B,\gamma_B)\|\langle u_A,T_{\gamma_A}\psi_{\rm tw}''\rangle_{L^2}|\\&\qquad\qquad\quad+|\kappa_\sigma(u_B,\gamma_B)\|\langle u_A,T_{\gamma_A}\psi_{\rm tw}''\rangle _{L^2}-\langle u_B,T_{\gamma_B}\psi_{\rm tw}''\rangle _{L^2}|, \label{eq:EIII}
\end{aligned}    
\end{equation}
where $|\kappa_\sigma(u,\gamma)|\leq C_3$
on account of \eqref{eq:nl:bnd:on:n:sigma}.
Appealing to \eqref{eq:nl:bnd:lip:on:n:sigma}, we arrive at a bound for the first term of \eqref{eq:EIII} that can be absorbed in \eqref{eq:nl:a:lip:l2}.

Since in fact $\psi_{\rm tw}\in H^3(\mathbb R;\mathbb R^n)$ holds,
we may use
\begin{equation}
    T_{\gamma_A}\psi_{\rm tw}-T_{\gamma_B}\psi_{\rm tw}=\int_0^{\gamma_B-\gamma_A}T_{\gamma_A+s}\psi'_{\rm tw} \mathrm ds
\end{equation}
to conclude
\begin{equation}
    \|T_{\gamma_A}\psi_{\rm tw}-T_{\gamma_B}\psi_{\rm tw}\|_{H^2(\mathbb R;\mathbb R^n)}\leq |\gamma_A-\gamma_B|\|\psi_{\rm tw}\|_{H^{3}(\mathbb R;\mathbb R^n)}.
\end{equation}
In order to find an estimate for the second part of \eqref{eq:EIII}, 
it suffices to compute
\begin{align}
   &|\langle u_A,T_{\gamma_A}\psi_{\rm tw}''\rangle _{L^2}-\langle u_B,T_{\gamma_B}\psi_{\rm tw}''\rangle _{L^2}|\nonumber \\
   & \quad\quad\leq|\langle u_A-u_B,T_{\gamma_B}\psi_{\rm tw}''\rangle _{L^2}|+|\langle u_A,T_{\gamma_A}\psi_{\rm tw}''-T_{\gamma_B}\psi_{\rm tw}''\rangle _{L^2}|\\
   &\quad\quad\leq \nonumber \|u_A-u_B\|_{L^2}\|\psi_{\rm tw}\|_{H^2(\mathbb R;\mathbb R^n)}+\|u_A\|_{L^2 }\|\psi_{\rm tw}\|_{H^3(\mathbb R;\mathbb R^n)}|\gamma_A-\gamma_B|,
\end{align}
completing the proof.
\end{proof}

\subsection{Stochastic terms}
\label{subsec:nl:st}

We here collect the estimates that we will need for the stochastic terms in our evolution systems. As a preparation, we provide bounds for $g$ in $H^1$, $H^k,$ and $H^{k+1}$, respectively, noting that the first and the last will be required in {\S}\ref{subsec:nl:det} below.\pagebreak

\begin{lemma}
\label{lem:nl:est:g:h1}
Pick $k > d/2$ and suppose that \textnormal{(HSt)} and \textnormal{(HPar)} are satisfied. Then there exists a constant $K > 0$, which does not depend on the pair $(\Phi,c)$, so that for any $v \in H^1$
we have the bound
\begin{equation}
\label{eq:nl:g:h1:bnd}
\begin{array}{lcl}
    \|  g(\Phi  + v) \|_{H^1(\mathcal D;\mathbb R^{n \times m})} &\le &
     K 
     ( 1 + \|v\|_{H^1} ),
     \end{array}
\end{equation}
while for any pair $v_A,v_B \in H^k$ we have the bound
\begin{equation}
\begin{array}{lcl}
    \|  g(\Phi  + v_A) -
        g(\Phi + v_B) \|_{H^1(\mathcal D; \mathbb R^{n \times m})} 
        & \le &  
     K (1 + \|v_A\|_{H^k} ) \|v_A - v_B\|_{H^1}
     \\[0.2cm]
     & & \qquad
     +\,K (1 + \|v_B\|_{H^1} ) \|v_A - v_B\|_{H^k}.
\end{array}
\end{equation}
\end{lemma}
\begin{proof}
These bounds follow directly by 
noting that $Dg$ is uniformly bounded and
    by inspecting the pointwise estimate \eqref{eq:nl:pw:est:for:partial:x}.
\end{proof}

\begin{corollary}
\label{cor:nl:bnd:g:hk}
Pick $k > d/2$ and suppose that \textnormal{(Hq)}, \textnormal{(HSt)} and \textnormal{(HPar)} are satisfied. Then for any $v \in H^k$ 
we have the bound
\begin{equation}
\label{eq:nl:bnd:g:abs:hs:hk}
\begin{array}{lcl}
\|g(\Phi+v)\|_{HS(L^2_Q; H^k)}
 &\le&   K(1+\|v\|_{H^k}^{k+1}),
\end{array}
\end{equation}
while for any pair $v_A, v_B \in H^k$ we have the estimate
\begin{equation}
\begin{array}{lcl}
    \|g(\Phi+v_A)-g(\Phi+v_B)\|_{HS(L^2_Q; H^k)}&\leq &K\big(1+\|v_A\|_{H^k}^k+\|v_B\|_{H^k}^k\big)\|v_A-v_B\|_{H^k}.
    \end{array}
\end{equation}
\end{corollary}
\begin{proof}
Using the bound \eqref{eq:bnd:theta:hk:additional},
    these estimates follow by combining
    Lemmas \ref{lem:Hk:func} and \ref{lem:HS:z}.
\end{proof}

\begin{corollary}
\label{cor:nl:bnd:g:hkp1}
Pick $k > d/2$ and suppose that \textnormal{(Hq)}, \textnormal{(HSt)} and \textnormal{(HPar)} are satisfied. Then there exists a constant $K > 0$, which does not depend on the pair $(\Phi,c)$, so that for any $v \in H^{k+1}$ 
we have the bound
\begin{equation}
\label{eq:nl:bnd:g:abs:hs:hkp1}
\begin{array}{lcl}
\|g(\Phi+v)\|_{HS(L^2_Q; H^{k+1})}
 &\le&   K(1+\|v\|_{H^k}^{k+1})
              (1+\|v\|_{H^{k+1}}),
\end{array}
\end{equation}
while for any pair $v_A, v_B \in H^{k+1}$ the expression
\begin{equation}
\begin{array}{lcl}
    \Delta_{AB}g &=& g(\Phi + v_A ) - g(\Phi + v_B)
    \end{array}
\end{equation}
satisfies the estimate
\begin{equation}
\begin{array}{lcl}
    \|\Delta_{AB} g\|_{HS(L^2_Q; H^{k+1})}
    &\leq & K(1+\|v_A\|_{H^k}^k+\|v_B\|_{H^k}^k)
    \\[0.2cm]
    & & \qquad \qquad \times
    (1+\|v_A\|_{H^{k+1}}+\|v_B\|_{H^{k+1}})
    \|v_A-v_B\|_{H^k}
\\[0.2cm]
& & \qquad
+ \,K(1+\|v_A\|_{H^k}^k+\|v_B\|_{H^k}^k)
  \|v_A-v_B\|_{H^{k+1}}.
\end{array}
\end{equation}
\end{corollary}
\begin{proof}
    Using the bound \eqref{eq:bnd:theta:hk:additional:hkp1}, these estimates follow by combining
    Lemmas \ref{lem:Hk:func:hkp1} and \ref{lem:HS:z}.
\end{proof}

We now turn to the functions $\mathcal{S}$ and $\mathcal{M}_\sigma$
that act on $\xi\in L^2_Q$ as 
\begin{equation}
\label{eq:nl:def:s:m}
\begin{array}{lcl}
\mathcal S(v;\Phi)[\xi] &=&g(\Phi+v)[\xi]+\partial_x(\Phi+v)b(\Phi + v, 0)[\xi],
\\[0.2cm]
\mathcal{M}_\sigma(v)[\xi] & = & \kappa_\sigma(\Phi_\sigma + v, 0)^{-1/2} \mathcal{S}(v;\Phi_\sigma)[\xi].
\end{array}
\end{equation}

\begin{corollary}
\label{cor:nl:bnd:fin:s:m}
Pick $k > d/2$ and suppose that \textnormal{(Hq)}, \textnormal{(HSt)} and \textnormal{(HPar)} are satisfied. Then there exists a constant $K > 0$, which does not depend on the pair $(\Phi,c)$, so that for any $v \in H^{k+1}$ 
and any sufficiently small $\sigma \ge 0$
we have the bounds
\begin{equation}
\begin{array}{lcl}
\|\mathcal{S}(v;\Phi)\|_{HS(L^2_Q; H^k)}
& \le &   K(1+\|v\|_{H^k}^{k+1} 
+ \|v\|_{H^{k+1}   }) ,
\\[0.2cm]
\|\mathcal{M}_\sigma(v)\|_{HS(L^2_Q; H^k)}
& \le &   K(1+\|v\|_{H^k}^{k+1} 
+ \|v\|_{H^{k+1}   }) .
 \end{array}
\end{equation}
\end{corollary}
\begin{proof}
    These estimates follow readily from 
    \eqref{eq:nl:bnd:g:abs:hs:hk}, the inequality in \eqref{eq:nl:bnd:on:n:sigma} for $\theta=-1/2$,
    and the bounds for $b$ in Corollary \ref{cor:nl:bnds:for:b:and:kc}.
\end{proof}

\subsection{Deterministic terms}
\label{subsec:nl:det}

Proceeding with the terms in the deterministic part of our evolution equations, our ultimate goal is to obtain a bound for the nonlinearity $\mathcal{N}_\sigma$ that will feature in our stability arguments. One of the main tasks is to track the dependence on the $H^{k+1}$-norm of our perturbation $v(t)$. Indeed, in the sequel we will only have integrated control over this norm, as opposed to the pointwise control that we will have over the $H^k$-norm. We start by providing basic bounds for the nonlinearity $f$ and correction term $h$.

\begin{corollary}
\label{cor:nl:est:f:h:hk}
Pick $k > d/2$ and suppose that \textnormal{(Hf-Lip)}, \textnormal{(HCor)} and \textnormal{(HPar)} are satisfied. Then there exists a constant $K > 0$, which does not depend on the pair $(\Phi,c)$, so that for any $v \in H^k$ 
we have the bound
\begin{equation}
\label{eq:nl:bnd:f:abs:hs:hk}
\|f(\Phi+v)\|_{ H^k}
+ \|h(\Phi+v)\|_{ H^k}
 \le   K(1+\|v\|_{H^k}^{k+1}),
\end{equation}
while for any pair $v_A, v_B \in H^k$ we have the estimate
\begin{equation}
    \|f(\Phi+v_A)-f(\Phi+v_B)\|_{ H^k}\leq K(1+\|v_A\|_{H^k}^k+\|v_B\|_{H^k}^k)\|v_A-v_B\|_{H^k} ,
\end{equation}
which also holds for $h$.
Finally, for any $v \in H^k$ we have the quadratic bound
\begin{equation}
\label{eq:nl:bnd:delta:f:drv}
    \|f(\Phi+v)-f(\Phi) - Df(\Phi)[v] \|_{H^k}\leq K(1+\|v\|_{H^k}^{k+1})\|v\|^2_{H^k}.
\end{equation}    
\end{corollary}
\begin{proof}
    These estimates follow from
    Lemmas  \ref{lem:Hk:func} and \ref{lem:Hk:func:quadr}
    together with the bound \eqref{eq:bnd:theta:hk:additional}.
\end{proof}

We proceed by studying the term $\mathcal{K}_C$, which we must understand in $H^{k+1}$. For any multi-index $\beta \in \mathbb{Z}^{d}_{\ge 0}$ with $|\beta| =1$, we will use the estimate
\begin{equation}
\label{eq:est:splitting:for:tc}
\begin{array}{lcl}
    \| \partial^\beta \mathcal{K}_C(u,\gamma) \|_{H^k}
    & \le & \| \chi_h(u,\gamma) \partial^\beta  g(U) \|_{HS(L^2_Q; H^k)}
      \|\widetilde{\mathcal{K}}_C(u, \gamma) \|_{L^2_Q}
\\[0.2cm]
  & & \qquad
      + \,\|\chi_h(u,\gamma) g(U) \|_{HS(L^2_Q;H^k)}
      \|\partial^\beta \widetilde{\mathcal{K}}_C(u, \gamma) \|_{L^2_Q},
\end{array}
\end{equation}
together with its natural analogue for 
$\partial^\beta [\mathcal{K}_C(u_A,\gamma) - \mathcal{K}_C(u_B,\gamma)]$. We start by considering the
term $\partial^\beta \widetilde{K}_C$, for which it suffices to understand $g$ in $H^1$.

\begin{lemma}
\label{lem:est:bnd:wtc:deriv}
Pick $k > d/2$ together with a multi-index $\beta \in \mathbb{Z}^{d}_{\ge 0}$ with $|\beta| = 1$
and suppose that \textnormal{(Hq)}, \textnormal{(HSt)} and \textnormal{(HPar)} are satisfied.
Then there exists a constant $K > 0$,
which does not
depend on the pair $(\Phi, c)$, so that the following holds true.
For any $v \in H^1$ and $\gamma \in \mathbb{R}$
we have the bound
\begin{equation}
 \label{eq:est:wtc:l2:bnds}
    \begin{array}{lcl}
     \|\partial^\beta \widetilde{\mathcal{K}}_C(\Phi + v,\gamma)\|_{L^2_Q}  & \le &
        K 
        ( 1 + \|v\|_{H^1} ) ,
    \end{array}
\end{equation}
while for any pair $v_A, v_B  \in H^k$
and $\gamma \in \mathbb{R}$,
the expression
\begin{equation}\begin{array}{lcl}
    \Delta_{AB}\widetilde{\mathcal{K}}_C  
     &=& \widetilde{\mathcal{K}}_C(\Phi + v_A,\gamma)
        - \widetilde{\mathcal{K}}_C(\Phi + v_B,\gamma)
        \end{array}
\end{equation}
satisfies the estimate
\begin{equation}
  \label{eq:est:wtc:lip:bnds:h1}
    \begin{array}{lcl}
       \|\partial^\beta \Delta_{AB}\widetilde{\mathcal{K}}_C\|_{L^2_Q}
       & \le &
      K 
      (1 + \|v_A\|_{H^1}) \|v_A - v_B\|_{L^2}
      \\[0.2cm]
      &  & \qquad
      + \,K (1 + \|v_A\|_{H^k} ) \|v_A - v_B\|_{H^1}
      \\[0.2cm]
      & & \qquad
     +\,K (1 + \|v_B\|_{H^1} ) \|v_A - v_B\|_{H^k}
    \end{array}
\end{equation}
\end{lemma}
\begin{proof}
Note first
that for any $z \in H^1(\mathcal{D},\mathbb{R}^{m \times n})$ and any $\gamma\in\mathbb R$
we have
\begin{equation} 
\begin{array}{lcl}
   \| \partial^\beta Q zT_\gamma\psi_{\rm tw}\|_{L^2_Q}^2
   &=& \|  Q \partial^\beta [zT_\gamma\psi_{\rm tw}]\|_{L^2_Q}^2
   \\[.2cm]& \le & \|q\|_{L^1(\mathcal{D};\mathbb{R}^{m \times m})} \|\partial_\xi [zT_\gamma\psi_{\rm tw}]\|_{L^2(\mathcal{D};\mathbb{R}^m)}^2
   \\[0.2cm]
    & \le & \|q\|_{L^1(\mathcal D;\mathbb{R}^{m \times m})} \|z\|^2_{H^1(\mathcal{D};\mathbb{R}^{m \times n})}
    \big[\|\psi_{\rm tw}\|_\infty + \|\psi_{\rm tw}'\|_\infty\big]^2 .
\end{array}    
\end{equation}
The bound \eqref{eq:est:wtc:l2:bnds} hence follows
directly from \eqref{eq:nl:g:h1:bnd}.
Turning to the  estimate
\eqref{eq:est:wtc:lip:bnds:h1}, 
we recall the identity \eqref{eq:Delta:AB:KtildeC} and set $\gamma=\gamma_A=\gamma_B$.
The stated bound can now be obtained by combining
Lemmas \ref{lem:nl:chi:wtkc:l2:ests} and \ref{lem:nl:est:g:h1}.
\end{proof}

\begin{lemma}
\label{lem:nl:bnd:on:k:c:hkp1}
Pick $k > d/2$ 
and suppose that \textnormal{(Hq)}, \textnormal{(HSt)} and \textnormal{(HPar)} are satisfied.
Then there exists a constant $K > 0$,
which does not
depend on the pair $(\Phi, c)$, so that the following holds true. For any $v \in H^{k+1}$ and $\gamma \in \mathbb{R}$
we have the bound
\begin{equation}
\label{eq:nl:kc:H1:bnd}
\begin{array}{lcl}
    \|\mathcal{K}_C(\Phi + v, \gamma)\|_{H^{k+1}} 
    & \le & 
    K(1 + \|v\|^{k+1}_{H^k})
    ( 1 + \|v\|_{H^{k+1}} 
    ),
\end{array}
\end{equation}
while for any pair $v_A, v_B  \in H^{k+1} $
and $\gamma \in \mathbb{R}$,
the expression
\begin{equation}
    \Delta_{AB} \mathcal{K}_C = \mathcal{K}_C(\Phi + v_A, \gamma)
     - \mathcal{K}_C(\Phi + v_B, \gamma)
\end{equation}
satisfies the estimate
\begin{equation}
\begin{array}{lcl}
    \|\Delta_{AB} \mathcal{K}_C \|_{H^{k+1}}
& \le & 
K(1+\|v_A\|_{H^k}^{k+1}+\|v_B\|_{H^k}^{k+1})
    \\[0.2cm]
    & & \qquad \qquad \times
    (1+
    \|v_A\|_{H^{k+1}}+\|v_B\|_{H^{k+1}})
    \|v_A-v_B\|_{H^k}
\\[0.2cm]
& & \qquad
+ \,K(1+\|v_A\|_{H^k}^{k}+\|v_B\|_{H^k}^{k})
  \|v_A-v_B\|_{H^{k+1}}.
\end{array}
\end{equation}
\end{lemma}
\begin{proof}
The bound \eqref{eq:nl:kc:H1:bnd}
follows from the decomposition \eqref{eq:est:splitting:for:tc}, 
using
\eqref{eq:nl:bnds:chi:hl:k:tw:l2},
\eqref{eq:nl:bnd:g:abs:hs:hk}, \eqref{eq:nl:bnd:g:abs:hs:hkp1},
and  \eqref{eq:est:wtc:l2:bnds}.
Turning to the Lipschitz estimate in $H^k$, we 
pick a multi-index  $\beta \in \mathbb{Z}^{d-1}_{\ge 0}$ and consider the splitting
\begin{equation}
\begin{array}{lcl}\label{eq:splitting:EIEII}
    \partial^\beta \Delta_{AB}  \mathcal{K}_C
    & = & \mathcal{E}_I + \mathcal{E}_{II},
\end{array}
\end{equation}
in which we have defined
\begin{equation}
\label{eq:nl:bnd:kc:def:e:i}
\begin{array}{lcl}
\mathcal{E}_I &  = &     
    (\chi_h(\Phi + v_A, \gamma) - \chi_h(\Phi + v_B, \gamma))
    [\partial^\beta g(\Phi + v_A) ] \widetilde{\mathcal{K}}_C(\Phi + v_A, \gamma)
\\[0.2cm]
& & \qquad
 + \chi_h(\Phi + v_B)
 [ \partial^\beta g(\Phi + v_A) - \partial^\beta g(\Phi + v_B)]
  \widetilde{\mathcal{K}}_C(\Phi + v_A, \gamma)
\\[0.2cm]
& & \qquad
 + \chi_h(\Phi+ v_B)  \partial^\beta g(\Phi + v_B) \big[
  \widetilde{\mathcal{K}}_C(\Phi+ v_A, \gamma) -  \widetilde{\mathcal{K}}_C(\Phi + v_B, \gamma) \big],
\\[0.2cm]
\end{array}
\end{equation}
together with
\begin{equation}
\label{eq:nl:bnd:kc:def:e:ii}
\begin{array}{lcl}
\mathcal{E}_{II} &  = &     
 (\chi_h(\Phi + v_A, \gamma) - \chi_h(\Phi + v_B, \gamma))
    [g(\Phi + v_A) ] \partial^\beta \widetilde{\mathcal{K}}_C(\Phi + v_A, \gamma)
\\[0.2cm]
& & \qquad
 + \chi_h(\Phi + v_B)
 [ g(\Phi + v_A) -  g(\Phi + v_B)]
  \partial^\beta \widetilde{\mathcal{K}}_C(\Phi + v_A, \gamma)
\\[0.2cm]
& & \qquad
 + \chi_h(\Phi+ v_B)  g(\Phi + v_B)
  [\partial^\beta \widetilde{\mathcal{K}}_C(\Phi+ v_A, \gamma) -  \partial^\beta \widetilde{\mathcal{K}}_C(\Phi + v_B, \gamma)].
\\[0.2cm]
\end{array}
\end{equation}
Using decompositions analogous to 
\eqref{eq:est:splitting:for:tc},
the bounds in Lemmas \ref{lem:nl:chi:wtkc:l2:ests}
and \ref{lem:est:bnd:wtc:deriv},
and Corollaries \ref{cor:nl:bnd:g:hk} and
\ref{cor:nl:bnd:g:hkp1}, lead to expressions that can all be absorbed in the stated estimate. Indeed, we obtain
\begin{equation}
\begin{array}{lcl}
\|\mathcal{E}_I \|_{H^k} & \le &
    K  ( 1 + \|v_A\|^{k+1}_{H^k}
    + \|v_B\|^{k+1}_{H^k})(1 + \|v_A\|_{H^{k+1}}
     + \|v_B\|_{H^{k+1}} )
     \|v_A-v_B\|_{L^2} 
\\[0.2cm]  
& & \qquad
+  \, K(1+\|v_A\|_{H^k}^k+\|v_B\|_{H^k}^k)
    (1+\|v_A\|_{H^{k+1}}+\|v_B\|_{H^{k+1}})
    \|v_A-v_B\|_{H^k}
\\[0.2cm]
& & \qquad
+ \,K(1+\|v_A\|_{H^k}^k+\|v_B\|_{H^k}^k)
  \|v_A-v_B\|_{H^{k+1}},
\end{array}
\end{equation}
together with
\begin{equation}
\begin{array}{lcl}
\|\mathcal{E}_{II}\|_{H^k}
& \le &     
 K(1 + \|v_A\|^{k+1}_{H^k}
  + \|v_B\|^{k+1}_{H^k})
    ( 1 + \|v_A\|_{H^1})
    \|v_A - v_B\|_{L^2} 
\\[0.2cm]
& & \qquad
+\, K ( 1 + \|v_A\|^{k+1}_{H^k} + \|v_B\|^{k+1}_{H^k}) )
(1 + \|v_A\|_{H^1} + \|v_B\|_{H^1})
\|v_A - v_B\|_{H^k} 
\\[0.2cm]
      &  & \qquad
      + \, K ( 1 + \|v_B\|_{H^k}^{k+1})
            (1 + \|v_A\|_{H^k} ) \|v_A - v_B\|_{H^1},
\end{array}
\end{equation}
which proves the assertion.
\end{proof}

Upon introducing the expressions
\begin{equation}
\label{eq:nl:def:xi:i:ii}
    \Xi_{I}(u, c) = f(u) + c \partial_x u,
    \qquad \qquad
    \Xi_{II;\sigma}(u,\gamma) = \sigma^2 [h(u) + \partial_x \mathcal{K}_C(u,\gamma)],
\end{equation}
we point out that the function $\mathcal{J}_\sigma$
defined in \eqref{eq:list:def:j:sigma:new}
can be written in the form
\begin{align}\mathcal J_\sigma(u,\gamma;c)&=\kappa_\sigma(u,\gamma)^{-1}
\big[ \Xi_{I}(u, c) + \Xi_{II;\sigma}(u, \gamma) \big].
\end{align}
The estimates above can be used to formulate convenient
bounds for $\Xi_{I}$ and $\Xi_{II;\sigma}$.

\begin{corollary}
\label{cor:nl:bnds:xi:i:ii}
Pick $k > d/2$ and suppose that \textnormal{(Hf-Lip)}, \textnormal{(HCor)}, \textnormal{(Hq)}, \textnormal{(HSt)} and \textnormal{(HPar)} are satisfied.
Then there exists a constant $K > 0$,
which does not
depend on the pair $(\Phi, c)$, so that the following holds true. For any $v \in H^{k+1}$, any $\gamma \in \mathbb{R}$, and any $\sigma \ge  0$,
we have the bounds
\begin{equation}
    \begin{array}{lcl}
    \langle \Xi_{I}(\Phi + v,c), \psi_{\rm tw}
    \rangle_{L^2}
    & \le & 
       K(1 + \|v\|_{L^2}),
    \\[0.2cm]
    \langle  \Xi_{II;\sigma}(\Phi + v,\gamma)  , \psi_{\rm tw} \rangle_{L^2}
    & \le & K \sigma^2  
    (1 + \|v\|_{L^2}  ),
    \end{array}
\end{equation}
together with
\begin{equation}
    \begin{array}{lcl}
    \| \Xi_{I}(\Phi + v,c) \|_{H^k}
    & \le & 
       K\big(1 + \|v\|_{H^k}^{k+1}  + \|v\|_{H^{k+1}}\big) ,
    \\[0.2cm]
    \| \Xi_{II;\sigma}(\Phi + v,\gamma) \|_{H^k}
    & \le & K \sigma^2  
    (1 + \|v\|_{H^k}^{k+1}  )( 1 + 
    \|v\|_{H^{k+1}} 
    ).
    \end{array}
\end{equation}
In addition, for any pair $v_A, v_B  \in H^{k+1}$,
any $\gamma \in \mathbb{R}$,
and any  $\sigma \ge  0$,
the expressions
\begin{equation}
\begin{array}{lcl}
    \Delta_{AB} \Xi_I
     &= & \Xi_I(\Phi + v_A, c)
     - \Xi_I(\Phi + v_B, c),
\\[0.2cm]
    \Delta_{AB} \Xi_{II;\sigma}
     &= & \Xi_{II;\sigma}(\Phi + v_A, \gamma)
     - \Xi_{II;\sigma}(\Phi + v_B, \gamma),
\end{array}
\end{equation}
satisfy the estimates
\begin{equation}
    \begin{array}{lcl}
    \langle \Delta_{AB} \Xi_{I} , \psi_{\rm tw} \rangle_{L^2}
    & \le & 
      K \|v_A - v_B\|_{L^2} ,
    \\[0.2cm]
    \langle \Delta_{AB} \Xi_{II;\sigma}, \psi_{\rm tw} 
    \rangle_{L^2}
    & \le & 
    \sigma^2 K
    \|v_A-v_B\|_{L^2} ,
\\[0.2cm]
    \end{array}
\end{equation}
together with
\begin{equation}
    \begin{array}{lcl}
    \| \Delta_{AB} \Xi_{I} \|_{H^k}
    & \le & 
      K (1 + \|v_A\|^k_{H^k} +\|v_B\|_{H^k}^k) \|v_A - v_B\|_{H^k}
      + 
      K \| v_A - v_B\|_{H^{k+1}},
    \\[0.2cm]
    \| \Delta_{AB} \Xi_{II;\sigma} \|_{H^k}
    & \le & 
    \sigma^2 K(1+\|v_A\|_{H^k}^{k+1}+\|v_B\|_{H^k}^{k+1})
    \\[0.2cm]
    & & \qquad \qquad \times
    (1+
    \|v_A\|_{H^{k+1}}+\|v_B\|_{H^{k+1}})
    \|v_A-v_B\|_{H^k}
\\[0.2cm]
& & \qquad
+ \,\sigma^2 K(1+\|v_A\|_{H^k}^k+\|v_B\|_{H^k}^k)
  \|v_A-v_B\|_{H^{k+1}}.
    \end{array}
\end{equation}
\end{corollary}
\begin{proof}
Inspecting the
    definitions \eqref{eq:nl:def:xi:i:ii},
    the $L^2$-bounds follow from
    Lemma \ref{lem:nw:l2:ests:f:g:h} and Corollary \ref{cor:nl:bnds:for:b:and:kc} after transferring the spatial derivatives onto $\psi_{\rm tw}$,
    while the $H^k$-bounds follow from
    Corollary \ref{cor:nl:est:f:h:hk}
    and Lemma \ref{lem:nl:bnd:on:k:c:hkp1}.
\end{proof}

We are now ready to consider the
nonlinearity $\mathcal{N}_\sigma$
defined in \eqref{eq:app:def:n:m:sigma}. It is exactly at this point where we need the pair $(\Phi,c)=(\Phi_\sigma,c_\sigma)$ to be the instantaneous stochastic waves, which requires $\sigma\geq 0$ to be sufficiently small.
In particular,
for any $v \in H^{k+1}$
we will use the representation
\begin{equation}
\label{eq:nl:repr:n:sigma:final}
 \mathcal{N}_\sigma(v)
  = \mathcal{N}_{I;\sigma}(v) 
  -\chi_l(\Phi_\sigma + v,\gamma)
  \langle \mathcal{N}_{I;\sigma}(v) ,  \psi_{\rm tw} \rangle_{L^2} [ \partial_x \Phi_\sigma + \partial_x v],
\end{equation}
which involves the intermediate function
\begin{equation}
    \mathcal{N}_{I;\sigma}(v) =  
    \Phi_\sigma'' + \mathcal{J}_\sigma(\Phi_\sigma + v, 0;c_\sigma)
        - c_0 \partial_x v - Df(\Phi_0) v;
\end{equation}
see \eqref{eq:app:def:n:i:sigma} and \eqref{eq:app:repr:n:sigma:final}.
By construction, we have $\Phi_\sigma'' + \mathcal{J}_\sigma(\Phi_\sigma;0;c_\sigma) =0$,
which allows us to write
\begin{equation}
\label{eq:nl:decomp:for:n:i:sigma}
\begin{array}{lcl}
    \mathcal{N}_{I;\sigma}(v) 
    &=&  \mathcal{J}_\sigma(\Phi_\sigma + v, 0;c_\sigma) - \mathcal{J}_\sigma(\Phi_\sigma, 0;c_\sigma)
    -c_0 \partial_x v - Df(\Phi_0) v 
\\[0.2cm]
& = & 
    f(\Phi_\sigma + v)-  f(\Phi_\sigma) - Df(\Phi_\sigma) v
    + \big(Df(\Phi_\sigma)- Df(\Phi_0)\big) v + (c_\sigma - c_0) \partial_x v
\\[0.2cm]
& & \qquad
 + \,\kappa_\sigma(\Phi_\sigma, 0)^{-1} \big( \Xi_{II;\sigma}(\Phi_\sigma + v, 0) - \Xi_{II;\sigma}(\Phi_\sigma ,0) \big)
 \\[0.2cm]
& & \qquad
    +\, (\nu^{(-1)}_\sigma( \Phi_\sigma + v,0) - \nu^{(-1)}_{\sigma} (\Phi_\sigma))
    \big[\Xi_I( \Phi_\sigma + v,c_\sigma) + \Xi_{II;\sigma}(\Phi_\sigma+ v,0) \big]
\\[0.2cm]
& & \qquad
        + \,\nu^{(-1)}_\sigma(\Phi_\sigma,0) \big(\Xi_{I}(\Phi_\sigma+ v,c_\sigma)- \Xi_{I}(\Phi_\sigma,c_\sigma)\big),
\end{array}
\end{equation}
recalling the function $\nu^{(-1)}_\sigma(u,\gamma) = \kappa_\sigma^{-1}(u,\gamma) - 1$.

\begin{corollary}
\label{cor:nl:est:n:final}
Pick $k > d/2$ and suppose that \textnormal{(Hf-Lip)}, \textnormal{(HCor)}, \textnormal{(Hq)} and \textnormal{(HSt)}  are satisfied.
Then there exists a constant $K > 0$
so that for all 
sufficiently small 
$\sigma \ge 0$
and all $v \in H^{k+1}$
we have the estimates
\begin{equation}
\begin{array}{lcl}
\langle \mathcal{N}_{I;\sigma}(v) , \psi_{\rm tw} 
\rangle_{L^2} & \le  &
K \|v\|_{L^2} \|v\|_{H^k}
+ \sigma^2  K \|v\|_{L^2},
\\[0.2cm]
    \|\mathcal{N}_{I;\sigma}(v)\|_{H^k}
    & \le & 
    K (1 + \|v\|^{k+1}_{H^{k}})\|v\|_{H^k}^2
        + \sigma^2 K \|v\|_{H^{k+1}} \big( 1 + \|v\|_{H^k}^{k+2} \big).
\end{array}
\end{equation}
\end{corollary}
\begin{proof}
The $L^2$-inner product bound follows from 
Lemma \ref{lem:nl:kappa_sigma} and
Corollary \ref{cor:nl:bnds:xi:i:ii},
together with the estimate $\|v^2\|_{L^2} \le \|v\|_{L^2}\|v\|_{H^k}$.
The $H^k$-bounds follow in a similar fashion,
using \eqref{eq:nl:bnd:delta:f:drv} to obtain the intermediate estimate
\begin{equation}
\begin{array}{lcl}
    \|\mathcal{N}_{I;\sigma}(v)\|_{H^k}
    & \le & 
    K (1 + \|v\|^{k+1}_{H^k})\|v\|_{H^k}^2
    + \sigma^2 K \|v\|_{H^k}\big( 1 + \|v\|_{H^k}^{k+1})
    \\[0.2cm]
    & & \qquad
    + \,\sigma^2 K \|v\|_{H^{k+1}} \big( 1 + \|v\|_{H^k}^{k+2} \big),
\end{array}
\end{equation}
which can be absorbed in the stated bound.
\end{proof}

\begin{corollary}
\label{cor:nl:ct:fin:bnd:n:sigma}
Pick $k > d/2$ and suppose that \textnormal{(Hf-Lip)}, \textnormal{(HCor)}, \textnormal{(Hq)} and \textnormal{(HSt)}  are satisfied.
Then there exists a constant $K > 0$
so that for all 
sufficiently small 
$\sigma \ge 0$
and all $v \in H^{k+1}$
we have the estimate
\begin{equation}
\begin{array}{lcl}
    \|\mathcal{N}_{\sigma}(v)\|_{H^k}
    & \le & 
    K \|v\|_{H^k}^2 (1 + \|v\|^{k}_{H^k} + \|v\|_{H^{k+1}})
        + \sigma^2 K \|v\|_{H^{k+1}} \big( 1 + \|v\|_{H^k}^{k+2} \big).
\end{array}
\end{equation}
\end{corollary}
\begin{proof}
Inspecting the representation
\eqref{eq:nl:repr:n:sigma:final},
this bound follows  from
\eqref{eq:nl:bnds:chi:hl:k:tw:l2}
and Corollary 
\ref{cor:nl:est:n:final}.
\end{proof}

\subsection{Low dimensional bounds}
\label{subsec:nl:low}

In this section, we consider the case $1 \le d \le 4$ 
and search for bounds in $H^1$.
We will make frequent use of the Sobolev embeddings
\begin{equation}
\label{eq:nl:low:sob:emb:i}
H^{1} \hookrightarrow L^4,
\qquad
H^{4/3}\hookrightarrow L^6,
\qquad
H^{3/2}\hookrightarrow L^8,
\end{equation}
and
\begin{equation}
\label{eq:nl:low:sob:emb:ii}
H^{3/2}\hookrightarrow W^{1,8/3},
\qquad
H^{7/4}\hookrightarrow L^{16},
\qquad
H^{2}\hookrightarrow W^{1,4}.
\end{equation}
In addition, we often encounter the interpolation bounds
\begin{equation}
\label{eq:nl:low:intp:bnds}
\begin{array}{lcl}
     \|v\|_{H^{4/3}} 
     &\le& K \|v\|_{H^1}^{2/3}\|v\|_{H^2}^{1/3},
     \\[0.2cm]
 \|v\|_{H^{3/2}} & \le & K \|v\|_{H^1}^{1/2} \|v\|_{H^2}^{1/2},
 \\[0.2cm]
 \|v\|_{H^{7/4}} & \le & K \|v\|_{H^1}^{1/4} \|v\|_{H^2}^{3/4}.
\end{array}
\end{equation}

\begin{lemma}
\label{lem:nl:low:lip:theta:h1:h1p5:h2}
Pick $1 \le d \le 4$, assume that $\Phi$ is bounded with $\Phi' \in H^2$, and consider a $C^1$-smooth function $\Theta: \mathbb R^n \to \mathbb R^N$ for which $D^{\ell}\Theta$ is globally Lipschitz for all $0 \le \ell \le 1$.
Then there exists a constant $K > 0$ so that
for any pair $v_A,v_B \in H^{3/2}$ we have the bound
\begin{equation}\label{eq:nl:low:lip:theta:h1:h1p5}
     \|\Theta(\Phi+v_A)-\Theta(\Phi+v_B)\|_{H^1(\mathcal D;\mathbb R^N)}\leq
     K  (1+\|v_B\|_{H^{3/2}})\|v_A-v_B\|_{H^{3/2}}.
\end{equation}
If in addition $v_B \in H^2$, then we have
     \begin{equation}\label{eq:nl:low:lip:theta:h1:h2}
     \|\Theta(\Phi+v_A)-\Theta(\Phi+v_B)\|_{H^1(\mathcal D;\mathbb R^N)}\leq K (1+\|v_B\|_{H^2})\|v_A-v_B\|_{H^1}.
     \end{equation}
\end{lemma}
\begin{proof}
The bound is clear in $L^2$. Inspecting the pointwise estimate \eqref{eq:nl:pw:est:for:partial:x},
the bound \eqref{eq:nl:low:lip:theta:h1:h1p5} follows from applying  H\"older's inequality which yields
\begin{equation}
\|(v_A-v_B)v_B'\|_{L^2} \le  \|(v_A-v_B)\|_{L^8} \|v_B'\|_{L^{8/3}}   
\end{equation}
and by appealing to the Sobolev embeddings
\eqref{eq:nl:low:sob:emb:i}--\eqref{eq:nl:low:sob:emb:ii}.
On the other hand, the estimate \eqref{eq:nl:low:lip:theta:h1:h2}
follows from 
\begin{equation}
    \|(v_A-v_B)v_B'\|_{L^2}\leq \|v_A-v_B\|_{L^4}\|v_B'\|_{L^4}
\end{equation}
and using the embeddings $H^{1}\hookrightarrow L^4$
and $H^{2}\hookrightarrow W^{1,4}$.
\end{proof}

Observe, in the next result  we are restricting
our Lipschitz estimate to the special case $v_A=v$ and 
$v_B = 0$. This is on account
of the pointwise term $|v_A-v_B| |\partial^{\gamma_1} v_B| |\partial^{\gamma_2} v_B|$ arising from the third line of \eqref{eq:nl:low:g:in:h2:sec:drv:expr},
which requires higher order norms.

\begin{lemma}
\label{lem:nl:low:ests:lip:theta:h2:h2}
Pick $1 \le d \le 4$, assume that $\Phi$ is bounded with $\Phi' \in H^2$, and consider a $C^2$-smooth function $\Theta: \mathbb R^n \to \mathbb R^N$ for which $D^{\ell}\Theta$ is globally Lipschitz for all $0 \le \ell \le 2$.
Then there exists a constant $K > 0$ so that
for any $v \in H^{2}$ we have the bound
    \begin{equation}\label{eq:nl:low:lip:theta:h2:h2}
     \|\Theta(\Phi+v)-\Theta(\Phi)\|_{H^2(\mathcal D;\mathbb R^N)}\leq K (1 + \|v\|_{H^2})\|v\|_{H^2} .
     \end{equation}
\end{lemma}
\begin{proof}
Pick two multi-indices $\gamma_1, \gamma_2 \in \mathbb{Z}^d_{\ge 0}$ with $|\gamma_1| = | \gamma_2 | = 1$. Writing
\begin{equation}
    \Delta_{\gamma_1, \gamma_2} \Theta = \partial^{\gamma_1 + \gamma_2} \big[ \Theta(\Phi + v_A) - \Theta(\Phi + v_B) \big],
\end{equation}
with $v_A,v_B\in H^2$, we may compute
\begin{equation}
\label{eq:nl:low:g:in:h2:sec:drv:expr}
\begin{array}{lcl}
  \Delta_{\gamma_1, \gamma_2} \Theta
& = & D^2 \Theta(\Phi + v_A)[ \partial^{\gamma_1} \Phi + \partial^{\gamma_1} v_A][ \partial^{\gamma_2} v_A - \partial^{\gamma_2} v_B]
\\[0.2cm]
& & \qquad
 + \,D \Theta(\Phi + v_A)[ \partial^{\gamma_1 + \gamma_2} v_A - \partial^{\gamma_1 + \gamma_2} v_B]
\\[0.2cm]
& & \qquad
  +\, \big( D^2  \Theta(\Phi + v_A) - D^2 \Theta(\Phi + v_B) \big)[\partial^{\gamma_1} \Phi +\partial^{\gamma_1} v_B, \partial^{\gamma_2} \Phi +\partial^{\gamma_2} v_B]
\\[0.2cm]
& & \qquad
 + \,\big( D^2  \Theta(\Phi + v_A)  \big)[\partial^{\gamma_1} v_A -\partial^{\gamma_1} v_B, \partial^{\gamma_2} \Phi +\partial^{\gamma_2} v_B] .
\end{array}  
\end{equation}
Note that $D \Theta$ and $D^2 \Theta$ are globally bounded, and that for any pair $w_1, w_2 \in H^2$ we have
\begin{equation}
    \| \partial^{\gamma_1} w_1 \partial^{\gamma_2} w_2\|_{L^2} 
    \le \|\partial^{\gamma_1} w_1\|_{L^4} \|\partial^{\gamma_2} w_2\|_{L^4}
    \le \|w_1\|_{H^2} \|w_2\|_{H^2}.
\end{equation}
In particular, for $v_A=v$ and $v_B = 0$ we may
proceed term by term to compute
\begin{equation}
    \begin{array}{lcl}
    \|  \Delta_{\gamma_1, \gamma_2} \Theta \|_{L^2(\mathcal D;\mathbb R^N)}
    & \le &
      K\big[ |v\|_{H^1} + \|v\|_{H^2}^2
      + \|v \|_{H^2} 
      + \|v\|_{L^2}
      + \|v\|_{H^1} 
      \big],
    \end{array}
\end{equation}
which can be absorbed in the stated bound.
\end{proof}

\begin{lemma}
\label{lem:nl:low:lip:theta:cub:l2}
Pick $1 \le d \le 4$, assume that $\Phi$ is bounded with $\Phi' \in H^2$, and consider a $C^3$-smooth function $\Theta: \mathbb R^n \to \mathbb R^N$ for which $D^{3}\Theta$ is globally bounded.
Then there exists a constant $K > 0$ so that
for any pair $v_A,v_B \in H^{4/3}$ we have the bound
\begin{equation}\label{eq:nl:low:lip:theta:quadr:l2}
     \|\Theta(\Phi+v_A)-\Theta(\Phi+v_B)\|_{L^2(\mathcal D;\mathbb R^N)}\leq
     K  (1+\|v_A\|_{H^{4/3}}^2 
     + \|v_B\|_{H^{4/3}}^2
     )\|v_A-v_B\|_{H^{4/3}}.
\end{equation}
\end{lemma}
\begin{proof}
In view of the pointwise bound 
\begin{equation}
    |\Theta(\Phi+v_A)-\Theta(\Phi+v_B)|\leq [1+|v_A|^2+|v_B|^2]|v_A-v_B|,
\end{equation}
we may compute
    \begin{equation}
        \begin{aligned}
        \|\Theta(\Phi+v_A)-\Theta(\Phi+v_B)\|_{L^2(\mathcal D;\mathbb R^N)}&\leq K[1+\|v_A\|^2_{L^6}+\|v_B\|^2_{L^6}]\|v_A-v_B\|_{L^6},\\
        \end{aligned}
    \end{equation}
    which leads to the stated estimate using the embedding $H^{4/3}\hookrightarrow L^6$.
\end{proof}

\begin{lemma}
Pick $1 \le d \le 4$, assume that $\Phi$ is bounded with $\Phi' \in H^2$, and consider a $C^3$-smooth function $\Theta: \mathbb R^n \to \mathbb R^N$ for which $D^{3}\Theta$ is globally bounded.
Then there exists a constant $K > 0$ so that
for any pair $v_A,v_B \in H^{2}$ we have the bound
\begin{equation}\label{eq:nl:low:lip:theta:cub:h1}
\begin{array}{lcl}
     \|\Theta(\Phi+v_A)-\Theta(\Phi+v_B)\|_{H^1(\mathcal D;\mathbb R^N)}
     & \leq &
     K (1 + \|v_A\|_{H^1} + \| v_B\|_{H^2}) 
     \|v_A-v_B\|_{H^1}
\\[0.2cm]
& & \qquad
    + K \|v_A\|_{H^{7/4}}^2\| (v_A -v_B)\|_{H^{3/2}}
\\[0.2cm]
& & \qquad
    + K (\|v_A\|_{H^{7/4}} + \|v_B\|_{H^{7/4}})
\\[0.2cm]
& & \qquad \qquad \times
    \| v_B\|_{H^{3/2}} \|v_A - v_B\|_{H^{7/4}} .
\end{array}
\end{equation}
\end{lemma}
\begin{proof}
Pick a multi-index $\gamma \in \mathbb{Z}^d_{\ge 0}$ with $|\gamma| =  1$. Writing
\begin{equation}
    \Delta_{\gamma} \Theta = \partial^{\gamma} [\Theta(\Phi + v_A) - \Theta(\Phi + v_B)],
\end{equation}
we may inspect \eqref{eq:nl:id:for:partial:theta} to obtain the pointwise bound
\begin{equation}
    |\Delta_{\gamma} \Theta|
    \le 
    K (1 + |v_A|^2) | \partial^\gamma v_A - \partial^\gamma v_B |
    +K (1 + |v_A| + |v_B|) |v_A - v_B|
    ( |\partial^\gamma \Phi| + |\partial^\gamma v_B |).
\end{equation}
Proceeding term by term, we obtain
\begin{equation}
\begin{array}{lcl}
\|\Delta_{\gamma} \Theta\|_{L^2(\mathcal D;\mathbb R^N)} 
& \le & K \|v_A-v_B\|_{H^1}
    + K \|v_A\|_{L^{16}}^2\|\partial^\gamma (v_A -v_B)\|_{L^{8/3}}
\\[0.2cm]
& & \qquad
    + K (\|v_A\|_{L^4} + 
    \|v_B\|_{L^4})\|v_A - v_B\|_{L^4}
\\[0.2cm]
& & \qquad
    + K \|v_A - v_B\|_{L^4}\|\partial^\gamma v_B\|_{L^4}
\\[0.2cm]
& & \qquad
    + K (\|v_A\|_{L^{16}} + \|v_B\|_{L^{16}})\|v_A - v_B\|_{L^{16}} \|\partial^\gamma v_B\|_{L^{8/3}},
\end{array}
\end{equation}
which leads to the stated bound upon applying
the Sobolev embeddings \eqref{eq:nl:low:sob:emb:i}--\eqref{eq:nl:low:sob:emb:ii}.
\end{proof}

\begin{corollary}
Pick $1 \le d \le 4$, assume that $\Phi$ is bounded with $\Phi' \in H^2$, and consider a $C^3$-smooth function $\Theta: \mathbb R^n \to \mathbb R^N$ for which $D^{3}\Theta$ is globally bounded.
Then there exists a constant $K > 0$ so that
for any $v \in H^{2}$ we have the bound
\begin{equation}\label{eq:nl:low:lip:theta:cub:h1:va:only}
\begin{array}{lcl}
\|\Theta(\Phi+v)-\Theta(\Phi)\|_{L^2(\mathcal D;\mathbb R^N)}
     & \leq &
         K \|v\|_{H^2}\big(1 + \|v\|_{H^1}^2 \big),
\\[0.2cm]
     \|\Theta(\Phi+v)-\Theta(\Phi)\|_{H^1(\mathcal D;\mathbb R^N)}
     & \leq &
         K \|v\|_{H^2}\big(1 + \|v\|_{H^1} 
+ \|v\|_{H^1} \|v\|_{H^2}\big).
\end{array}
\end{equation}
\end{corollary}
\begin{proof}    
Applying the interpolation estimates \eqref{eq:nl:low:intp:bnds}
to the bound \eqref{eq:nl:low:lip:theta:quadr:l2} with $v_A=v$ and $v_B = 0$
leads directly to the first estimate. Performing the same
for \eqref{eq:nl:low:lip:theta:cub:h1}, we find
\begin{equation}
\begin{array}{lcl}
\|\Theta(\Phi+v)-\Theta(\Phi)\|_{H^1(\mathcal D;\mathbb R^N)}
    & \le & K \|v\|_{H^1}
    + K \|v\|_{H^{1}}\|v\|_{H^{2}}^{2}
    + K \|v\|_{H^1}^2 ,
  \end{array}
\end{equation}
which can be absorbed in the stated bound.
\end{proof}

\begin{lemma}
Pick $1 \le d \le 4$, assume that $\Phi$ is bounded with $\Phi' \in H^2$, and consider a $C^3$-smooth function $\Theta: \mathbb R^n \to \mathbb R^N$ for which $D^{3}\Theta$ is globally bounded.
Then there exists a constant $K > 0$ so that
for any $v \in H^{2}$ we have the bounds
\begin{equation}\label{eq:nl:low:lip:theta:cub:h1h2}
\begin{array}{lcl}
\|\Theta(\Phi+v)-\Theta(\Phi)
     - D\Theta(\Phi)[v]\|_{L^2(\mathcal D;\mathbb R^N)}
     & \leq &
     K ( 1 + \|v\|_{H^2})\|v\|_{H^1}^2 ,
\\[0.2cm]
     \|\Theta(\Phi+v)-\Theta(\Phi)
     - D\Theta(\Phi)[v]\|_{H^1(\mathcal D;\mathbb R^N)}
     & \leq &
       K (1 + \|v\|_{H^1}) \|v\|_{H^2}^2.
\end{array}
\end{equation}
\end{lemma}
\begin{proof}
Recall the identity \eqref{eq:pointwise:D2Theta}. Writing
\begin{equation}
    \mathcal{Q}= 
    D^2(\Phi  + st v) [v, v],
\end{equation}
for $0\leq s,t\leq 1,$ and picking a multi-index $\gamma$ with $|\gamma|=1$, 
we first compute
\begin{equation}
    \partial^\gamma \mathcal{Q} = 2 D^2\Theta(\Phi + stv)[\partial^\gamma v, v]
     + D^3 \Theta(\Phi + stv)[ \partial^\gamma \Phi +st\partial^\gamma v, v, v],
\end{equation}
which provides the pointwise bounds
\begin{equation}
\begin{array}{lcl}
| \mathcal{Q} | & \le &
K   (1 + |v|) |v|^2 ,
\\[0.2cm]
| \partial^\gamma \mathcal{Q} |
 & \le &
 K (1 + |v|) |v| |\partial^\gamma v|
    + K (1 + |\partial^\gamma v| ) |v|^2
    \le K |v| \big( |v| +  |\partial^\gamma v|
    + |v| |\partial^\gamma v| \big).
\end{array}
\end{equation}
In particular, we find
\begin{equation}
\begin{array}{lcl}
    \|\mathcal Q\|_{L^2(\mathcal D;\mathbb R^N)} & \le & K \|v\|_{L^4}^2 + K \|v\|_{L^6}^3 ,
\\[0.2cm]
\| \partial^\gamma \mathcal Q\|_{L^2(\mathcal D;\mathbb R^N)}
    & \le & K \|v\|_{L^4}^2  + K \|v\|_{L^8} \|\partial^\gamma v\|_{L^{8/3}}
    + K \|v\|_{L^{16}}^2 \|\partial^\gamma v\|_{L^{8/3}},
\end{array}
\end{equation}
which in view of the Sobolev embeddings \eqref{eq:nl:low:sob:emb:i}--\eqref{eq:nl:low:sob:emb:ii}
leads to the bounds
\begin{equation}
\begin{array}{lcl}
\|\mathcal Q\|_{L^2(\mathcal D;\mathbb R^N)} &\le & K \Big( \|v\|_{H^1}^2 + \|v\|_{H^{4/3}}^3 \Big),
\\[0.2cm]
    \| \partial^\gamma Q\|_{L^2(\mathcal D;\mathbb R^N)}
   & \le & K \Big( \|v\|_{H^1}^2  +  \|v\|_{H^{3/2}} \| v\|_{H^{3/2}}
    +  \|v\|_{H^{7/4}}^2 \| v\|_{H^{3/2}} \Big).
\end{array}
\end{equation}
The stated estimates now follow by applying the
interpolation bounds \eqref{eq:nl:low:intp:bnds}.
\end{proof}

\begin{corollary}
\label{cor:nl:low:ests:on:k:c}
Suppose that \textnormal{(Hq)}, \textnormal{(HSt)} and \textnormal{(HPar)} are satisfied with $k=1$
and pick a multi-index $\beta \in \mathbb Z^d_{\ge 0}$ with
$|\beta| = 1$.
Then there exists a constant $K > 0$,
which does not
depend on the pair $(\Phi, c)$, so that the following holds true. 
For any $v\in H^2$ and $\gamma \in \mathbb R$ we have the bounds
\begin{equation}
\label{eq:nl:low:est:on:kc}
\begin{array}{lcl}
\| \partial^\beta \widetilde{\mathcal{K}}_C(\Phi + v, \gamma) - \partial^\beta \widetilde{\mathcal{K}}_C(\Phi,\gamma) \|_{L^2_Q}
    & \le & 
         K (1 +\|v\|_{L^2}) \|v\|_{H^1} ,
\\[0.2cm]
    \| \mathcal{K}_C(\Phi + v, \gamma) - \mathcal{K}_C(\Phi,\gamma) \|_{H^2}
    & \le &   
    K  (1 + \|v\|_{H^1}^3) (1 + \|v\|_{H^2})\|v\|_{H^2} .
\end{array}
\end{equation}
\end{corollary}
\begin{proof}

The first bound  follows from \eqref{eq:nl:g:h1:bnd} and by inspecting \eqref{eq:Delta:AB:KtildeC} with $\gamma=\gamma_A=\gamma_B$, $v_A=v$ and $v_B=0$.
To obtain the second bound, we  again consider the splitting  in \eqref{eq:splitting:EIEII}, but now with $v_A=v$ and $v_B=0$.
Lemma \ref{lem:nl:chi:wtkc:l2:ests}
together with \eqref{eq:nl:low:lip:theta:h2:h2}
yields
\begin{equation}
    \|\mathcal{E}_I\|_{H^1} \le 
    K \|v\|_{L^2} (1 + \|v\|_{H^2}^2) 
    + K (1 + \|v\|_{H^2}) \|v\|_{H^2}.
\end{equation}
whereas \eqref{eq:nl:low:lip:theta:h1:h2} together with the
first estimate of \eqref{eq:nl:low:est:on:kc}
allows us to find
\begin{equation}
\begin{array}{lcl}
    \|\mathcal{E}_{II}\|_{H^1}
    &\le& 
    K \|v\|_{L^2} (1 + \|v\|_{H^1})  (1 +\|v\|_{L^2}) 
    (1 + \|v\|_{H^1})
    \\[0.2cm]
    & & \qquad
    + \,
     K \|v\|_{H^1}(1 +\|v\|_{L^2}) 
    (1 + \|v\|_{H^1})
    \\[0.2cm]
    & & \qquad
    +\, K(1 + \|v\|_{H^1}) (1 +\|v\|_{L^2}) \|v\|_{H^1}.
\end{array}
\end{equation} 
Note that both bounds can be absorbed by the stated estimate.
\end{proof}

Thanks to the preparations above, we are now ready to work towards
the final estimates for $\mathcal{N}_{\sigma}$. To this end,
we recall the expressions $\Xi_I$ and $\Xi_{II;\sigma}$ defined in \eqref{eq:nl:def:xi:i:ii}, together with
the intermediate function $\mathcal{N}_{I;\sigma}$
and the associated decomposition \eqref{eq:nl:decomp:for:n:i:sigma}.
\begin{corollary}
\label{cor:nl:low:bnds:xi:i:ii}
Suppose that \textnormal{(Hf-Cub)}, \textnormal{(HCor)}, \textnormal{(Hq)}, \textnormal{(HSt)} and \textnormal{(HPar)} are satisfied
with $k=1$.
Then there exists a constant $K > 0$,
which does not
depend on the pair $(\Phi, c)$, so that the following holds true. 
For any $v \in H^2$, any $\gamma \in \mathbb R$, and any $\sigma\ge 0$,
the expressions
\begin{equation}
\begin{array}{lcl}
    \Delta_v \Xi_I
     &= & \Xi_I(\Phi + v, c)
     - \Xi_I(\Phi , c),
\\[0.2cm]
    \Delta_v \Xi_{II;\sigma}
     &= & \Xi_{II;\sigma}(\Phi + v, \gamma)
     - \Xi_{II;\sigma}(\Phi, \gamma)
\end{array}
\end{equation}
satisfy the bounds
\begin{equation}
    \begin{array}{lcl}
    \langle  \Delta_v \Xi_{I} , \psi_{\rm tw} \rangle_{L^2}
    & \le & 
      K  (1 +  \|v\|_{H^1}^2) \|v\|_{H^2} ,
    \\[0.2cm]
    \langle \Delta_v \Xi_{II;\sigma}, \psi_{\rm tw} 
    \rangle_{L^2}
    & \le & 
    \sigma^2 K
    \|v\|_{L^2},
\\[0.2cm]
    \end{array}
\end{equation}
together with
\begin{equation}
    \begin{array}{lcl}
    \| \Delta_v \Xi_{I} \|_{H^1}
    & \le & 
      K  (1 + \|v\|_{H^1} + \|v\|_{H^1 }\|v\|_{H^2})\|v\|_{H^2} ,
    \\[0.2cm]
    \| \Delta_v \Xi_{II;\sigma} \|_{H^1}
    & \le & \sigma^2
    K \big( 1+ \|v\|_{H^1}^3)(1  + \|v\|_{H^2}) \|v\|_{H^2} .
    \end{array}
\end{equation}
\end{corollary}
\begin{proof}
    These bounds readily follow by inspecting the definitions \eqref{eq:nl:def:xi:i:ii}
    and applying \eqref{eq:est:chig:lip:bnds}, 
    \eqref{eq:nl:low:lip:theta:h1:h2},
    \eqref{eq:nl:low:lip:theta:cub:h1:va:only}
    and
    \eqref{eq:nl:low:est:on:kc}.
\end{proof}

\begin{corollary}
\label{cor:nl:est:n:i:final:low}
Suppose that \textnormal{(Hf-Cub)}, \textnormal{(HCor)}, \textnormal{(Hq)} and \textnormal{(HSt)} are satisfied
with $k=1$.
Then there exists a constant $K > 0$
so that for all sufficiently small $\sigma\geq 0$
and all $v \in H^{2}$
we have
\begin{equation}
\begin{array}{lcl}
\langle \mathcal{N}_{I;\sigma}(v) , \psi_{\rm tw} 
\rangle_{L^2} & \le  &
K ( 1 + \|v\|_{H^2})\|v\|_{H^1}^2
+ \sigma^2   K  (1 +  \|v\|_{H^1}^2) \|v\|_{H^2},
\\[0.2cm]
    \|\mathcal{N}_{I;\sigma}(v)\|_{H^1}
    & \le & 
    K (1 + \|v\|_{H^1})\|v\|_{H^2}^2
    + \sigma^2 K (1 + \|v\|_{H^1}^3)(1 + \|v\|_{H^2} ) \|v\|_{H^2}.
\end{array}
\end{equation}
\end{corollary}
\begin{proof}
The $L^2$-bound follows from 
Lemma \ref{lem:nl:kappa_sigma},
Corollary \ref{cor:nl:low:bnds:xi:i:ii}
and the quadratic estimate \eqref{eq:nl:low:lip:theta:cub:h1h2}.
Similar computations for the $H^1$-bound lead to the 
estimate
\begin{equation}
\begin{array}{lcl}
\| \mathcal{N}_{I;\sigma}(v)\|_{H^1} & \le & K (1 + \|v\|_{H^1})\|v\|_{H^2}^2
\\[0.2cm]
& & \qquad
    +\,
     \sigma^2 K \|v\|_{L^2}  
     (1 + \|v\|_{H^2})
     (1  + \|v\|_{H^1} + \|v\|_{H^1} \|v\|_{H^2} )
\\[0.2cm]
& & \qquad
    +\, \sigma^2 K (1 + \|v\|_{H^1}^3)(1 + \|v\|_{H^2} ) \|v\|_{H^2} ,
\end{array}
\end{equation}
which can be absorbed in the stated bound.
\end{proof}

\begin{corollary}
\label{cor:nl:low:fin:bnd:n:sigma}
Suppose that \textnormal{(Hf-Cub)}, \textnormal{(HCor)}, \textnormal{(Hq)} and \textnormal{(HSt)} are satisfied
with $k=1$.
Then there exists a constant $K > 0$
so that for all sufficiently small $\sigma\geq 0$
and all $v \in H^{2}$
we have 
\begin{equation}
\begin{array}{lcl}
    \|\mathcal{N}_{\sigma}(v)\|_{H^1}
    & \le & 
K (1 + \|v\|_{H^1}^2)\|v\|_{H^2}^2
    + \sigma^2 K (1 + \|v\|_{H^1}^3)(1 + \|v\|_{H^2} ) \|v\|_{H^2} .
\end{array}
\end{equation}
\end{corollary}
\begin{proof}
By inspecting the representation
\eqref{eq:nl:repr:n:sigma:final},
we see that Corollary \ref{cor:nl:est:n:i:final:low}
together with
\eqref{eq:nl:bnds:chi:hl:k:tw:l2}
lead to 
\begin{equation}
\begin{array}{lcl}
    \|\mathcal{N}_{\sigma}(v)\|_{H^1}
    & \le & 
K (1 + \|v\|_{H^1})\|v\|_{H^2}^2
\\[0.2cm]
& & \qquad
    + \,\sigma^2 K (1 + \|v\|_{H^1}^3)(1 + \|v\|_{H^2} ) \|v\|_{H^2}
\\[0.2cm]
& & \qquad
+\, K ( 1 + \|v\|_{H^2})\|v\|_{H^1}^2 (1 + \|v\|_{H^2})
\\[0.2cm]
& & \qquad
+\, \sigma^2   K  (1 +  \|v\|_{H^1}^2) \|v\|_{H^2}(1 + \|v\|_{H^2}) ,
\end{array}
\end{equation}
which can be absorbed in the stated bound.
\end{proof}

\section{Variational solutions}\label{sec:variational}

In this section, we establish Proposition \ref{prop:existLip} and study the regularity of the obtained solutions. In particular, we provide existence and uniqueness results for the coupled SPDE-SDE system 
\begin{equation}\begin{cases}
    \mathrm d u=[\Delta u + f(u)+\sigma^2h(u)]\mathrm dt+\sigma g(u)\mathrm d
W_t^Q, \\
\mathrm d\gamma = [c+a_\sigma(u,\gamma;c)]\mathrm dt+\sigma b(u,\gamma) \mathrm d  W_t^Q,
\end{cases}\label{eq:system}
\end{equation}
where $u$ takes values in the affine space $\mathcal U_{H^k(\mathcal D;\mathbb R^n)}$, i.e., $u = \Phi_{\mathrm{ref}} + v$ with $v \in H^k(\mathcal D;\mathbb R^n)$, for some  $k\geq 0$.
We start by considering global existence for the case
$k=0$ in {\S}\ref{sec:weak} and move on to study local existence for $k >d/2$ in {\S}\ref{sec:regularity}, all under our global Lipschitz assumption (Hf-Lip).
We provide an alternative approach
for global existence results in {\S}\ref{subsec:small:dim} that works for $k=1$ and $1 \le d \leq 4$ under the cubic growth condition (Hf-Cub). 

Our main findings beyond the $L^2$-based statements in Proposition \ref{prop:existLip} are summarised in the following result.
The precise interpretation of the diffusion operator in the identity
\eqref{eq:Hk*} is discussed throughout this section. We emphasise that the various types of solutions that we construct agree where they overlap; see, e.g., Lemma \ref{lem:ex:hk:agree:l2} below.

\begin{proposition}[Higher regularity]
\label{prop:var:higher:full}
Suppose that either \textnormal{(Hf-Lip)} is satisfied with $k > d/2$ or 
that $k=1$ and \textnormal{(Hf-Cub)} is satisfied. Assume furthermore that
\textnormal{(HCor)}, \textnormal{(Hq)} and \textnormal{(HSt)} are satisfied.
Fix $T>0$ and $0\leq \sigma\leq 1$. Then for any initial condition 
\begin{equation}
    (z_0,\gamma_0)\in H^k(\mathcal D;\mathbb R^n)\times \mathbb R,
\end{equation}
there exists an increasing sequence of stopping times $(\tau_\ell)_{\ell\geq 0}$ and a stopping time $\tau_\infty$, with  $\tau_\ell\to \tau_\infty$ and $0<\tau_\infty\leq T$ $\mathbb P$-a.s., together with progressively measurable maps
\begin{equation}
    z:[0,T]\times \Omega\to H^k(\mathcal D;\mathbb R^n),\quad \gamma:[0,T]\times \Omega\to\mathbb R,
\end{equation}
that satisfy the following properties:
\begin{enumerate}[\rm (i)]
    \item For almost every $\omega\in\Omega$, the map
    $ 
        t\mapsto (z(t,\omega),\gamma(t,\omega))
    $
    is of class $C([0,\tau_\infty(\omega));H^{k}(\mathcal D;\mathbb R^n)\times \mathbb R)$;
   \item We have the integrability condition $(z,\gamma)\in L^2(\Omega;L^2([0,\tau_\ell];H^{k+1}(\mathcal D;\mathbb R^n)\times \mathbb R))$,  for any $\ell\geq 0$;
   \item The $H^k(\mathcal D;\mathbb R^n)$-valued identity\footnote{   At first, the equality in \eqref{eq:Hk*} should be understood as an equality in $[H^{k+1}]^*\equiv H^{k-1}$,   but by (ii) we can conclude that we have equality in $H^k$. 
   }
   \begin{equation}
       \begin{aligned}
           z(t)=z_0+\int_0^t [\Delta z(s)+\Phi_{\rm ref}'']\mathrm ds&+\int_0^tf(z(s)+\Phi_{\rm ref})\mathrm ds\\&+\sigma^2\int_0^th(z(s)+\Phi_{\rm ref})\mathrm ds+\sigma \int_0^tg(z(s)+\Phi_{\rm ref})\mathrm dW_s^Q,
       \end{aligned}\label{eq:Hk*}
   \end{equation}
   together with the scalar identity
   \begin{equation}
\gamma(t)=\gamma_0+\int_0^t[c+a_\sigma(z(s)+\Phi_{\rm ref},\gamma(s);c)]\mathrm ds+\sigma\int_0^t b(z(s)+\Phi_{\rm ref},\gamma(s))\mathrm dW_s^Q,
   \end{equation}
   hold $\mathbb P$-a.s. for all $0\leq t<\tau_\infty$;
   \item Suppose there are  other progressively measurable maps $\tilde z$ and $\tilde \gamma$ that satisfy \textnormal{(i)--(iii)} with another stopping time $\tilde \tau_\infty$ and localising sequence $(\tilde \tau_\ell)_{\ell \ge 0}$. Then for almost all $\omega\in\Omega,$ we have  $\tilde \tau_\infty(\omega)\leq \tau_\infty(\omega)$ together with 
   \begin{equation}
      \quad \tilde z(t,\omega)=z(t,\omega)\quad\text{and}\quad \tilde \gamma(t,\omega)=\gamma(t,\omega),\quad \text{for all}\quad 0\leq t\leq \tilde \tau_\infty(\omega).
   \end{equation}
\end{enumerate}
In the setting where $k=1$ and \textnormal{(Hf-Cub)} is satisfied, we may take
$\tau_{\ell} = \tau_\infty = T$.
\end{proposition}


\subsection{Global existence in \texorpdfstring{$L^2$}{L2}}\label{sec:weak}
 Throughout this part, we consider $k=0$ and  solutions where $(v, \gamma)$ is measured with respect to the Gelfand triple $(\mathcal V,\mathcal H,\mathcal V^*)$ given by
    \begin{equation}
    \label{eq:var:weak:setting}
        \mathcal V=H^1(\mathcal D;\mathbb R^n)\times\mathbb R,\quad \mathcal H=L^2(\mathcal D;\mathbb R^n)\times\mathbb R,\quad \mathcal V^*=  H^{-1}(\mathcal D;\mathbb R^n)\times\mathbb R. 
    \end{equation}
    The associated inner products are given by
    \begin{equation}
    \begin{aligned}\langle (v,\gamma_A),(w,\gamma_B)\rangle_{\mathcal V}&=\langle v,w\rangle_{H^1(\mathcal D;\mathbb R^n)}+\langle \gamma_A,\gamma_B\rangle_{\mathbb R},\\\langle (v,\gamma_A),(w,\gamma_B)\rangle_{\mathcal H}&=\langle v,w\rangle_{L^2(\mathcal D;\mathbb R^n)}+\langle \gamma_A,\gamma_B\rangle_{\mathbb R},\end{aligned}
    \end{equation}
    while the duality pairing acts as
    \begin{equation}\langle (v,\gamma_A),(w,\gamma_B)\rangle_{\mathcal V^*;\mathcal V}=\langle v,w\rangle_{H^{-1}(\mathcal D;\mathbb R^n);H^1(\mathcal D;\mathbb R^n)}+\langle \gamma_A,\gamma_B\rangle_{\mathbb R},
    \end{equation} 
    where we follow \cite{liu2010spde} to interpret
    the duality pairing between the spaces $H^{-1}(\mathcal D;\mathbb R^n)$ and $H^1(\mathcal D;\mathbb R^n)$.
    In particular, for any $v \in L^2(\mathcal D;\mathbb R^n)$
and $w \in H^1(\mathcal D;\mathbb R^n)$ we have
\begin{equation}
    \langle v , w \rangle_{H^{-1}(\mathcal D;\mathbb R^n);H^1(\mathcal D;\mathbb R^n)}
     = \langle v, w \rangle_{L^2(\mathcal D;\mathbb R^n)}.
\end{equation}
In addition, the Laplacian can be interpreted as a bounded linear operator mapping from $H^1(\mathcal D;\mathbb R^n)$ into $H^{-1}(\mathcal D;\mathbb R^n)$ by writing
\begin{equation}
\label{eq:var:def:laplace:hm1}
    \langle \Delta v,w\rangle_{H^{-1}(\mathcal D;\mathbb R^n);H^{1}(\mathcal D;\mathbb R^n)}=
    - \langle \nabla v, \nabla w \rangle_{L^2(\mathcal D;\mathbb R^n)}
    =
    -\langle v,w\rangle_{H^1(\mathcal D;\mathbb R^n)}+\langle v,w\rangle_{L^2(\mathcal D;\mathbb R^n)}
\end{equation}
for any pair $v,w \in H^1(\mathcal D;\mathbb R^n)$.
In the parlance of  \cite{agresti2022critical}, this is analogous to the so-called \textit{weak setting}.


We are now ready to prove Proposition \ref{prop:existLip},
primarily using the estimates for $a_\sigma$ and $b$
obtained in Corollary \ref{cor:nl:bnds:for:b:and:kc}
and Lemma \ref{lem:a_sigma}. As a consequence
of the global Lipschitz assumption on $f$, these
estimates  allow us to establish Proposition \ref{prop:existLip} in a more direct manner than the approach in \cite{hamster2019stability,hamster2020}.
 Indeed, we are able to directly embed the  coupled system \eqref{eq:system} within a variational framework,
 rather than solving first for $u$
 and then interpreting the equation for $\gamma$ as an SDE with random coefficients.

%
It is worthwhile to point out that 
we can either use the classical variational framework \cite{liu2010spde} of Liu and R\"ockner or the critical variational framework \cite{agresti2022critical} developed by Agresti and Veraar. Indeed, the Lipschitz estimates
\eqref{eq:nl:est:b:kc:lip:bnds} and
\eqref{eq:nl:a:lip:l2}
 do  not depend on  both $v_A$ and $v_B$ simultaneously and only involve $L^2$-norms, enabling us to verify the local monotonicity condition
required by \cite{liu2010spde} for the full coupled
system \eqref{eq:system}. 

\begin{proof}[Proof of Proposition \ref{prop:existLip} under \textnormal{(Hf-Lip)}]
It suffices to check the conditions in \cite[Thm. 1.1]{liu2010spde}, which provides our statements when applied with $\alpha =2$ and $\beta=2$.
%
We shall merely focus on the parts regarding the local monotonicity, coercivity, and the growth condition  induced by the SDE of $\gamma$. The remaining conditions can be readily verified by the reader by  exploiting the identity 
\eqref{eq:var:def:laplace:hm1}
and appealing to Lemma \ref{lem:nw:l2:ests:f:g:h} and Corollary \ref{cor:g:HS:L2}.


 Using the bounds in Lemma \ref{lem:a_sigma}, we see  for any $v_A,v_B\in H^1(\mathcal D;\mathbb R^n)$ and any $\gamma_A,\gamma_B\in\mathbb R$ that
\begin{equation}\begin{aligned}
\langle a_\sigma(\Phi+v_A,\gamma_A,c)&-a_\sigma(\Phi+v_B,\gamma_B,c),\gamma_A-\gamma_B\rangle_{\mathbb R}\\&\leq K_a\rho(v_A)\big[\|v_A-v_B\|_{L^2(\mathcal D:\mathbb R^n)}+|\gamma_A-\gamma_B|\big]|\gamma_A-\gamma_B|\\
&\leq 2K_a\rho(v_A)\big[\|v_A-v_B\|_{L^2(\mathcal D:\mathbb R^n)}^2+|\gamma_A-\gamma_B|^2\big],\label{eq:local_mon}
\end{aligned}
    \end{equation}
 in which we have exploited  the  scalar identity $xy\leq x^2+y^2$ and introduced the function \begin{equation}
\rho(v)=\big[1+\|v\|_{L^2(\mathcal D;\mathbb R^n)}\big]^2 .
    \end{equation}
In addition, for any $v\in H^1(\mathcal D;\mathbb R^n)$ and $\gamma\in\mathbb R$ we
may use the scalar identity above again together with  $x\leq 1+x^2$ to conclude
\begin{equation}\begin{aligned}
    \langle a_\sigma(\Phi+v,\gamma,c),\gamma\rangle_{\mathbb R} &\leq |a_\sigma(\Phi+v,\gamma,c)\|\gamma|\\&\leq K_a\big[1+\|v\|_{L^2(\mathcal D;\mathbb R^n)}\big]|\gamma|\\&\leq 2K_a\big[1+\|v\|_{L^2(\mathcal D;\mathbb R^n)}^2+|\gamma|^2\big].\label{eq:coer-growth}\end{aligned}
\end{equation}
The second estimate in \eqref{eq:nl:est:b:kc:lip:bnds} together with the bound in  \eqref{eq:local_mon} now yield the local monotonicity property \cite[(H2)]{liu2010spde}, 
whereas the coercivity condition \cite[(H3)]{liu2010spde} and the growth condition \cite[(H4)]{liu2010spde} follow from \eqref{eq:nl:est:kc:b}  and \eqref{eq:coer-growth}. 
%
\end{proof}



\subsection{Local existence in \texorpdfstring{$H^k$}{Hk}}
\label{sec:regularity}

In this part  we investigate the regularity of the solution  found in Proposition \ref{prop:existLip}, which we refer to as $(\tilde z,\tilde \gamma)\in L^2(\mathcal D;\mathbb R^n)\times \mathbb R$. We do this by establishing the (local) existence and uniqueness of solutions in  $H^k(\mathcal D;\mathbb R^n)\times\mathbb R$, which we write as 
$(z,\gamma).$ If $z$ has an initial value in $H^k(\mathcal D;\mathbb R^n)$, then it remains in $H^k(\mathcal D;\mathbb R^n)$ for at least a short time.
In particular, we show that  $(z,\gamma)$ coincides with  the solution $(\tilde z,\tilde \gamma)$, where they overlap, showing that $z$ persists globally as a continuous $L^2$-valued solution even after the $H^k$-smoothness is lost.

Indeed, one can follow 
\cite{agresti2022critical,brezis2011functional,krylov1981stochastic} to show
that for $k \ge 1$ the bilinear map
\begin{equation}
\label{eq:var:dual:pairing:h:km1:kp1}
    \langle v, w \rangle_{H^{k-1};
    H^{k+1}}
    =  \langle v, w \rangle_{H^{k-1}}
     - \sum_{|\alpha| = k-1 } \langle \partial^\alpha v , \partial^\alpha \Delta w \rangle_{L^2}
\end{equation}
allows $H^{k-1}$
to be interpreted as the dual of $H^{k+1}$. In order to confirm that this   duality pairing is compatible with the inner product of $H^k$, it suffices to compute
\begin{equation}
    \langle v, w \rangle_{H^{k-1};
    H^{k+1}}
    = 
    \langle v, w \rangle_{H^{k-1}}
    + \sum_{|\alpha| = k-1} \langle \partial^\alpha \nabla v, \partial^\alpha \nabla w \rangle_{L^2}
    =     
    \langle v, w \rangle_{H^k},
\end{equation}
whenever $v \in H^k$
and $w \in H^{k+1}$.
In addition, the diffusion operator $\Delta$ can|as usual|be seen as an element of $\mathscr L(H^{k+1};H^{k-1})$.
The definition \eqref{eq:var:dual:pairing:h:km1:kp1}
yields
\begin{equation}
\label{eq:var:delta:v:w:hkm1:hkp1}
    \langle \Delta v,w\rangle_{H^{k-1};H^{k+1}} = -\langle \nabla v, \nabla w \rangle_{H^k}=-\langle v,w\rangle_{H^{k+1}}+\langle v,w\rangle_{L^2}
\end{equation}
for any pair $v,w\in H^{k+1}$,
hence generalising \eqref{eq:var:def:laplace:hm1} and providing an alternative yet equivalent definition for $\Delta$ from the space $H^{k+1}$ into its dual. 

    As explained in the introduction, the presence of derivatives generates cross terms
 that violate the monotonicity requirements  in \cite{liu2010spde}; see e.g., Lemma \ref{cor:nl:est:f:h:hk}. 
 We  therefore appeal to the critical variational framework \cite{agresti2022critical} instead. We remark that the proof below only requires the $C^k$-Lipschitz smoothness on both $f$ and $g$.

\begin{proof}[{Proof of Proposition \ref{prop:var:higher:full} under \textnormal{(Hf-Lip)}}] 
For any integer $n\in\mathbb N,$ 
the bounds in 
Corollaries \ref{cor:nl:bnd:g:hk}
and \ref{cor:nl:est:f:h:hk}
allow us to find constants
$C_{1,n},C_{2,n}>0$ for which
 the estimates
\begin{equation}
        \|f(\Phi+v_A)-f(\Phi+v_B)\|_{H^{k-1}}
        + \|h(\Phi+v_A)-h(\Phi+v_B)\|_{H^{k-1}}
        \leq C_{1,n}\|v_A-v_B\|_{H^k}, 
    \end{equation}
and
\begin{equation}
     \|g(\Phi+v_A)-g(\Phi+v_B)\|_{HS(L^2_Q;H^k)}\leq C_{2,n} \|v_A-v_B\|_{H^k}
\end{equation}
 hold whenever $\|v_A\|_{H^k},\|v_B\|_{H^k}\leq n$.
Together with the fact that \eqref{eq:var:delta:v:w:hkm1:hkp1} implies 
\begin{equation}
         \langle \Delta v,v\rangle_{H^{k-1};H^{k+1}}\leq -\|v\|_{H^{k+1}}^2+\|v\|_{H^k}^2,
     \end{equation}
and recalling the estimates related to the $\gamma$-variable
obtained in the proof of Proposition \ref{prop:existLip},
we note that the result follows by appealing to
\cite[Thm 3.3]{agresti2022critical}.
In particular, the latter yields the existence and uniqueness of a maximal solution $((z,\gamma),\sigma_\infty)$ with a corresponding localising sequence $(\sigma_\ell)_{\ell\geq 0}$. Upon 
     defining the stopping times
    \begin{equation}
    \label{eq:var:def:tau:l:stopping:time}
    \tau_\ell=\sigma_\ell \wedge \inf\{t\geq 0:\sup_{t\geq 0}\|z(t)\|_{H^k}^2+\int_0^t\|z(s)\|_{H^{k+1}}^2\mathrm ds+\int_0^t|\gamma(s)|^2\mathrm ds>\ell \},
    \end{equation}
    the blow up criterion of $\sigma_\infty$ in \cite[Thm 3.3]{agresti2022critical} shows that $\tau_\ell \to \sigma_\infty $ holds. We may hence set $\tau_\infty=\sigma_\infty$, 
    from which all the claims follow.
%
%
\end{proof}

\begin{lemma}\label{lem:ex:hk:agree:l2}
Consider the setting of Proposition \ref{prop:var:higher:full} where \textnormal{(Hf-Lip)} is satisfied.
Writing 
$(\tilde{z}, \tilde{\gamma})$ for the $L^2 (\mathcal D;\mathbb R)\times \mathbb{R}$-valued solution constructed in Proposition
\ref{prop:existLip}, we have
         \begin{equation}
             \tilde z(t,\omega)=z(t,\omega)\quad\text{and}\quad \tilde \gamma(t,\omega)=\gamma(t,\omega),
         \end{equation}
         for all $0\leq t<\tau_\infty(\omega)$ and almost every $\omega\in\Omega$. 
\end{lemma}
\begin{proof}
Without loss, set $h=0.$ Although we can identify the identity in \eqref{eq:Hk*} as an equality in $H^k$, it is actually an equality in $[H^{k+1}]^*$, which we will now  exploit. 
Indeed, for all $\zeta\in H^{k+1}$ we have 
\begin{equation}\begin{aligned}
    \langle z(t),\zeta\rangle_{H^k}&=\langle z_0,\zeta\rangle_{H^k}+\int_0^t\langle \Delta z(s),\zeta\rangle_{H^{k-1};H^{k+1}} \mathrm ds
    +\int_0^t \langle f(z(s)+\Phi_{\rm ref}),\zeta\rangle_{H^k}\mathrm ds
    \\&\hspace{3cm}+\sigma \int_0^t\langle g(z(s)+\Phi_{\rm ref})\mathrm dW_s^Q,\zeta\rangle_{H^k}.
    \end{aligned}
\end{equation}

Picking an arbitrary $\eta\in C_c^\infty(\mathcal D;\mathbb R^n)$, we write $\zeta=(1-\Delta)^{-k}\eta$ and note that $\zeta \in H^{\ell}$ for any $\ell \ge 0$. For any $v \in H^k$, Parseval's identity (see Appendix \ref{appendix:Fourier}) yields
\begin{equation}
\begin{array}{lcl}
\langle v,\zeta \rangle_{H^k} &=&
\langle \boldsymbol{\xi}\mapsto (1+|\boldsymbol{\xi}|^2)^{k/2} \hat v(\boldsymbol{\xi}),\boldsymbol{\xi}\mapsto (1+|\boldsymbol{\xi}|^2)^{k/2}\hat \zeta(\boldsymbol{\xi})\rangle_{L^2(\widehat{\mathcal D};\mathbb R^n)}
\\[0.2cm]
& = &
\langle \boldsymbol{\xi}\mapsto \hat v(\boldsymbol{\xi}),\boldsymbol{\xi}\mapsto (1+|\boldsymbol{\xi}|^2)^k\hat \zeta(\boldsymbol{\xi})\rangle_{L^2(\widehat{\mathcal D};\mathbb R^n)}
\\[0.2cm]
& = & \langle v ,(1-\Delta)^{k}\zeta \rangle_{L^2}    
\\[0.2cm]
& = & \langle v,  \eta \rangle_{L^2}.
\end{array}
\end{equation}
In addition,  we have the pathwise identities
\begin{equation}
\langle \Delta z(s),\zeta\rangle_{H^{k-1}; H^{k+1}}=\langle z(s) ,\Delta \zeta\rangle_{H^k} = \langle z(s), \Delta \eta \rangle_{L^2},
\end{equation}
for any $0\leq s< \tau_\infty$. We find that $z(t)$ is an analytically weak solution in $L^2,$ i.e.,
\begin{equation}\begin{aligned}
    \langle z(t),\eta\rangle_{L^2}&=\langle z_0,\eta\rangle_{L^2}+\int_0^t\langle  z(s),\Delta\eta\rangle_{L^2}\mathrm ds+\int_0^t\langle f(z(s)+\Phi_{\rm ref}),\eta\rangle_{L^2}\mathrm ds\\&\hspace{3cm}+\sigma \int_0^t\langle g(z(s)+\Phi_{\rm ref})\mathrm dW_s^Q,\eta\rangle_{L^2},
    \end{aligned}
\end{equation}
and conclude that $z(t)$ is  a solution in the sense of Proposition \ref{prop:existLip} by invoking a standard density argument. The fact that $\gamma(t) $ and $\tilde\gamma(t)$ coincide directly follows.
\end{proof}


\subsection{Global existence in \texorpdfstring{$H^1$}{H1} for \texorpdfstring{$1 \le d \le 4$}{1<=d<=4}}
\label{subsec:small:dim}

We conclude by considering an alternative approach towards global existence that is valid in low spatial dimensions, namely $1\leq d\leq 4$. The relevant Gelfand triple is given by
 \begin{equation}
 \label{eq:var:strong:setting}
        \mathcal V=H^2(\mathcal D;\mathbb R^n)\times\mathbb R,\quad \mathcal H=H^1(\mathcal D;\mathbb R^n)\times\mathbb R,\quad \mathcal V^*=  L^2(\mathcal D;\mathbb R^n)\times\mathbb R,
    \end{equation}
which is analogous to the \textit{strong setting}
in \cite{agresti2022critical}.  
The duality pairings are the same as those in {\S}\ref{sec:regularity}
upon fixing $k=1$ and we consider the cubic growth condition (Hf-Cub).

For  spatial dimension $d=1$, the classical framework
\cite{liu2010spde} can be used to construct
solutions to \eqref{eq:system}
in the weak setting \eqref{eq:var:weak:setting},
under the additional variational condition \eqref{eq:int:mono:sign} \cite{hamster2019stability,hamster2020}.
However, the Sobolev embeddings simply do not work out in $d > 1$, while the analogue of \eqref{eq:int:mono:sign}
fails in the strong setting \eqref{eq:var:strong:setting}.
Fortunately, such an inequality is not needed
for the approach developed in the critical variational framework \cite{agresti2022critical}.

Let us point out that  we cannot solve $(z,\gamma)$ in one go as in the previous settings. This is because the estimate 
\eqref{eq:var:strong:cub:bnd:f}
for $f$ below implies
\begin{equation}\label{eq:var:bnd:a:cub}\begin{aligned}
|a_\sigma(\Phi+v_A,\gamma_A,c)-a_\sigma(\Phi+v_B,\gamma_B,c)|\leq& \,K\big[1+\|v_A\|^2_{H^{4/3}}+\|v_B\|^2_{H^{4/3}}\big]\|v_A-v_B\|_{H^{4/3}}\\&\quad\quad+K\big[1+\|v_B\|_{H^{4/3}}^3\big]|\gamma_A-\gamma_B|,
\end{aligned}
    \end{equation}
   which does not satisfy Assumption 3.1 in \cite{agresti2022critical} due to the cubic growth term. 
Nevertheless, we can solve for $z$ first and follow the random-coefficient approach developed in \cite{hamster2019stability,hamster2020} to understand $\gamma$.

We remark that within this setting the dimension restriction $1\leq d\leq 4$ is induced by both the nonlinearity $f$ and the noise term $g$. It is related
to the critical exponents $4/3$ and $3/2$ that appear
in \cite{agresti2022critical}
and the proof below.

\begin{proof}[{Proof of Proposition \ref{prop:var:higher:full} under \textnormal{(Hf-Cub)}}]
As before, the estimates pertaining to the $\gamma$-variable 
 in the proof of Proposition \ref{prop:existLip} also suffice here, so we only need to focus on the terms related to the SPDE of $z$. 
Without loss, we  take $h=0$ again.
On account of (Hf-Cub), Lemma \ref{lem:nl:low:lip:theta:cub:l2}
provides the bound
    \begin{equation}
       \label{eq:var:strong:cub:bnd:f}
        \begin{aligned}
        \|f(\Phi+v_A)-f(\Phi+v_B)\|_{L^2}
        &\leq C_1[1+\|v_A\|^2_{H^{4/3}}+\|v_B\|^2_{H^{4/3}}]\|v_A-v_B\|_{H^{4/3}},
        \end{aligned}
    \end{equation}
    which involves the exponent $4/3$ that is critical in \cite{agresti2022critical}. 
In addition, Lemmas \ref{lem:HS:z}
and \ref{lem:nl:low:lip:theta:h1:h1p5:h2}
together yield the estimate
\begin{equation}\label{eq:var:H1g}
     \|g(\Phi+v_A)-g(\Phi+v_B)\|_{HS(L^2_Q;H^1)}\leq C_2 (1+\|v_B\|_{H^{3/2}})\|v_A-v_B\|_{H^{3/2}},
\end{equation}
where again the value $3/2$ is critical.
The remaining conditions needed to invoke \cite[Thm. 3.4]{agresti2022critical} for the first equation 
in \eqref{eq:system} can be readily verified.
As a result, we may write $u = \Phi_{\mathrm{ref}} + z$
   for the solution to this equation, 
   and conclude that $z=(z(t))_{t\in[0,T]}$  lives in the Bochner spaces
    \begin{equation}
     C([0,T]; H^1)\quad \text{and}\quad L^2([0,T]; H^2).
    \end{equation}
    Using the interpolation bound $\|v\|_{H^{4/3}} \le \|v\|_{H^1}^{2/3}\|v\|_{H^2}^{1/3}$,
    we obtain
    \begin{equation}
    \begin{aligned}
        \int_0^T  \|z(t)\|_{H^{4/3}}^3 \, \mathrm dt
        & \le  \|z\|_{C([0,T];H^1)}\int_0^T\|z(t)\|_{H^{2}} \, \mathrm dt 
        \\
        &\le \sqrt{T}  \|z\|_{C([0,T];H^1)} \|z\|_{L^2([0,T];H^{2})} ,
    \end{aligned}
    \end{equation}
    which implies $z\in L^3([0,T];H^{4/3})$. 
    In particular, the bound
    \eqref{eq:var:bnd:a:cub} yields
    \begin{equation}\begin{aligned}
|a_\sigma(\Phi+z(t),\gamma_A,c)-a_\sigma(\Phi+z(t),\gamma_B,c)|\leq K\rho\big(z(t)\big)|\gamma_A-\gamma_B|
\end{aligned}
\end{equation}
with the weight function 
\begin{equation}
\rho(v)=1+\|v\|_{H^{4/3}}^3.
\end{equation}
Note that $t \mapsto \rho\big(z(t)\big)$ is integrable.
This allows us to follow the approach
 in the proof of \cite[Prop. 4.5.2]{hamster2020}|which refers to \cite[Ch. 3]{prevot2007concise}|to
 establish the global existence and uniqueness of  $\gamma(t)$.
\end{proof}

We remark that under the weaker condition (Hf-Lip),
the same result can be obtained by simply 
applying the classical variational framework \cite{liu2010spde}
directly to the full problem \eqref{eq:system}.
In particular, this approach uses the estimates 
\begin{equation}
\begin{array}{lcl}
        \|f(\Phi+v_A)-f(\Phi+v_B)\|_{H^1}
        &\leq &  K (1+\|v_B\|_{H^2})\|v_A-v_B\|_{H^1},
    \\[0.2cm]
     \|g(\Phi+v_A)-g(\Phi+v_B)\|_{HS(L^2_Q;H^1)}
     &\leq & K (1+\|v_B\|_{H^2})\|v_A-v_B\|_{H^1}, 
    \end{array}
    \end{equation}
    that follow from
     Lemma \ref{lem:nl:low:lip:theta:h1:h1p5:h2} and  are again valid for $1 \le d \le 4$ only. Alternatively, one can invoke the framework 
     in \cite{agresti2022critical}, for which it suffices to use
     the bound
     \begin{equation}
         \|f(\Phi+v_A)-f(\Phi+v_B)\|_{L^2}\leq K\|v_A-v_B\|_{L^2,}
     \end{equation}
     in combination with \eqref{eq:var:H1g}.

\begin{proof}[Proof of Proposition \ref{prop:existLip} under \textnormal{(Hf-Cub)}]
This is simply a restatement of the global $H^1$-results
in Proposition \ref{prop:var:higher:full}.
\end{proof}


\section{Evolution equations of the perturbation}\label{sec:pert}
In this section, we establish  equations for the evolution of the perturbation
\begin{equation}
    v(t)=T_{-\gamma(t)}u(t)-\Phi=T_{-\gamma(t)}[z(t)+\Phi_{\rm ref}]-\Phi,\label{eq:v}
\end{equation}
where $z(t)$ and $\gamma(t)$ were constructed in {\S}\ref{sec:variational}. In   {\S}\ref{sec:pert_SPDE} we show that $v(t)$ satisfies
\begin{equation}
\label{eq:pert:ev:eq:for:v}
   \mathrm dv=\mathcal R_\sigma(v;c,\Phi)\mathrm dt+\sigma \mathcal S(v;\Phi)\mathrm dW_s^Q,
\end{equation}
in the variational
sense, where the nonlinearities $\mathcal R_\sigma$ and $\mathcal S$ are given by
\begin{equation}
    \mathcal R_\sigma(v;c,\Phi)=
    \Delta_y v +
    \kappa_\sigma(\Phi +v , 0) 
    \big[ \partial_x^2 v + \Phi'' + \mathcal{J}_\sigma( \Phi + v, 0; c)
    \big] 
    +a_\sigma(\Phi + v, 0, c)\partial_x(\Phi+v)\label{eq:pert}
\end{equation}
and
\begin{equation}\mathcal S(v;\Phi)[\xi]=g(\Phi+v)[\xi]+\partial_x(\Phi+v)b(\Phi +v,0)[\xi],\quad \xi\in L^2_Q;\end{equation}
see Appendix \ref{list} for the definitions of $a_\sigma$,  $b$, $\kappa_\sigma$ and $\mathcal{J}_\sigma$, which are analogous to the $d=1$ case
and do not involve second derivatives.

It is worthwhile to point out that  $v(t)$ explicitly depends on the phase $\gamma(t)$, yet $\gamma(t) $ is absent in the equations above. This is a consequence of the translational invariance of our system, which allows us to reduce the coupled system \eqref{eq:system} to a single system of equations; see \eqref{eq:app:list:comm:rels:f:g}--\eqref{eq:app:list:comm:rels:a:b}.

Subsequently, in {\S}\ref{sec:time_transform} we consider a stochastic time transformation to change the coefficient in front of $\partial_x^2$ into unity; as in $\mathcal{L}_{\rm tw}$. In particular, we show that the
transformed function $\bar v(t)$ satisfies the system
\begin{equation}
    \mathrm d \bar v=[
    \mathcal L_{\rm tw} \bar v +\kappa_\sigma(\Phi + \bar v,0)^{-1}\Delta_y \bar v
    +\mathcal N_\sigma(\bar v)]\mathrm dt+\mathcal M_\sigma(\bar v)\mathrm d {W}_t^Q,\label{eq:towards4}
\end{equation}
where the definitions of $\mathcal{N}_\sigma$ and $\mathcal{M}_\sigma$
can be found in \eqref{eq:app:def:n:m:sigma}. As before, note that these functions do not contain second derivatives and are analogous to the $d = 1$ case. The main message is that we have cleanly isolated
the linear operator $\mathcal{L}_{\rm tw}$, allowing us to pass to a mild formulation where we can exploit the semigroup $S_{\rm tw}(t)$.
Indeed, in {\S}\ref{sec:mild_proof} we   show that the variational solution to  \eqref{eq:towards4} is also a mild solution in some sense, however, the time-dependent coefficient in front of the Laplacian $\Delta_y$  forces us to consider random evolution families, causing significant complications for our $d>1$ case that were absent  in \cite{hamster2019stability,hamster2020}.


\subsection{An application of It\^o's formula}\label{sec:pert_SPDE}

Our starting point here is the local $H^k$-valued solutions
constructed in Proposition \ref{prop:var:higher:full},
which we will use to 
provide a rigorous interpretation 
for \eqref{eq:pert:ev:eq:for:v}. To this end,
we note that the second derivatives in $\mathcal{R}_{\sigma}$
can be interpreted as an element in $\mathscr L(H^{k+1},H^{k-1})$ 
in the usual sense, which corresponds with the duality pairing 
\begin{equation}
    \langle [ \Delta_y + \kappa \partial_x^2  ]v,w\rangle_{H^{k-1};H^{k+1}}=
    - \langle \nabla_y v, \nabla_y w \rangle_{H^k}
    -\kappa \langle \partial_x v, \partial_x w \rangle_{H^k},
\end{equation}
for any $\kappa \in \mathbb{R}$, $v \in H^{k+1}$, and $w \in H^{k+1}$.

\begin{proposition}\label{prop:pert-higher}
    Consider  the setting of   Proposition \ref{prop:var:higher:full} and suppose  that condition \textnormal{(HPar)} holds. Then the map
    \begin{equation}
        v:[0,T]\times \Omega \to H^k(\mathcal D;\mathbb R^n)
    \end{equation}
    defined in \eqref{eq:v} is progressively measurable and satisfies the following properties:
    \begin{enumerate}[\upshape(i)]
            \item For almost every $\omega\in\Omega$, the map
    $ 
        t\mapsto v(t,\omega)
    $
    is of class $C([0,\tau_\infty(\omega));H^k(\mathcal D;\mathbb R^n))$;
   \item For any $\ell\geq 0$ we have the integrability condition $v\in L^2(\Omega;L^2([0,\tau_\ell];H^{k+1}(\mathcal D;\mathbb R^n))$;
            \item For almost every $\omega\in\Omega$ and any $\ell \ge 0$, we have $\mathcal R_\sigma(v(\cdot,\omega);c,\Phi)\in L^1([0,\tau_\ell(\omega)];H^{k-1}(\mathcal D;\mathbb R^n))$, together with $
            \mathcal S(v;\Phi)\in L^2( \Omega; L^2([0,\tau_\ell];HS( L^2_Q;H^k(\mathcal D;\mathbb R^n)))$;
        \item The $H^{k}(\mathcal D;\mathbb R^n)$-valued identity 
        \begin{equation}
            v(t)=v(0)+\int_0^t\mathcal R_\sigma(v(s);c,\Phi)\mathrm ds+\sigma\int_0^t\mathcal S(v(s);\Phi)\mathrm dW_s^Q,
        \end{equation}
            holds $\mathbb P$-a.s. for all $0\leq t<\tau_\infty$. 
    \end{enumerate}
\end{proposition}

\begin{proof}     
    Properties (i)--(iii) follow rather directly from Proposition \ref{prop:var:higher:full}. The proof of (iv) is completely analogous to that of in \cite[Sec. 5.4]{hamster2020}, because an application of It\^o's formula   only results into terms with derivatives with respect to the direction of the wave, since $\gamma(t)$ only affects the $x$-coordinate.
    In particular, the Laplacian with respect to $y$ introduces no unexpected terms.
    
    In  more detail, we pick an arbitrary test function $\zeta\in C_c^\infty(\mathcal D;\mathbb R^n)$ and consider the maps
    \begin{equation}
        \psi_{1;\zeta}:H^{k-1}\times \mathbb R\to\mathbb R,\quad  \psi_{2;\zeta}:\mathbb R\to\mathbb R,
    \end{equation}
    defined by
\begin{equation}
    \psi_{1;\zeta}(z,\gamma)=\langle z,T_{\gamma}\zeta\rangle_{H^{k-1};H^{k+1}}
\end{equation}
and
\begin{equation}
    \psi_{2;\zeta}(z,\gamma)=\langle \Phi_{\rm ref}-T_\gamma \Phi_\sigma,T_{\gamma}\zeta\rangle_{H^{k-1};H^{k+1}}=\langle \Phi_{\rm ref}-T_\gamma \Phi_\sigma,T_{\gamma}\zeta\rangle_{H^k}.
\end{equation}
 By construction, we have the identity
 \begin{equation}
     \langle v(t),\zeta\rangle_{H^k} =\psi_{1;\zeta}(z(t),\gamma(t))+\psi_{2;\zeta}(\gamma(t)).
 \end{equation}
Performing  computations as in \cite[Lem. 5.5]{hamster2020} and \cite[Cor. 5.6]{hamster2020} leads to the expression
\begin{equation}
        \langle v(t),\zeta\rangle_{H^k}=\langle v(0),\zeta\rangle _{H^k}+\int_0^t\langle \mathcal R_\sigma(v(s);c,\Phi),\zeta\rangle_{H^{k-1};H^{k+1}}\mathrm ds+\sigma\int_0^t\langle \mathcal S(v(s);\Phi)\mathrm d\widetilde W_s^Q,\zeta\rangle_{H^k},
    \end{equation}
    where \begin{equation}\widetilde W_t^Q=\sum_{k=0}^\infty \int_0^tT_{-\gamma(s)}\sqrt{Q}e_k  \mathrm d\beta_k(s)\end{equation} 
    is a stochastic process that is, in fact, indistinguishable from the cylindrical $Q$-Wiener process $W^Q_t$ on account of the translational invariance of $Q$; see the proof of \cite[Prop 5.4]{hamster2020} for more information. By convention, we may replace the stochastic process $\widetilde W^Q_t$ by $W^Q_t$, which proves the assertion.
\end{proof}

\begin{remark} As a matter of fact, throughout the proof of Proposition \ref{prop:pert-higher}, one only needs to exploit the translation invariant property of our noise and our system in the $x$-coordinate. Therefore, it is also possible to consider other types of noise with a more general $y$-dependence.
\end{remark}

\subsection{Stochastic time transformation}\label{sec:time_transform}
Our goal here is to introduce the time transformation that will lead to the system \eqref{eq:towards4}. As in the one-dimensional setting \cite{hamster2019stability,hamster2020}, 
we will rescale time homogeneously over space in order
to divide out the troublesome $\kappa_\sigma$ coefficient in front of the $\partial_{x}^2$ term in \eqref{eq:pert:ev:eq:for:v}. However, it will
reappear in front of the $\Delta_y$ term, which means
that the resulting problem is still quasi-linear instead of semi-linear as in the one-dimensional setting.
Nevertheless, this repositioning is an essential part of our analysis.


Consider $v=(v(t))_{t\in[0,T]}$ to be the map defined 
in 
\eqref{eq:v}, where we let $(z,\gamma)$ be the global solution constructed
in Proposition \ref{prop:existLip}, and where we take
$(\Phi,c) = (\Phi_\sigma,c_\sigma)$ from this point forward.
%
Let us introduce the notation
\begin{equation}
    \tau_{\sigma;v}(t,\omega)=\int_0^t \kappa_\sigma(\Phi_\sigma+v(s,\omega),0) \, \mathrm ds.
\end{equation}
Lemma \ref{lem:nl:kappa_sigma} shows that $t\mapsto \tau_{\sigma;v}(t)$ is a continuous strictly increasing $\mathbb F$-adapted process. This implies that it admits an inverse map $t_{\sigma;v}:[0,T]\times\Omega\to[0,T]$ that satisfies
\begin{equation}
    \tau_{\sigma;v}(t_{\sigma;v}(\tau,\omega),\omega)=\tau,\qquad t_{\sigma;v}(\tau_{\sigma;v}(t,\omega),\omega)=t,
\end{equation}
for (almost) every $\omega\in\Omega.$
Specifically, we may use \eqref{eq:nl:bnd:on:n:sigma} to find a  constant $K_{\kappa} \ge 1$ such that
\begin{equation}
t\leq \tau_{\sigma;v}(t)\leq K_\kappa t, \qquad K_{\kappa}^{-1}\tau \leq t_{\sigma;v}(\tau)\leq \tau, \label{eq:Kkappa}
\end{equation}
for all $0\leq t,\tau\leq T$.

We now define the stochastic time transformed perturbation as
\begin{equation}
    \bar{v}(\tau,\omega)=v(t_{\sigma;v}(\tau,\omega),\omega).\label{eq:bar_v}
\end{equation}
Applying standard time transformation rules \cite[Lem. 6.2]{hamster2019stability}
to \eqref{eq:pert:ev:eq:for:v} formally leads to 
the system
\begin{equation}
    \mathrm d \bar v=\kappa_\sigma^{-1}(\Phi_\sigma + \bar v, 0)\mathcal R_\sigma(\bar v;c_\sigma,\Phi_\sigma)\mathrm d\tau +
    \sigma \kappa_\sigma^{-1/2}(\Phi_\sigma +\bar v,0)\mathcal S(\bar v;\Phi_\sigma) {\mathrm d\overline W}_\tau^Q,\label{eq:rewrite}
\end{equation}
in which
$\overline W_\tau^Q$ is again a $Q$-cylindrical Wiener process, but now adapted to the filtration 
$\overline{\mathbb F}=(\overline{\mathcal F}_\tau)_{\tau\geq 0}$ given by \begin{equation}\overline{\mathcal F}_\tau=\{A\in\mathcal F:A\cap \{t_{\sigma;v}(\tau)\leq t\}\in\mathcal F_t,\,\text{for all }t\geq 0\}.\label{eq:filtration}
\end{equation}
In particular, we have
\begin{equation}
    \overline W^Q_\tau=\sum_{k=0}^\infty\sqrt{Q}e_k\bar \beta_k(\tau),\qquad  \bar\beta_k(
    \tau)=\int_0^\tau \frac1{\sqrt{\partial_\tau t_{\sigma;v}(\tau')}}\mathrm d\beta_k(t_{\sigma;v}(\tau')).
\end{equation}
For our purposes here, it suffices to note that $\overline{W}^Q_\tau$ has the same statistical properties as $W^Q_t$. Upon recalling the definitions
\eqref{eq:app:def:n:m:sigma},
we see that \eqref{eq:rewrite} can be written in the form
\begin{equation}
\label{eq:pert:id:for:bar:v:informal}
    \mathrm d \bar v=[
    \mathcal L_{\rm tw} \bar v +\kappa_\sigma(\Phi_\sigma + \bar v,0)^{-1}\Delta_y \bar v
    +\mathcal N_\sigma(\bar v)]\mathrm dt+\mathcal M_\sigma(\bar v)\mathrm d {\overline W}_t^Q.
\end{equation}
These computations are made rigorous in the following result,
in which we have introduced the transformed stopping times
\begin{equation}
  \bar{\tau}_{\ell} = \tau_{\sigma;v}( \tau_{\ell} ) \wedge T ,  \qquad \bar{\tau}_\infty = \tau_{\sigma;v}( \tau_\infty ) \wedge T.
\end{equation}

\begin{proposition}[stochastic time transform]  \label{prop:pert-transform} Consider  the setting of Proposition \ref{prop:var:higher:full}
and suppose that condition \textnormal{(HPar)} holds. 
    Then the map
    \begin{equation}
        \bar v:[0,T]\times \Omega \to H^k(\mathcal D;\mathbb R^n)
    \end{equation}
    defined in \eqref{eq:bar_v} is progressively measurable with respect to the  filtration $\overline{ \mathbb F}=(\overline{\mathcal F}_t)_{t\geq 0}$  in \eqref{eq:filtration} and  satisfies the following properties:
    \begin{enumerate}[\upshape(i)]
            \item For almost every $\omega\in\Omega$, the map
    $ 
        t\mapsto \bar{v}(t,\omega)
    $
    is of class $C([0,\bar{\tau}_\infty(\omega));H^k(\mathcal D;\mathbb R^n))$;
   \item For any $\ell \ge 0$ we have the integrability condition $\bar{v}\in L^2(\Omega;L^2([0,\bar{\tau}_\ell];H^{k+1}(\mathcal D;\mathbb R^n)))$;
   \item For almost every $\omega\in\Omega$ and any $\ell \ge 0$, we have $\mathcal N_\sigma(\bar{v}(\cdot,\omega))\in L^1([0,\bar{\tau}_\ell(\omega)];H^{k}(\mathcal D;\mathbb R^n))$, together with $
            \mathcal M_\sigma(\bar{v})\in L^2(\Omega; L^2([0,\bar{\tau}_\ell];HS( L^2_Q,H^k(\mathcal D;\mathbb R^n)))$;
        \item The $H^k(\mathcal D;\mathbb R^n)$-valued identity
        \begin{equation}\begin{aligned}
            \bar v(t)&=\bar v(0)+\int_0^t 
            \big[ \mathcal L_{\rm tw}\bar v(s)
            + \kappa_\sigma (\Phi_\sigma + \bar v(s) , 0)^{-1}
            \Delta_y \bar v(s)
            \big] 
            \mathrm ds\\&\qquad\qquad\qquad\qquad\qquad\quad+\int_0^t\mathcal N_\sigma(\bar v(s))\mathrm ds+\sigma\int_0^t\mathcal M_\sigma(\bar v(s))\mathrm d\overline W_s^Q
            \end{aligned}
        \end{equation}
        holds $\mathbb P$-a.s. for all $0\leq t<\bar{\tau}_\infty$.
    \end{enumerate}

\end{proposition}
\begin{proof}
The result can be obtained by following the proof of 
\cite[Prop. 6.3]{hamster2019stability} and 
\cite[Lem. 6.3]{hamster2020}.
\end{proof}

\subsection{Mild formulation}
\label{sec:mild_proof}

We are now ready to recast the system \eqref{eq:pert:id:for:bar:v:informal} into an appropriate mild formulation. The quasi-linear nature of the problem
causes several complications that need to be addressed, which we achieve by utilising the theory of forward integration discussed in {\S}\ref{sec:forward}.

Let $\bar v=(\bar v(t))_{t\in[0,T]}$  denote
the stochastic process defined by \eqref{eq:bar_v}, seen as the globally continuous $L^2$-valued  process on account of Proposition \ref{prop:existLip}. Observe that this process is indistinguishable from the one in Proposition \ref{prop:pert-transform} on the interval $[0,\bar{\tau}_\infty)$ as a result of Lemma \ref{lem:ex:hk:agree:l2}. 
Having a globally defined $\bar v$  allows us to
introduce the  random function
\begin{equation}
\nu_\sigma(t,\omega):=\kappa_\sigma(\Phi_\sigma + \bar v(t,\omega), 0)^{-1} 
\end{equation}
for all time $0\leq t\leq T.$ 
Notice that this function is progressively measurable and continuous with respect to  $t$ for almost every $\omega\in\Omega$ due to Lemma \ref{lem:nl:kappa_sigma}  and the fact 
\begin{align}
    |\nu_\sigma(t)-\nu_\sigma(s)|=|\kappa_\sigma(\Phi + \bar v(t),0)^{-1}-\kappa_\sigma(\Phi + \bar v(s),0)^{-1}|
    \leq K\sigma^2\|\bar v(t)-\bar v(s)\|_{L^2},\quad t,s\geq 0.
\end{align}
Indeed, recall $\bar v \in C([0,T],L^2).$
In particular, property (H$\nu$) is satisfied,
with constants $K_\nu=1$ and $k_\nu=\frac{1}{1+\sigma^2K}>0$ for some $K>0.$ Note that for $\sigma=0$ we simply have $\nu(t)\equiv 1$.

By construction, we can now use the family of
random linear operators $\mathcal{L}_{\nu_\sigma}(t)$
defined as in \eqref{eq:lin_gen}
with $\nu=\nu_\sigma$ to recast \eqref{eq:pert:id:for:bar:v:informal} into the form
\begin{equation}
\label{eq:per:id:for:bar:v:with:l:nu}
    \mathrm d \bar v=[
    \mathcal L_{\nu_\sigma} (t) \bar v
    +\mathcal N_\sigma(\bar v)]\mathrm dt+ \sigma \mathcal M_\sigma(\bar v)\mathrm d \overline{W}_t^Q.
\end{equation}
Furthermore, we denote by $E(t,s)$  the associated
evolution family \eqref{eq:evolution}
that features in our mild formulation below.
We emphasise that the
stochastic integral in \eqref{eq:forward_g} is
referred to as a generalised forward integral
and needs to be understood as $J_2\big( \mathcal{M}_\sigma(\bar v) \big)$;
see Corollary \ref{cor:important}.
We point out that it is unknown
whether this integral is a (strict) forward integral
in the sense of Definition \ref{def:forward},
i.e., whether we have $I^-\big(E(t,\cdot)\mathcal{M}_\sigma(\bar v(\cdot))\big)=J_2\big(\mathcal{M}_\sigma(\bar v)\big)$.

\begin{proposition}[mild solution]\label{prop:forward}
    Consider either the setting of Proposition  \ref{prop:var:higher:full} and suppose that condition \textnormal{(HPar)} holds. Then the map 
    \begin{equation}
        \bar v:[0,T]\times \Omega \to H^k(\mathcal D;\mathbb R^n) 
    \end{equation}
    defined in \eqref{eq:bar_v} satisfies the
    $H^k(\mathcal D;\mathbb R^n)$-valued identity
    \begin{equation}
        \bar v(t)=E(t,0)\bar v(t)+\int_0^tE(t,s)\mathcal N_\sigma(\bar v(s))\mathrm ds+\sigma\int_0^tE(t,s)\mathcal M_\sigma(\bar v(s))\mathrm d \overline{W}_s^-,\label{eq:forward_g}
    \end{equation}
    which holds $\mathbb P$-a.s. for all $0\leq t< \bar{\tau}_\infty$. 
\end{proposition}
\begin{proof}
Recall that $\bar v$ is in $C([0,T];L^2)$. For convenience, we introduce the shorthand notations
\begin{equation}
    \textbf{f}(t)=\mathcal N_\sigma(\bar v(t))\quad\text{and}\quad \textbf g(t)=\mathcal M_\sigma(\bar v(t)),
\end{equation}
and define the stochastic process $\bar w=(\bar w(t))_{t\in[0,T]}$ by
\begin{equation}\bar w(t)=E(t,0)\bar v(0)+\int_0^tE(t,s)\textbf f(s)\mathrm ds+\sigma\int_0^tE(t,s)\textbf g(s)\mathrm d \overline{W}_s^-, \quad 0\leq t\leq T.\end{equation}
By inspecting the definitions \eqref{eq:app:def:n:m:sigma} for $\mathcal{N}_\sigma$ and $\mathcal{M}_\sigma$ and appealing to Corollary \ref{cor:important}, we can conclude that $\bar w$ is well-defined and a continuous $L^2$-valued process on $[0,T]$ as well. 

By a slight adaptation of the proof in  \cite[Thm. 6.6]{van2021maximal},
using the proof of \cite[Thm. 3.2]{gawarecki2010stochastic}
to take into account the deterministic term $\textbf{f}$,
one can show that $\bar w$ is  an analytically weak solution 
 in $L^2$ to the linear problem
\begin{equation}
\label{eq:pert:mild:prob:for:z}
  \mathrm   dz=[\mathcal L_{\nu }(t)z+\textbf f(t)]\mathrm dt+\textbf g(t)\mathrm d \overline{W}_t^Q,
\end{equation}
which means pathwise that
\begin{equation}
  \langle \bar{w}(t), \zeta \rangle_{L^2} =
  \langle \bar{w}(0), \zeta \rangle_{L^2} 
  +
  \int_0^t \langle \bar{w}(s) , \mathcal L^{\mathrm{adj}}_{\nu }(s) \zeta \rangle_{L^2} \, \mathrm ds
  + \int_0^t \langle \textbf f(s), \zeta \rangle_{L^2} \, \mathrm ds+\int_0^t \langle \textbf g(s)\mathrm d \overline{W}_s^Q,
   \zeta \rangle_{L^2}
\end{equation}
holds for all $\zeta \in C_c^{\infty}(\mathcal D;\mathbb R^n)$ and $0 \le t < \bar{\tau}_\infty$. 
On the other hand, Proposition \ref{prop:pert-transform} implies that  $\bar{v}$ is also an analytically weak solution to \eqref{eq:pert:mild:prob:for:z}  on $[0, \bar{\tau}_\infty)$.
Upon defining $\bar \phi=\bar v-\bar w$, we observe  that $\bar \phi$ solves the linear initial value problem
\begin{equation}
    \partial_t \phi=\mathcal L_{\nu}(t) \phi,\quad \phi(0)=0,
\end{equation}
on $[0, \bar{\tau}_\infty)$,
which clearly only has the zero solution $\phi=0$.
We therefore see that $\bar v=\bar w$ holds with equality in $L^2$ on $[0, \bar{\tau}_\infty)$,
which means that they are equal almost everywhere on this interval. Since $\bar v$ is known to take values in $H^k$ on this interval, we conclude that $\bar w$ does the same and that we have equality in this space, completing the proof.
%
\end{proof}

\section{Nonlinear stability}\label{sec:stability}

In this section we prove the estimate in Proposition \ref{prop:stab-overview}, which leads to the stochastic metastability of planar waves over exponentially long timescales as described in Theorem \ref{thm:main}. From this point onward, we replace the mild form in  \eqref{eq:forward_g} for the time transformed perturbation with the  generic equation
\begin{equation}
w(t)=E(t,0)w(0)+\int_0^tE(t,s)[\sigma^2F_{\rm lin}(w(s))+F_{\rm nl}(w(s))]\mathrm ds+\sigma \int_0^t E(t,s)B(w(s))\mathrm dW_s^-,\label{eq:generic}
\end{equation}
in which $E(t,s)$ is any random evolution family that is covered by the results in {\S}\ref{sec:forward}.
Throughout this section we assume  there exists a  continuous process $w:[0,\infty)\times \Omega\to H^k$ for some $k\geq 0$ and a stopping time $\tau_\infty$
so that equation \eqref{eq:generic} is satisfied for $0\leq t<\tau_\infty$. 

The maps
\begin{equation}
    F_{\rm lin}:H^{k+1}\to H^k,\quad F_{\rm nl}:H^{k+1}\to H^k,\quad B:H^{k+1}\to HS(L^2_Q;H^k),
\end{equation}
are assumed to satisfy the bounds
\begin{align}
    \|F_{\rm lin}(w)\|_{H^k}&\leq K_{\rm lin}\|w\|_{H^{k+1}}\label{9.2},\\
    \|F_{\rm nl}(w)\|_{H^k}&\leq K_{\rm nl}\|w\|_{H^{k+1}}^2\label{9.3},\\
    \|B(w)\|_{HS(L^2_Q,H^k)}&\leq K_{B}(1+\|w\|_{H^{k+1}}),\label{9.4}\end{align}
    whenever $\|w\|_{H^k}\leq 1$. 
In addition, if $\|w\|_{L^2}\leq \eta_0$ holds for some sufficiently small $\eta_0>0$, then we have the orthogonal identities 
\begin{equation}
    \langle \sigma F_{\rm lin}(w)+F_{\rm nl}(w),\psi_{\rm tw}\rangle_{L^2}=0\quad\text{and}\quad\langle B(w)[\xi],\psi_{\rm tw}\rangle_{L^2}=0.\label{eq:ort}
\end{equation}
These choices reflect the bounds that arise from (Hf-Cub); see Corollary \ref{cor:nl:low:fin:bnd:n:sigma}. Of course, they also cover the case where (Hf-Lip) holds, but the computations in Lemma \ref{lem:maximalregH1} could be simplified slightly by replacing  \eqref{9.3} with  $\|F_{\rm nl}(w)\|_{H^k}\leq K_{\rm nl}\|w\|_{H^k}\|w\|_{H^{k+1}}$; see Corollaries
\ref{cor:nl:bnd:fin:s:m} and \ref{cor:nl:ct:fin:bnd:n:sigma}. Lastly, we refer to 
the remarks on the orthogonality conditions in Appendix \ref{list}.

Fix $\epsilon \in (0, 2\mu)$ with $\mu$  as in Lemma \ref{lem:decay_E}. Recall that our main objective is to control the size of
\begin{equation}
N_{\epsilon;k}(t)=\|w(t)\|_{H^k}^2+\int_0^te^{-\epsilon(t-s)}\|w(s)\|_{H^{k+1}}^2\mathrm ds.\label{eq:size}
\end{equation}
 In particular, for any $0 < \eta < \eta_0$, we write 
\begin{equation}
    t_{\rm st}(\eta;k)=\inf\{t\geq 0:\|w(t)\|_{H^k}^2+\int_0^te^{-\epsilon(t-s)}\|w(s)\|_{H^{k+1}}^2\mathrm ds>\eta\}\label{eq:st}
\end{equation}
for the associated stopping time and observe that the  conditions in \eqref{eq:ort} are automatically satisfied for $w=w(t)$ with $t\leq t_{\rm st}(\eta;k)$.  In addition, the definition of the localising sequence \eqref{eq:var:def:tau:l:stopping:time} implies $t_{\mathrm{st}}(\eta;k) < \tau_\infty$. Finally, note that $\langle w(t),\psi_{\rm tw}\rangle_{L^2}=0$ for $t\leq t_{\rm st}(\eta;k)$ if and only if $\langle w(0),\psi_{\rm tw}\rangle_{L^2}=0$.

Our main result here provides logarithmic growth bounds
for the expectation of the maximal value that $N_{\epsilon;k}$ attains as we increase $T$.

\begin{proposition}\label{prop:E_stab}
 Consider the generic setting above. Pick two sufficiently small constants $\delta_\eta>0$ and $\delta_\sigma>0$. Then there exists a constant $K>0$ so that for any integer $T\geq 2,$ any $0<\eta<\delta_\eta,$ any $0\leq \sigma\leq \delta_\sigma,$ and any integer $p\geq 1$, we have the bound
 \begin{equation}
     \mathbb E\left[\sup_{0\leq t\leq t_{\rm st}(\eta;k)\wedge T}|N_{\epsilon;k}(t)|^p\right]\leq K^p\left[\|v(0)\|_{H^{k}}^{2p}+\sigma^{2p}\big(p^p+\log(T)^p\big)\right].
 \end{equation}
\end{proposition}

\begin{remark}
\label{rem:st:sharpening:d1}
    The estimates obtained in \cite{hamster2019stability,hamster2020diag,hamster2020}
    for $d = 1$ also fit into the generic framework of this section upon choosing $k=0$. In particular, Proposition \ref{prop:E_stab} sharpens \cite[Prop. 5.1]{hamster2020expstability},
    which means that the bound \eqref{eq:second} extends
    to this setting and improves \cite[Thm. 1.1]{hamster2020expstability}.
\end{remark}

Following the earlier work \cite{hamster2019stability}, 
we proceed by providing separate estimates for the integrals in \eqref{eq:generic}. To this end, we introduce the notation
\begin{equation}
\mathcal{E}_0(t)=E(t,0) P^\perp w(0),
\end{equation}
together with the integrals
\begin{equation}
\begin{aligned}
 \mathcal{E}_{F ; \text{lin}}(t)&=\int_0^{t} E(t,s) P^\perp F_{\operatorname{lin}}(w(s)) 1_{s<t_{\rm st}(\eta;k)} \mathrm ds, \\
 \mathcal{E}_{F ; \mathrm{nl}}(t)&=\int_0^{t} E(t,s) P^\perp F_{\mathrm{nl}}(w(s)) 1_{s<t_{\rm st}(\eta;k)} \mathrm ds, 
\\
 \mathcal{E}_{B}(t)&=\int_0^{t} E(t,s) P^\perp B(w(s)) 1_{s<t_{\rm st}(\eta;k)} \mathrm d W_s^-.
\end{aligned}
\end{equation}
The presence of the projection $P^\perp$ in the above is simply to emphasise condition \eqref{eq:ort}. Using these expressions, we obtain the estimate
     \begin{equation}\begin{aligned}
    \label{eq:st:a:priori:bnd:e}
        &\mathbb E\sup_{0\leq t\leq t_{\rm st}(\eta;k)\wedge T}\|w(t)\|_{H^k}^{2p}\\&\qquad\quad\leq 4^{2p}\mathbb E\sup_{0\leq t\leq T}\left[\|\mathcal E_0(t)\|_{H^k}^{2p}+\sigma^{4p}\|\mathcal E_{F;\rm lin}(t)\|_{H^k}^{2p}+\|\mathcal E_{F; \rm nl}(t)\|_{H^k}^{2p}+\sigma^{2p}\|\mathcal E_B(t)\|_{H^k}^{2p}\right].
        \end{aligned}
    \end{equation}
Turning to the integrated $H^{k+1}$-bound, we introduce the integrals
\begin{align}
\label{eq:st:a:priori:bnd:i}
\mathcal{I}_{  0}(t)&=\int_0^t e^{-\varepsilon(t- s)}\| \mathcal{E}_0(s)\|_{H^{k+1}}^2 \mathrm ds,\\
 \mathcal{I}_{ F ; \mathrm{lin}}(t)&=\int_0^t e^{-\varepsilon (t-s)}\| \mathcal{E}_{F ; \mathrm{lin}}(s)\|_{H^{k+1}}^2 \mathrm ds, \\
 \mathcal{I}_{ F ; \mathrm{nl}}(t)&=\int_0^t e^{-\varepsilon(t- s)}\|\mathcal{E}_{F ; \mathrm{nl}}(s)\|_{H^{k+1}}^2 \mathrm ds, \\
   \mathcal{I}_{ B}(t)&=\int_0^t e^{-\varepsilon(t- s)}\| \mathcal{E}_{B}(s)\|_{H^{k+1}}^2 \mathrm ds.
\end{align}
This leads directly to the estimate
  \begin{equation}
        \begin{aligned}
            &\mathbb E \sup_{0\leq t\leq t_{\rm st}(\eta;k)\wedge T}\left[\int_0^t e^{-\epsilon(t-s)}\|w(s)\|_{H^{k+1}}^2\mathrm ds\right]^p \\&\qquad\quad\leq 4^{2p}\mathbb E \sup_{0\leq t\leq T}\Big[\mathcal I_{0}(t)^p+\sigma^{4p}\mathcal I_{F;\rm lin}(t)^p
+ \mathcal I_{F;\rm nl}(t)^p+\sigma^{2p}\mathcal I_{B}(t)^p\Big].\\
        \end{aligned}
        \end{equation}

\begin{lemma}\label{lem:E0ElinEnl}
    For any $0<\eta<\eta_0$, any $T>0$, and any $p \ge 1$, we have the pathwise bounds 
    \begin{equation}\begin{aligned}
 \sup_{0\leq t\leq T}\|\mathcal E_0(t)\|_{H^k}^{2p}&\leq M^{2p}\|w(0)\|_{H^k}^{2p},\\
         \sup_{0\leq t\leq T}\|\mathcal E_{F;\rm lin}(t)\|_{H^k}^{2p}&\leq M^{2p}K_{\rm lin}^{2p} \sup_{0\leq t\leq t_{\rm st}(\eta;k)\wedge T} N_{\epsilon;k}(t)^p,\\
         \sup_{0\leq t\leq T}\|\mathcal E_{F;\rm nl}(t)\|_{H^k}^{2p}&\leq \eta^p M^{2p}K_{\rm nl}^{2p}\sup_{0\leq t\leq t_{\rm st}(\eta;k)\wedge T} N_{\epsilon;k}(t)^p.
    \end{aligned}\end{equation}
\end{lemma}
\begin{proof}
    These results 
    follow directly from  straightforward norm estimates; see also \cite[Lem. 5.3]{hamster2020expstability}.
\end{proof}

\begin{lemma}\label{lem:maximalregH1}
There exist a constant $K > 0$ so that for any $0<\eta<\eta_0$, any $T>0$, and any  $p \ge 1$, we have the pathwise bounds
    \begin{equation}\begin{aligned}
        \sup_{0\leq t\leq T}\mathcal I_{0}(t)^p &\leq 
        K^{2p}\|w(0)\|_{H^{k}}^{2p},\\
        \sup_{0\leq t\leq T} \mathcal I_{F;\rm lin}(t)^p&\leq
        K^{2p}\sup_{0\leq t\leq t_{\rm st}(\eta;k)\wedge T} N_{\epsilon;k}(t)^p,\\
          \sup_{0\leq t\leq T} \mathcal I_{F;\rm nl}(t)^p&\leq \eta^p
      K^{2p}\sup_{0\leq t\leq t_{\rm st}(\eta;k)\wedge T} N_{\epsilon;k}(t)^p .\\
      \end{aligned}\end{equation}
\end{lemma}
\begin{proof}
The key to establish these estimates is to use the time-dependent inner products $\langle \cdot, \cdot \rangle_{s; 0;k+1}$, as
defined in \eqref{eq:Hk+1nu}, and use the equivalence in \eqref{eq:Hk+1nubounds}. 
Indeed, the bound for
$\mathcal{I}_0$ follows 
by applying Lemma \ref{lem:max:reg:bnds:j:delta} with $\delta = 0$. The remaining estimates can be obtained
by following the computations in \cite[Lem. 9.8--9.12]{hamster2019stability}, using the identity
\eqref{eq:fw:deriv:e:delta:delta} with $\delta =0$.
\end{proof}

\begin{lemma}\label{lem:EB}
For any $0<\eta<\eta_0$, any integer $T \ge 2$, and any integer $p \ge 1$, we have the bound 
%
    \begin{equation}
        \mathbb E\sup_{0\leq t\leq T}\|\mathcal E_B(t)\|_{H^k}^{2p}
        + \mathbb E \sup_{0\leq t\leq T}\mathcal I_{B}(t)^p
        \leq 
        2^p K_{\rm gr} ^{2p} M^{2p}K_B^{2p}[\epsilon^{-1}+\eta]^p
        (p^p+\log(T)^p).
    \end{equation}
\end{lemma}
\begin{proof}
The moment estimate will  follow from  Proposition \ref{prop:E_B},
once we have verified that condition (HB) holds.
To see this, we compute 
\begin{equation}
    \int_0^te^{-\epsilon(t-s)}\|B(w(s))\textbf{1}_{s\leq t_{\rm st}(\eta;k)}\|_{HS(L^2_Q;H^k)}^2\leq 2 K_B^2(\epsilon^{-1}+\eta),
\end{equation}
for any $0\leq t\leq T.$
In addition, we note that \eqref{9.4} implies
\begin{align}   \|E_{\rm tw}(1)B(w)\|_{HS(L^2_Q;H^k)}&\leq K_{B}M(1+\|w\|_{H^{k}}),
\end{align}
which yields the bound
\begin{equation}
    \|E_{\rm tw}(1)B(w(t))\textbf{1}_{t\leq t_{\rm st}(\eta;k)}\|_{HS(L^2_Q;H^k)}^2\leq 2 M^2K_B^2(1+\eta),
\end{equation}
for any $0\leq t\leq T.$ In conclusion, condition (HB) is satisfied with $\Theta_*^2=2 M^2K_B^2(\epsilon^{-1}+\eta)$.
\end{proof}

\begin{proof}[Proof of Proposition \ref{prop:E_stab}]
Collecting the results in Lemmas \ref{lem:E0ElinEnl}--\ref{lem:EB}, the estimates \eqref{eq:st:a:priori:bnd:e}
and \eqref{eq:st:a:priori:bnd:i} can be combined to yield
\begin{equation}\begin{aligned}
            &\mathbb E\left[\sup_{0\leq t\leq t_{\rm st}(\eta;k)\wedge T}N_{\epsilon;k}(t)^p\right]\\&\qquad\quad\leq K^p\Bigg(\|v(0)\|_{H^{k}}^{2p}+\sigma^{2p}(p^p+\log(T)^p)
            +(\sigma^{4}+\eta)^p\mathbb E\left[\sup_{0\leq t\leq t_{\rm st}(\eta;k)\wedge T}N_{\epsilon;k}(t)^p\right]\Bigg).
            \end{aligned}
        \end{equation}
        The result hence readily follows by
        restricting the size of $\sigma^4+{\eta}$, .
\end{proof}

\begin{proof}[Proof of Proposition \ref{prop:stab-overview}] The bound can immediately be deduced from Proposition 
\ref{prop:E_stab} by following the proof of
\cite[Thm. 2.4]{hamster2019stability} to undo the time transformation, which only affects the value of the constants. In more detail,
recall the perturbation $v(t)$ defined in \eqref{eq:v}
together with the stochastic time transformed perturbation $\bar v(\tau)$ defined in \eqref{eq:bar_v}.
We now write
    \begin{equation}
        \bar N_{\epsilon,k}(\tau)=\|\bar v(\tau )\|_{H^k}^2+\int_0^\tau e^{-\epsilon(\tau-\sigma)}\|\bar v(\sigma)\|_{H^{k+1}}^2 \mathrm d\sigma,
    \end{equation}
    define the associated stopping time $\bar t_{\rm st}(\eta;k),$
    and recall the similar expressions 
    \eqref{eq:mr:def:n:eps:k} and \eqref{eq:mr:def:n:eps:k:st:time} for $v(t)$.
Observe the  inequality
    \begin{equation}
        \sup_{0\leq t\leq t_{\rm st}(\eta;k)\wedge T}N_{\epsilon;k}(t)\leq K_\kappa\sup_{0\leq \tau\leq \bar t_{\rm st}(K_\kappa^{-1}\eta;k)\wedge K_\kappa T}\bar N_{K_\kappa^{-1}\epsilon;k}(\tau), \label{eq:timetransformVSoriginal}
    \end{equation}
    where the constant $K_\kappa$ is as in \eqref{eq:Kkappa}. This inequality can be obtained by  tracing through the proof of \cite[Prop. 6.4]{hamster2019stability} and utilising the implication
    \begin{equation}
        \sup_{0\leq t\leq T}N_{\epsilon;k}(t)>\eta\implies \sup_{0\leq \tau\leq K_\kappa T}\bar N_{K_\kappa^{-1}\epsilon;k}(\tau)>K_\kappa^{-1}\eta ,
    \end{equation}
    which is equivalent to
    \begin{equation}
        t_{\rm st}(\eta)<T\implies \bar t_{\rm st}(K_\kappa^{-1}\eta)<K_\kappa T.
    \end{equation}
    This completes the proof.
\end{proof}

\appendix

\section{List of main functions}\label{list}
In this appendix we provide an overview of the main functions that are used in this paper. 
Throughout this section, we take
$k = 0$ if (Hf-Lip) is satisfied or $k=1$
if (Hf-Cub) is satisfied. We assume furthermore
that (HSt), (Hq), and (HCor)  hold with this choice of $k$.
In addition, we take $c,\gamma \in \mathbb R$ and $\xi \in L^2_Q$. We emphasise that we have kept our naming conventions as close as possible to those considered in  \cite{hamster2019stability,hamster2020} to prevent confusion.

We start by 
choosing  a smooth non-decreasing cut-off function
\begin{equation}
    \chi_{\rm low}:\mathbb R\to [\tfrac14,\infty)
\end{equation}
that satisfies the properties 
\begin{equation}
    \chi_{\rm low}(\theta)=\frac14 |\mathbb T|^{d-1},\quad \theta\leq \frac14 |\mathbb T|^{d-1},\quad \chi_{\rm low}(\theta)=\theta,\quad \theta\geq \frac12 |\mathbb T|^{d-1},
\end{equation}
together with a smooth non-increasing cut-off function 
\begin{equation}
    \chi_{\rm high}:\mathbb R_+
    \to [0,1]
\end{equation}
that satisfies the properties 
\begin{equation}
    \chi_{\rm high}(\theta)=1,\quad \theta\leq 2 + \|\Phi_0 - \Phi_{\rm ref}\|_{L^2(\mathcal D;\mathbb R^n)},\quad  \chi_{\rm high}(\theta)=0,\quad \theta\geq 3 + \|\Phi_0 - \Phi_{\rm ref}\|_{L^2(\mathcal D;\mathbb R^n)}.
\end{equation}
For any $u \in \mathcal U_{L^2(\mathcal D;\mathbb R^n)}$,
these cut-offs can be used to define 
\begin{equation}\label{eq:list:def:chi:h:l}\chi_h(u,\gamma)=\chi_{\rm high}(\|u-T_\gamma\Phi_{\rm ref}\|_{L^2(\mathcal D;\mathbb R^n)})\quad\text{and}\quad
    \chi_l(u,\gamma)=\big[ \chi_{\rm low}\big(-\langle  u,T_\gamma\psi_{\rm tw}'\rangle_{L^2(\mathcal D;\mathbb R^n)} \big) \big]^{-1}.
\end{equation}

We note that when (HPar) is satisfied and we take
$u = T_{\gamma}[\Phi + v]$ with
\begin{equation}
\label{eq:app:bnd:on:v:l2}
    \|v\|_{L^2(\mathcal D;\mathbb R^n)}
    \le \min \{ 1,  |\mathbb T|^{\frac{d-1}{2}} [4 \|\psi_{\rm tw}\|_{H^1(\mathbb R;\mathbb R^n)}]^{-1} \},
\end{equation}
then we have
\begin{equation}
    \chi_h(u, \gamma) = 1 
    \quad\text{and}\quad
    \chi_l(u, \gamma) = - \big[\langle u,  T_{\gamma} \psi_{\rm tw}' \rangle_{L^2(\mathcal D;\mathbb R^n)} \big]^{-1}.
\end{equation}
Indeed, we may compute
\begin{equation}
\begin{array}{lcl}
    \|u - T_\gamma \Phi_{\rm ref}\|_{L^2(\mathcal D;\mathbb R^n)}
    & \le & \|\Phi - \Phi_0\|_{L^2(\mathcal D;\mathbb R^n)}
    + \|\Phi_0 - \Phi_{\rm ref}\|_{L^2(\mathcal D;\mathbb R^n)}
    + \|v\|_{L^2(\mathcal D;\mathbb R^n)}
\\[0.2cm]
     &\le & 2 + \|\Phi_0 - \Phi_{\rm ref}\|_{L^2(\mathcal D;\mathbb R^n)},
\end{array}
\end{equation}
together with
\begin{equation}
\begin{array}{lcl}
| \langle \Phi_0 -  \Phi - v ,\psi_{\rm tw}'\rangle_{L^2(\mathcal D;\mathbb R^n)} | 
& \le & 
\big[ \|\Phi_0 - \Phi\|_{L^2(\mathcal D;\mathbb R^n)} + \|v\|_{L^2(\mathcal D;\mathbb R^n)} \big] |\mathbb T|^{\frac{d-1}{2}} \|\psi'_{\rm tw}\|_{L^2(\mathbb R; \mathbb R^n)}
\\[.2cm]&\le&  \frac{1}{2} |\mathbb T|^{d-1},
\end{array}
\end{equation}
to conclude
\begin{equation}
\begin{array}{lcl}
        - \langle  \Phi + v,\psi_{\rm tw}'\rangle_{L^2(\mathcal D;\mathbb R^n)}
        & = & - \langle  \Phi_0 ,\psi_{\rm tw}'\rangle_{L^2(\mathcal D;\mathbb R^n)}
        + \langle \Phi_0 -  \Phi - v ,\psi_{\rm tw}'\rangle_{L^2(\mathcal D;\mathbb R^n)}
\\[0.2cm]
& \ge & |\mathbb T|^{d-1} - \frac{1}{2}|\mathbb T|^{d-1}.
\end{array}
\end{equation}

Again taking $u \in \mathcal U_{L^2(\mathcal D;\mathbb R^n)}$, we introduce the scalar function
\begin{equation}
    b(u,\gamma)[\xi]=-\chi_h(u,\gamma)^2\chi_l(u,\gamma)\langle g(u)[\xi],T_{\gamma}\psi_{\rm tw}\rangle_{L^2(\mathcal D;\mathbb R^n)},\label{eq:b}
\end{equation}
together with
\begin{equation}
\label{eq:list:def:kappa}
\kappa_\sigma(u,\gamma)=1+\frac{\sigma^2}{2}\|b(u,\gamma)\|^2_{HS( L^2_Q;\mathbb R)},
\end{equation}
and the associated quantities
\begin{align}
\label{eq:list:def:nu}
    \nu_\sigma^{(1)}(u,\gamma)&=\kappa_\sigma(u,\gamma)-1,
    \\
    \nu_\sigma^{(-1)}(u,\gamma)&=\kappa_\sigma(u,\gamma)^{-1}-1,\\
    \nu_\sigma^{(-1/2)}(u,\gamma)&=\kappa_\sigma(u,\gamma)^{-1/2}-1.
\end{align}
These expressions are all well-defined by
Corollary \ref{cor:nl:bnds:for:b:and:kc}
and Lemma \ref{lem:nl:kappa_sigma}.

For $u \in \mathcal U_{L^2(\mathcal D;\mathbb R^n)}$,
the map $g(u): L^2_Q \to L^2(\mathcal D;\mathbb R^n)$ has a formal adjoint
$
    g^{\rm adj}(u):L^2(\mathcal D;\mathbb R^n)\to  L^2_Q,
$
that acts as
\begin{equation}
g^{\rm adj}(u)[\zeta]=Qg(u)^\top[\zeta],
\end{equation}
where the matrix transpose is taken in a pointwise fashion. Indeed, 
for $\xi \in L^2_Q$ and $\zeta \in L^2(\mathcal D;\mathbb R^n)$ we compute
\begin{equation}
\langle g(u)[\xi],\zeta\rangle_{L^2(\mathcal D;\mathbb R^n)}=
\langle  Q^{-1/2} \xi , Q^{-1/2} g^{\rm adj}(u)[\zeta]\rangle_{L^2(\mathcal D;\mathbb R^m)}
=
\langle \xi,g^{\rm adj}(u)[\zeta]\rangle_{L^2_Q}.
\end{equation}
The fact that both maps $g(u)$ and $g^{\rm adj}(u)$
are well-defined follows from Lemmas \ref{lem:HS:z}
and
\ref{lem:nw:l2:ests:f:g:h}, together with
Lemma \ref{lem:p:in:hk} and
the computation
\begin{equation}
\begin{array}{lcl}
    || Q g(u)^\top[\zeta] ||_{L^2_Q}
    & = & || Q^{1/2} g(u)^\top[\zeta] ||_{L^2(\mathcal D;\mathbb R^m)}
    \\[0.2cm]
    & \le & ||p||_{L^2(\mathcal D;\mathbb R^{m \times m})} || g(u)^\top[\zeta]||_{L^1(\mathcal D;\mathbb R^{m})}
    \\[0.2cm]
    & \le & ||p||_{L^2(\mathcal D;\mathbb R^{m \times m})} || g(u)^\top||_{L^2(\mathcal D;\mathbb R^{m \times n})} ||\zeta||_{L^2(\mathcal D;\mathbb R^{n})}.
\end{array}
\end{equation}
This allows us to introduce the $L^2_Q$-valued function
\begin{equation}
\label{eq:list:def:wt:k:c}
    \widetilde{\mathcal{K}}_C(u, \gamma) =
       \chi_l(u,\gamma)
         \chi_h(u,\gamma) Q g(u)^\top T_{\gamma} \psi_{\rm tw} ,
\end{equation}
together with the $L^2(\mathcal D;\mathbb R^n)$-valued function 
\begin{equation}
       \mathcal{K}_C(u,\gamma) = - \chi_h(u,\gamma) g(u) \widetilde{\mathcal{K}}_C(u, \gamma)
       =
       -\chi_h(u,\gamma)^2\chi_l(u,\gamma)g(u)[g^{\rm adj}(u) T_{\gamma}\psi_{\rm tw}].
\end{equation}
Note that a short computation shows 
\begin{equation}
    \|b(u,\gamma)\|_{HS(L^2_Q;\mathbb R)}^2=\chi_h(u,\gamma)^4\chi_l(u,\gamma)^2\langle g(u)g^{\rm adj}(u)T_\gamma\psi_{\rm tw},T_\gamma\psi_{\rm tw}\rangle_{L^2(\mathcal D;\mathbb R^n)},
\end{equation}
which provides a more explicit representation for \eqref{eq:list:def:kappa}.

Now taking\footnote{If (Hf-Lip) holds, it suffices to take $u \in \mathcal{U}_{L^2(\mathcal D;\mathbb R^n)}$. However \eqref{eq:app:alt:expr:for:a:sigma} is not necessarily well-defined in this case.}  $u \in \mathcal{U}_{H^{k+1}(\mathcal D;\mathbb R^n)}$, we are  ready to define the scalar function
\begin{equation}
\label{eq:list:def:a}
\begin{array}{lcl}
    a_\sigma(u,\gamma;c) &=&-\chi_l(u,\gamma)
    \Big[
      \langle f(u) +\sigma^2h(u), T_{\gamma} \psi_{\rm tw} \rangle_{L^2(\mathcal D;\mathbb R^n)}
      - \langle c u + \sigma^2 \mathcal{K}_C(u,\gamma), \partial_x \psi_{\rm tw} \rangle_{L^2(\mathcal D;\mathbb R^n)}
\\[0.2cm]
& & \qquad \qquad\qquad
      + \kappa_\sigma(u,\gamma) \langle u, T_{\gamma} \psi''_{\rm tw} \rangle_{L^2(\mathcal D;\mathbb R^n)}
    \Big];
\end{array}
\end{equation}
see Lemmas \ref{lem:a_sigma} and \ref{lem:nl:low:lip:theta:cub:l2}.
In addition, we define
the $H^k(\mathcal D;\mathbb R^n)$-valued function
\begin{align}\label{eq:list:def:j:sigma:new}
\mathcal J_\sigma(u,\gamma;c)&=
\kappa_\sigma(u,\gamma)^{-1}\left(f(u)+c\partial_x u+\sigma^2h(u)+\sigma^2[\partial_x \mathcal K_C(u,\gamma)]\right),\end{align}
where the well-posedness follows from
the bounds \eqref{eq:nl:g:h1:bnd}
and \eqref{eq:est:wtc:l2:bnds},
together with  Lemma
\ref{lem:nw:l2:ests:f:g:h}
and
Corollary \ref{cor:nl:low:bnds:xi:i:ii}.
Note that 
\begin{equation}
\label{eq:app:alt:expr:for:a:sigma}
\begin{array}{lcl}
    a_\sigma(u,\gamma;c) &=&    
    -\chi_l(u,\gamma)\kappa_\sigma(u,\gamma)\big(\langle \mathcal J_\sigma(u,\gamma;c),T_\gamma \psi_{\rm tw}\rangle_{L^2(\mathcal D;\mathbb R^n)}+\langle  u,T_\gamma \psi_{\rm tw}''\rangle_{L^2(\mathcal D;\mathbb R^n)}\big).
\end{array}
\end{equation}

Exploiting the translational invariance of our nonlinearities and the noise,
we obtain the commutation relations
\begin{equation}
\label{eq:app:list:comm:rels:f:g}
    T_\gamma f(u)=f(T_\gamma u),\qquad T_\gamma g(u)[\xi]=g(T_\gamma u)[T_\gamma \xi],\qquad T_\gamma g^{\rm adj}(u)[\zeta]=g^{\rm adj}(T_\gamma u)[T_\gamma\zeta].
\end{equation}
In particular, we see that
\begin{equation}
\label{eq:app:list:comm:rels:a:b}
     a_\sigma(u,\gamma;c)=a_\sigma(T_{-\gamma}u,0;c) ,
     \qquad  
     b(u,\gamma)[\xi]=b(T_{-\gamma}u,0)[T_{-\gamma}\xi],
\end{equation}
and similar identities hold for $\kappa_\sigma$, $\mathcal{J}_\sigma$ and the cut-off functions 
\eqref{eq:list:def:chi:h:l}.
This subsequently allows us to eliminate the dependence on $\gamma$ in the sequel.

Assuming (HPar) and taking $v \in H^{k+1}(\mathcal D;\mathbb R^n)$,
we introduce the expressions
\begin{equation}
\label{eq:app:def:R:sigma:S}
\begin{array}{lcl}
    \mathcal R_\sigma(v;c,\Phi) &= &
    \Delta_y v +
    \kappa_\sigma(\Phi +v , 0) 
    \big[ \partial_x^2 v + \Phi'' + \mathcal{J}_\sigma( \Phi + v, 0; c)
    \big] 
    +a_\sigma(\Phi +v, 0,c)\partial_x(\Phi+v),
\\[0.2cm]
\mathcal S(v;\Phi)[\xi]
&=&g(\Phi+v)[\xi]+\partial_x(\Phi+v)  b(\Phi+v,0)[\xi].
\end{array}
\end{equation}
We remark that $\mathcal S(v;\Phi)[\xi] \in H^k(\mathcal D;\mathbb R^n)$
and $\mathcal{R}_{\sigma}(v; c, \Phi) \in H^{k-1}(\mathcal D;\mathbb R^n)$,
and refer to {\S}\ref{sec:variational}
for the subtle interpretation of the latter.
By construction, we have $\langle \mathcal S(v;\Phi)[\xi] , \psi_{\rm tw} \rangle_{L^2(\mathcal D;\mathbb R^n)} = 0$ whenever \eqref{eq:app:bnd:on:v:l2}
is satisfied.

Finally, recalling the family $(\Phi_\sigma, c_\sigma)$ constructed in Proposition \ref{prop:wave},
we define the $H^k(\mathcal D;\mathbb R^n)$-valued expressions
\begin{equation}
\label{eq:app:def:n:m:sigma}
    \begin{array}{lcl}
    \mathcal{N}_\sigma(v)
     & = & 
\Phi_\sigma'' + \mathcal{J}_\sigma(\Phi_\sigma + v, 0;c_\sigma)
        - c_0 \partial_x v - Df(\Phi_0) v
        \\[0.2cm]
        & & \qquad \qquad
        + \kappa_\sigma(\Phi_\sigma + v;0)^{-1} a_\sigma(\Phi_\sigma + v,0,c_\sigma) \partial_x(\Phi_\sigma + v) ,
        \\[0.2cm]
        \mathcal{M}_{\sigma}(v)[\xi]
         & = & \kappa_{\sigma}( \Phi_\sigma + v, 0)^{-1/2}\mathcal{S}(v;\Phi_\sigma)[\xi],
    \end{array}
\end{equation}
for any $v \in H^{k+1}(\mathcal D;\mathbb R^n)$.
Indeed, note that these expressions no longer involve second derivatives of $v$.
In fact, upon introducing the intermediate 
function
\begin{equation}
\label{eq:app:def:n:i:sigma}
\begin{array}{lcl}
    \mathcal{N}_{I;\sigma}(v) 
    &=&  
    \Phi_\sigma'' + \mathcal{J}_\sigma(\Phi_\sigma + v, 0;c_\sigma)
        - c_0 \partial_x v - Df(\Phi_0) v
\\[0.2cm]
    &= & \Phi_\sigma''   + \mathcal{J}_\sigma(\Phi_\sigma + v,0;c_\sigma) - [\mathcal{L}_{\rm tw} - \partial_{xx} ] v,
\end{array}
\end{equation}
we may use the 
identity $\mathcal{L}_{\rm tw}^{\rm adj} \psi_{\rm tw} = 0$ 
to arrive at the convenient representation
\begin{equation}
\label{eq:app:repr:n:sigma:final}
 \mathcal{N}_\sigma(v)
  = \mathcal{N}_{I;\sigma}(v) 
  -\chi_l(\Phi_\sigma + v,0)
  \langle \mathcal{N}_{I;\sigma}(v) ,  \psi_{\rm tw} \rangle_{L^2(\mathcal D;\mathbb R^n)} [ \partial_x \Phi + \partial_x v].
\end{equation}
One now readily verifies that $\langle \mathcal{N}_\sigma(v), \psi_{\rm tw} \rangle_{L^2(\mathcal D;\mathbb R^n)} = 0$ 
holds whenever \eqref{eq:app:bnd:on:v:l2} is satisfied.


\section{Moment bounds and tail probabilities}
We briefly review here the technique of passing back and forth between moment estimates and tail probabilities. Similar results can be found in \cite[Sec. 2]{hamster2020expstability}; see also \cite{talagrand2005generic,veraar2011note}.

\begin{lemma}\label{lem:moment:to:tail}
   Consider a nonnegative random variable $X$. Suppose that  there exists two constants $\Theta_1>0$ and $\Theta_2>0$ so that the moment bound
\begin{equation}
\mathbb E\left[ X^{p}\right] \leq p^p \Theta_1^{ p}+\Theta_2^{p}
\end{equation}
holds for all integers $p \geq 1$. Then for every $\vartheta>0$ we have the estimate
\begin{equation}
\mathbb P(X\geq\vartheta) \leq 3\exp\left(\frac{\Theta_2}{2 e \Theta_1}\right)\exp \left(-\frac{\vartheta}{2 e \Theta_1}\right).
\end{equation}
\end{lemma}
\begin{proof}
The proof is similar to that of \cite[Lem 2.2]{hamster2020expstability}. Let $\lambda>0$ be arbitrary. Then by an exponential Markov inequality, we obtain
    \begin{equation}
    \begin{aligned}
        \mathbb P(X\geq\theta)&\leq e^{-\lambda\theta}\mathbb E [e^{\lambda X}]\\
        &=e^{-\lambda\theta}\sum_{p=0}^\infty\frac{\lambda^p}{p!} \mathbb E[X^p]\\&
        \leq e^{-\lambda\theta}\left[\sum_{p=0}^\infty \lambda^pe^p\Theta_1^{p}+\sum_{p=0}^\infty \frac{\lambda^p}{p!}\Theta_2^{2p}\right],
        \end{aligned}
    \end{equation}
    exploiting the identity $p!\geq p^p e^{-p}$ in the second inequality above.
    Upon choosing $\lambda=(2e\Theta_1)^{-1}$, we obtain the tail probability $
        \mathbb P(X\geq\theta) \leq  e^{-\lambda\theta}[2+e^{\lambda \Theta_2}]\leq 3e^{\lambda \Theta_2}e^{-\lambda\theta}$, which proves the assertion.
\end{proof}
 
\begin{lemma}\label{lem:tail:to:moment}
    Fix two constants $ A \geq 2$ and $\Theta_1>0$ and consider a nonnegative random variable $X$ that satisfies the estimate
\begin{equation}
\mathbb P(X\geq\vartheta) \leq 2  A \exp \left(-\frac{\vartheta}{2 e \Theta_1}\right)
\end{equation}
for all $\vartheta>0$. Then for any $p \geq 1$ we have the moment bound
\begin{equation}
\mathbb E\left[X^{p}\right] \leq\big(p^p+\log ( A)^p\big)(8 e \Theta_1)^p.
\end{equation}
\end{lemma}
\begin{proof} This follows directly from \cite[Lem. 2.3]{hamster2020expstability}.
\end{proof}

\begin{corollary}\label{cor:moment:tail}
    Consider $N \geq 2$ nonnegative random variables $X_1, X_2, \ldots, X_N$ and suppose that there exists two constants $\Theta_1>0$ and $\Theta_2>0$ so that the moment bound
\begin{equation}
\mathbb E[X_i^{ p}] \leq p^p \Theta_1^{p}+\Theta_2^{p}
\end{equation}
holds for all integers $p \geq 1$ and each $i \in\{1, \ldots, N\}$. Then for any $p \geq 1$ we have the maximal bound
\begin{equation}
\mathbb E \max _{i \in\{1, \ldots, N\}} X_i^{p} \leq\big(p^p+\log (N)^p+\Theta_2^{p}\Theta_1^{-p}\big)\left(24 e \Theta_1\right)^p.
\end{equation}
\end{corollary}

\begin{proof} For any $\vartheta>0$, we invoke Lemma \ref{lem:moment:to:tail} and the observation $(2e)^{-1}<1$ to obtain the estimate
\begin{equation}
\mathbb P\left(\max _{i \in\{1, \ldots, N\}} X_i\geq \vartheta\right) \leq \sum_{i=1}^N \mathbb P(X_i\geq \vartheta) \leq 3 N \exp(\Theta_2\Theta_1^{-1})\exp \left(-\frac{\vartheta}{2 e \Theta_1}\right).
\end{equation}
    The assertion follows by appealing to Lemma \ref{lem:tail:to:moment}. In particular,  we take $A=\frac32N\exp(\Theta_2 \Theta_1^{-1})$, use the inequality $(a+b+c)^p\leq 3^{p-1}(a^b+b^p+c^p)$  and note that $\log(\frac32)<1$.
\end{proof}

\section{Fourier analysis}\label{appendix:Fourier}




The Fourier transform can  be defined for any locally compact Abelian group \cite{bogachev2007measure,deitmar2014principles,friedman1982foundations,rudin2017fourier,scalamandre2017harmonic}, and thus in particular for $S\in\mathcal\{\mathbb T^d,\mathcal D\}$ with $\mathcal D=\mathbb R\times \mathbb T^{d-1}$.
The Pontryagin dual of $S$, denoted by $\widehat S$,  for the spaces  $S=\mathbb T^d$ and $S=\mathcal D$ are $\widehat S=\mathbb Z^n$ and $\widehat S=\mathbb R\times \mathbb Z^{d-1},$ respectively. 

Let $V$  be a seperable Banach space. 
For any function $u\in L^1(\mathbb T^d;V)\cap L^2(\mathbb T^d;V)$,
we define the Fourier transform to be
\begin{equation}
    \hat u(\xi)=\frac1{|\mathbb T|^d}\int_{\mathbb T^d}u(y)e^{-\frac{2\pi i}{|\mathbb T|} \langle \xi, y\rangle }\mathrm dy,\quad \xi\in\mathbb Z^d,
\end{equation}
while for any function $u\in L^1(\mathcal D;V)\cap L^2(\mathcal D;V)$,  we have as Fourier transform 
\begin{equation}
    \hat u(\omega,\xi)=\frac{1}{|\mathbb T|^{d-1}}
    \int_{\mathbb R\times \mathbb T^{d-1}} u(x,y)e^{- 2\pi i\omega x}e^{-\frac{2\pi i}{|\mathbb T|} \langle \xi, y\rangle }\mathrm dx\,\mathrm dy,\quad (\omega,\xi)\in\mathbb R\times \mathbb Z^{d-1}.
\end{equation} The mapping $\mathcal F:f\to\hat f$ extends to an isometric isomorphism from $L^2( S;V)$ to $L^2(\widehat{ S};V)$ for any choice of $S$.
In particular,  the inversion formula for $S=\mathcal D$ is given by \cite{kluvanek1970fourier,okada1985fourier}
\begin{equation}
u(x,y)=\sum_{\xi\in\mathbb Z^{d-1}}\int_{\mathbb R}\hat u(\omega,\xi)e^{2\pi i \omega x}e^{\frac{2\pi i}{|\mathbb T|} \langle \xi,y\rangle}\mathrm dx
\end{equation}
and Plancherel's identity  holds, i.e.,
\begin{equation}
\begin{aligned}
    \|f\|_{L^2(\mathcal D;V)}^2&=\int_{\mathcal D}\|f(\mathbf{x})\|_V^2\mathrm d\mathbf{x}=\int_{\mathbb R\times \mathbb T^{d-1}}\|f(x,y)\|_V^2\mathrm dx\,\mathrm dy\\&=\frac{1}{|\mathbb T|^{d-1}}\sum_{\xi\in\mathbb Z^{d-1}}\int_{\mathbb R}\|\hat f(\omega,\xi)\|_V^2\mathrm d\omega=\frac{1}{|\mathbb T|^{d-1}}\int_{\widehat{\mathcal D}}\|\hat f(\boldsymbol{\xi})\|_V^2\mathrm d\boldsymbol{\xi}=\frac{1}{|\mathbb T|^{d-1}}\|\hat f\|_{L^2(\widehat {\mathcal D};V)}.
\end{aligned}
\end{equation}
Here we have introduced the notation $\mathrm d\boldsymbol{\xi}=\mathrm d\omega \mathrm d\xi,$ where $\mathrm d\xi$ is the counting measure on $\mathbb Z^{d-1}.$
The factor $ 1/{|\mathbb T|^{d-1}}$ is a consequence of  not having normalised the Lebesgue induced measure on $\mathbb T^{d-1}$.

Recall that the Sobolev spaces $H^k(S;V)$ can be characterised by means of the Fourier transform  \cite{nau2012vector}. Indeed, an equivalent norm is given by 
\begin{equation}
    \vvvert u\vvvert_{H^k(S;V)}=\int_{\widehat S}(1+|\boldsymbol \xi|^2)^{k}\|\hat f(\boldsymbol \xi)\|_V^2\mathrm d\boldsymbol \xi,
\end{equation}
where $\boldsymbol{\xi}\in \widehat S$.
This equivalence follows readily from the fact  that $\widehat {\partial^\alpha f}=\xi^\alpha \hat f$ holds and by exploiting Plancherel's identity. In a similar fashion, Parseval's identity holds, which yields an inner product on $H^k(S;V) $ in terms of the Fourier transform. The norm also gives  the inner product by polarisation.



\printbibliography

\end{document}